%% file: cosystolic_full.tex
\documentclass[11pt]{report}
\PassOptionsToPackage{numbers, compress}{natbib}
\usepackage{amsmath,amssymb,amsthm,fullpage,mathrsfs,pgf,tikz,caption,subcaption,mathtools}
\usepackage[ruled,vlined,linesnumbered]{algorithm2e}
\usepackage{natbib}
\usepackage{amsmath,amssymb,amsthm,mathtools}
\usepackage[utf8]{inputenc} 
\usepackage[T1]{fontenc}    
\usepackage{hyperref}       
\usepackage{url}            
\usepackage{booktabs}       
\usepackage{amsfonts}       
\usepackage{nicefrac}       
\usepackage{microtype}      
\usepackage[margin=1in]{geometry}
\usepackage{bbm}
\usepackage{enumitem}
\usepackage{xcolor}
\usepackage{cleveref}

\usepackage[T1]{fontenc}
\usepackage{lmodern}

\usepackage{chngcntr} 

\usepackage{tocbibind} 

\input{amsart_header_18_05}

\usepackage{enumitem} 

\counterwithout{section}{chapter}

\numberwithin{equation}{section} 

\makeatletter
\renewcommand{\paragraph}{%
  \@startsection{paragraph}{4}%
  {\z@}{1.25ex \@plus 1ex \@minus .2ex}{-1em}%
  {\normalfont\normalsize\bfseries}%
}
\makeatother

\DeclareMathOperator{\res}{res} 

\DeclareMathOperator{\Sh}{Sh}
\newcommand{\from}{\leftarrow}

\newcommand{\aug}[1]{{#1}_+} 

\newcommand{\Ham}{{\mathrm{Ham}}} 

\newcommand{\schoose}[2]{{\textstyle {{#1} \choose {#2}}}} 

\newcommand{\ws}[1]{\|{#1}\|} 
\newcommand{\nws}[1]{\|{#1}\|_{\Ham}} 
 
\newcommand{\cupp}{\cup} 

\newcommand{\lMod}[1]{{#1}\textrm{-}{\mathrm{Mod}}}

\newcommand{\topl}[1]{{#1}^\circ} 
\newcommand{\smp}[1]{{#1}^\triangle} 

\newcommand{\Comm}[1]{{{#1}\textrm{-}\mathsf{Alg}}}
\newcommand{\catGrp}{{\mathsf{Grp}}}

\begin{document}
\title{
On Good $2$-Query Locally Testable Codes from Sheaves on High Dimensional Expanders
}

\date{}

\author{Uriya A.\ First \\
Department of Mathematics\\
University of Haifa
\and
Tali Kaufman\\
Computer Science Department\\ 
Bar-Ilan University}


\maketitle

\begin{abstract}

We expose a strong connection between good $2$-query locally testable codes (LTCs) and high dimensional expanders. 
Here,
an LTC is called  good if it has constant rate and linear distance. Our emphasis in this work is on LTCs testable  with only $2$ queries. These are are harder to construct than general LTCs, and are of particular interest to theoretical computer science. 

The connection we make between $2$-query LTCs and high dimensional expanders is done by introducing a new object called a {\em sheaf} that is put on top of a high dimensional expander. Sheaves are vastly studied in topology. Here, we introduce sheaves on simplicial complexes. Moreover, we define a notion of an {\em expanding sheaf} that has not been studied before. 

We present a framework to get good infinite families of $2$-query LTCs from expanding sheaves on high dimensional expanders, utilizing towers of coverings of these high dimensional expanders. 
Starting with a high dimensional expander and an expanding sheaf, 
our framework produces an infinite 
family of codes admitting a $2$-query tester.
If the initial sheaved high dimensional expander
satisfies some  conditions,
which can be checked in constant time,
then these codes form a family of good $2$-query LTCs.

We  give candidates for sheaved high dimensional expanders
which can be fed into our framework, in the form of an iterative process
which (conjecturally) produces such candidates given a high dimensional
expander and a special auxiliary sheaf. (We could not
verify the prerequisites of our framework for these candidates 
directly because of computational limitations.)
We analyze this process experimentally
and heuristically, and identify some  properties 
of the fundamental group of the high dimensional expander at hand
which
are sufficient (but not necessary) to get the desired sheaf, and consequently
an  infinite family of good $2$-query LTCs.

\end{abstract}


\setcounter{page}{1}
\newpage

\tableofcontents
\newpage

\section{Introduction}
\label{sec:intro}


\paragraph{Locally Testable Codes.}

A locally testable code (LTC) is an error correcting code
admitting a randomized algorithm --- called a \emph{tester} --- which, given access to a word,
can decide with high probability
whether it is close to a codeword or not
by querying just a few 
(i.e.\ $O(1)$) of its letters.
More formally, the tester must accept all codewords,
and the probability of rejecting a word outside the code
is at least proportional to its Hamming distance from the code.
Thus, upon transmitting a codeword along a noisy channel, the receiver can  
probe just a few letters to assess whether the codeword
was significantly corrupted or not, and decide accordingly whether to 
decode it, or ask   for retransmission.
The probability of rejecting a   word which is very far from the code
(i.e.\ of relative Hamming distance $\geq \eta$)
is the called ($\eta$-)soundness of the LTC.

LTCs also play a major role in the construction of  probabilistically
checkable proofs (PCPs), as   almost all known PCPs include them as building blocks.
The length of such PCPs is related to various properties
of the LTC, e.g., its distance, rate and the efficiency of the the testing;
see 
\cite{Goldreich_2011_locally_testable_codes_proofs} for a survey.

A family of LTCs is called \emph{good} if the codes in that family have constant rate and linear distance.

\paragraph*{$\mathbf{2}$-Query LTCs.}

A subclass of  LTCs of particular importance is the $2$-query LTCs,
i.e., LTCs admitting a tester probing just two letters;
the alphabet size may be large.
Such LTCs admit an even stronger connection to PCPs, and also to the Unique Games Conjecture (UGC).
However, they are also known to be 
somewhat constrained.
Indeed, by  \cite{BenSasson_2003_bounds_2query_codes},
there are no $2$-query LTCs with linear distance
and constant rate on a binary alphabet, and likewise
for \emph{linear}  
$2$-query LTCs on any finite field alphabet. 
See \cite{Kol_2016_bounds_two_query_LTCs,KolR12} for further restrictions.

\paragraph*{Some Notable Constructions of LTCs.}
LTCs are generally difficult to construct for the reason that
random low-density parity check (LDPC) codes  are usually not locally
testable. Rather, the parity checks should be designed
to admit   redundancy. $2$-query LTCs are even harder to come by.

Some notable examples of LTCs  include Reed--Muller codes, which have
linear distance and polylogarithmic message length 
\cite{Friedl_1995_total_degree_tests}, \cite{Rubinfeld_1996_characterizations_of_polys},
and the LTCs of Ben-Sasson--Sudan \cite{BenSasson_2008_short_PCPs_polylog_query}
and Dinur \cite{Dinur_2007_PCP_theorem_gap_amplification},
which have linear distance and inverse poly-logarithmic  rate.
It was further studied how the group under which a code is invariant affects its local testability. With respect to that, it was shown that affine invariant codes with the so called ``single orbit'' property are locally testable \cite{KaufmanS08-affine-LTC}. 
Dinur, Evra, Livne, Lubotzky and Mozes \cite{Dinur_2021_locally_testable_codes_preprint}
and Panteleev and Kalachev \cite{Panteleev_2021_good_quantum_LTCs}
have recently and independently of each other constructed  infinite families of LTCs with 
constant (large) query size,   linear distance and constant rate.\footnote{
	An earlier version of this paper  was written and circulated independently of these works.
}
(Panteleev and Kalachev also constructed good LDPC quantum codes in {\it op.\ cit.})

\paragraph*{Local-Testability Follows From an Underlying High-Dimensional Expander.}

The past decade had seen an emerging trend of using \emph{high-dimensional expanders} for
constructing locally testable codes and other property testers,
e.g., 
\cite{Kaufman_2014_high_dim_expanders_property_testing},
\cite{Kaufman_2016_isoperimetic_inequalities}, 
\cite{Evra_2016_cosystolic_expanders_arxiv_version}, 
\cite{Dikstein-LTC},
\cite{Kaufman_2021_amplified_local_testability_preprint}
to name just a few. 
Loosely speaking, these
works share a common theme: one uses a $2$-dimensional object, e.g., a $2$-dimensional simplicial
complex, or a $3$-layer partially ordered set, in order
to define a code. One then relates   expansion properties of the object
at hand to the testability of the code, making use of the $3$ layers of the object
(vertices, edges and triangles in the case of a simplicial complex).
See \cite{Kaufman_2021_amplified_local_testability_preprint} for an aximatization 
of this approach.

Despite the extensive research, so far, 
the prototypical high dimensional expanders did
not give rise to LTCs with linear distance and constant rate. 
It was particularly expected that 
the \emph{Ramanujan complexes} of Lubozky, Samuels and Vishne 
\cite{Lubotzky_2005_explicit_constructions_of_Ramanujan_complexes}
(see also Li \cite{Li_2004_ramanujan_hypergraphs}  and Sarveniazi \cite{Sarveniazi_2007_explicit_Ramanujan_hypergraphs}),
which are often considered  as the 
prototype of   high dimensional expanders, should give rise to LTCs.
For comparison, the recent LTCs constructed in \cite{Dinur_2021_locally_testable_codes_preprint} and 
\cite{Panteleev_2021_good_quantum_LTCs}
use special square complexes, which do not seem to admit higher-dimensional
analogues. 

\subsection{Main Contributions}

This work concerns with the construction of $2$-query
LTCs 
basing on high-dimensional expanders, e.g.\ Ramanujan complexes.
Our contributions are the following.


\paragraph{A Framework for Constructing Good 2-Query LTCs from Expanding Sheaves on High Dimensional Expanders.} We present a general framework ---
called the \emph{tower paradigm} --- for constructing 
$2$-query LTCs
from high dimensional expanders, e.g.\ Ramanujan complexes, by introducing a new piece of data: a \emph{sheaf} on the 
expander at hand.  

In more detail, our framework takes as input (constant sized)  initial data consisting of a ``small'' 
high dimensional expander and a sheaf. We also assume that the ``small'' high dimensional
expander admits an infinite family of \emph{coverings}, which is the case for many
known high-dimensional expanders.
Using the constant-sized initial data and the coverings, we construct an infinite family of codes (with length
tending to $\infty$) admitting a natural $2$-query tester.
We  then show that if the constant sized initial data satisfies a list of conditions,
which can be verified by a finite (constant sized) computation,
then the entire infinite family of codes is a family of $2$-query LTCs with linear distance and constant rate;
see  \S\ref{subsec:tower-paradigm}
for more details and Theorem~\ref{TH:tower-paradigm} for a precise statement.

This result consists of two components of independent interest.
The first is  a new \emph{local-to-global principle} 
which allows us to show that a $2$-query code arising from an \emph{expanding sheaf} on a high-dimensional expander
is locally testable and has linear distance by means of local conditions (Theorem~\ref{TH:cosystolic-exp-from-links}, 
Corollary~\ref{CR:cse-from-links-dim-2}, Remark~\ref{RM:main-ineqs-for-2dim-case-with-Oppenheim}); 
this is a vast generalization of  
\cite{Kaufman_2016_isoperimetic_inequalities},
\cite{Evra_2016_cosystolic_expanders_arxiv_version}, 
\cite{Kaufman_2018_cosystolic_expanders}  
which moreover works under milder assumptions.
The second is a \emph{rate conservation} method (Theorem~\ref{TH:rate-conservation}), used to maintain a constant 
rate among the infinite family of codes we construct.

\paragraph{Examples of $\mathbf2$-Query LTCs With Linear Distance and Conjectural Constant Rate.} 

We construct candidates for the constant-sized initial data 
--- consisting of a   (constant sized) expander and an expanding sheaf ---
required for our framework. 
To that end, we first construct  
sheaved high-dimensional expanders   fulfilling 
the   conditions guaranteeing testability and linear distance.
Then, we present an iterative process which takes such a sheaf and
modifies it to create a new sheaf which,
conjecturally, also satisfies the conditions guaranteeing a constant rate.
This gives  rise to explicit infinite
families   of $2$-query LTCs with linear distance and conjectural constant rate.
See \S\ref{subsec:explicit-example}
for an explicit example of how such a family of codes might look like, and
Theorem~\ref{TH:main-candidates} 
and Remark~\ref{RM:main-candidates} for  precise statements.
While verifying that each such family has constant rate could be done in a finite
(constant sized) computation involving the initial data, doing so is presently not
possible due 
to computational limitations. 


The iterative process   works by artificially
creating or eliminating \emph{cohomology classes}
(in the cohomology of the sheaf); see \S\ref{subsec:quotient}.
We analysed it using computer simulations, and 
identified conditions involving the fundamental group of the expander and the sheaf
to modified
which, once met,   guarantee that the process outputs a modified sheaf satisfying all the requirements (Conjecture~\ref{CJ:dimension-of-E2}).
We also justify these    conditions with a heuristic theoretical argument.
They are not necessary for the success of the process, though.

We remark that our $2$-query LTCs do not violate the restrictions
proved in \cite{BenSasson_2003_bounds_2query_codes} since they are not linear
and use a very large alphabet $\Sigma=\F_2^m$.
If we treat   each letter in the alphabet $\Sigma$ as an $m$-letter string in $\F_2$,
then they  become linear codes over the alphabet $\F_2$. 
We further note that the soundness of the codes does not depend on the alphabet size.
In addition, our LTCs
are not \emph{lifted codes}, in contrast to the presently known LTCs. 

\paragraph{Some Implications to Quantum Codes.} 
Our framework 
can also be used to construct infinite families of other low-query LTCs (non-linear, with large alphabet)
and (linear) \emph{quantum CSS codes} whose $X$-side is locally testable and has linear distance.
Our rate conservation method applies in these contexts, but securing the required conditions on the initial data is still out of reach.

\subsection{Conceptual and Methodological Contributions}


\paragraph{Exposing a Connection Between $\mathbf 2$-Query LTCs and  High Dimensional Expansion.} 
Our results 
expose a strong relation between $2$-query LTCs and high dimensional expanders. Connections between general local testability and high-dimensional expansion was studied previously, but in this work, we show that high dimensional expansion is moreover related to a stronger notion of local testability, namely, to $2$-query locally testable codes. 
This connection was already glimpsed on in the work
of
the second author 
and Lubotzky \cite{Kaufman_2014_high_dim_expanders_property_testing}
under the broader connection between high dimensional expansion (coboundary expansion) and general local testability.
Alas,  the only $2$-query LTC   obtained in that work {had}   to have  two codewords (a sting of $1$s and a string of $0$s), and so it was not regarded as a true LTC.
In this work we show that by using sheaves, we can (conjecturally) get good $2$-query LTCs from high dimensional expanders, namely, the intrinsic barrier that high dimensional expanders can not give $2$-LTCs with satisfactory rate is overcome. 

\paragraph{Introducing Sheaves on Simplicial Complexes and Expanding Sheaves.}
Loosely speaking, a sheaf is a layer of linear algebra data that is put on top of a simplicial complex. Sheaves are vastly studied in topology and algebraic geometry. Here, we introduce a discrete 
variation of the topological definition: sheaves on simplicial complexes. Moreover, we define a notion of an {\em expanding sheaf} that has not been studied before in topology, nor elsewhere. 

\paragraph{Utilizing Coverings of High Dimensional Expanders as a Way to  Reduce Obstructions
and Getting New Examples.} 
We use \emph{coverings} of high dimensional expanders,
and more specifically \emph{towers of coverings}, 
both in our framework for getting good $2$-query LTCs
and in finding initial data to feed into the framework. 

When establishing the tower paradigm, we use coverings to   obtain
new examples of expanding sheaves from existing ones, generating
infinitely many examples from a single base example. 
Specifically, given a ``big'' simplicial complex covering a ``small'' one
and a sheaf on the small complex, we can construct a sheaf on
the big complex by \emph{pulling back} the sheaf on the small complex along the covering map.
Pullback of sheaves  is a well-known construction in topology. At the basis of our framework
lies the observation that the pullback of a sheaf inherits many properties, e.g., local  expansion conditions, 
from the original sheaf.

Our second use of coverings is in applying our framework, as they allows us to reduce
obstructions. In more detail, the conditions on the initial data for our framework
which guarantee constant rate depend on the dimension of the first \emph{cohomology space}.
Loosely speaking, the larger it is, the further away we are from satisfying these conditions.
We use coverings together with the  \emph{pushforward} construction from topology 
to create sheaves of dimension that is significantly larger
than the dimension of the obstructing cohomology space.
For certain coverings of high dimensional expanders arising form number theory, 
this approach results in expanding sheaves of arbitrarily large dimension,
but such that the obstruction to rate conservation 
remains constant,
ultimately becoming negligible in dimension to the sheaf. It is those sheaves that we feed into our iterative
process, which (conjecturally) eliminates the relatively small obstruction.

\subsection{Acknowledgements}

We are grateful to Shmuel Weinberger for some useful conversations.

The computer simulations discussed in this work 
were written by the first author in {\tt Python},
and 
were executed on the Hive computer cluster at the University of Haifa, which is partly funded by ISF grant 2155/15.
We thank Eyal Confeld for writing earlier simulations    in {\tt Mathematica}.

Tali Kaufman's research is supported by an ERC grant and an ISF grant.

\section{Overview of The Main Results}
\label{sec:overview}


We now survey the main results of the paper.
In  \S\ref{subsec:hdx}, we give some relevant background on high dimensional expanders. 
We then introduce {\em sheaves on simplicial complexes}
in \S\ref{subsec:intro-sheaves}, and
the important notion of {\em expanding sheaves} in \S\ref{subsec:expanding-sheaves}.
Next, in \S\ref{subsec:coverings-intro}, we discuss   coverings of high dimensional expanders and highlight their important role in our framework for obtaining good $2$-query LTCs from expanding sheaves on high dimensional expanders. The framework itself,
called the \emph{tower paradigm},   is then
presented in  \S\ref{subsec:tower-paradigm}.
In \S\ref{subsec:finding-int-data-intro}, we present a method
which conjecturally produces   the constant-sized initial data (consisting
of a sheaved high-dimensional expander) required for our framework; here
we use coverings once more. 
Lastly, in \S\ref{subsec:explicit-example}, we give an explicit  
family of $2$-query LTCs
which arises from our framework\gap{}.
This family has linear distance and we conjecture that,
for an appropriate choice of parameters, it has constant rate.

The outline of the paper is given in \S\ref{subsec:paper-organization}.

\subsection{High Dimensional Expanders}
\label{subsec:hdx}

Of the various flavors of high dimensional expansion 
which have emerged in the past two decades --- all of which generalize expansion in graphs ---
the two   most   relevant
for our purpose are \emph{coboundary expansion} and \emph{cosystolic expansion}. 

\paragraph*{Some History.}
Coboundary expansion  originated in
the works of
Linial--Meshulam \cite{Linial_2006_homological_connectivity}
and Meshulam--Wallach \cite{Meshulam_2009_homological_connectivity}
on the cohomology of random simplicial complexes, and the work
of Gromov \cite{Gromov_2010_expanders_and_top_II} on the minimum amount of overlapping
forced by mapping
a simplicial complex to $\R^n$. Cosystolic expansion  is
a more relaxed version of coboundary expansion
developed in   \cite{Dotterrer_2018_topological_overlap},
\cite{Kaufman_2016_isoperimetic_inequalities}, \cite{Evra_2016_cosystolic_expanders_arxiv_version}
in order to extend the reach of Gromov's methods.

The first connections between high dimensional expansion and property
testing were observed and studied in \cite{Kaufman_2014_high_dim_expanders_property_testing}.

\paragraph*{Cochains, Cocycles and Coboundaries.}

Let $X$ be a simplicial complex.\footnote{
	All simplicial complexes are assumed to be finite, unless indicated otherwise.
} We write $X(i)$ to denote the set of $i$-dimensional faces
of $X$, e.g., $X(0)$, $X(1)$ and $X(2)$ stand for the vertices, edges and triangles of $X$, respectively.
The simplicial complex $X$ also has a single empty face, of dimension $-1$.

Let $i\in \N\cup\{ -1,0\}$.
Recall that an \emph{$i$-cochain} on $X$ with coefficients in $\F_2$
is an assignment of an element in $\F_2$ to each $i$-face of $X$, i.e., a vector $f\in \F_2^{X(i)}$.
We set $C^i=C^i(X,\F_2)=\F_2^{X(i)}$ and 
write the $x$-coordinate of $f\in C^i$ as $f(x)$.  
As usual, the \emph{$i$-th coboundary map} $d_i:C^i\to C^{i+1}$ is defined by
\begin{equation}\label{EQ:classical-coboundary}
(d_if)(y)=\sum_{\text{$x$ is an $i$-face of $y$}} f(x) 
\end{equation}
for all $f\in C^i$, $y\in X(i+1)$.
A standard computation shows that $d_{i+1}\circ d_i=0$.
The spaces of \emph{$i$-coboundaries} and \emph{$i$-cocycles} are now defined as
\begin{equation}\label{EQ:Bi-Zi-def}
B^i=B^i(X,\F_2)=\im d_{i-1}
\qquad\text{and} 
\qquad
Z^i=Z^i(X,\F_2)=\ker d_i,
\end{equation}
respectively,
where $d_{-2}=0$ by convention.
We have $B^i\subseteq Z^i\subseteq C^i$ because $d_i\circ d_{i-1}=0$, 
and the quotient space $Z^i/B^i$ is the  \emph{$i$-th cohomology}
space $\HH^i(X,\F_2)$.\footnote{
	What we have defined here is the \emph{reduced} cohomology of $X$ with $\F_2$-coeffients,  
	denoted $\tilde{\HH}^i(X,\F_2)$ later in the paper and elsewhere. The ordinary, non-reduced,
	cohomology is defined in the same manner with the difference
	that the empty face of $X$ is ignored, i.e., one sets $C^{-1}=0$ and $d_{-1}=0$.
}

\paragraph*{Expansion of Cochains is a Form of Local Testability of The Cocycle Code.}

Given a simplicial complex $X$,
we may regard the $i$-cocycles $Z^i=Z^i(X,\F_2)$ as a linear code inside $C^i=C^i(X,\F_2)=\F_2^{X(i)}$.
This code is called the \emph{$i$-cocycle code}, and it admits a natural $(i+2)$-query tester: given $f\in C^i$, choose a face $y\in X(i+1)$ uniformly at random and
accept $f$ if 
$(d_if)(y)=0$ (cf.\ \eqref{EQ:classical-coboundary}).
By definition,\footnote{
	Definitions concerning codes are recalled in Section~\ref{sec:ltc-and-css}.
} this tester makes $Z^i$ into an 
an \emph{$\veps$-testable code}  inside $C^i$  if and only if
\begin{equation}\label{EQ:cbe-F2-first}
\frac{\|d_i f\|_{\Ham}}{d_{\Ham}(f,Z^i)}\geq \veps
\qquad\forall \,f\in C^i-Z^i, 
\end{equation}
where $\|\cdot\|_{\Ham}$ and $d_{\Ham}$ denote the normalized Hamming norm and distance
(in $\F_2^{X(i)}$ or $\F_2^{X(i+1)}$), respectively.
As for the distance of $Z^i$, since $B^i$ typically
contains vectors with small support (unless $i=0$), the best we could for 
is the existence of $\delta>0$ such that
\begin{equation}\label{EQ:cbe-F2-second}
\|g\|_{\Ham}\geq \delta \qquad\forall\,g\in Z^i-B^i. 
\end{equation}

Conditions  \eqref{EQ:cbe-F2-first} and~\eqref{EQ:cbe-F2-second} can also 
be viewed as statements concerning the expansion of $i$-cochains under $d_i$.
When both of these conditions hold, $X$ is said to be 
an \emph{$(\veps,\delta)$-cosystolic expander} in dimension $i$.\footnote{
	Warning: The cosystolic expansion considered
	later in this work and also in other sources   involves weights on the faces of $X$,
	which we have suppressed here for simplicity. See \S\ref{subsec:coboundary-exp} for details.
}
If moreover $Z^i=B^i$ (equivalently $\HH^i(X,\F_2)=0$),
then   $X$ is said to be an \emph{$\veps$-coboundary expander} in dimension $i$
(the parameter $\delta$ plays no role as $Z^i-B^i=\emptyset$).
Thus, $X$ is an $\veps$-coboundary expander if and only if the code $B^i\subseteq C^i$
is $\veps$-testable with respect to the natural tester.

\paragraph*{Expansion in Dimension 0: The Case of Graphs.}
Since the $0$-cocycle code $Z^0$ of a simplicial complex $X$
is determined  by its underlying graph, we might as well assume that $X$ is a graph.
In this case, if $f\in C^0(X,\F_2)$ has
support $A\subseteq X(0)$, then the support of $d_0 f$ is precisely
the set of edges leaving $A$. Note further
that $B^0\subseteq\F_2^{X(0)}$ consists of exactly two vectors, namely $(0,\dots,0)$ and $(1,\dots,1)$.
Consequently, $X$ is an   $\veps$-coboundary expander
in dimension $0$ if and only if $X$ is an $\veps$-expander in usual sense, i.e.,
\begin{equation}\label{EQ:graph-expansion}
\frac{|E(A,X(0)-A)|}{\min\{|A|,|X(0)-A|\}}\geq \veps\frac{|X(1)|}{|X(0)|}\qquad\forall\,\emptyset\neq A\subsetneq X(0).
\end{equation}

Similarly, a graph $X$ is an $(\veps,\delta)$-cosystolic expander in dimension $0$ if and only
if each connected component of $X$ is an $\veps$-expander
consisting of at least $\delta$-fraction of the vertices in $X$.

Shifting back our point of view to codes,
we also note
that   $X$ is
an $(\veps,\delta)$-cosystolic expander if and only if the code $Z^0\subseteq \F_2^{X(0)}$ is $\veps$-testable
with respect to its natural $2$-query tester
and has    relative distance $\geq \delta$.
On the other hand, the message length of $Z^0$ is meager --- it is the number of connected components of $X$,
thus bounded from above by $\frac{1}{\delta}$.

\paragraph*{Expansion in Higher Dimensions.}

In contrast to the case of $0$-cocycle codes, if $i>0$, then the code $Z^i=Z^i(X,\F_2)$ 
may have constant rate, but its distance
is typically small, becuase $B^i$ usually contains   vectors of small support. However, in this case, the 
code $Z^i$ can be enriched into a \emph{quantum CSS code}.
Moreover, if   $X$ is an $(\veps,\delta)$-cosystolic expander in dimension $i$,
then the $X$-side of this quantum CSS code
is $\veps$-testable and has relative
distane $\geq \delta$; see \cite{Evra_2016_cosystolic_expanders_arxiv_version}, or \S\ref{subsec:css} for a generalization.

\paragraph*{Intrinsic Barrier to Good `Ordinary' Cocycle Codes.}
The last two paragraphs demonstrate an intrinsic  difficulty in trying
to construct good LTCs from high dimensional expanders: either the rate or the distance are small. We will see below that \emph{sheaves} allow us to bypass this natural limitation. 

\medskip


Before  we move to present sheaves, we recall an important method for obtaining ``global'' local testability from ``local'' local testability in the context
of cocycle codes coming from high dimensional expanders. A generalization
of this method to sheaves   will   play a major role in our new framework for constructing good 2-LTCs from high dimensional expanders.

\paragraph*{Local Local-Testability Implies Global Local-Testablity.}

We have seen that the 
testability of the $0$-cocycle code of $X$, i.e., the
cosystolic expansion of  $X$ in dimension
$0$, is determined directly by the expansion of its underlying graph.

For  higher dimensions, the most prominent method for proving that a simplicial complex
$X$ is a good cosystolic expander in a desired dimension  $i\in\{1,\dots,\dim X-2\}$ 
is a \emph{local-to-global
principle} established in \cite{Kaufman_2016_isoperimetic_inequalities} for
$i=1$ and \cite{Evra_2016_cosystolic_expanders_arxiv_version} in general.

Recall that the \emph{link} of a simplicial complex $X$ at a face $z\in X$
is $X_z=\{x-z\where z\subseteq x\in X\}$. If $z\neq\emptyset$, the link $X_z$ is
called a \emph{proper link} of $X$. 
The main result of \cite{Evra_2016_cosystolic_expanders_arxiv_version} states
that if each of the proper links $X_z$ is a  
good coboundary expander in a range of dimensions, and if the underlying
graph of $X$ is a sufficiently good expander, then $X$ is an $(\veps,\delta)$-coboundary
expander in dimension $i$  with $\veps$, $\delta$ depending
on the implicit expansion constants.
In fact, by Oppenheim's Trickling Down Theorem \cite[Theorem~1.4]{Oppenheim_2015_vanishing_of_cohomology},
we can replace the expansion condition on the underlying graph of $X$
with requiring that $X$ is connected  and all its proper links are sufficiently good coboundary expanders
in dimension $0$.
The main theorem of \cite{Evra_2016_cosystolic_expanders_arxiv_version} can therefore
be summarized as: good coboundary expansion at the links (informally called ``local'' local-testablity) implies  cosystolic expansion (informally called ``global'' local-testability).
Among our main results is a  generalization of this principle to \emph{sheaves}.

\subsection{Sheaves on Simplicial Complexes}
\label{subsec:intro-sheaves}

Loosely speaking,
a sheaf is a layer of linear-algebra data put on top of a simplicial complex.

We comment about the relation between our sheaves and related notions,
e.g.,  sheaves on topological spaces, Jordan and Livne's local
systems \cite[\S2]{Jordan_1997_Ramanujan_local_systems} and Friedman's sheaves on graphs
\cite{Friedman_2015_sheaves_on_graphs}
at the end.

\paragraph*{Sheaves on Graphs.}
Let $\F$ be a field. 
An $\F$-sheaf on  a graph $X$ consists of 
\begin{enumerate}[label=(\arabic*)]
	\item an $\F $-vector space $\calF(x)$ for every $x\in X(0)\cup X(1)$, and
	\item a  linear map $\res^{\calF}_{e\from u}:\calF(u)\to \calF(e)$
	for every edge $e \in X(1)$ and vertex $u\in X(0)$ with $u\subseteq e$.
\end{enumerate}
The maps $\res^{\calF}_{e\from u}$ are called \emph{restriction maps}.

\paragraph*{Examples.}
The most basic example of an $\F$-sheaf on a graph $X$ is obtained by taking $\calF(x)=\F$
for all $x\in X(0)\cup X(1)$ and setting all the restriction maps to be $\id_{\F}$.

\label{intro-ex:sheaf-exp-codes}
A more interesting  example that will be revisited later  can be constructed as follows:
Let $X$ be a $k$-regular graph. 
Given $v\in X(0)$, we write $E(v)$ to denote the set of edges containing $v$. 
For every   $v\in X(0)$,
choose  $n(v)\in \{0,1,\dots,k\}$
and an injective linear transformation $T_v:\F^{n(v)}\to \F^{E(v)}\cong \F^k$,
and write $C_v=\im T_v$. Using this data, we define an $\F$-sheaf on $X$ by setting
\begin{itemize}
	\item $\calF(v)=\F^{n(v)}$ for each vertex $v\in X(0)$,
	\item $\calF(e)=\F$ for each edge $e\in X(1)$, and
	\item $\res^{\calF}_{e\from v}=\mathrm{Proj}_e\circ T_v$ for every
	edge $e\in  X(1)$
	and vertex $v\subseteq e$, where $\mathrm{Proj}_e:\F^{E(v)}\to \F$ is the projection
	onto the $e$-coordinate.
\end{itemize}
We shall see below (\S\ref{subsec:expanding-sheaves}) 
that if all the $C_v$  are good codes, then
this example gives rise to a good \emph{$0$-cocycle code}.
In fact, this  
is a sheaf-theoretic variation on the famous expander codes of Sipser and Spielman
\cite{Sipser_1996_expander_codes}; their presentation
in \cite{Meshulam_2018_graph_codes_arxiv_version} demonstrates the similarity.

\paragraph*{Sheaves on Simplicial Complexes.}
Sheaves on simplicial complexes are defined in the same manner as sheaves on graphs,
with the difference that one needs to impose an extra assumption.
Formally, an
$\F $-sheaf $\calF$ on $X$ consists of 
\begin{enumerate}[label=(\arabic*)]
	\item an $\F $-vector space $\calF(x)$ for every nonemtpy face $x\in X$, and
	\item a  linear map $\res^{\calF}_{y\from x}:\calF(x)\to \calF(y)$
	for every pair of nonempty faces $x,y\in X$ with $x\subsetneq y$,
\end{enumerate}
subject to the requirement
$
\res^{\calF}_{z\from y}\circ \res^{\calF}_{y\from x}=\res^{\calF}_{z\from x}
$
whenever $x\subsetneq y\subsetneq z$.
This requirement is vacuous if $X$
is a graph, but is very restrictive if $\dim X>1$.
Figure~\ref{FG:sheaf} illustrates of the data of a sheaf on a $2$-dimensional simplicial
complex, the arrows representing the restriction maps; the extra requirement 
means that the diagram of vector spaces on the right commutes.

\begin{figure}[ht]
\caption{A simplicial complex $X$ (left) and the data of a sheaf $\calF$ on $X$ (right).}
\begin{center}
\includegraphics[height=3.5cm]{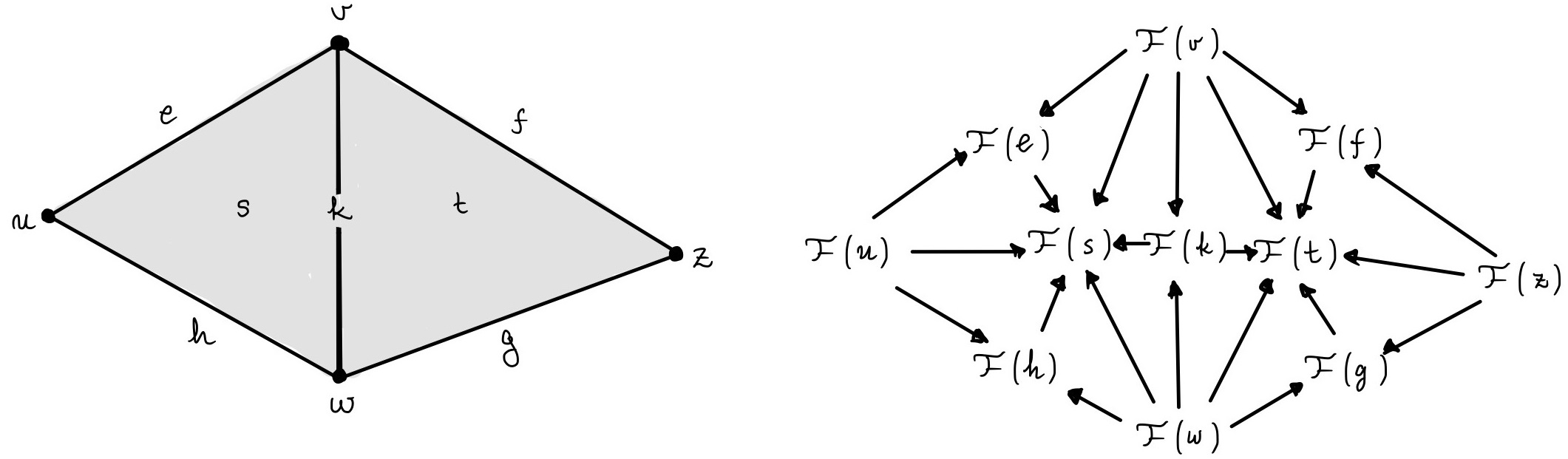}
\end{center}
\label{FG:sheaf}
\end{figure}

For the sake of simplicity, we henceforth  consider only $\F_2$-sheaves, and call them sheaves
for brevity. 

As with graphs, one can fix an $\F_2$-vector space $V$
and define    a sheaf $\calF_V$ on $X$ 
(denoted $V_X$
later on) 
by  taking $\calF_{V}(x)=V$
for every face $x\in X-\{\emptyset\}$
and setting all the restriction maps to be the identity. Such sheaves
are called \emph{constant}. More sophisticated examples will be considered (and needed) below.


\paragraph*{Augmented Sheaves.}

It is sometimes convenient to modify the definition
of a sheaf on $X$ by requiring that $\calF(x)$ and $\res_{y\from x}^\calF$
are also defined when $x$ is the empty face; we call  this extended structure
an \emph{augmented sheaf}.
Two examples of this kind will be important for our discussion.

The first is $\calF_{\F_2}^+$, obtained by taking $\calF(x)=\F_2$
for all $x\in X$ (including the empty face) and setting all the restriction maps to be $\id_{\F_2}$.
Replacing $\F_2$ with a general $\F_2$-vector space $V$ gives   the \emph{constant augmented sheaf}
$\calF_V^+$ (also denoted $V_+$ later on).

The second example is obtained by restricting a sheaf $\calF$ on $X$ to a proper link. Formally,
given a nonempty $z\in X$, let $\calF_z$ denote the \emph{augmented} sheaf on $X_z$
defined by $\calF_z(x)=\calF(x\cup z)$ and $\res^{\calF_z}_{y\from x}=\res^{\calF}_{y\cup z\from x\cup z}$
for all $x,y\in X_z$ with $x\subsetneq y$.

\paragraph*{Cochains, Cocycles and Coboundaries}

Similarly to \S\ref{subsec:hdx}, given
a sheaf (resp.\ augmented sheaf) $\calF$
on a simplicial complex $X$ and
$i\in\N\cup \{0\}$ (resp.\ $i\in\N\cup\{0,-1\}$), we can construct 
vector spaces of
$i$-cochains, $i$-cocycles and $i$-coboundaries with coefficients in $\calF$. 
The only difference is that
we evoke the restriction maps of $\calF$ when defining $d_i$.
Specifically, put $C^i=C^i(X,\calF)=\prod_{x\in X(i)}\calF(x)$ and
define  $d_i:C^i\to C^{i+1}$ by
\begin{equation}\label{EQ:local-cocycle-cond}
(d_i f)(y)=\sum_{\text{$x$ is an $i$-face of $y$}} \res^{\calF}_{y\from x} (f(x))
\end{equation}
for every $f\in C^i$ and $y\in X(i+1)$;
cf.\ \eqref{EQ:classical-coboundary}.\footnote{
	In order to generalize this to  
	$\F$-sheaves with $\F$ a field  of characteristic not $2$, one  should be introduce
	signs to \eqref{EQ:local-cocycle-cond}; 
	see \S\ref{subsec:sheaf-coh}. 
} 
The vector spaces of $i$-coboundaries and $i$-cochains
are defined as in~\eqref{EQ:Bi-Zi-def} and denoted $B^i(X,\calF)$ and $Z^i(X,\calF)$,
respectively, and the $i$-th cohomology group of $\calF$ is $\HH^i(X,\calF)=Z^i(X,\calF)/B^i(X,\calF)$.

If we take $\calF$ to be the constant augmented sheaf $\calF^+_{\F_2}$,
then this recovers $C^i(X,\F_2)$, $Z^i(X,\F_2)$ and $B^i(X,\F_2)$ considered in \S\ref{subsec:hdx}.

\paragraph*{Related Notions to Sheaves.}

The sheaves we have defined here are inspired by \emph{sheaves on topological spaces},
which are ubiquitous to topology and algebraic geometry,
see \cite[Chapter~II]{MacLane_1994_intro_to_topos_theory}
or \cite{Iversen_1986_cohomology_of_sheaves}, for instance. In fact, 
our sheaves
can be reinterpreted as sheaves on certain topological spaces
(associated to the simplicial complex at hand)
in such  a  way that the cohomology spaces  remain  the same;
see Appendix~\ref{sec:comparison}.
In contrast, augmented sheaves and their cohomology do not fit nicely into this setting.

The \emph{local systems  on graphs} defined 
by Jordan and Livne \cite[\S2]{Jordan_1997_Ramanujan_local_systems}
can be viewed as $\R$-sheaves in which all the restriction maps are isomorphisms.
The cohomology theory developed in {\it op.\ cit.} then  agrees with ours.

Friedman's   sheaves on graphs  \cite{Friedman_2015_sheaves_on_graphs} 
are defined like our sheaves, but with the restriction maps going from the edges to the vertices.
From the perspective of our work, they should perhaps be called \emph{co-sheaves}, or sheaves valued
in the opposite category of $\F$-vector spaces. They admit a   \emph{homology} theory,
rather than a cohomology theory.

\subsection{Expanding  Sheaves}
\label{subsec:expanding-sheaves}

Having   introduced  sheaves on simplicial complexes,
which are discrete versions of the sheaves   commonly studied in topology,
we turn to   present  the   new concept of an \emph{expanding sheaf},
which      has not been previously studied   in a topological context.
Expanding sheaves will play a pivotal role in our framework for constructing  good $2$-query LTCs.

Henceforth, $X$ is a simplicial complex  and $\calF$ is a sheaf or an augmented sheaf on $X$.

\paragraph*{Expansion of Sheaves.}

Given $f\in C^i=C^i(X,\calF)$,
let $\supp f=\{x\in X(i)\suchthat f(x)\neq 0\}$. The normalized Hamming norm
of $f$ is $\|f\|_{\Ham}=\frac{|\supp f|}{|X(i)|}$,
and   normalized Hamming distance between $f,g\in C^i$ is $d_{\Ham}(f,g):=\|f-g\|_{\Ham}$.

With this notation at hand, the notions   of cosystolic expansion and coboundary
expansion recalled in \S\ref{subsec:hdx} extend verbatim to sheaves.
That is, the sheaf $\calF$, or a sheaved complex $(X,\calF)$, 
is said be an $(\veps,\delta)$-cosystolic
expander in dimension $i$ if \eqref{EQ:cbe-F2-first} and \eqref{EQ:cbe-F2-second} hold for cocycles
with coefficients in $\calF$, and an
$\veps$-cosystolic expander in dimension $i$
if $\HH^i(X,\calF)=0$ (equiv.\ $B^i=Z^i$)  and
\eqref{EQ:cbe-F2-first} holds.

The situation  considered in \S\ref{subsec:hdx} now arises as the special
case where $\calF$ is the constant augmented sheaf $\calF^+_{\F_2}$.
Note, however, that we have shifted the focus from the expansion of the simplicial complex
$X$ to the expansion of the  sheaf
$\calF$. Indeed, $\calF$ may have poor expansion even
when $X$ is an excellent high-dimensional expander (e.g., take all the restriction maps to be $0$).
 
\label{page:cb-exp-in-dim-neg-one}%
Considering the expansion of (augmented) sheaves also illuminates an important condition which was transparent in the
case of
$\calF^+_{\F_2}$: coboundary expansion in dimension $-1$. Suppose that $\calF$ is an augmented
sheaf such    
that there is $m\in\N$ with $\calF(v)=\F_2^m$ for all $v\in X(0)$, and put $\Sigma=\F_2^m$.
Then $d_{-1}:C^{-1}\to C^0$ is an $\F_2$-linear map
from $\calF(\emptyset)$ to $\Sigma^{X(0)}$.
By definition, the augmented sheaf $\calF$ is an $\veps$-coboundary expander in dimension $-1$
if and only if $d_{-1}$ is injective, and its image  in  $\Sigma^{X(0)}$
is a code with relative
distance $\geq \veps$.

\paragraph*{$\mathbf{0}$-Cocycle Codes and ${\bf 2}$-Query LTCs.}
Suppose now that $\calF$ is a sheaf with $\calF(v)=\F_2^m$ for every vertex
$v\in X(0)$, where $m\in\N$ fixed, and put $\Sigma=\F_2^m$.
Then $C^0=C^0(X,\calF)= \Sigma^{X(0)}$, and   we may
regard $Z^0=Z^0(X,\calF)$ as a code inside  $\Sigma^{X(0)}$; we call
$Z^0$ a \emph{$0$-cocycle code}. There is a natural $2$-query tester for $Z^0$:
given $f\in C^0$, choose an edge $e\in X(1)$ uniformly at random and
accept $f$ if 
\[
\res^{\calF}_{e\from u}(f(u))=\res^{\calF}_{e\from v}(f(v)),
\]
where $u$ and $v$ are the vertices of $e$,
cf.\ \eqref{EQ:local-cocycle-cond}.
As in \S\ref{subsec:hdx},
$\calF$ is an $(\veps,\delta)$-coboundary expander in dimension $0$ if and only
if the code $Z^0\subseteq \Sigma^{X(0)}$ is  $\veps$-testable with respect to this tester, and has relative distance $\geq\delta$
(note that $B^0=0$).
The rate of $Z^0$ is $|X(0)|^{-1}\dim_{\F_2}\HH^0(X,\calF)$.
Note that we may
also  view $Z^0$ as code inside $(\F_2^m)^{X(0)}=\F_2^{X(0)\times\{1,\dots,m\}}$
--- the alphabet being $\F_2$ --- in which case, we get a
$2m$-query $\veps$-testable \emph{linear} code  with relative distance
$\geq \frac{\delta}{m}$; the rate remains the same.  

\medskip

The $2$-query LTCs that we will construct  arise as $0$-cocycle
codes of sheaves with good cosystolic exansion in dimension $0$.

\paragraph*{Higher-Cocycle Codes and Quantum CSS Codes.}

We can similarly consider the space of $i$-coycles $Z^i=Z^i(X,\calF)$
as a  code inside $C^i=C^i(X,\calF)$ when $i>0$.
If $\calF(x)=\Sigma:=\F_2^m$ for every $x\in X(i)$,
then the alphabet can be taken to be $\Sigma$,
and this code has an   $(i+2)$-query tester.
Such codes have different potential applications depending on whether $B^i(X,\calF)= 0$
or $B^i(X,\calF)\neq 0$. (The situation $B^i=0$ with $i>0$ is impossible
if we only consider $\F_2$-valued cocycles as in \S\ref{subsec:hdx},
but is possible for general sheaves, e.g., if  
$\calF(x)=0$ for all $x\in X(i-1)$.)

If $B^i=0$, then, as in the case $i=0$, the code
$Z^i$ is  $\veps$-testable  with relative distance $\geq \delta$
if and only if $\calF$ is an $(\veps,\delta)$-cosystolic expander in dimension $i$;
its rate is $|X(i)|^{-1}\dim_{\F_2} \HH^i(X,\calF)$.
See \S\ref{subsec:higher-cocycle-codes} for an example of an infinite family
of good
$1$-cocycle codes. (We do not know if this is a family of LTCs.)

If, on the other hand,   $B^i\neq 0$, then by viewing $Z^i(X,\calF)$
as a \emph{linear} codes inside $C^i(X,\calF)$ with alphabet $\F_2$,
we can enrich  it into a \emph{quantum CSS codes} over the alphabet $\F_2$;
see \S\ref{subsec:css} for details.
The  rate of this quantum CSS code 
is again $|X(i)|^{-1}\dim_{\F_2} \HH^i(X,\calF)$,
and if 
$\calF$ is an $(\veps,\delta)$-coboundary expander in dimension $i$, 
then 
its $X$-side has relative distance $\geq \delta$  and is   $\veps$-testable
(up to  scaling of the constants).


\paragraph*{Local Local-Testablity Implies Global Local-Testibility:
a Sheafy Version.}

Let $(X,\calF)$ be a sheaved simplicial complex.
Our first main result (Theorem~\ref{TH:cosystolic-exp-from-links};
see also 
Corollary~\ref{CR:cse-from-links-dim-2},
Remark~\ref{RM:main-ineqs-for-2dim-case-with-Oppenheim})
is a generalization of the local-to-global principle of Evra--Kaufman \cite{Evra_2016_cosystolic_expanders_arxiv_version}
recalled in \S\ref{subsec:hdx}. 
In more detail, given $i\in\{0,\dots,{\dim X-2}\}$,
we show  that if for every $z\in X-\{\emptyset\}$,
the augmented sheaf $\calF_z$ is a good coboundary expander in a range of dimensions,
and the underlying graph of $X_z$ is a sufficiently good expander (``local'' properties of $X$ and $\calF$),
then $\calF$ is an $(\veps,\delta)$-cosystolic expander in dimension $i$ (a ``global'' property of $\calF$)
with $\veps$ and $\delta$ depending on the suppressed expansion constants.
This also extends the main result of 
\cite{Kaufman_2018_cosystolic_expanders}  which, in our terminology,
addresses the special case of \emph{constant sheaves}. 

Our proof is more efficient than \cite{Evra_2016_cosystolic_expanders_arxiv_version}
and \cite{Kaufman_2018_cosystolic_expanders} in the sense that it
makes milder assumptions on the expansion of $X_z$ and $\calF_z$,
and at the same time produces larger
expansion constants $\veps$ and $\delta$. We also
show that the LTCs arising as $i$-cocycles codes
of sheaves to which our theorem applies   admit a linear-time decoding algorithm
able to correct a linear number of errors (up to a vector in $B^i(X,\calF)$ if $i>0$),
see Proposition~\ref{PR:decoding}. 

\paragraph*{A Side-Application: Sheafy Expander Codes.}

We return to discuss the sheaf-variation of expander codes defined on page~\pageref{intro-ex:sheaf-exp-codes},
focusing on its $0$-cocycle code.

Recall that $X$ is a $k$-regular graph, and we used codes $C_v\subseteq \F_2^{E(v)}$
to define a sheaf  $\calF$ on $X$. 
Suppose that all the $C_v$ have a common dimension $m$  and relative distance $\geq \veps$.
Then, writing $\Sigma=\F_2^m$, we have $C^0(X,\calF)=\Sigma^{X(0)}$.
Moreover, our   assumption on the distance of the $C_v$
says that, for every   $v\in X(0)$, the augmented sheaf $\calF_v$ (on the link $X_v$)
is an $\veps$-coboundary expander in dimension $-1$.
In other words, ``locally'', $\calF$ has good coboundary
expansion in dimension $-1$.

Since $X$ is merely $1$-dimensional (rather than $2$-dimensional),
this is not enough to apply our  
Theorem~\ref{TH:cosystolic-exp-from-links}
to assert that $\calF$ is a  good cosystolic expander in dimension $0$,
or equivalently, that $Z^0(X,\calF)\subseteq \Sigma^{X(0)}$ is an LTC with linear distance. 
Indeed, the code $Z^0(X,\calF)$ is usually \emph{not} testable if $m>\frac{k}{2}$, because
removing one of its defining constraints (i.e., removing an edge from $X$) will typically
enlarge $Z^0(X,\calF)$. Also, even if $X$ were the $1$-dimensional skeleton of a triangle complex
$Y$, it is usually not possible to extend $\calF$ in a non-redundant  way to $Y$. Indeed,
if $\calF$ could be extended to a sheaf on $Y$, then for any triangle $t=\{u,v,w\}\in Y(2)$, 
we would have $\res^{\calF}_{t\from \{u,v\}}\circ \res^{\calF}_{\{u,v\}\from \{u\}}=
\res^{\calF}_{t\from \{u\}}=\res^{\calF}_{t\from \{u,w\}}\circ \res^{\calF}_{\{u,w\}\from \{u\}}$.
A simple linear-algebra argument now shows that
if $C_u$ contains a word $f\in\F_2^{E(u)}$ with $f_{\{u,v\}}\neq f_{\{u,w\}}$,
then  $\res^{\calF}_{t\from \{u,v\}}$ and $\res^{\calF}_{t\from\{u,w\}}$ must be $0$.

Testability aside, if
the second eigenvalue of the adjacency matrix
of $X$ is $\lambda k$ ($\lambda\in [-1,1]$),  then we can still infer
that $Z^0(X,\calF)\subseteq \Sigma^{X(0)}$ has relative distance at least $ \veps-\lambda$, 
see \S\ref{subsec:0cocycle-graphs}.
If   instead we view $Z^0(X,\calF)$ as a \emph{linear} code inside $\F_2^{X(0)\times\{1,\dots,m\}}$,
then the relative distance is $\geq \frac{ \veps-\lambda }{m}$.
Since by dimension considerations, the  rate of $Z^0(X,\calF)  $ 
is at least $(1-\frac{k}{2m}) $ (with respect to either alphabet),
we conclude that the code   $Z^0(X,\calF)$ is good   if $\veps>\lambda$ and $m>\frac{k}{2}$.
These bounds are similar to the expander
codes of \cite{Sipser_1996_expander_codes} (see also \cite{Meshulam_2018_graph_codes_arxiv_version}). 
Note, however, that $Z^0(X,\calF)$ is not a \emph{lifted
code}, and thus \emph{not} an expander code in the   sense of \cite{Sipser_1996_expander_codes}.

\subsection{Utilizing Coverings}
\label{subsec:coverings-intro}
Coverings of high dimensional expanders   play an important role in our framework for getting good $2$-query LTCs from sheaved high dimensional expanders. 
Broadly speaking, coverings   allow  us to produce many expanding sheaves from a single example,
and  in a different context,   provide an ``inflation'' effect  that reduces obstructions. 
The former will facilitate our framework while the latter 
would be useful to applying it. 
We now recall what are coverings, and explain why they are important in our framework.

Henceforth, all simplicial complexes are assumed to be connected.

\paragraph*{Coverings of Simplicial Complexes.}

Let $X$ and $Y$ be (connected) simplicial complexes. Recall that a simplicial map
$p:Y\to X$ is called a \emph{covering} map if for every nonempty
face $z\in Y $,
the restriction of $p$ to the link $Y_z$ defines an bijection between
$Y_z$ and the link $X_{p(z)}$. Equivalently, $p$ is a covering
map if it induces a covering map of topological spaces between
the topological realizations of $Y$ and $X$. 
In this case,   the connectivity of $X$ implies that the 
number of faces in $Y$ mapping to a nonempty face $x\in X$ is independent of $x$;
this common number is
called the \emph{degree} of $p$. We say that $p$ is a \emph{double covering},
or that $Y$ is a double covering of $X$ (via $p$),
if the degree of $p$ is $2$. In this case, $|Y(i)|=2|X(i)|$ for all $i\in\N\cup \{0\}$.

It is a standard fact from
algebraic topology that    there is a one-to-one correspondence
between the (connected) coverings of $X$ (considered up to isomorphism over $X$)
and subgroups of the fundamental group $\pi_1(X)$. This restricts to a bijection
between the   degree-$d$ coverings of $X$ and the index-$d$ subgroups of $\pi_1(X)$.

\paragraph*{Pulling Back a Sheaf Along a Covering.}

If $p:Y\to X$ is a covering map, and $\calF$ is a sheaf on $X$,
then we can   define a sheaf $p^*\calF$ on $Y$ by pulling back $\calF$ along $p$, i.e., by
setting
\[
p^*\calF(y)=\calF(p(y))\qquad\text{and}\qquad
\res^{p^*\calF}_{y'\from y} = \res^{\calF}_{p(y')\from p(y)}
\]
for all $y,y'\in Y$ with $y\subsetneq y'$. The sheaf $p^*\calF$ called the \emph{pullback}
of $\calF$ along $p:Y\to X$. 

Since $p$ is a covering map, it restricts to an isomorphism
$Y_z\to X_{p(z)}$ for every nonempty $z\in Y$. Under this isomorphism,
the restriction of $p^*\calF$ to $Y_z$, i.e.\ $(p^*\calF)_z$,
is just $\calF_z$. Thus, up to isomorphism, 
the sheaves $p^*\calF$ and $\calF$ have the same restrictions to proper links.


\paragraph*{Local Local-Testability Lifts Along Coverings.}

Let $p:Y\to X$ be a covering map and let $\calF$ be a sheaf on $X$.
Recall that our Theorem~\ref{TH:cosystolic-exp-from-links}
says that if the pairs $(X_z,\calF_z)_{z\in X-\{\emptyset\}}$ satisfy some expansion conditions
(informally called ``local'' local-testability), then $(X,\calF)$ will be an
$(\veps,\delta)$-cosystolic expander in dimension $i$ (``global'' local-testability).
Since $p^*\calF$ and $\calF$ have the same restrictions to   proper links up to isomorphism, 
once the assumptions of Theorem~\ref{TH:cosystolic-exp-from-links} are satisfied for $\calF$,
they are also satisfied
for   $p^*\calF$, meaning that $p^*\calF$ is also an $(\veps,\delta)$-cosystolic expander in dimension $i$.

We apply this observation in the following context:
Let $m\in\N$, $\Sigma=\F_2^m$,
and 
suppose that $\calF(v)=\F_2^m=\Sigma$
for all $v\in X(0)$.
If   $(X,\calF)$
satisfies the conditions of Theorem~\ref{TH:cosystolic-exp-from-links} with $i=0$,
then for every covering $p:Y\to \calF$, we have that $p^*\calF(u)=\Sigma$ for all $u\in Y(0)$,
and $p^*\calF$ is an $(\veps,\delta)$-cosystolic expander in dimension $0$ with $\veps,\delta>0$
independent of $Y,p$. Consequently, for every
covering $p:Y\to X$, the code $Z^0(Y,p^*\calF)\subseteq \Sigma^{Y(0)}$
is $\veps$-testable and has relative distance $\geq \delta$.
Otherwise said, the family
of codes $\{Z^0(Y,p^*\calF)\subseteq \Sigma^{Y(0)}\}_{Y,p}$ with $p:Y\to X$
ranging over the coverings of $X$
is a   family of $2$-query LTCs with linear distance.

\paragraph*{Rate Conservation in Coverings.}

We continue to assume that $p:Y\to X$ is a covering map 
and $\calF$ is a sheaf on $X$ with $\calF(v)=\F_2^m=\Sigma$ for all $v\in X(0)$.
Similarly to the situation with testability and distance, we would like
to be able to guarantee that ``pullback code'' $Z^0(Y,p^*\calF)$ has roughly
the same   rate as $Z^0(X,\calF)$.

Suppose that $p:Y\to X$ is of degree $\ell$ and factors as a composition
of double coverings $Y=X_r\to X_{r-1}\to \dots\to X_0=X$ (thus $\ell=2^r$). 
In Theorem~\ref{TH:rate-conservation}, we show that in this special case, we  
have $\dim Z^0(Y,p^*\calF)=\Theta(\ell)$, i.e., the rate of $Z^0(Y,p^*\calF)$ is constant, provided
that
\[\dim\HH^1(X,\calF)<  \dim\HH^0(X,\calF).\]
We call this result \emph{rate conservation}.
In particular, if $\{X_r\}_{r\geq 0}$ is an \emph{infinite tower} of connected  double coverings
of $X_0=X$, i.e., each   $X_r$ is a double covering of $X_{r-1}$,
and if $\calF_r$ is the pullback of $\calF$ to $X_r$,
then the family
of codes $\{Z^0(X_r,\calF_r)\subseteq \Sigma^{X_r(0)}\}_{r\geq 0}$ has constant rate.

The  
integer $\dim\HH^0(X,\calF)-\dim\HH^1(X,\calF)$  can be considered as measuring the   obstruction  to rate conservation.
Indeed, we can apply rate conservation
precisely when it is positive, and the larger it   is, the larger the rate of $Z^0(Y,p^*\calF)$
will be for $p:Y\to X$ as above.
We will see   in \S\ref{subsec:finding-int-data-intro} that if
$\calF$ is  sheaf on $X$ such that   $\dim\HH^1(X,\calF)$
is significantly smaller than $m$ (recall that $\Sigma=\F_2^m=\calF(v)$ for $v\in X(0)$), then 
there is a way
to modify $\calF$ in order to (conjecturally)
decrease $\dim\HH^1(X,\calF)$ even further, thus achieving the threshold
for rate conservation. 

We also remark that,
as stated here, rate conservation is specific to double coverings, $\F_2$-sheaves
and $0$-cocycle codes.
If one wishes to replace $\F_2$ with another field of characteristic $p>0$,
then the requirement that each $X_r$ is a double covering of $X_{r-1}$
should be replaced by $X_r\to X_{r-1}$ being \emph{Galois covering} of degree $p$,
i.e., that $\pi_1(X_r)$ is a normal subgroup of index $p$ in $\pi_1(X_{r-1})$.
In order to apply rate conservation to $i$-cocycle  codes with $i>0$,
one needs to add the extra hypothesis $\HH^{i-1}(X,\calF)=0$.

\subsection{The Tower Paradigm: A Framework for Constructing Good $2$-Query LTCs from Expanding Sheaves }
\label{subsec:tower-paradigm}

We now put together the observations of \S\ref{subsec:coverings-intro}
to give a method --- the \emph{tower paradigm} --- for constructing
an infinite 
family of LTCs with linear distance and constant rate from auxiliary \emph{finite}  initial data.
This method   overcomes
the intrinsic barrier in constructing cocycle codes with linear distance and constant rate
noted in \S\ref{subsec:hdx}.

\paragraph*{The Initial Data.}

The initial data   consists of an integer $m\in\N$,
a $2$-dimensional simplicial complex $X$, and a sheaf $\calF$ on $X$
such that  $\calF(v) =\F_2^m$ for all $v\in X(0)$.
We write $\Sigma=\F_2^m$; this will be the alphabet of the $2$-query LTCs 
that will be constructed from these data.

\paragraph*{Requirements on The Initial Data.}

The initial data $(X,\calF)$   is required to satisfy the following three requirements:
\begin{enumerate}[label=(t\arabic*)]
	\item \label{item:tower-cond-intro}
	\label{item:tower-intro:first}
	There is a sequence of (connected) simplicial complexes $\{X_r\}_{r\geq 0}$
	such that   $X_0=X$ and   $X_r$ is a double covering of $X_{r-1}$ for all $r\in\N$.
	We call $\{X_r\}_{r\geq 0}$ a \emph{tower of double coverings} of $X$.
	\item \label{item:LTC-cond-into}
	For every nonempty $z\in X$,
	the sheaf $\calF_z$ (on the link $X_z$)
	is a   good coboundary expander 
	and the underlying graph of $X_z$ is a sufficiently
	good expander; see condition \ref{item:TH:tower-paradigm:LTC} of Theorem~\ref{TH:tower-paradigm}
	for a precise statement.
	Informally, this means
	that $X$ is a high-dimensional expander, and $\calF$ satisfies ``local'' local-testability.
	\item \label{item:rate-cond-intro}
	\label{item:tower-intro:last}
	$\dim\HH^0(X,\calF)>\dim \HH^1(X,\calF)$.
\end{enumerate}

Note that conditions \ref{item:LTC-cond-into} and 
\ref{item:rate-cond-intro} can be verified for a given $(X,\calF)$
by a finite computation.
Condition \ref{item:tower-cond-intro} does not
involve the sheaf $\calF$, 
and can be readily arranged by using
existing constructions of high-dimensional expanders, e.g., \cite{Lubotzky_2005_explicit_constructions_of_Ramanujan_complexes} or \cite{Kaufman_2018_local_spectral_high_dim_exps}.

\paragraph*{The Induced Family of Good ${\mathbf{2}}$-Query LTCs}

Using  the initial data $(X,\calF)$ and the tower of double coverings $\{X_r\}_{r\geq 0}$,
we 
define    an infinite family of codes
on the alphabet $\Sigma=\F_2^m$ 
as follows:
Denote by $\calF_r$ the pullback of $\calF$ along the covering map $X_r\to X_0=X$.
Then $Z_r:=Z^0(X_r,\calF_r)$ is a code inside $C^0(X_r,\calF_r)=\Sigma^{X_r(0)}$.
Writing $n_r=|X_r(0)|=2^r|X(0)|$, this defines a family
of codes 
\[\{Z_r\subseteq \Sigma^{n_r}\}_{r\in\N}\] 
with length tending to infinity.

\begin{thm}[Informal; see Theorem~\ref{TH:tower-paradigm}]
	If conditions \ref{item:tower-intro:first}--\ref{item:tower-intro:last} hold,
	then the  codes $\{{Z_r\subseteq \Sigma^{n_r}}\}_{r\in\N}$ together with their natural
	$2$-query testers form an infinite family of $2$-query LTCs with constant rate and linear distance.
	Moreover, they admit a linear-time decoding algorithm.
\end{thm}

As explained in \S\ref{subsec:coverings-intro}, condition \ref{item:LTC-cond-into} 
and our local-to-global principle
(Theorem~\ref{TH:cosystolic-exp-from-links}) imply that this is a family of 
$2$-query LTCs with linear distance,
and condition \ref{item:rate-cond-intro} 
allows us to apply \emph{rate conservation} (Theorem~\ref{TH:rate-conservation})
to conclude that the rate of the family is constant.

We remark that the soundness of the LTCs $\{{Z_r\subseteq \Sigma^{n_r}}\}_{r\in\N}$    depends
only on the expansion of $X$ and the   coboundary expansion of the restriction of $\calF$
to the proper links of $X$. It does not depend on the alphabet size $|\Sigma|=2^m$. 


\subsection{Finding Initial Data for The Tower Paradigm}
\label{subsec:finding-int-data-intro}

It remains to find examples of initial data for the \emph{tower paradigm}
which satisfy all three requirements \ref{item:tower-intro:first}--\ref{item:tower-intro:last}.
While we   demonstrate that every two of these conditions can be met
(see \S\ref{subsec:three-cond-examples}), finding sheaved high-dimensional expanders
satisfying all three is  surprisingly difficult, and unfortunately remains open.
Instead, we construct candidates satisfying conditions \ref{item:tower-cond-intro}
and \ref{item:LTC-cond-into}, and conjecturally
also \ref{item:rate-cond-intro}.

More precisely, we introduce an iterative process which takes a sheaved high dimensional expander   
satisfying
\ref{item:tower-cond-intro}
and \ref{item:LTC-cond-into} and modifies its sheaf. 
We show that if the process ends 
quickly enough, then the resulting sheaf 
will satisfy
\ref{item:rate-cond-intro} as well.
We conjecture that the process will terminate quickly when performed on examples coming from number theory (Conjecture~\ref{CJ:process-for-buildings}),
hence our aforementioned candidates.
Moreover, we identify  conditions,
phrased by means of representations of the  fundamental group of the high-dimensional expander at hand,  
which imply that the process terminates after just $1$ step. What this means in practice 
is that if one could find an \emph{arithmetic group} 
with a finite-dimensional \emph{$\F_2$-representation}
satisfying certain conditions (see assumption (1) in Theorem~\ref{TH:main-candidates}
and the following comment), 
then they would give rise to initial data for the tower paradigm,
and thus to an infinite family of good $2$-query LTCs. There exist 
arbitrarily large 
\emph{finite} groups with  representations meeting these conditions.


\medskip

We now explain in broad strokes how our candidates for initial data    for the tower paradigm are constructed.
An example of how the resulting family of codes may look like 
is given in \S\ref{subsec:explicit-example}.
 
\paragraph*{The Tower.}

In order to construct the tower $\{X_r\}_{r\geq 0}$,
we fix an \emph{affine building} $Y$ of dimension $d\geq 2$,
e.g., the affine building
of $\nSL{\Q_p}{d+1}$ (see \cite[\S6.9]{Abramenko_2008_Buildings}).
Informally, $Y$ is a highly-symmetric $d$-dimensional \emph{infinite} simplicial complex.
Each of the $X_r$ is obtained as a finite quotient $\Gamma_r\leftmod Y$, where $\Gamma_r$
is a group acting freely on $Y$.
By choosing the groups $\{\Gamma_r\}_{r\geq 0}$ to be a decreasing sequence $\Gamma_0\geq \Gamma_1\geq \Gamma_2\geq\dots$
such that   $[\Gamma_{r-1}:\Gamma_r]=2$ for all $r\in\N$,
the $X_r$ arrange naturally into an infinite tower of double coverings of $X=X_0$.
(The covering map $X_r=\Gamma_r\leftmod Y\to\Gamma_{r-1}\leftmod Y=X_{r-1}$ 
sends $\Gamma_r y$ to $\Gamma_{r-1} y$.)
See \S\ref{subsec:construction-of-quotients} for particular examples of $Y$, $\{\Gamma_r\}_{r\geq 0}$.
Algorithms for constructing certain   quotients 
$\Gamma_r\leftmod Y$ 
explicitly can be found in \cite{Lubotzky_2005_explicit_constructions_of_Ramanujan_complexes}, for instance.


Choosing   $X:=X_0$ to be a quotient  of an affine building $Y$ by a group $\Gamma_0$
also means that its proper links  
are \emph{spherical buildings}, which are known to be excellent expanders. This guarantees
the that the  expansion assumptions on the proper links $X_z$ 
mentioned in    \ref{item:LTC-cond-into} will hold automatically
as soon as $Y$ is \emph{thick} enough. (For example, the thickness of the affine building
of $\nSL{\Q_p}{d+1}$ is $p+1$.) It also has the advantage that $\pi_1(X)=\Gamma_0$ is an
\emph{arithmetic group}; a fact that will be put to use later on.

Assuming that $X=X_0$ and the tower $\{X_r\}_{r\geq 0}$ have been   chosen, we
set to look for a sheaf $\calF$ for which conditions \ref{item:LTC-cond-into}
and \ref{item:rate-cond-intro} hold. 
We do this in two stages. First, a certain \emph{locally constant sheaf}
${\calF}$ is chosen. Then, the sheaf ${\calF}$ is modified to produce a sheaf
$\quo{\calF}$ satisfying \ref{item:LTC-cond-into} and conjecturally \ref{item:rate-cond-intro}.


\paragraph*{Locally Constant Sheaves.}

A sheaf $\calG$ on $X$ is called \emph{locally constant}
if for every $v\in X(0)$, the augmented sheaf  $\calG_v$ is (isomorphic to) a constant augmented sheaf 
on $X_v$. This is equivalent to all the restriction maps of $\calG$
being isomorphisms. Since $X$ is connected,
this means that there is  $m\in\N\cup\{0\}$, denoted $\dim\calG$ and called the dimension of $\calG$, 
such that $m=\dim\calG(x)$
for all $x\in X-\{\emptyset\}$.
Locally constant sheaves are abundant: every $n$-dimensional  
$\F_2$-representation of $\pi_1(X)$
gives rise to an $n$-dimensional locally constant sheaf on $X$, with the trivial representation corresponding
to the constant sheaf $\calF_{\F_2 }$.

\paragraph*{Locally Constant Sheaves as Initial Data for The Tower Paradigm.}

We are interested in locally constant sheaves because condition
\ref{item:LTC-cond-into} is satisfied for any locally constant
sheaf ${\calF}$ on the $X$ we chose. Indeed,
if $z\in X-\{\emptyset\}$, then  
${\calF}_z$ is a constant sheaf on the spherical building $X_z$.
It was shown in \cite{Lubotzky_2016_expansion_of_buildings} and \cite{Kaufman_2018_cosystolic_expanders}
(see also \cite{First_2021_weighted_mixing_lemmas_preprint}) that such sheaves
are excellent coboundary exapnders in all dimensions (no matter how thick $Y$ is),
which means that \ref{item:LTC-cond-into} holds.

%

The reason why we do not apply the tower paradigm to   locally constant sheaves
is because it turns out that conditions \ref{item:tower-cond-intro}
(an infinite tower of double coverings)
and \ref{item:rate-cond-intro} (rate conservation) cannot hold simultaneously for such
sheaves (Proposition~\ref{PR:loc-constant-failure}). What we suggest to 
do instead is 
taking a  special  locally constant sheaf on $X$ and   modifying it slightly so that is satisfies   \ref{item:rate-cond-intro}   as well.

We explain the modification process and the choice of the special sheaf separately.


\paragraph*{Modifying   Locally Constant Sheaves.}

Let $ {\calF}$ be a locally constant sheaf on $X$ of a large dimension $m$.
We think of $ \calF$ as varying with  $m$ as it goes to $\infty$, but ultimately,
both $ \calF$ and $m$ will be fixed and regarded as ``small''. 

We just observed that $ \calF$ satisfies \ref{item:LTC-cond-into}
but not \ref{item:tower-intro:last}. In \S\ref{subsec:quotient},
we present an iterative process that takes $ \calF$ as input and outputs
a modified sheaf $\quo{\calF}$ which satisfies \ref{item:tower-intro:last}.
If the iterative process terminates quickly, and if $h:=\dim\HH^1(X,\calF)-\dim\HH^0(X,\calF)+1$
is very small compared to $ \dim\calF$, then we can show that the output sheaf
$\quo{\calF}$ is ``very close'' to the original $\calF$, to the extent that it also satisfies
the local expansion condition \ref{item:LTC-cond-into}.
If that is indeed the case, then $(X,\quo{\calF})$
can serve as initial data for the tower paradigm.

Note that $h$   quantifies how ``far'' we are from being able to apply rate conservation. 
Informally, the iterative process eliminates this obstruction when its size is negligible to $ \dim\calF$.

In more detail, $\quo{\calF}$ is constructed as the \emph{quotient} 
of $\calF$ by a ``tiny'' non-locally
constant subsheaf $\calC$,
chosen to artificially increase $\dim\HH^0(X,\calF/\calC)$
and decrease $\dim \HH^1(X,\calF/\calC)$.
To construct $\calC$, we choose a ``tiny'' subspace $E\subseteq Z^1(X,\calF)$
and let $\calC$ be the smallest subsheaf of $\calF$ such that $E\subseteq C^1(X,\calC)$; 
see Construction~\ref{CN:modification-prototype} or \S\ref{subsec:quotient}.
The subsheaf $\calC$ has the feature that $\calC(v)$ is $0$ for every vertex $v\in X(0)$
while (typically) $\calC(x)\neq 0$ for faces $x$ of dimension $>0$. In particular, $\quo{\calF}(v)=\calF(v)/0\cong \F_2^m=:\Sigma$ for all $v\in X(0)$. 
We show in \S\ref{subsec:quotient} that elements
in $E\cap B^1(X,\calF)$ give rise to ``new'' classes in $\HH^0(X,\quo{\calF})$
while elements in $E$ which map to a nonzero class in $\HH^1(X,\calF)$
eliminate that class in $\HH^1(X,\quo{\calF})$. In total, we expect
to get
\[
\dim\HH^0(X,\quo{\calF})-\dim\HH^1(X,\quo{\calF})=\dim\HH^0(X,\calF)-
\dim\HH^1(X,\calF)+\dim E.
\]
%
That is, by passing from $\calF$ to $\quo{\calF}$,
we increase $\dim\HH^0(X,-)-\dim\HH^1(X,-)$ by $\dim E$. If this 
prediction works,
then we could choose $E$ such that $\dim E=\dim\HH^1(X,\calF)-\dim\HH^0(X,\calF)+1$ 
and get
$\dim\HH^0(X,\quo{\calF})>\dim\HH^1(X,\quo{\calF})$, i.e.,
\ref{item:tower-intro:last} would hold for $\quo{\calF}$. 
However, passing from $\calF$ to $\quo{\calF}$ may result in ``new'' cohomology
classes in $\HH^1(X,\quo{\calF})$. We can eliminate these classes by enlarging $E$ further and repeat this process until there are no
more excess cohomology classes in $\HH^1(X,\quo{\calF})$; 
this is formalized in 
Construction~\ref{CN:modification-process}.
The resulting sheaf $\quo{\calF}$ always satisfies \ref{item:tower-intro:last}, although
not necessarily \ref{item:LTC-cond-into}.

However, we show that if    $\dim E\ll  \dim\calF $ when the process ends   
--- which is what we mean by saying that the process ends quickly ---,
or if we simply terminate the process when  $\dim E\ll  \dim\calF $, 
then, with high probability, $\quo{\calF}$ still satisfies the necessary
local expansion condition \ref{item:LTC-cond-into}; see 
Corollary~\ref{CR:securing-expansion-for-quotients} (which builds
on 
Theorem~\ref{TH:sheaves-on-quo-of-aff-buildings-quotients}).  

\paragraph*{Condition for The Modification Process to End Quickly.}


We simulated the iterative modification process for sheaves on 
$3$-dimensional tori, small $2$-dimensional $3$-thick Ramanujan complexes and other examples.\footnote{The {\tt Python} code of
the simulations was  written by the first named author
and 
is attached to the {\tt arXiv} version of this paper.}
Through  the simulations, we came out with formulas that 
predict
the growth of the subspace $E$ when more and more cohomology
classes are eliminated. This is formalized in Conjecture~\ref{CJ:cup-product-explains-everything}, which, loosely speaking, says that the growth of $E$
is governed by the \emph{cup product} bilinear map
$\cupp:\HH^1(X,\F_2)\times \HH^1(X,\calF)\to\HH^2(X,\calF)$
(see \S\ref{subsec:cup-prod}).
In particular, our analysis suggests that:

\begin{cnj}\label{CJ:into-easy}
	(Simplified; see Conjecture~\ref{CJ:dimension-of-E2})
	If $\calF$ is a sheaf on $X$ such that
	the   linear map
	$\alpha\otimes f\mapsto \alpha \cupp f:\HH^1(X,\F_2)\otimes_{\F_2} \HH^1(X,\calF)	
	\to\HH^2(X,\calF)$
	is injective and $\HH^0(X,\calF)\neq 0$,
	then, once applied to $\calF$, the modification process 
	stops after one step (i.e., there are no ``new'' cohomology
	classes in $\HH^1(X,\quo{\calF})$ in the above sense) with high probability. 
	More precisely, when the modification process ends,
	we have $\dim E=\dim\HH^1(X,\calF)-\dim\HH^0(X,\calF)+1$ 
\end{cnj}

This conjecture (more precisely, the finer Conjecture~\ref{CJ:dimension-of-E2})
is supported by all of our simulations. 
Since the universal covering of $X$ is contractible (it is an affine building), $\calF$ corresponds to a representation $\rho:
\Gamma_0=\pi_1(X)\to \nGL{\F_2}{m}$, and the assumption on $\calF$
is equivalent to saying that 
$\HH^1(\Gamma_0,\F_2)\otimes_{\F_2} \HH^1(\Gamma_0,\rho)	
\to\HH^2(\Gamma_0,\rho)$ is injective and $\rho$ has nontrivial invariant
vectors. We do not know if there is an arithmetic
group with an $\F_2$-representation satisfying
this condition, but there are arbitrarily large $2$-groups
for which this holds (see the {\tt MathOverflow} answer \cite{SashaP_2022_MO_answer}).

We also make a bolder conjecture which predicts that the modification
process ends quickly if the affine building
$Y$ which covers $X$ is sufficiently thick.

\begin{cnj}\label{CJ:into-hard}
	(Simplified; see Conjecture~\ref{CJ:process-for-buildings})
	There are $d,q\in \N$ ($d\geq 2$) and a function
	$f:\N\cup \{0\}\to \N$ such that if   $X$ is 
	covered by a $q$-thick affine building of dimension $d$
	and $\calF$ is a sheaf on $X$, then applying the modification
	process to $\calF$ results in a subspace $E$ such that $\dim E\leq f(\dim\HH^1(X,\calF))$
	with high probability.
\end{cnj}

\paragraph*{Finding a Locally Constant Sheaf to Modify Using Coverings.}

If we take Conjectures~\ref{CJ:into-easy}
and~\ref{CJ:into-hard} for granted, all that remains in order to
find initial data for the tower paradigm is to find a sheaf
$\calF$ on $X$ such that $\dim\HH^1(X,\calF)\ll \dim\calF$.
(In order to use Conjecture~\ref{CJ:into-easy}, we also need
to require that the additional assumption of that conjecture holds for $\calF$.)
This sheaf would be modified into a quotient sheaf $\quo{\calF}$
satisfying \ref{item:LTC-cond-into} and~\ref{item:rate-cond-intro}.

We show in Theorem~\ref{TH:sheaves-with-small-coh} 
that there exist finite simplicial
complexes $X$ covered by arbitrarily thick affine buildings
such that $X$ admits   locally constant sheaves $\calF$
of arbitrarily large dimension which satisfy
$\dim\HH^1(X,\calF) =0$.\footnote{
	This forces $\dim\HH^1(X,\calF)=0$ 
	if $X$ has an infinite tower of double coverings
	(Proposition~\ref{PR:loc-constant-failure}).
}
In particular, the requirement  $\dim\HH^1(X,\calF)\ll \dim\calF$
can be met. Alternatively, we could start with any
locally constant sheaf $\calG$ on $X$, and replace
it by $\calG_s:=\calG \times \calF^s$ for some large $s\in\N$ 
in order to increase $\dim\calG_s$ without affecting
$\dim\HH^1(X,\calG_s)=\dim\HH^1(X,\calG)$.



The idea behind the   construction   is to  once more  utilize coverings.
Let $p:X'\to X$ be a covering of degree $m$, 
and let $\calF$
be the \emph{pushforward} of the the constant sheaf $\calF_{\F_2}$
on $X'$ along $p$ (see \S\ref{subsec:pushforward}).
The sheaf $\calF$ is locally constant of dimension $m$
and has the additional property
that  $\HH^i(X, \calF )\cong \HH^i(X, \calF_{\F_2})=\HH^i(X',\F_2)$ 
(Lemma~\ref{LM:Shapiro}).
We  now put into use the fact that $\pi_1(X)$ is an \emph{arithmetic group}.
Using   deep facts about such groups,
we show in Theorem~\ref{TH:good-quotient}, that if the covering building $Y$
is carefully chosen, then $X'$ can be chosen 
to satisfy $\dim\HH^1(X',\F_2)=O(1)$ as $m$ grows. 
This is already enough if we want sheaves $\calF$
satisfying $\dim\HH^1(X,\calF)=O(1)$
as a function of $\dim\calF$,
and  a  
more sophisticated construction of this flavor achieves  $\dim\HH^1(X,\calF)=0$.
(Note that this holds despite the fact that $X$ has an infinite
tower of double coverings, which means in particular that
$\HH^1(X,\F_2)\neq 0$.)
More generally, it is expected that if Serre's Conjecture on 
the \emph{congruence subgroup property} (Conjecture~\ref{CJ:csp}) holds, then 
for every affine building $Y$ of dimension $\geq 2$, 
one  could choose
$X'$ with  $\dim\HH^1(X',\F_2)=O(\log m)$,
and thus get 
$\dim\HH^1(X,\calF)=O(\log \dim\calF)$.

\paragraph{Conclusion.}

We construct candidates for   initial data for the tower paradigm
as follows: We choose a  simplicial complex $X$ covered by a sufficiently thick
affine building of dimension $\geq 2$, and such that $X$ admits an infinite tower
of double coverings (condition~\ref{item:tower-cond-intro}).
Using other coverings of $X$, we find a locally constant sheaf $\calF$ 
such that $\dim\HH^1(X,\calF)\ll \dim\calF$;
the pair $(X,\calF)$ satisfies \ref{item:LTC-cond-into}.
We then apply an iterative process to modify $\calF$ into a quotient sheaf $\quo{\calF}=\calF/\calC$.
However, we terminate the process if $\calC$ becomes ``close'' to $\calF$
in dimension, in order to keep the validity
of \ref{item:LTC-cond-into} for $(X,\quo{\calF})$.
If the process terminated on its own, then
$(X,\quo{\calF})$ also satisfies \ref{item:rate-cond-intro}.
In this case, $(X,\quo{\calF})$ are   initial
data for the tower paradigm,
and therefore give rise to
an infinite family of good $2$-query   LTCs;
their common alphabet
is $\Sigma:=\F_2^m$ for   $m=\dim\calF$. (The soundness
of the testing is independent of $m$, however.)

Our Conjecture~\ref{CJ:into-hard} predicts that
the modification process will indeed terminate on its own.
Alternatively, our Conjecture~\ref{CJ:into-easy},
that is supported by computer simulations, says that
this will also be the case if
$\calF$ satisfies an additional property concerning the cup product.
See Theorem~\ref{TH:main-candidates} and Remark~\ref{RM:main-candidates}
for precise statements.

\subsection{Explicit 2-Query LTCs with Linear Distance and Conjectural Constant Rate}
\label{subsec:explicit-example}

We finish with giving an example of an infinite family of $2$-query LTCs
with linear distance and conjectural constant rate
that arises from our framework.  Sheaves  are not explicitly
mentioned, but are needed for the proofs.

\medskip


For the example, 
we use the Ramanujan complexes
constructed by Lubotzky, Samuels and Vishne in 
\cite[\S9]{Lubotzky_2005_explicit_constructions_of_Ramanujan_complexes} as a black box,
making reference only to the auxiliary finite field $\F_q$ used in {\it op.\ cit.}, which we assume
to be of characteristic $2$.
Alternatively, it is also possible
to use the simplicial complexes from Theorem~\ref{TH:good-quotient}
below; this has the advantage of not replying on Serre's Conjecture
on the congruence subgroup property, and the disadvantage of not having
an efficient algorithm to explicitly construct the complexes.

Fix $d\geq 3$.
The construction in {\it op.\ it.}   gives an explicit infinite sequence of $d$-dimensional   simplicial complexes,
each mapping into the former:
\[
\dots\to X_r\to \dots \to X_2\to X_1\to X_0,
\]
and which can be refined into a tower of \emph{double coverings} of $X_0$.

\paragraph*{Constructing The Codes From    Initial Data.}

Suppose that the following finite set of \emph{initial data} is provided (note that this data  
is fixed and does not grow with the parameter $r$):
\begin{enumerate}[label=(\arabic*)]
	\item $m\in \N$,
	\item a linear transformation $T_{e,u}:\F_2^m\to \F_2^m$ for every edge $e=\{u,v\}\in X_0(1)$,
	\item a subspace $C_e\subseteq \F_2^m$ for every edge $e\in X_0(1)$.
\end{enumerate} 
The integer $m$ is the same constant $m$ from \S\ref{subsec:finding-int-data-intro}.
While it is fixed, it will be convenient to think of it as growing.

Once the    data (1)--(3) is provided, we construct an infinite family of codes
$\{C_r\subseteq \Sigma^{n_r}\}_{r\in\N}$ on the alphabet $\Sigma:=\F_2^m$
as follows. Write $n_r=|X_r(0)|$ and identify $\Sigma^{n_r}$ with $\Sigma^{X_r(0)}$.
We write the $v$-coordinate of $f\in \Sigma^{X_r(0)}$
as $f(v)$. Then $f\in C_r$ if  for every edge $e=\{u,v\}\in X_r(1)$
with image  $e_0=\{u_0,v_0\}$ in $X_0$, we have
\begin{equation}\label{EQ:two-query-test}
T_{e_0,u_0} (f(u)) + T_{e_0,v_0}(f(v))\in C_{e_0}.
\end{equation}
A $2$-query tester for $C_r$ is given by choosing $e\in X_r(0)$ uniformly at
random and accepting the given $f\in \Sigma^{n_r}$ if \eqref{EQ:two-query-test} holds.

\paragraph*{Construction of The Initial Data.}
 
Provided Conjecture~\ref{CJ:into-hard} holds,
a possible choice for the initial data 
of $m$, $\{T_{e,v}\}_{e,v}$ and $\{C_e\}_e$ can be   
as follows. One chooses 
another simplicial complex $X'$ that
is a degree-$m$  covering of $X_0$, e.g., one of the $X_r$.
The size of $m$ needs to be sufficiently large relative to $X_0$, but   is otherwise fixed
(i.e.\ independent of $r$).
Given a face $x\in X_0-\{\emptyset\}$, we number the faces in $X'$ mapping onto $x$
as $\hat{x}_1,\dots,\hat{x}_m$. Thus, for every
edge $e =\{u,v\}\in X_0(1)$, there is a unique permutation $\sigma_{e,u}:\{1,\dots,m\}\to \{1,\dots,m\}$
such that $\hat{u}_i\in \hat{e}_{\sigma(i)}$ for all $i\in\{1,\dots,m\}$.
We take $T_{e,u}:\F_2^m\to \F_2^m$ to be the linear transformation
given by
\[
T_{e,u}(\alpha_1,\dots,\alpha_m)=(\alpha_{\sigma_{e,u}^{-1}(1)},\dots,\alpha_{\sigma_{e,u}^{-1}(m)}).
\]
%
To construct the spaces $\{C_e\}_{e\in X_0(1)}$, 
we view $\F_2^{X'(1)}$ as the space $C^1(X',\F_2)$ of $1$-cochains on $X'$
with $\F_2$-coefficients (see \S\ref{subsec:hdx}).
We then take
\[
C_e=\{(h(\hat{e}_1),\dots,h(\hat{e}_m))\where h\in E\}\subseteq \F_2^m,
\]
where $E$ is a subspace of $C^1(X',\F_2)$   corresponding 
to the space with same name constructed in the iterative
process of \S\ref{subsec:finding-int-data-intro}
(or formally in Construction~\ref{CN:modification-process}).
We terminate the process when $\dim E$ becomes close
to $m$ to guarantee that $\dim E\ll m$. 
%
%
%
%
%
%
%
%

For the sake of simplicity, we only explain what
$E$ would be if the process were to end after $1$ step (which is not the case here,
see \S\ref{subsec:growth}).
Let $z_1,\dots,z_t\in Z^1(X',\F_2)$ be $1$-cocycles representing a basis of $ \HH^1(X',\F_2)$.
Then, after one step of the process, $E$ will be the $\F_2$-span
of $z_1,\dots,z_t$. The next steps of the process add more vectors to $E$.
We would like to choose $m$ be enough in advance so that
$\dim E\ll m$. It is expected that if Serre's Conjecture
on the congruence subgroup property (Conjecture~\ref{CJ:csp}) holds, then
$t=\dim\HH^1(X',\F_2)$   grows logarithmically in $m$.
(However,
choosing $X_0$ and $X'$ to be $X$ and one of the $X'_r$ from Theorem~\ref{TH:good-quotient}  
guarantees $t=O(1)$ as
$m$ grows.)

\paragraph*{Validity of The Construction.}

Provided that our assertion on  the logarithmic growth of $\dim\HH^1(X',\F_2)$
holds,  we have
the following:

\begin{thm}[{informal; see Theorem~\ref{TH:main-candidates}, Remark~\ref{RM:main-candidates}}]\label{TH:example-code}
	Under the previous assumptions, there is $q_0\in \N$ such that if $q\geq q_0$,
	then, with probability $1-o(1)$ as $m\to\infty$, we have the following:
	\begin{enumerate}[label=(\roman*)]
		\item The family $\{C_r\subseteq \Sigma^{n_r}\}_{r\in\N}$
	is a family of $2$-query LTCs with linear distance.
		The soundness of the testing does not depend on $m$.
		If Conjecture~\ref{CJ:into-hard} holds, then the codes in the family
		have constant rate.
		\item There is $\eta>0$
	such that each $C_r$ 
	admits a  linear-time decoding
	algorithm able to correct up to $\eta n_r$ errors.
	\end{enumerate}
	In particular, there is $m_0=m_0(q)$ such that for every $m\geq m_0$,
	there is a choice of $E$ for which both (i) and (ii) hold.
%
\end{thm}

The dependence on the logarithmic growth of $\dim\HH^1(X',\F_2)$
can be avoided by replacing the Ramanujan complexes
of \cite{Lubotzky_2005_explicit_constructions_of_Ramanujan_complexes}
with the complexes from Theorem~\ref{TH:good-quotient} below,
but we still rely on Conjecture~\ref{CJ:into-hard}
for having constant rate.


\subsection{Organization of The Paper}
\label{subsec:paper-organization}

The paper is divided into three chapters and includes two appendices.

Chapter~\ref{chap:foundations} sets the foundations for the connection between
expanding sheaves and codes:
The preliminary Section~\ref{sec:preliminaries}
recalls necessary facts about simplicial complexes.
The subject matter of Section~\ref{sec:sheaves} is sheaves
on simplicial complexes, their cohomology, and associated tools
such as the pushforward and pullback constructions.
In Section~\ref{sec:coboundary}, we introduce and discuss 
coboundary and cosystolic expansion of sheaves.
Section~\ref{sec:locally-min} concerns with  the  notion
of a \emph{locally minimal cochain} and the  expansion 
of such cochains as a mean to get cosystolic expansion.
In Section~\ref{sec:ltc-and-css}, we explain how sheaves
give rise to  codes with a tester and quantum CSS codes, and formalize the connection
between the expansion of the sheaf and various properties of the code.

Chapter~\ref{chap:tower} presents the tower paradigm:
Section~\ref{sec:cse-from-links}
presents a \emph{local-to-global principle} which allows us to establish cosystolic expansion
of sheaves
from information about their restrictions to proper links. 
This principle is applied to some examples
of cocycle codes in Section~\ref{sec:quo-aff-buildings}.
In Section~\ref{sec:rate}, we prove our \emph{rate conservation} result. 
Section~\ref{sec:recap} puts the previous results together to give 
the \emph{tower paradigm}, a framework
for constructing 
an  infinite family of good $2$-query LTCs from a sheaved high-dimensional expander.

Finally, Chapter~\ref{chap:initial-data} concerns with 
constructing sheaved complexes which can serve as candidates for the initial data of the tower paradigm.
In Section~\ref{sec:modifying}, 
we introduce an iterative process which takes   special
locally constant sheaves as input and  produces the desired candidates.
Section~\ref{sec:congruence} concerns with constructing   simplicial complexes
covered by affine buildings with some special properties, e.g., an infinite tower
of double coverings. These examples are then used in Section~\ref{sec:candidates}
to construct the locally constant sheaves required for the iterative process,
as well as other examples of interest.

Appendix~\ref{sec:comparison} explains the connection between sheaves on simplicial complexes
as defined here and the familiar sheaves on topological spaces. 
Appendix~\ref{sec:derived} shows that the sheaf cohomology we define 
in this work by elementary means is  
actually a \emph{right derived functor} and thus (from a mathematical point of view) deserves the
name ``cohomology''.

\chapter{Foundations}
\label{chap:foundations}

\section{Preliminaries}
\label{sec:preliminaries}

If not indicated otherwise, throughout this work,  simplicial complexes are finite,
and vector spaces are finite dimensional.
We always let $X$ denote a simplicial complex and $\F$   a field.

\subsection{Simplicial Complexes}
\label{subsec:complexes}

As usual, a simiplicial complex $X$ with vertex set $V=V(X)$
is a nonempty set consisting of finite subsets of $V$ such that  
$\{v\}\in X$
for all $v\in V$ and every subset of a set in $X$ is also in $X$.
Elements of $X$ are the faces of $X$ and elements of $V(X)$
are the vertices of $X$. A face with $k+1$
vertices is said to be of dimension $k$, or a $k$-face. 
The set of $k$-faces of $X$ is denoted $X(k)$.
Faces of dimension $1$ are called edges, faces of dimension $2$
are called triangles, and so on.
Note, however, that a $0$-face and a vertex are not the thing --- the $0$-face corresponding
to a vertex $v\in V(X)$ is the singleton $\{v\}$.
The dimension of $X$, denoted $\dim X$, is the maximal $k\in \{-1,0,\infty\}\cup\N$
for which $X(k)\neq\emptyset$.

A graph is a $1$-dimensional simplicial complex. The underlying graph
of a simplicial complex $X$ is $X(\leq 1):=X(-1)\cup X(0)\cup X(1)$.

If $(Y,\leq )$ is a partially ordered set, we will say that $Y$
is a simplicial complex if there is an isomorphism of partially
ordered set $(Y,\leq)\cong (X,\subseteq)$ with $X$ a simplicial complex.
We   then ascribe all the notation involving $X$ to $Y$ via this isomorphism;
the choice of the isomorphism will always be  be inconsequential.

The topological realization of a simplicial complex $X$ is denoted $|X|$.
We ascribe topological properties of $|X|$ to $X$, e.g.,
$X$ is said to be connected if $|X|$ is connected.
This condition is equivalent to saying
that the underlying graph of $X$ is connected.

Given $A\subseteq X$ and $z\in X$, we
write 
\[ A_{\supseteq z}=\{x\in A\suchthat x\supseteq z\}
\qquad
\text{and}
\qquad 
A_{\subseteq z}=\{x\in A\suchthat x\subseteq z\}.\]
In particular, $X(k)_{\supseteq z}$ (resp.\ $X(k)_{\subseteq z}$)
is the set of $k$-faces containing (resp.\ contained in) $z$. We further
let 
\[
A_z =\{x-z\where x\in A_{\supseteq z}\}.
\]
The set $X_z$ is a simplicial complex known as the \emph{link} of $X$ at $z$.
Note that $X_\emptyset=X$; 
when $z\neq \emptyset$, we call $X_z$ a \emph{proper link} of $X$.
We say that $X$ is \emph{strongly connected} if all the links of $X$ are connected.

Given $0\leq i<j$, we define the $(i,j)$-degree of $X$ to be
\[
D_{i,j}(X)=\max\{\#X(j)_{\supseteq z}\where z\in X(i)\},
\]
i.e., the largest possible number of $j$-faces containing a fixed $i$-face.
The degree of $X$ is
$D(X)=D_{0,\dim X}(X)$.
If every face of $X$ is contained in a $d$-face,
then the degree of $X$ is related to the $(i,j)$-degree by
\begin{equation}\label{EQ:degree-ratio}
D_{i,j}(X)\leq \schoose{d+1}{j+1} D_{i,d}(X)\leq \schoose{d+1}{j+1} D (X).
\end{equation}

An ordered face in $X$ is a face $x\in X$
together with a total ordering of its vertices. Ordered
faces will be written as tuples of vertices, e.g.\ $x=(v_0,\dots,v_i)$, which
indicates that $v_0<\dots<v_i$. 
We let
$X_{\ord}$ denote the set of ordered faces in $X$, and $X_{\ord} (k)$
 the subset of ordered $k$-faces. 
If $x\in X_{\ord}(k)$, then we write $x_i$ for the ordered face
obtained from $x=(v_0,v_1,\dots,v_k)$ by removing the vertex $v_i$.
We shall freely regard ordered faces
as non-ordered faces by forgetting the ordering.
If $x=(u_0,\dots,u_i),y=(v_0,\dots,v_j)\in X_{\ord}$ are ordered faces such that
$x\cap y=\emptyset$ and 
$x\cup y\in X$ (here we regarded $x$, $y$ as non-ordered faces),
then the concatenation $xy$ denotes the ordered
face $(u_0,\dots,u_i,v_0,\dots,v_j)$.

\subsection{Weights}
\label{subsec:weights}

Recall that
a simplicial complex $X$ is called
\emph{pure of dimension $d$} ($d\geq 0$), or a \emph{$d$-complex} for short, if every face of $X$
is contained in a $d$-face.
In this case, following \cite{Lubotzky_2016_expansion_of_buildings}, 
\cite{Evra_2016_cosystolic_expanders_arxiv_version}, \cite{Kaufman_2018_cosystolic_expanders} and other sources, we  
define the   \emph{canonical weight} of a $k$-face $x\in X(k)$ to be
\[
w(x)=w_X(x)={\textstyle {d+1\choose k+1}}^{-1}|X(d)|^{-1} |X(d)_{\supseteq x}|.
\]
Given $A\subseteq X(k)$, 
we also write   $w(A)=\sum_{x\in A}w(x)$.
The weight $w(x)$ is the probability of obtaining $x$ by 
choosing a $d$-face $y \in X(d)$ uniformly at random and then   choosing
a $k$-face of $y$ uniformly at random. This readily implies that $w(X(k))=1$ for all $-1\leq k\leq d$,
and
\begin{align}\label{EQ:weight-of-containing-cells}
w(X(\ell)_{\supseteq x})=\schoose{\ell+1}{k+1} w(x)
\end{align}
for all $-1\leq k\leq \ell\leq d$ and $x\in X(k)$.

\begin{example}\label{EX:regular-graph-weights}
	If $X$ is a $k$-regular graph with $n$ vertices, then $X$
	has $\frac{1}{2}nk$ edges and so the canonical weight function 
	of $X$ is given by
	\[
	w(x) = \left\{\begin{array}{cl}
	1 & x=\emptyset \\[2pt]
	\frac{1}{n} & x\in X(0) \\[2pt]
	\frac{2}{kn} & x\in X(1).
	\end{array}\right.
	\]
\end{example}

Suppose that $X$ is a $d$-complex
and let $z\in X(i)$. Then the link $X_z$
is a $(d-1-i)$-complex. It is straightforward to check
that  the canonical weight functions
of $X$ and   $X_z$ are related by   the formula
\begin{align}\label{EQ:weight-in-link}
w_{X }(x ) &=
\schoose{k+1}{i+1}  w_X(z)  w_{X_z}(x-z),
\end{align} 
which holds for all $x\in X(k)_{\supseteq z}$ and $k\geq i$. 
 

\subsection{Coverings}
\label{subsec:coverings}

Let $X$ and $Y$ be    simplicial complexes. A morphism of simplicial complexes
from $Y$ to $X$   is a function $f:V(Y)\to V(X)$ such that $f(y):=\{f(v)\where v\in y\}\in X$ for all 
$y\in Y$. The morphism $f:Y\to X$ is \emph{dimension-preserving} if $\dim f(y)=\dim y$
for all $y\in Y$, and a \emph{covering map} if
$f$ is onto and it induces  a bijection from 
$Y_{\supseteq y}$ to $X_{\supseteq f(y)}$ for all $y\in Y(0)$. 
The latter is equivalent to saying that the continuous map $|f|:|Y|\to |X|$
is a covering map of topological spaces.
Covering maps
are dimension-preserving.

If there exists a covering map $f:Y\to X$, we say that $Y$ covers $X$. In
this case, if $Y$ is a $d$-complex if and only if $X$. In addition, if
$Y$ is connected, then $Y$ is strongly connected if and only if $X$ is.

A covering $f:Y\to X$ is said to be of \emph{degree $e$}
if $|f^{-1}(v)|=e$ for every $e\in V(X)$; we then write $\deg f=e$
or $[Y:X]=e$ (suppressing $f$). In this case, for every non-empty
face of $X$, there are exactly $e$ faces in $Y$ which map to it under $f$.
If $X$ is connected and $f:Y\to X$ is a covering, then the size of $|f^{-1}(v)|$ is independent of $v$,
so every covering of a connected simplicial complex 
has a well-defined degree. A covering map of degree $2$ is called a 
\emph{double covering}.

\medskip

Let $G$ be a group. A \emph{$G$-Galois covering} of simplicial
complexes consists of a covering map $p:Y\to X$ and $G$-action $G\times Y\to Y$
such that 
\begin{enumerate}[label=(\arabic*)]
	\item for every $g\in G$, the map $v\mapsto gv:V(Y)\to V(Y)$  is an     automorphism of $Y$,
	\item $p(gy) =p(y)$ for all $y\in Y$, and
	\item for every $x\in X$, the action of $G$ on $Y$
	restricts to an action on $p^{-1}(x)$, and $G$ acts simply and transitively
	on $p^{-1}(x)$.
\end{enumerate} 
We will often simply say that $p:Y\to X$ is a $G$-Galois covering,
suppressing the $G$-action.

Condition (3) implies that  a $G$-Galois covering must be of degree $|G|$.
The converse is false, however --- $X$ may admit coverings of degree $|G|$ which cannot 
be realized a $G$-Galois coverings. In more detail, if both $X$ and $Y$
are connected and $p:Y\to X$ is a covering map, then, by fixing a base point
$y\in |Y|$, we may realize $\pi_1(X):=\pi_1(|X|,p(y))$  
as a subgroup of $\pi_1(Y)=\pi_1(|Y|,y)$. The covering $p:Y\to X$ can be realized
as a $G$-Galois covering 
if and only if $\pi_1(Y)$ is a normal subgroup of $\pi_1(X)$. In this case,
$G\cong \pi_1(X)/\pi_1(Y)$, and the evident map $G\to \Aut(Y/X):=\{f:Y\to Y\suchthat p\circ f=p\}$
is an isomorphism.

\begin{example}\label{EX:coverings-basic}
	(i) Let $C_2$ denote the cyclic group with two elements. Every
	double covering $p:Y\to X$ is $C_2$-Galois in an  unique way.
	Simply let the nontrivial element of $C_2$ act on $Y$ by sending $v\in V(Y)$
	to the over vertex of $Y$ mapping to $p(v)$.
	
	(ii) If $p:Y\to X$ is a covering map and $Y$ is contractible (and hence connected), then
	$|Y|$ must coincide with the universal covering of $|X|$.
	This means that $p:Y\to X$ is Galois with Galois group
	$\Aut(Y/X)\cong \pi_1(X)$.
	
	(iii) Let $G$ be any group, let $Y$ denote the disjoint
	union of $|G|$ copies of $X$ and give it the $G$-action permuting these copies.
	Let $p:Y\to X$ be the map which restricts to the identity
	on each copy of $X$. Then $p:Y\to X$ is a $G$-Galois covering
	called the \emph{trivial $G$-Galois covering} of $G$. Note that $Y$
	is not connected if $|G|>1$. Moreover, the evident map $G\to \Aut(Y/X)$
	is not an isomorphism if $|G|>2$.
\end{example}

\subsection{Skeleton and Spectral Expansion}
\label{subsec:skeleton}


Let $X$ be a $d$-complex.
Given a set of $0$-faces $S\subseteq X(0)$, we write
$E(S)$ for the set of edges in $X$ having both of their $0$-faces in $S$.
Recall from \cite[Definition~2.5]{Kaufman_2018_cosystolic_expanders} that   $X$ is said to be an 
\emph{$\alpha$-skeleton expander} ($\alpha\in [0,\infty)$) 
if for every   $S\subseteq X(0)$, we have
\[
w(E(S))\leq w(S)^2+\alpha w(S),
\]
where $w$ is the canonical weight function of $X$ defined in \S\ref{subsec:weights}.
The complex $X$ is considered more skeleton expanding the smaller $\alpha$ is.

Let $C^0(X,\R)$ denote the $\R$-vector space
of functions $f:X(0)\to \R$.
Following 
\cite{Oppenheim_2015_vanishing_of_cohomology},
\cite[\S2A]{First_2021_weighted_mixing_lemmas_preprint}
and similar sources, we define the weighed adjacency operator of
$X$ to be the linear operator $\calA:C^0(X,\R)\to C^0(X,\R)$ 
defined by $(\calA f)(x)=\sum_{e\in X(1)_{\supseteq x}} \frac{w(e)}{2w(x)}f(e-x)$
($f\in C^0(X,\R)$, $x\in X(0)$).
For example, if $X$ is a $k$-regular graph, then $\calA$ is the usual
adjacency operator of $X$ scaled by a factor of $\frac{1}{k}$. 
Let $C^0_{\circ}(X,\R)$ denote the subspace of $C^0(X,\R)$ consisting of functions
$f$ with $\sum_{v\in X(0)} f(v)=0$. 
Given an interval $I\subseteq [-1,1]$, we say that the underlying 
weighted graph of 
$X$ is a \emph{spectral $I$-expander}
if the spectrum of $\calA:C^0_{\circ}(X,\R)\to C^0_{\circ}(X,\R)$
is contained in $I$.\footnote{
	Caution: The spectral expansion
	of the underlying weighted graph of $X$
	takes into account the canonical weight function of
	$X$ and therefore depends on the higher-dimensional faces of $X$.
}

It follows readily from the Weighted Expander Mixing Lemma
\cite[Theorem~3.3(ii)]{First_2021_weighted_mixing_lemmas_preprint} that
if the underlying weighted graph of $X$ is a $[-1,\lambda]$-spectral expander for some
$\lambda\in [0,1]$, then $X$ is a $\lambda$-skeleton expander. 

\medskip

Given $k\in\{-1,0,\dots,d-1\}$ and a $d$-complex $X$,
we   say that $X$ is a \emph{$k$-local $[-1,\lambda]$-spectral expander}
(resp.\ \emph{$k$-local $\alpha$-skeleton expander}) if,
for every $z\in X(k)$,
the underlying weighted graph of $X_z$ is a $[-1,\lambda]$-spectral expander
(resp.\ $X_z$ is an $\alpha$-skeleton expander).

\subsection{Buildings}
\label{subsec:buildings}

Buildings are possibly-infinite connected simplicial complexes which have 
certain remarkable  structural properties.
They will play a role in some of the examples we consider later on.
Contrary to our standing assumption
that simplicial complexes are finite, \emph{buildings can 
be infinite  if not otherwise stated}. 

We omit the techincal definition of a buidling, which can be found in \cite{Abramenko_2008_Buildings},
for instance, and satisfy with recalling here some facts about buildings needed for this work. 
We shall only consider buildings $Y$ admitting a \emph{strongly transitive action}
in the sense of \cite[\S6.1.1]{Abramenko_2008_Buildings},
and the word ``building'' will always mean a ``building admitting a strongly transitive action''. 
This means that there is a group
$G$ acting on $Y$ and satisfying the transitivity properities listed in {\it op.\ cit.}.

To every building $Y$ one can attach a Coxeter diagram $T=T(Y)$, called the type of $Y$, which is a
finite undirected graph whose edges
are given labels from the set $\{3,4,5,\dots\}\cup\{\infty\}$. The complex $Y$ is pure of dimension 
$|V(T)|-1$. In fact, there is a labeling $t:V(Y)\to V(T)$
such that every face in $Y$ consists of vertices with different labels.
Coxeter diagrams appearing on the list in
\cite[p.~50]{Abramenko_2008_Buildings}
are called \emph{spherical}, whereas the ones
described in \cite[Remark~10.33(b)]{Abramenko_2008_Buildings} 
are called \emph{affine}.
We call $Y$ spherical or affine if $T$ is spherical or affine, respectively.
If $Y$ is spherical, then $|Y|$ is homotopy equivalent
to a bouquet of   spheres of dimension $\dim Y$. 
If $Y$ is affine, then $Y$ is contractible.
Finite buildings are spherical.

It will be convenient  to treat any nonemtpy $0$-dimensional simplicial complex 
as a spherical building  of dimension $0$ with Coxeter diagram consisting of a single point.\footnote{
	The correct analogue of a $0$-dimensional building with
	a strongly transitive action a \emph{Moufang set}, but this will not be needed in this work.
}

Let $Y$ be a $d$-dimensional building.
If $z\in Y$ is a face of dimension $\leq d-1$, then the link $Y_z$ is also a building.
If $Y$ is spherical or affine and $z\neq \emptyset$, then $Y_z$ is spherical (in both cases).
The building $Y$ is called $q$-thick ($3\leq q\in\N$) if every $x\in Y(d-1)$ is contained in at least $q$ 
$d$-faces. 

\begin{example}\label{EX:An-building}
	Let $\F$ be a field and let $n\in\N$.
	Write $A_n(\F)$ for the incidence complex  of nontrivial subspaces of $\F^{n+1}$.
	That is,
	the 
	vertices of $A_n(\F )$ are the nonzero proper subspaces of $\F^{n+1}$
	and its faces are the sets of vertices which are totally ordered by inclusion.
	Then $A_n(\F)$ is an $(n-1)$-dimenional spherical building. Its type is $A_n$ --- the
	Coxeter diagram consisting of a single path with $n$ vertices and having all edges labeled $3$.
	If $|\F|\geq q$, then $A_n(\F)$ is $(q+1)$-thick.
	
	When $n=2$, the graph $A_2(\F)$ is nothing but the incidence graph
	of points and lines in the $2$-dimensional projective space over $\F$.
\end{example}

We refer the reader to \cite[\S6.9]{Abramenko_2008_Buildings} 
and \cite{Abramenko_2002_models_for_affine_buildings}
for the description of some affine buildings. 
More generally, Bruhat and Tits \cite{Bruhat_1972_groupes_reductifs_corps_local_I}
(see also \cite{Tits_1979_reductive_groups_over_local_fields}) showed that one can attach
to every almost-simple simply-connected algebraic group $\bfG$ over a local field
$F$ an affine building $Y$ equipped with a strongly transitive
action by the group $G=\bfG(F)$. For example, given a prime
number $p$, the group $G$ can be taken to be $\nSL{\Q_p}{n}$
(with $\bfG=\uSL_n$, $F=\Q_p$), in which case the 
corresponding affine building is the one described in \cite[\S6.9]{Abramenko_2008_Buildings}. 
It has type $\tilde{A}_{n-1}$ (a cycle graph on $n$ vertices with
all edges labeled $3$), dimension $n-1$
and it is $(p+1)$-thick. 
Moreover, it is locally finite, i.e., every nonempty face is contained in finitely many faces. 

\medskip 

We will be particularly interested in finite spherical buildings,
and finite simplicial complexes $X$ admitting a covering map $f:Y\to X$
with $Y$ being an affine building. (In the latter case, $|Y|$ is the universal covering of $|X|$,
because $|Y|$ is contractible.)
Such complexes $X$ are good spectral expanders,
and by \S\ref{subsec:skeleton}, also good skeleton expanders. Formally:

\begin{thm}[{\cite[Theorem~7.2]{First_2021_weighted_mixing_lemmas_preprint}}]
\label{TH:skeleton-exp-building}
	Let $q\in\{3,4,5,\dots\}$ 
	and let $X$ be a (finite) simplicial complex such that one of the following holds:
	\begin{enumerate}[label=(\arabic*)]
		\item $X$ is a finite $q$-thick spherical building of dimension $d\geq 1$.
		\item There is a covering map $f:Y\to X$ such that $Y$ is a $q$-thick affine building
		of dimension $d\geq 2$.
	\end{enumerate}
	Let $L$ denote
	the set of edge	labels appearing in the Coxeter diagram of the building
	mentioned in (1) or (2), and let $m=\max (L\cup \{2\})$.\footnote{
		We   have $m\leq 8$ if (1) holds and $m\leq 6$ if (2) holds,
		see \cite[Chapter~9]{Abramenko_2008_Buildings}.
	}
	Write $d=\dim X$ and suppose that
	$q\geq d^2(m-2)$. Then the underlying graph of 
	$X$ is a $[-1,\alpha]$-spectral expander
	(and thus $X$ is an $\alpha$-skeleton expander) for
	\[\alpha = \frac{\sqrt{m-2}}{\sqrt{q}-(d-1)\sqrt{m-2}}.\]
\end{thm}

The Ramanujan complexes of \cite{Lubotzky_2005_Ramanujan_complexes_Ad} 
and \cite{Li_2004_ramanujan_hypergraphs} are
simplicial complexes covered by affine buildings of type $\tilde{A}_n$.
However, the 
spectral (resp.\ skeleton) expansion of their underlying weighted graph
is much better than the bound provided by Theorem~\ref{TH:skeleton-exp-building}. 
We demonstrate this in the $2$-dimensional case.

\begin{prp}\label{PR:skeleton-exp-Ramanujan}
	Let $Y$ be the affine building of $\nSL{F}{3}$, where $F$ is a local non-archimedean field,
	and let $q$ denote the number of elements in the residue field of $F$.
	(The thickness of $Y$ is $q+1$.)
	If $X$ is a simplicial complex covered
	by   $Y$ and moreover    a Ramanujan complex
	in the sense of \cite{Lubotzky_2005_Ramanujan_complexes_Ad} (see
	also \cite{Carwright_2003_Ramanujan_geometries_An}), then
	the underlying weighted 
	graph of $X$ is a $[-1,\frac{3q}{q^2+q+1}]$-spectral expander.
\end{prp}

\begin{proof}
	The links of $X$ are spherical buildings of the form $A_2(\F_q)$ (notation as in Example~\ref{EX:An-building}).
	This means that every vertex is contained in $2(q^2+q+1)$ edges and $(q+1)(q^2+q+1)$ triangles.
	Now, in the notation of \cite{Lubotzky_2005_Ramanujan_complexes_Ad}, the weighted adjacency
	operator of $X$ is $\frac{1}{2(q^2+q+1)}(A_1+A_2)$. When $X$ is Ramanujan,
	the joint spectrum of $(A_1,A_2)$ was computed in \cite[Theorem~2.11]{Lubotzky_2005_Ramanujan_complexes_Ad}.
	It follows from that computation that the underlying graph of $X$ is a $[-1,\frac{3q}{q^2+q+1}]$-spectral
	expander.
\end{proof}

\section{Sheaves }
\label{sec:sheaves}

In this section we introduce sheaves on simplicial complexes and various related
notions. Until the end of the section, 
\emph{simplicial complexes are allowed to be infinite},
and   $\F$ denotes a field.

\subsection{Sheaves on Simplical Complexes}
\label{subsec:sheaves}

Let $X$ be a simplicial complex.
A \emph{sheaf} $\calF$ on  
$X$ consists of 
\begin{enumerate}
	\item[(1)] an abelian group $\calF(x)$ for every $ x\in X-\{\emptyset\}$, and
	\item[(2)] a group homomorphism $\res_{y\from x}^{\calF}:\calF(x)\to \calF(y)$ for 
	all $\emptyset\neq x\subsetneq y\in X$,
\end{enumerate}
subject to the condition 
\begin{equation}\label{EQ:sheaf-cond} 
\res_{z\from y}^{\calF}\circ \res_{y\from x}^{\calF}=\res_{z\from x}^{\calF}
\end{equation}
for all $\emptyset\neq x\subsetneq y\subsetneq z\in X$.
We also say that $(X,\calF)$ is a sheaved  simplicial  complex.
Elements of $\calF(x)$ are called $x$-\emph{sections},
and
the homomorphisms $\res^{\calF}_{y\from x}$ are called \emph{restriction maps}.
If there is no risk of confusion, 
we will often abbreviate $\res_{y\from x}^{\calF}f$ ($f\in \calF(x)$)
to $\res_{y\from x} f$, $f|_{x\to y}$ or $f|_y$. Note that
condition \eqref{EQ:sheaf-cond} is vacuous if $\dim X\leq 1$.

An \emph{augmented sheaf} $\calF$ on $X$ is defined similarly, except
we also include the empty face $\emptyset$. That is, $\calF(\emptyset)$
and $\res_{y\from \emptyset}^{\calF}$ are defined, and \eqref{EQ:sheaf-cond} is required
to told with $x=\emptyset$ as well.
We may  regard any sheaf $\calF$ as an augmented sheaf by setting $\calF(\emptyset)=0$
and $\res_{y\from \emptyset}^{\calF}=0$ for all $y\in X-\{\emptyset\}$, 
so all the statements we prove for augmented sheaves
also apply to sheaves.

A \emph{sheaf of $\F$-vector spaces}, or an \emph{$\F$-sheaf} for short, on $X$
is a sheaf $\calF$ on $X$
such that  
$\calF(x)$ is an $\F$-vector space for all $x\in X-\{\emptyset\}$
and    the restriction maps
of $\calF$ are $\F$-linear.
One can define in the same manner  (augmented) sheaves of groups, rings, modules, sets, and so on.

\begin{example}\label{EX:sheaves-basic-examples}
	Let $X$ be a simplicial complex.

	(i) Given an abelian group $A$,
	we   define a sheaf $\calF_A$ on $X$
	by setting $\calF_A(x)=A$ for all $x\in X-\{\emptyset\}$
	and $\res_{y\from x}^{\calF_A}=\id_A$
	for all $\emptyset\neq x\subsetneq y\in X $.
	The   sheaf $\calF_A$ is  called the \emph{constant sheaf}
	associated to $A$.
	Abusing the  notation, we will usually denote $\calF_A$
	simply as $A$, or $A_X$.
	
	(ii) Continuing (i), one can also define  an augmented
	sheaf $\calF'_A$ on $X$
	by setting $\calF'_A(x)=A$ and $\res_{y\from x}^{\calF'_A}=\id_A$
	for all $x,y\in X $ with $x\subsetneq y$. We call
	$\calF'_A$ the \emph{constant augmented sheaf} on $X$
	and denote it by $\aug{A}$ when $X$ is clear from the context.
	
	(iii) Fix arbitrary abelian groups $(A_x)_{x\in X }$
	and set $\calF(x)=A_x$ and $\res^{\calF}_{y\from x}=0\in\Hom_{\Z}(A_x,A_y)$ for all $x,y$.
	Then $\calF$ is an augmented sheaf on $X$ (albeit,  not a very interesting one).
	If $A_\emptyset=0$, then we may regard $\calF$ as a sheaf.
	
	(iv) If $\calF$ and $\calG$ are   sheaves on $X$,
	then one can form the  product   sheaf
	$\calF \times \calG$ defined 
	by $(\calF\times \calG)(x)=\calF(x)\times \calG(x)$ and
	$\res_{y\from x}^{\calF\times \calG}=\res_{y\from x}^{\calF}\times \res_{y\from x}^{\calG}$.
	
	(v) Let $\calF$ be a sheaf on $X$. Suppose that we are given   subgroups
	$\calG(x)\subseteq \calF(x)$ for all $x\in X-\{\emptyset\}$
	such that $\res^{\calF}_{y\from x}(\calG(x))\subseteq \calG(y)$
	for all $\emptyset\neq x\subsetneq y\in X$. Then the collection $\{\calG(x)\}_{x\in X-\{\emptyset\}}$ 
	can be made into a sheaf $\calG$
	on $X$ by setting $\res^{\calG}_{y\from x}=\res^{\calF}_{y\from x}|_{\calG(x)}$.
	We call such $\calG$ a \emph{subsheaf} of $\calF$.  
	
	(vi) If $\calG$ is a subsheaf of $\calF$, then we   define the \emph{quotient sheaf}
	$\calF/\calG$ by setting $(\calF/\calG)(x)=\calF(x)/\calG(x)$ and  
	$\res^{\calF/\calG}_{y\from x}(f+\calG(x)) = (\res^{\calF}_{y\from x} f)+\calG(y)$
	for all $x\in X-\{\emptyset\}$ and $f\in \calF(x)$. 
	
	(vii) Let $\calF$ be an $\F$-sheaf on $X$
	and let $\bbK$ be a field extension of $\F$.
	The \emph{base change} of $\calF$ from $\F$ to $\K$
	is the $\K$-sheaf $\calF_\K$ on $X$ 
	determined by
	$\calF_{\K}(x)=\calF(x)\otimes_\F \K$ and $\res^{\calF_\K}_{y\from x}=\res^{\calF}_{y\from x}\otimes \id_{\K}$.
\end{example}

Examples (iii)-(vii) generalize verbatim to augmented sheaves.	


\begin{example}\label{EX:sheaves-from-reps}
	Let $X$ be a connected simplicial complex  and let $\F$ be a field.
	Then every representation $\rho:\pi_1(X)\to \nGL{\F}{n}$
	gives rise to an $\F$-sheaf $\calF=\calF_\rho$.
	To define it, we must first introduce some general notation.	
	
	Write $\Gamma=\pi_1(X)$ and let $\pi:Y\to X$ be the universal covering of $X$.
	Then we can (non-canonically) 
	identify $\Gamma$ with the group of deck transformations of $\pi:Y\to X$
	(i.e., the group of automorphisms $g:Y\to Y$ satisfying $\pi\circ g=\pi$).
	Then $\Gamma$ acts freely on $Y$ via simplicial automorphisms, and for every non-empty
	$x\in X$, the preimage $\pi^{-1}(x)$ is an orbit under $\Gamma$.
	
	For every nonempty $x\in X$, choose some representative 
	\[\hat{x}\in \pi^{-1}(x);\]
	equivalently, $\{\hat{x}\where x\in X-\{\emptyset\}\}$
	is a set of representatives for $\Gamma\leftmod(Y-\{\emptyset\})$.
	Suppose  that $\emptyset\neq x\subsetneq x'\in X$.
	Then it may not be the case that $\hat{x}\subseteq \hat{x}'$.
	However, since $\pi:Y\to X$ is a covering, there is a unique $y \in Y$
	such that $\hat{x}\subseteq y$ and $\pi(y)=x'=\pi(\hat{x}')$.
	This means that there is a unique element $\gamma\in \Gamma$
	such that $\gamma y=\hat{x}'$. We denote this $\gamma$ by
	\[
	\gamma(x',x).
	\]
	It is routine to check that if $\emptyset\neq x\subsetneq x'\subsetneq x''\in X$,
	then we have
	\begin{equation}\label{EQ:gamma-compatibility}
	\gamma(x'',x')\gamma(x',x)=\gamma(x'',x).
	\end{equation}

	Now, given a representation $\rho:\Gamma=\pi_1(X)\to \End_{\F}(V)$,
	where $V$ is an $\F$-vector space, we may define an $\F$-sheaf $\calF=\calF_\rho$  on 
	$X$ by setting 
	\begin{itemize}
		\item $\calF(x)=V$ for all $\emptyset\neq x\in X$, and
		\item $\res^{\calF}_{x'\from x}=\rho(\gamma(x',x)):V\to V$ for all $\emptyset\neq x\subsetneq x'\in X$. 
	\end{itemize}
	It follows readily from \eqref{EQ:gamma-compatibility} that $\calF $ is a sheaf.
	While $\calF(x)=V$ for every $x\in X-\{\emptyset\}$, in general, $\calF$ is not
	the constant sheaf $V_X$ of Example~\ref{EX:sheaves-basic-examples}(i). (In fact, $\calF$ is \emph{isomorphic}
	to $V_X$ if and only if 
	$\rho$ 
	is the trivial representation of $\Gamma$ on $V$.)
\end{example}

If $\calF$ and $\calG$ are two   sheaves on $X$, then a \emph{morphism} $\vphi:\calF\to \calG$
consists of a collection     of
abelian group homomorphisms   $\{\vphi_x:\calF(x)\to \calG(x)\}_{x\in X-\{\emptyset\} }$
which are compatible with the restriction maps, namely,
\[
\vphi_y\circ \res^{\calF}_{y\from x}=\res^{\calG}_{y\from x}\circ \vphi_x 
\]
for all $ \emptyset\neq  x\subseteq y\in X$. 
The collection of all morphisms from $\calF$ to $\calG$ forms an abelian group
with addition given by $\vphi +\vphi' = (\vphi_x +\vphi'_x)_{x\in X-\{\emptyset\} }$
The composition of $\vphi$ with another morphism $\psi:\calG\to \calH$
is $\psi\circ \vphi := (\psi_x\circ \vphi_x)_{x\in X-\{\emptyset\} }$.
We call $\vphi$ an \emph{isomorphism} if each $\vphi_x$ is an isomorphism. If
there is an isomorphism $\vphi:\calF\to \calG$, we say that $\calF$
and $\calG$ are isomorphic and write $\calF\cong \calG$.

Given a morphism $\vphi:\calF\to \calG$ of sheaves on $X$, its kernel,   $\ker \vphi$,
is the subsheaf of $\calF$ determined by $(\ker\vphi)(x)=\ker(\vphi_x:\calF(x)\to \calG(x))$,
its image,   $\im \vphi$, is the subsheaf of $\calG$
determined by $(\im\vphi)(x)=\im(\vphi_x:\calF(x)\to \calG(x))$, and its cokernel, $\coker \vphi$,
is the sheaf $\calG/\im \vphi$.
We call $\vphi$ injective (resp.\ surjective) if $\ker \vphi$ is the zero subsheaf of $\calF$
(resp.\ $\im \vphi=\calG$). 

Morphisms of augmented sheaves, their kernels, images and cokernels are defined in
the same manner, by
including the empty face.

Morphisms of   sheaves of $\F$-vector spaces (resp.\ rings, groups, etc.)  are defined
similarly with the extra requirement that each $\vphi_x$ is   $\F$-linear (resp.\ a ring homomorphism,
a group homomorphism, etc.). The kernel, image  and cokernel  of 
a morphism of $\F$-sheaves are $\F$-sheaves as well.

\begin{remark}\label{RM:sheaf-category}
	The class of  sheaves (resp.\   $\F$-sheaves) 
	on $X$ together with the morphisms
	just defined is an \emph{abelian category}, denoted
	$\Sh (X)$.  Similarly, augmented sheaves on $X$ also form an abelian category.
\end{remark}

For the relation between the   sheaves defined here and the well-known
notion of a sheaf on a topological space, see Appendix~\ref{sec:comparison}.

\subsection{Sheaf Cohomology}
\label{subsec:sheaf-coh}

Sheaf cohomology generalizes ordinary cohomology of simplicial complexes with coefficients
in an abelian group. It is defined as follows.


Let $\calF$ be an augmented sheaf on a simplicial complex $X$.
Recall (\S\ref{subsec:complexes}) 
that $X_{\ord}$ denotes the set of ordered faces in $X$.
For every
$  k\in\N\cup\{-1,0\}$,
define 
\[\tilde{C}^k(X,\calF)=\prod_{ {x}\in X_{\ord}(k)}\calF(x)  \]
(we forget the ordering of $x$ in the expression ``$\calF(x)$'').
Given $f\in \tilde{C}^i(X,\calF)$ and $ {x}\in X_{\ord}(k)$,
we write the $ {x}$-coordinate of $f$ as $f( {x})\in\calF(x)$.
The group of \emph{$\calF$-valued $k$-cochains} is 
\[
C^k(X,\calF)=
\{f\in \tilde{C}^k(X,\calF)\suchthat
f(\pi  {x})=\mathrm{sgn}(\pi)  f( {x})
~\text{for all}~
\pi\in \Sigma_{\{0,\dots,k\}},  {x}\in X_{\ord}(k)\},
\]
where the the permutation group $\Sigma_{\{0,\dots,k\}}$
acts on $X_{\ord}(k)$ by permuting the vertex ordering of every
ordered $k$-face $x=(v_0,\dots,v_k)$.

The \emph{coboundary map}
$
d_k=d_k^{\calF}:C^k(X,\calF)\to C^{k+1}(X,\calF)
$ is defined
by
\[
(d_k f)( {y})=\sum_{i=0}^{k+1}(-1)^i\res_{y\from y_i}f( {y}_i),
\]
where the ordered face $ {y}_i$ is obtained from $ {y}=(v_0,\dots,v_{k+1})$ by removing $v_i$. 
It is routine to check that $d_k f$ is  in $C^{k+1}(X,\calF)$ and
$
d_{k+1}\circ d_k=0
$. The latter is equivalent to saying that
\[
0\to C^{-1}(X,\calF)\xrightarrow{d_{-1}}
C^{0}(X,\calF)\xrightarrow{d_{0}}
C^{1}(X,\calF)\xrightarrow{d_{1}}
\cdots
\]
is a cochain complex. Note that
$C^{-1}(X,\calF)=0$ if $\calF$ is a sheaf.
The
\emph{$\calF$-valued $k$-cocycles} and \emph{$\calF$-valued
$k$-coboundaries} are
\[
Z^k(X,\calF)=\ker d_k\qquad\text{and}\qquad 
B^k(X,\calF)=\im d_{k-1},
\]
respectively, with the convention that $d_{-2}=0$.
The \emph{$k$-th cohomology group} of $\calF$ is
\[
\HH^k(X,\calF):=Z^k(X,\calF)/B^k(X,\calF).
\]
The cohomology class represented by $f\in Z^k(X,\calF)$
is denoted $[f]$ or $[f]_{\calF}$.

If $\calF$ is a sheaf (i.e.\ $\calF(\emptyset)=0$),
then $B^0(X,\calF)$ is $0$ by definition,
and thus    $\HH^0(X,\calF)=Z^0(X,\calF)$.
The elements of $Z^0(X,\calF)$
consist of families
$(f(x))_{x\in X(0)}\in \prod_{x\in X(0)} \calF(x)$
such that $f(\{u\})|_{\{u,v\}}=f(\{v\})|_{\{u,v\}}$ for every edge $\{u,v\}\in X(1)$.
They are
called the \emph{global sections} of $\calF$.

\begin{example}
	(i) Let $\calF$ be the augmented sheaf constructed in Example~\ref{EX:sheaves-basic-examples}(iii).
	Then $\HH^k(X,\calF)=C^k(X,\calF)\cong \prod_{x\in X(k)} \calF(x)$, because
	$d_k=0$ for all $k$.
	
	(ii) If $\calF$
	and $\calG$ are sheaves on $X$,
	then we have a canonical isomorphism
	$C^k(X,\calF\times \calG)\xrightarrow{\sim} C^k(X,\calF)\times C^k(X,\calG)$,
	which restricts to isomorphisms
	$Z^k(X,\calF\times \calG)\xrightarrow{\sim} Z^k(X,\calF)\times Z^k(X,\calG)$
	and
	$B^k(X,\calF\times \calG)\xrightarrow{\sim} B^k(X,\calF)\times B^k(X,\calG)$;
	the details are left to the reader.
	Consequently, there is a canonical isomorphism
	$\HH^k(X,\calF\times \calG)\cong \HH^k(X,\calF)\times\HH^k(X,\calG)$. 
\end{example}

\begin{remark}\label{RM:ordered-cohomology}
	Fix a linear ordering $L$ on the vertices of $X$.
	Then $L$ induces an ordering on the vertices of every face $x\in X$; 
	we write $x_L$ to denote $x$ endowed with this ordering.
	We can now identify $C^k(X,\calF)$ with $C_L^k(X,\calF):=\prod_{x\in X}\calF(x)$
	by mapping $f\in C^k(X,\calF)$ to $(f(x_L))_{x\in X(k)}\in C^k_L(X,\calF)$. 
	It is straightforward to check that under this identification,
	the coboundary map $d_k$ corresponds to $d_{k,L}:C_L^k(X,\calF)\to C_L^{k+1}(X,\calF)$
	determined by
	\begin{equation}\label{EQ:d-k-L}
		(d_{k,L} f) (y)=\sum_{x\in X(k)_{\subseteq y}} [y:x]_L\res_{y\from x} f(x) 
	\end{equation}
	where $y\in X(k+1)$ and 
	$
	[y:x]_L:=(-1)^i$ for the unique $i\in\{0,\dots,k+1\}$ 
	such that $x_L$ is obtained from $y_L=(v_0,\dots,v_{k+1})$
	by removing $v_i$. 
	Consequently, the cohomology of 
	$\calF$ can be computed using the the cochain complex
	\[
	0\to 
	C^{-1}_L(X,\calF)\xrightarrow{d_{-1,L}} 	
	C^0_L(X,\calF)\xrightarrow{d_{0,L}} C^1_L(X,\calF)
	\xrightarrow{d_{1,L}} C^2_L(X,\calF) 
	\xrightarrow{d_{2,L}} \dots 
	\]
	Since for $x\in X(k)$, the factor $\calF(x)$ occurs once in $C^k_L(X,\calF)$
	and $(k+1)!$ times in $\tilde{C}^k(X,\calF)$, it is sometimes 
	convenient to use 
	$C^k_L(X,\calF)$ instead of $C^k(X,\calF)$. The disadvantage of 
	defining $\HH^k(X,\calF)$ using 
	the chain complex $C^\bullet_L(X,\calF)$ is the ostensible dependency
	on $L$.
	
	If $\calF$ is augmented $\F_2$-sheaf,
	then the factor $[y:x]$ in \eqref{EQ:d-k-L} has no effect, and can be removed.
	As a result, $d_{k,L}$ is independent of $L$, so
	the   isomorphism $C^k(X,\calF)\cong C^k_L(X,\calF)=\prod_{x\in X(k)}\calF(x)$
	is also independent of $L$.
\end{remark}

By comparing the description of $\HH^k(X,\calF)$ in 
Remark~\ref{RM:ordered-cohomology} and the definition of the singular
cohomology of $|X|$ with coefficients in an abelian group $A$,
we see that   the cohomology
of the constant (augmented) sheaf associated to  $A$
(Example~\ref{EX:sheaves-basic-examples}) 
is isomorphic to the (reduced) singular cohomology of $|X|$ with coefficents in $A$.
We record this observation in the following corollary.

\begin{cor}\label{CR:cohomology-comparison}
	Let $A$ be an abelian group regarded
	as a constant sheaf on $X$ (see Example~\ref{EX:sheaves-basic-examples}),
	and let $i\geq 0$.
	Then $\HH^i(X,A)\cong \HH^i(|X|,A)$, where   the right hand is the
	singular cohomology of $|X|$ with coefficients in $A$.
	Likewise, $\HH^i(X,\aug{A})\cong \tilde{\HH}^i(|X|,A)$, 
	where the right hand side
	denotes the reduced singular cohomology of $|X|$ with coefficients in $A$.
\end{cor}

As usual, a short exact sequence of sheaves on $X$ is a diagram 
\[0\to \calF\xrightarrow{\vphi} \calG \xrightarrow{\psi} \calH\to 0\]
of sheaves on $X$ such that $\vphi$ is injective, $\psi$ is surjective, and $\im \vphi=\ker\psi$.
In this case, there is a long cohomology   exact sequence of abelian groups
\begin{align}\label{EQ:long-coh-exact-seq}
	0 \to & \HH^0(X,\calF)\xrightarrow{\vphi_*} \HH^0(X,\calG) \xrightarrow{\psi_*} \HH^0(X,\calH) 
	\xrightarrow{\delta_0} \\
	&\HH^1(X,\calF)\xrightarrow{\vphi_*} \HH^1(X,\calG) \xrightarrow{\psi_*} \HH^1(X,\calH) 
	\xrightarrow{\delta_1}\cdots.\nonumber
\end{align}
The map $\vphi_*:\HH^i(X,\calF)\to \HH^i(X,\calG)$ is defined
by sending the cohomology class represented
by $f\in Z^i(X,\calF)$ to the one represented by
$(\vphi_x(f(x))_{x\in X_{\ord}(i)}\in Z^i(X,\calG)$,
and $\psi_*$ is defined similarly.
The map $\delta_i:\HH^i(X,\calH)\to \HH^{i+1}(X,\calF)$ is defined as follows:
Given $\gamma\in \HH^i(X,\calH)$ represented by some $h\in Z^i(X,\calH)$,
the surjectivity of $\psi$ implies that there is $g\in C^i(X,\calG)$ such
that $h(x)=\psi_x( g(x))$ for all $x\in X_{\ord}(i)$. Using the exactness, one
can show that there exists a unique
$f\in Z^{i+1}(X,\calF)$ such that $\vphi_y f(y)= (d_i g)(y)$ for all $y\in X_{\ord}(i+1)$,
and we define $\delta_i\gamma:=[f]_\calF$.
The proof that $\delta_i$ is well-defined and \eqref{EQ:long-coh-exact-seq} is exact is standard
and left to the reader.

\begin{remark}
	As expected, the  functors $\{\HH^i(X,-)\}_{i\geq 0}$ are the right derived functors
	of the left-exact functor $\HH^0(X,-)$ from the category of sheaves on $X$ to abelian groups,
	see Appendix~\ref{sec:derived}.
\end{remark}

\begin{remark}
	Given a short exact sequence of augmented sheaves
	$0\to \calF\to  \calG\to \calH\to 0$ on $X$, one can define a  long exact sequence
	similar to \eqref{EQ:long-coh-exact-seq}, but starting at $\HH^{-1}(X,\calF)$
	instead of $\HH^0(X,\calF)$. We omit the details. 
\end{remark}

If $\calF$ is an $\F$-sheaf on $X$, then the cohomology
groups $\HH^i(X,\calF)$ are $\F$-vector spaces. When $\F$ is clear from the context,
we shall often write
\[
h^i(\calF)=h^i(X,\calF):=\dim_{\F}\HH^i(X,\calF).
\]

\begin{lem}\label{LM:field-ext}
	Let $\calF$ be an $\F$-sheaf on $X$ and let $\K$ be a field extension of $\F$.
	Then $\dim_{\F} \HH^i(X,\calF)=\dim_{\K}\HH^i(X,\calF_{\K})$ for all $i\in\N\cup\{0\}$
	(notation as in Example~\ref{EX:sheaves-basic-examples}(vii)).
\end{lem}

\begin{proof}
	This follows by observing that the cochain complex
	$C^\bullet(X,\calF_\K)$ is isomorphic to 
	the cochain complex obtained by tensoring
	$C^\bullet(X,\calF)$ with $\K$.
\end{proof}

\subsection{Pushforward and Pullback}
\label{subsec:pushforward}

	Throughout, let $u:Y\to X$ denote a morphism of simplicial complexes (see~\S\ref{subsec:coverings}).
	Given a sheaf $\calG$ on $Y$, there is a natural way of ``pushing it'' along $u$ to a sheaf on $X$,
	and conversely, given a sheaf $\calF$ on $X$, there is a natural way of ``pulling it back'' along $u$
	to a sheaf on $Y$. We now explain these constructions.
	They will be extremely useful later on for producing new   examples of sheaved complexes
	from  old ones.
	
	\medskip

	Let $\calF$ be a sheaf on $X$. The \emph{pullback} or \emph{inverse image}
	of $\calF$ along $u:Y\to X$ is the sheaf $u^*\calF$ on $Y$ defined
	by
	\[
	u^*\calF(y)=\calF(u(y))
	\qquad 
	\text{and} 
	\qquad 
	\res^{u^*\calF}_{y'\from y} =
	\res^{\calF}_{u(y')\from u(y)} \]
	for all $\emptyset\neq y\subsetneq y'\in Y$, with the convention
	that $\res^{\calF}_{y\from y}=\id_{\calF(y)}$.
	
	\begin{example}
		If $A_X$ is the constant sheaf on $X$ associated to the abelian group $A$
		(Example~\ref{EX:sheaves-basic-examples}(i)), then $u^*A_X$ is the constant
		sheaf on $Y$ associated to $A$,
		that is, $u^*A_X=A_Y$.
	\end{example}
	
	Now let $\calG$ be a sheaf on $Y$ and suppose that
	$u:Y\to X$ is dimension preserving, i.e., $\dim y=\dim u(y)$ for all $y\in Y$.
	Given $x\in X$, we write $u^{-1}(x)$ for the set $\{y\in Y\suchthat u(y)=x\}$.
	Our assumption on $u$ implies that if $x'\in X$, $y'\in u^{-1}(x')$ and $x$ is a face of
	$x'$, then there exists a unique face $y$ of $y'$ such that $u(y)=x$; we denote this face
	$y$ by $y'(x)$.
	With this notation at hand, we
	define    \emph{pushforward} or
	\emph{direct image} of a sheaf $\calG$  along $u$ to be the sheaf $u_*\calG$
	on $X$ determined by
	\[
	(u_*\calG)(x)=\prod_{y\in u^{-1}( x )} \calG(y)
	\qquad
	\text{and}
	\qquad
	\res^{u_*\calG}_{x'\from x}((f_y)_{y\in u^{-1}(x)})=
	( \res^{ \calG}_{y'\from   y'(x)}(f_{y'(x)}))_{y'\in u^{-1}(x')}
	\]
	where $\emptyset\neq x\subsetneq x'\in X$ and $(f_y)_{y\in u^{-1}(x)}\in u_*\calG(x)=
	\prod_{y\in u^{-1}( x )}\calG(y)$.
	It routine to check that the sheaf condition \eqref{EQ:sheaf-cond} is satisfied for
	$u_*\calG$.

	One can also define the pushforward $u_*\calG$ without assuming that $u$ is dimension preserving.
	This construction is more involved and explained in Appendix~\ref{subsec:comparision-of-structure}; we will not make use of it 
	in this work.

\medskip

	The following lemma relates the cohomology
	of $\calG$ and $u_*\calG$. It can be regarded as a version of Shapiro's Lemma for sheaf cohomology.
	
	\begin{lem}\label{LM:Shapiro}
		Let $u:Y\to X$ be a dimension-preserving morphism of simplicial complexes
		and let $\calG$ be a sheaf on $Y$. Then, for all $i\geq 0$, there is an isomorphism
		$\HH^i(Y,\calG)\cong \HH^i(X,u_*\calG)$ which is natural in $\calG$.
	\end{lem}
	
	\begin{proof}
		It is enough to prove that the cochain complexes
		$C^\bullet(Y,\calG)$ and $C^\bullet(X,u_*\calG)$ are naturally isomorphic,
		i.e., that there is are isomorphisms
		$ t_{i,\calG}:C^i(Y,\calG)\to C^i(X,u_*\calG)$ such that
		$t_{i+1,\calG}\circ d_{i}^{\calG}=d_{i}^{u_*\calG}\circ t_{i,\calG}$
		and $t_{i,\calG'}\circ \vphi_*=\vphi_*\circ t_{i,\calG}$ for every
		morphism of sheaves on $Y$, $\vphi:\calG\to \calG'$.
		The desired isomorphism
		$t_{i,\calG}$ is the restriction
		of the identification
		$\tilde{C}^i(Y,\calG)=\prod_{y\in X_{\ord}(i)}\calG(y)
		\cong \prod_{x\in X_{\ord}(i)} \prod_{y\in u^{-1}(x)} \calG(y)
		=\prod_{x\in X_{\ord}(i)}u_*\calG(x)=\tilde{C}^i(X,u_*\calG)$
		to $C^i(Y,\calG)$. It is routine to check that it satisfies all the requirements.
	\end{proof}

\subsection{Restricting Sheaves to  The Links}
\label{subsec:sheaf-at-link}

Let $X$ be a simplicial complex and let $z\in X(i)$. 
Recall (\S\ref{subsec:complexes}) that $X_z$ denotes the link of $X$ at $z$.
Every  augmented  sheaf $\calF$ on $X$ restricts to an  augmented  sheaf
${\calF}_z$ on $X_z$ by setting
${\calF}_z(x)=\calF(x\cup z)$ and
$\res^{{\calF}_{z}}_{y\from x}=\res^{\calF}_{y\cup z\from x\cup z}$.
(This is how augmented sheaves arise naturally from sheaves!)

\begin{example}
	Let $A$ be an abelian group and let $\aug{A}$ denote
	the associated augmented sheaf on $X$ (Example~\ref{EX:sheaves-basic-examples}(ii)).
	Then $(\aug{A})_z$ is the augmented sheaf on $X_z$ associated to $A$.
\end{example}

Suppose now   that $z\in X_{\ord}(i)$, namely, we are also given
an ordering on the vertices of $z$.
With this extra data, it possible to take a cochain $f\in C^k(X,\calF)$
($i\leq k $)
and restrict it to a cochain  $f_z\in C^{k-i-1}(X_z,\calF_z)$ by
setting
\[
f_z(x)=f(xz)\qquad\forall x\in X_{z,\ord}(k-i-1).
\]
Conversely, given $g\in C^{k-i-1}(X,\calF_z)$,
there exists a unique  cochain $g^z\in C^k(X,\calF)$
such that
\[
g^z(xz)=g(x)\qquad \forall x\in X_{\ord}(k-i-1),
\]
and $g^z(y)=0$ for all $y\in X_{\ord}(k)$ with $z\nsubseteq y$.
Clearly, $(g^z)_z=g$.

\begin{lem}\label{LM:f-eta-basic-props}
In the previous setting, we have
$(d_{k-i-1}g)^z=d_{k }(g^z)$. In particular,
if $g$ is a cocycle  (resp.\ coboundary), then so is  $g^z$.
\end{lem}

\begin{proof}
Let $x\in X_{\ord}(k+1)$. 
We need to show that $(d_{k-i-1}g)^z(x) =d_{k }(g^z)(x)$.
If $z\nsubseteq x$ as sets, then $(d_{k-i-1}g)^z(x)=0=d_{k }(g^z)(x)$,
so assume that $z\subseteq x$ as sets.
By reordering the vertices of $x$,
we may assume that $x=yz$ for some $y\in X_{\ord}(k-i)$.
Then $(d_{k-i-1}g)^z(x)=(d_{k-i-1}g)(y)=\sum_{j=0}^{k-i} (-1)^j\res_{x\from z\cup y_j} g(y_j)$.
On other hand, since $g^z(x_j)=0$  if $z\nsubseteq x_j$,
we have
$d_{k }(g^z)(x)=\sum_{j=0}^{k+1} (-1)^j \res_{x\from x_j}g^z(x_j)=
\sum_{j=0}^{k-i} (-1)^j\res_{x\from z\cup y_j} g(y_j)$,
so
$(d_{k-i-1}g)^z(x)=d_{k }(g^z)(x)$.
\end{proof}

Let $\calP$ be a property of sheaved simplicial complexes
(written ``$(X,\calF)$ is $\calP$'' when it holds), 
and let $(X,\calF)$ be a sheaved simplicial complex. We will
say that $(X,\calF)$ is
a \emph{$k$-local $\calP$}  if $(X_z,\calF_z)$
is $\calP$ for all 
$z\in X(k)$. If $\calP$ also makes reference to a particular
dimension $i$ (as in ``$X$ is $\calP$ in dimension $i$''),
we will say that $(X,\calF)$
is a \emph{$k$-local $\calP$ in dimension $i$} if 
$(X_z,\calF_z)$
is $\calP$ in dimension $i-k-1$ for all $z\in X(k)$.

\subsection{Locally Constant Sheaves}
\label{subsec:loc-const}

Let $X$ be a simplicial complex. A sheaf $\calF$ on $X$ is called
\emph{constant} if there is an abelian group $A$ such that $\calF$ is isomorphic
to the constant sheaf $A $ on $X$ (Example~\ref{EX:sheaves-basic-examples}(i)). Similarly, an augmented sheaf 
$\calF'$ on $X$ is called constant if $\calF'\cong \aug{A}$ for some abelian group $A$.
If $\calF$ (resp.\ $\calF'$) has the additional structure of an $\F$-sheaf,
we further require $A$ to be an $\F$-vector space and the isomorphism $\calF\to A$
(resp.\ $\calF'\to\aug{A}$) to be $\F$-linear.

A sheaf $\calF$ on $X$ is called \emph{locally constant} if $\calF_z$ is a constant
augmented sheaf on $X_z$ for every $z\in X-\{\emptyset\}$.
Every constant sheaf is locally constant, but the converse is false in general.
 
\begin{lem}\label{LM:loc-const-crit}
	A sheaf $\calF$ on a simplicial complex $X$ is locally constant if and only if
	all the restriction maps $\res^\calF_{y\from x}$ ($\emptyset\neq x\subsetneq y\in X$)
	are isomorphisms.
\end{lem}

\begin{proof}
	If $\emptyset\neq x\subsetneq y\in X$ and $\calF$ is locally constant, then
	$\res^\calF_{y\from x}$ is equal to $\res^{\calF_x}_{y-x\from\emptyset}$,
	which is an isomorphism because $\calF_x$ is constant.
	Conversely, if all the restriction maps of $\calF$ are isomorphisms and $z\in X-\{\emptyset\}$,
	take $A=\calF(z)$ and note that $(\res_{x\cup z\from z})_{x\in X_z}:\aug{A}\to \calF_z $ is
	an isomorphism of augmented sheaves.
\end{proof}

\begin{example}
	(i) Let $X$ be a cycle graph on $n$ vertices. Fix an edge $e\in X(1)$
	and a $0$-face $z\subseteq e$. Define an $\R$-sheaf $\calF$ on $X$
	by setting $\calF(x)=\R$ for every $x\in X-\{\emptyset\}$,
	$\res^{\calF}_{y\from x}=\id_\R$ for $(y,x)\neq (e,z)$ and
	$\res^{\calF}_{e\from z}=-\id_\R$.
	By Lemma~\ref{LM:loc-const-crit}, $\calF$ is a locally constant sheaf.
	However, $\calF$  is not constant.
	Indeed, one readily checks that $Z^0(X,\calF)=0$. However, if $\calF$ were constant,
	then it would be isomorphic to the constant sheaf $\R_X$, and $Z^0(X,\R)\cong \R$.
	
	(ii) 
	Generalizing (i), we can construct
	locally constant sheaves on any graph $X$.	
	Simply take an abelian group $A$, set $\calF(x)=A$ for all $x\in X-\{\emptyset\}$
	and choose each restriction map $\res^{\calF}_{y\from x}$ ($\emptyset\neq x\subsetneq y\in X(1)$)
	to be some automorphism of $A$. As in (i), sheaves obtained in this manner
	are often not constant.  
	
	(iii) If $u:Y\to X$ is a covering map and $\calF$ is a locally constant sheaf on $Y$,
	then the pushforward $u_*\calF$ is a locally constant sheaf on $X$. The sheaf $u_*\calF$ may be non-constant
	even when $\calF$ is.

	(iv) Suppose that $X$ is connected and let $\rho:\pi_1(X)\to \nGL{V}{\F}$
	be representation of $\pi_1(X)$ on an $\F$-vector space $V$.
	Then $\F$-sheaf $\calF_\rho$ constructed in Example~\ref{EX:sheaves-from-reps}
	is locally constant. Moreover, it can be shown that $\calF_\rho$ is constant if and only
	if $\rho$ is a trivial representation (i.e., $\rho(\gamma)=\id_V$
	for all $\gamma\in\pi_1(X)$).
\end{example}

\begin{remark}
	Locally constant $\R$-sheaves on graphs are 
	equivalent as a category to the local sysetms on graphs 
	introduced by Jordan and Livne \cite{Jordan_1997_Ramanujan_local_systems}.
\end{remark}

Let $\calF$ be a locally constant $\F$-sheaf on $X$. 
If $X$ is connected, then
Lemma~\ref{LM:loc-const-crit} implies
that all the vector spaces $\{\calF(x)\}_{x\in X-\{\emptyset\}}$
have the same dimension. When the latter holds, we denote this common dimension by 
\[\dim \calF \]
and call it the dimension of $\calF$.

\begin{lem}\label{LM:coh-loc-const}
	Let $X$ be a connected simplicial complex and let $\calF$ be a locally constant
	$\F$-sheaf on $X$. Then $\dim \HH^0(X,\calF)\leq \dim \calF$. 
\end{lem}

\begin{proof}
	Fix some $0$-face $x_0\in X$.
	It is enough to show that any $0$-cocycle
	$f\in Z^0(X,\calF)$   is uniquely determined by $f(x_0)$.
	Indeed, if $y\in X(0)$ is another $0$-face, then there exists
	a sequence of $0$-faces $x_0,x_1,\dots,x_n=y$ in $X$
	such that $x_{i-1}\cup x_{i}\in X(1)$ for all $i$. Since $f\in Z^0(X,\calF)$,
	we have $\res_{x_i\cup x_{i-1}\from x_{i-1}}f(x_{i-1})=
	\res_{x_i\cup x_{i-1}\from x_i}f(x_{i })$ for all $i\in\{1,\dots,n\}$.
	The restriction maps are isomorphisms
	(Lemma~\ref{LM:loc-const-crit}),  so $f(y)$
	is uniquely determined by $f(x_0)$.  
\end{proof}

\begin{remark}
It is not difficult to see that a every 
$n$-dimensional locally constant  sheaf $\calF$
on a connected simplicial complex $X$ gives rise to a group homomorphism $\pi_1(X)\to \nGL{\F}{n}$. The converse
is also true: every
representation $\rho:\pi_1(X)\to \nGL{\F}{n}$ gives rise to an $n$-dimensional locally
constant $\F$-sheaf on $X$, and if the universal covering of $X$ is contractible,
then all locally constant sheaves are obtained in this manner, up to isomorphism. 
Moreover, in this case, $\HH^i(X,\calF)$ and the group
cohomology $\HH^i(\pi_1(X),\rho)$ are isomorphic. 
We omit the details as they will not be needed here.\gap{}
\end{remark}

\subsection{The Cup Product}

\label{subsec:cup-prod}

The cup product is a well-known operation on cohomology groups in topology.
We now present the analogous notion for sheaves on simplicial complexes, which will be needed
only for Section~\ref{sec:modifying} below.
For the sake of simplicity, we shall restrict the discussion to 
the cup-product action of $C^i(X,\F) $ on $C^j(X,\calG)$ where $\calG$
is an $\F$-sheaf.


As before, $X$ is a simplicial complex. We fix, once and for all, a linear ordering
$L$ on $V(X)$
and use it to identify $C^j(X,\calG)$ with $\prod_{x\in X(j)}\calG(x)$
for any sheaf $\calG$ as in Remark~\ref{RM:ordered-cohomology}.
If $v_0,\dots,v_i$ are the vertices of $x\in X(i)$ and $v_0<\dots<v_i$, then
we shall denote the ordered face $x_L$ simply as $v_0v_1\cdots v_i$.

Let $\calG$ be an $\F$-sheaf on $X$. For every $\alpha\in C^i(X,\F)$
and $g\in C^j(X,\calG)$, the \emph{cup product} 
of $\alpha$ and $g$  
is the element  $\alpha\cupp g \in C^{i+j}(X,\calG)$ defined by:
\[
(\alpha\cupp g)(v_0v_1\cdots v_{i+j})=\alpha(v_0\cdots v_i)g(v_i\cdots v_{i+j}).
\]
The properties of the cup product that we shall need are summarized in the following proposition.

\begin{prp}\label{PR:cup-product-basic-properties}
	Let $\calG$ be an $\F$-sheaf on $X$, and let $f,g\in C^j(X,\calG)$,
	$\alpha\in C^i(X,\F)$, $\beta\in C^k(X,\F)$. Then:
	\begin{enumerate}[label=(\roman*)]
		\item $\cupp:C^i(X,\F)\times C^j(X,\calG)\to C^{i+j}(X,\calG)$ is an $\F$-bilinear pairing.
		\item $d_{i+j}(\alpha\cupp f) = d_i\alpha\cupp f +(-1)^i\alpha\cupp d_j f$.
		\item $(\alpha\cupp \beta)\cupp f = \alpha\cupp (\beta\cupp f)$.
	\end{enumerate}
	Moreover, if   $\calG'$ is another $\F$-sheaf on $X$ and   $\vphi:\calG\to \calG'$
	is a  morphism,
	then: 
	\begin{enumerate}[resume, label=(\roman*)]
		\item[(v)] $\vphi_*(\alpha\cup g)=\alpha\cup \vphi_* g$,
		where $\vphi_*:C^r(X,\calG)\to C^r(X,\calG')$ is the map induced by $\vphi$
		(see \S\ref{subsec:sheaf-coh}).
	\end{enumerate}
\end{prp}

\begin{proof}
	Everything is straightforward except for (ii). 
	To prove (ii), 
	suppose that $x=v_0\cdots v_{i+j+1}$ is an ordered $(i+j+1)$-face.
	Then 
	\begin{align*}
		d_{i+j}(&\alpha\cupp f)(x)  =
		\sum_{t=0}^{i+j+1} (-1)^t(\alpha\cupp f)(v_0\cdots\hat{v}_t\cdots v_{i+j+1}) \\
		&= \sum_{t=0}^i(-1)^t\alpha(v_0\cdots\hat{v}_t\cdots v_{i+1})f(v_{i+1}\cdots v_{i+j+1})
		+ (-1)^{i+1}\alpha(v_0\cdots v_i\hat{v}_{i+1})f(v_{i+1}\cdots v_{i+j+1}) \\
		&\phantom{=} + (-1)^i\alpha(v_0\cdots v_i )f(\hat{v}_i v_{i+1}\cdots v_{i+j+1})
		+ \sum_{t={i+1}}^{i+j+1}(-1)^t\alpha(v_0\cdots v_i) f(v_i\cdots\hat{v}_t\cdots v_{i+j+1}) \\
		&=d_i\alpha(v_0\cdots v_{i+1})f(v_{i+1}\cdots v_{i+j+1}) 
		+ (-1)^i \alpha(v_0\cdots v_i) d_jf(v_i\cdots v_{i+j+1})\\  
		&=((d_i\alpha)\cupp f + (-1)^i\alpha\cupp (d_j f))(x).
	\end{align*}
	(Here, $\hat{v}_t$ means that  we omit   $v_t$.)
	As this holds for all $x$, (ii) follows.
\end{proof}

Proposition~\ref{PR:cup-product-basic-properties}(ii)
	implies readily that the bilinear pairing 
	\[[\alpha]\cup[g]\mapsto [\alpha\cup g]:\HH^i(X,\F)
	\times \HH^j(X,\calG)\to \HH^{i+j}(X,\calG),\] 
	also called the cup product,
	is well defined. It can further be shown
	that this pairing is independent of the ordering on $V(X)$.

%

\section{Coboundary and Cosystolic Expansion}
\label{sec:coboundary}

In this section we introduce expanding
sheaves. In fact,
we shall consider two types of expansion --- \emph{coboundary expansion}
and \emph{cosystolic expansion} --- and both make use of
an auxiliary \emph{norm} on the sheaf.

\subsection{Norms on Abelian Groups}
\label{subsec:norms}

Recall that a \emph{seminorm} on an abelian group $A$
is a function $\|\cdot\|:A\to \R$
such that $\|a\|\geq 0$, $\|a\|=\|-a\|$ and $\|a+b\|\leq \|a\|+\|b\|$ for all $a,b\in A$.
If $\|a\|=0$ implies $a=0$, we say that $\|\cdot\|$ is \emph{norm} on $A$.
In this case, $(x,y)\mapsto \|x-y\|:A\times A\to \R$
is a  translation-invariant metric on $A$.

A seminorm $\|\cdot\|:A\to \R$ is bounded if $\sup\{\|a\|\where a\in A\}<\infty$.
\emph{All norms and seminorms in this work are assumed to be bounded.}

If $\|\cdot\|_A$ is a seminorm (resp.\ norm)
on $A$ and $B$ is a   subgroup of $A$, then the restriction of   $\|\cdot\|_A$
to $B$ is   a seminorm (resp.\ norm) on $B$. The map $\|\cdot\|_{A/B}:A/B\to \R$
defined by
$\|a+B\|_{A/B}=\inf_{b\in B}\|a+b\|_A$ is a seminorm on $A/B$, called
the quotient seminorm.  If $\|\cdot\|_A$ is a norm and $B$ is finite,
then $\|\cdot\|_{A/B}$ is a norm.

\begin{example}\label{EX:norm-on-ab-grps}
	(i) The \emph{discrete norm} on an abelian group $A$ maps all nonzero elements
	of $A$ to $1$ and the zero element to $0$.
	
	(ii) The \emph{Hamming norm} on $\F^n$ sends $v\in \F^n$ to the number
	of its non-zero coordinates.
	More generally, if $V$ is a finite dimensional $\F$-vector space
	with a finite basis $B$,
	then the   Hamming norm on    $V$ relative to   $B$   sends
	the vector
	$v=\sum_{b\in B}\alpha_b b$  to the number of 
	nonzero $\alpha_b$-s. 
\end{example}

\subsection{Normed Sheaves}
\label{subsec:normed-sheaves}

Let $\calF$ be an augmented sheaf on a simplicial complex $X$.
A norm on $\calF$ is a collection $\|\cdot\|=\{\|\cdot\|_x\}_{x\in X}$
of norms $\|\cdot\|_x:\calF(x)\to \R$. We also say
that $(\calF,\|\cdot\|)$ is a normed augmented sheaf.
In this case, the \emph{mass} of $x\in X$ (relative to $\calF$ and $\|\cdot\|$) is defined
as $m(x)=\sup\{\|f\|_x \where f\in \calF(x)\}$; it is finite by our standing assumption
that all norms are bounded. 
For $A\subseteq X$, we write $m(A)=\sum_{x\in A}m(x)$ 
and let $m(i)=m(X(i))$.

In this work, we will be concerned with the following   examples of normed augmented 
sheaves.

\begin{example}[Weighted support norm]\label{EX:weighted-support}
Suppose that $X$ is a $d$-complex (i.e., pure of dimension $d$) and let $w$ denote
the canonical weight function on $X$ (see \S\ref{subsec:weights}).
Let $\calF$ be an augmented  sheaf on $X$. The \emph{weighted support norm} 
of $\calF$ is the norm
$\|\cdot\|_{\mathrm{ws}}=\{\|\cdot\|_{\mathrm{ws},x}\}_{x\in X}$,
where $\|\cdot\|_{\mathrm{ws},x}:\calF(x)\to \R$
is defined by $\|f\|_{\mathrm{ws},x}=w(x)$ if $f\neq 0$ and $\|f\|_{\mathrm{ws},x}=0$ otherwise.
Provided that $\calF(x)\neq 0$, the mass of $x$ is just $w(x)$.
Consequently, if $\calF(x)\neq 0$ for all $x\in X(i)$, then $m(i)=w(X(i))=1$.
\end{example}

The weighted support norm will be the default norm on every
sheaf we consider, and will be denoted simply
as $\ws{\cdot}$ when there is no risk of confusion.
The following norms, however, are more useful for 
coding theory applications of sheaves.

\begin{example}[Normalized Hamming norm]\label{EX:non-weighted-support}
The \emph{normalized Hamming norm} on an augmented sheaf $\calF$
is the norm $\nws{\cdot}=\{\|\cdot\|_{\Ham,x}\}_{x\in X}$,
where $\|\cdot\|_{\Ham,x}:\calF(x)\to \R$
is
defined by $\|f\|_{\Ham,x}=\frac{1}{|X(\dim x)|}$ if $f\neq 0$ and $\|f\|_{\Ham,x}=0$ otherwise.
This is the norm used in Section~\ref{sec:overview}.
If $\calF(x)\neq 0$ for all $x\in X(i)$, then $m(i)= 1$.
\end{example}

\begin{example}[Hamming norm relative to a basis]\label{EX:Hamming-norm}
	Suppose that $\calF$ is an augmented $\F$-sheaf, i.e., an augmented 
	sheaf  of  $\F$-vector spaces.
	A \emph{basis} for $\calF$ is a collection $B=\{B(x)\}_{x\in X}$
	such that $B(x)$ is an $\F$-basis of $\calF(x)$  for all $x\in X$.
	Let $\|\cdot\|_{B,x}$ denote the Hamming norm on $\calF(x)$ relative
	to the basis $B(x)$ (Example~\ref{EX:norm-on-ab-grps}(ii)).
	Then $\|\cdot\|_B :=\{\|\cdot\|_{B,x}\}_{x\in X}$ is a norm on $\calF$
	called the \emph{Hamming norm} relative to the basis 
	$B=\{B(x)\}_{x\in X}$.
	Writing $m=m_B$ for the corresponding mass function,
	we have $m (x)=\dim \calF(x)$ and $m (i)=\sum_{x\in X(i)}\dim\calF(x)=\dim C^i(X,\calF)$.
\end{example}

Every norm $\|\cdot\|$ on $\calF$
induces a norm on $C^k(X,\calF)$, also denoted $\|\cdot\|$,   given by
\[
\|f\|=\|f\|_{C^k}=\sum_{x\in X(k)} \|f(x)\|_x,
\]
where in the expression $f(x)$, we regard $x$ as an ordered cell by arbitrarily ordering its vertices.
The mass of all $i$-faces, $m(i)$, is nothing but 
$\sup\{\|f\|\where f\in C^i(X,\calF)\}$.
For example, if $\|\cdot\|$ is the weighted support norm,
then we   have $\|f\|\leq 1$ for all $f\in C^i(X,\calF)$.

\begin{example}\label{EX:cohomology-norms}
	(i) Let $X$ be $d$-complex with weight function $w=w_X$, 
	let $\calF$ an augmented sheaf on $X$
	and let $f\in C^i(X,\calF)$.
	Then, relative to the weighted support norm 
	$\|\cdot\|_{\mathrm{ws}}$ (Example~\ref{EX:weighted-support}),
	we have
	\[\|f\|_{\mathrm{ws}} =w(\supp f),\]
	where $\supp f:=\{x\in X(k)\suchthat f(x)\neq 0\}$. (This explains the name ``weighted support''.)
	By contrast, with respect to the normalized Hamming norm $\nws{\cdot}$ (Example~\ref{EX:non-weighted-support}),
	we have 
	\[\|f\|_{\Ham }=\frac{|\supp f|}{|X(k)|} . \]	
	Suppose further that there is an abelian group $\Sigma$
	such that $\calF(x)\cong\Sigma$ for all $x\in X(i)$.
	Let us fix a linear ordering on $V(X)$
	and use it identify  $C^i(X,\calF)$ with $\prod_{x\in X(i)}\calF(x)\cong
	\Sigma^{X(i)}$ as in Remark~\ref{RM:ordered-cohomology}.
	Then 
	the norm $\nws{\cdot}:C^i(X,\calF)\to \R$
	coincides with the normalized Hamming norm on $\Sigma^{X(i)}$.
	
	(ii) 
	Let $\calF$ be an augmented $\F$-sheaf 
	on $X$
	with a basis $B$, and let $\|\cdot\|_{B}$ 
	denote the associated Hamming norm 
	(Example~\ref{EX:Hamming-norm}).
	Again, fix a linear ordering on $V(X)$ and use it to identify
	$C^k(X,\calF)$ with $\prod_{x\in X(k)}\calF(x)$  as  in Remark~\ref{RM:ordered-cohomology}.
	Then, under this identification, $\|\cdot\|_{B }$ is the  Hamming norm
	of $\prod_{x\in X(k)}\calF(x)$ relative to the basis $\bigsqcup_{x\in X} B(x)$.
\end{example}

The norms $\ws{\cdot}=\|\cdot\|_{\mathrm{ws}}$, $\nws{\cdot}$, $\|\cdot\|_B$
of Examples~\ref{EX:weighted-support}--\ref{EX:Hamming-norm} 
are proportional under
mild assumptions on $X$ and $\calF$.

\begin{prp}\label{PR:Hamming-and-support-ratio}
	Let $X$ be a $d$-complex, let $k\in\{-1,0,\dots,d\}$, and
	let $\calF$ be an augmented sheaf on $X$.
	Put $Q=D_{k,d}(X)$ (see \S\ref{subsec:complexes}).
	Then, 
	\begin{enumerate}[label=(\roman*)]
		\item  
	$  \schoose{d+1}{k+1}\frac{|X(d)|}{ |X(k)|}Q^{-1}\ws{f}\leq 
	\nws{f} \leq 
	\schoose{d+1}{k+1}\frac{|X(d)|}{ |X(k)|} \ws{f}$ for all $f\in C^k(X,\calF)$.
	Furthermore, 
	$\schoose{d+1}{k+1}^{-1}\leq \frac{|X(d)|}{ |X(k)|}\leq Q$.
	\end{enumerate}
	If $\calF$ is an augmented $\F$-sheaf, $B$ is a basis of $\calF$,
	and $N=\max\{\dim\calF(x)\where x\in X(k)\}$,
	then we moreover have
	\begin{enumerate}[label=(\roman*), resume]
		\item $\schoose{d+1}{k+1}|X(d)|Q^{-1}\ws{f}\leq 
	\|f\|_{B}\leq 
	\schoose{d+1}{k+1}|X(d)|N\ws{f}$ for all $f\in C^k(X,\calF)$.
	\end{enumerate}
\end{prp}

\begin{proof}

	(i) The inequality $\schoose{d+1}{k+1}^{-1}\leq \frac{|X(d)|}{ |X(k)|}\leq Q$
	follows readily from the fact that every $d$-face
	contains exactly $\schoose{d+1}{k+1}$ $k$-faces and every
	$k$-face is contained in at most $Q$ $d$-faces.
	
	Let $x\in X(k)$. It is enough to show
	that for all $g\in \calF(x)$,
	we have 
	$ \schoose{d+1}{k+1}\frac{|X(d)|}{Q|X(k)|}\|g\|_{\mathrm{ws},x}\leq 
	\|g\|_{\Ham,x}\leq 
	\schoose{d+1}{k+1}\frac{|X(d)|}{ |X(k)|}\|g\|_{\mathrm{ws},x}$.
	This is clear if $g=0$, so assume $g\neq 0$.
	Then $\|g\|_{\mathrm{ws},x}=w(x)$ whereas
	$\|g\|_{\Ham,x}=\frac{1}{|X(k)|}$.
	The definition of $w(x)$ in \S\ref{subsec:weights} 
	implies that  
	$\schoose{d+1}{k+1}^{-1}|X(d)|^{-1} \leq w(x)
	\leq \schoose{d+1}{k+1}^{-1}|X(d)|^{-1} Q$.
	Thus,
	$ \schoose{d+1}{k+1}\frac{|X(d)|}{Q|X(k)|} w(x)
	\leq \frac{1}{|X(k)|}
	= \schoose{d+1}{k+1}\frac{|X(d)|}{ |X(k)|}\cdot \schoose{d+1}{k+1}^{-1}|X(d)|^{-1}\leq
	\schoose{d+1}{k+1}\frac{|X(d)|}{ |X(k)|} w(x)$, which is what we want.

	(ii) 
	As in (i),
	it is enough to show that for all $x\in X(k)$
	and $g\in \calF(x)-\{0\}$,
	we have $\schoose{d+1}{k+1}|X(d)|Q^{-1}w(x)\leq 
	\|g\|_{B(x)}\leq 
	\schoose{d+1}{k+1}|X(d)|Nw(x)$.
	We observed   that  
	$\schoose{d+1}{k+1}^{-1}|X(d)|^{-1} \leq w(x)
	\leq \schoose{d+1}{k+1}^{-1}|X(d)|^{-1} Q$.
	Since $1\leq \dim \calF(x)\leq N$,
	it follows  that
	$\schoose{d+1}{k+1}|X(d)|Q^{-1}w(x)\leq 1 
	\leq \|g\|_{B(x)}\leq N\leq 
	\schoose{d+1}{k+1}|X(d)|N w(x)$, as required.
\end{proof}

\subsection{Coboudary and Cosystolic Expansion}
\label{subsec:coboundary-exp}

Let $ (\calF,\|\cdot\|) $ be a normed augmented  sheaf on
a simplicial complex $X$ and let $m$
be its  mass function.
Let $k \in\N\cup\{0,-1\}$.
The norm $\|\cdot\|_{C^k}$
induces seminorms on 
$C^k(X,\calF)/B^k(X,\calF)$ and $C^k(X,\calF)/Z^k(X,\calF)$,
which we denote by $\|\cdot\|_{C^k/B^k}$ and $\|\cdot\|_{C^k/Z^k}$, respectively.
The subscripts will be dropped when there is no risk of confusion.

\begin{dfn}
Let $\veps,\delta\in [0,\infty)$. 
%
We say that $(X,\calF,\|\cdot\|)$   is an \emph{$\veps$-coboundary
expander} in dimension $k$
if 
\begin{enumerate}
	\item[(B)] $\|d_kf\|_{C^{k+1} } m(k)\geq \veps \|f+B^k(X,\calF))\|_{C^k/B^k } m(k+1)$ 
	for all $f\in C^k(X,\calF)$.
\end{enumerate}
We say that $(X,\calF,\|\cdot\|)$ is an \emph{$(\veps,\delta)$-cosystolic
expander} in dimension $k $ if
\begin{enumerate}
	\item[(C1)] $\|d_kf\|_{C^{k+1}} m(k)\geq \veps \|f+Z^k(X,\calF)\|_{C^k/Z^k} m(k+1)$ 
	for all $f\in C^k(X,\calF)$, and
	\item[(C2)] $\|f\|_{C^k}\geq \delta   m(k)$ for all $f\in Z^k(X,\calF)-B^k(X,\calF)$.
\end{enumerate}
\end{dfn}

When $X$ is a $d$-complex,
we say that the pair $(X,\calF)$ is an $\veps$-coboundary expander,
resp.\ $(\veps,\delta)$-cosystolic expander, in dimension $i$ if
this holds for $(X,\calF,\ws{\cdot})$ with $\ws{\cdot}$ being the  
weighted support norm of $(X,\calF)$ (Example~\ref{EX:weighted-support}).
Thus, $(X,\aug{(\F_2)})$ is an $\veps$-coboundary expander in dimension $i$
if and only if $X$ is an $\veps$-coboundary expander in dimension $i$
in the sense of Lubotzky, Meshulam and Mozes
\cite[Definition~1.1]{Lubotzky_2016_expansion_of_buildings}.

\begin{remark}\label{RM:properties-of-cbe}
The following properties of coboundary and cosystolic expansion are important to note:

(i) The triple $(X,\calF,\|\cdot\|)$ is an  $\veps$-coboundary
expander in dimension $k$ if and only if
it is an $(\veps,\delta)$-cosystolic  expander in dimension $k$ 
and $\HH^k(X,\calF)=0$.

(ii) Scaling the norms $\{\|\cdot\|_{x}\}_{x\in X(k)}$
by the same constant $c\in\R_+$ does not affect the coboundary
and cosystolic expansion in dimension $k$, and likewise for the for the norms
$\{{\|\cdot\|}_{x}\}_{x\in X(k+1)}$.
More generally, let $\|\cdot\|'$ be another norm on $\calF$,
and suppose that there are constants $u_k,v_k,u_{k+1},v_{k+1}\in \R_+$
such that $u_i\|f\|_x   \leq \| f\|'_x   \leq v_i\| f\|_x   $
for all $i\in\{k,k+1\}$, $x\in X(i)$ and $f\in \calF(x)$.
If $(X,\calF,\|\cdot\|)$ is an $(\veps,\delta)$-cosystolic expander (resp.\
$\veps$-coboundary expander) in dimension $k$,
then $(X,\calF,{\|\cdot\|}')$ is a $(\frac{u_k u_{k+1}}{v_k v_{k+1}}\veps,
\frac{u_k}{v_k}\delta)$-cosystolic expander (resp.\
$\frac{u_k u_{k+1}}{v_k v_{k+1}}\veps$-coboundary expander)
in dimension $k$.

(iii) If $\calF$ vanishes on $X(k)$ or on $X(k+1)$, equivalently
if $m(k)=0$ or $m(k+1)=0$,
then conditions (B) and (C1) hold with any $\veps\in \R_+$.
\end{remark}

\begin{remark}
	The normalization by $m(k)$ and $m(k+1)$ in (B), (C1) and (C2) is 
	made in order
	to keep $\veps$ and $\delta$ around the interval $[0,1]$.
	However, it is possible for $\veps$ to exceed $1$.
	Indeed, writing $\Delta_n$ for the $n$-dimensional simplex,
	it is easy to check that the coboundary expansion of $(\Delta_n,\aug{(\F_2)},\|\cdot\|_{\mathrm{ws}})$ 
	in dimension $0$ is $\frac{n+2-(n\bmod 2)}{n }$.\footnote{
		As for higher dimensions
		$k\in\{1,\dots,n-1\}$,
		Gromov \cite{Gromov_2010_expanders_and_top_II} and Meshulam--Wallach
		\cite{Meshulam_2009_homological_connectivity}
		showed that  
		$(\Delta_n,\aug{(\F_2)},\|\cdot\|_{\mathrm{ws}})$ is an
		$\frac{n+1}{n-k}$-coboundary expander in dimension $k$.
	}
	In contrast,
	if   $(X,\calF,\|\cdot\|)$ is an $(\veps,\delta)$-cosystolic expander
	in dimesion $k$ and $Z^k(X,\calF)\neq B^k(X,\calF)$, then $\delta$
	cannot exceed $1$. 
	
	If we use the weighted support norm, then the coboundary expansion  in dimension $k$
	cannot exceed $k+2$ by  the following lemma.
\end{remark}

\begin{lem}\label{LM:max-cbe}
	Let $(X,\calF)$ be a sheaved $d$-complex, let $k\in \{-1,\dots,d-1\}$
	and suppose that $\calF(x)\neq 0$ for all $x\in X(k+1)$. 
	Then the coboundary expansion of $(X,\calF)$ in dimension $k$ is at most $k+2$.
\end{lem}

\begin{proof}
	It is enough to show that $\ws{d_k f}m(k)\leq (k+2)\ws{f}m(k+1)$ for all $f\in C^k(X,\calF)$.
	Our assumptions imply that $m(k+1)=1$ and $m(k)\leq 1$. Using 
	this and \eqref{EQ:weight-of-containing-cells}, we get
	\begin{align*}
	\ws{d_k f}m(k)& \leq w(\supp(d_kf))\leq
	w(\bigcup_{x\in \supp f} X(k+1)_{\supseteq x})\\
	& \leq \sum_{x\in \supp f}w(X(k+1)_{\supseteq x})
	=(k+2)\sum_{x\in \supp f}w(x)=(k+2)\ws{f} m(k+1).
	\qedhere
	\end{align*}
\end{proof}

The meaning of begin an $\veps$-coboundary expander in $-1$ has
been worked out in the Overview section, page~\pageref{page:cb-exp-in-dim-neg-one}.
We recommend to recall it at this point.

\subsection{Some Examples of Coboundary Exapnders}
\label{subsec:cb-examples}

Only a few concrete examples of good coboundary expanders
in dimension $>0$ are known; 
see \cite{First_2021_weighted_mixing_lemmas_preprint} for  survey.
We now 
recall some of these examples which will be needed in this work.

In contrast, examples of  infinite families of cosystolic expanders 
of the form $(X,\aug{A},\|\cdot\|_{\mathrm{ws}})$ ($A$ is an abelian group) with $D(X)$ 
uniformly bounded 
appear
in \cite{Kaufman_2016_isoperimetic_inequalities} ($\dim X=2$, $A=\F_2$), \cite{Evra_2016_cosystolic_expanders_arxiv_version}
($A=\F_2$) and \cite{Kaufman_2018_cosystolic_expanders}. We shall give more examples in 
Sections~\ref{sec:cse-from-links} and~\ref{sec:quo-aff-buildings}.

\medskip

We begin with noting that if the underlying  weighted  graph of 
a $d$-complex $X$ is a good spectral expander
in the sense of \S\ref{subsec:skeleton},
then $(X,\aug{A})$ is a good coboundary expander in dimension $0$
(w.r.t.\ to $\|\cdot\|_{\mathrm{ws}}$)  for any abelian group $A$.

\begin{thm}[{\cite[Corollary~5.3]{First_2021_weighted_mixing_lemmas_preprint}}]
	\label{TH:cbe-of-spectral-exps}
	Suppose that $X$ is $d$-complex ($d\geq 1$)
	whose underlying weighted graph is a $[-1,\lambda]$-expander
	(in the sense of \S\ref{subsec:skeleton})
	for some $\lambda\in  [-1,1]$.
	Then $(X,\aug{A})$ is a $(1-\lambda)$-coboundary expander 
	in dimension $0$ for every
	abelian group $A$.
\end{thm}

Next, we recall that  finite 
buildings (see~\S\ref{subsec:buildings}) have
large (i.e.\ bounded away from $0$) coboundary expansion 
once endowed with  certain sheaves. The following theorem
summarizes results from \cite{Lubotzky_2016_expansion_of_buildings},
\cite{Kaufman_2018_cosystolic_expanders} and \cite{First_2021_weighted_mixing_lemmas_preprint}.

\begin{thm}\label{TH:cbe-for-buildings}
	Let $d\in \N\cup\{0\}$ and $q\in\N$. Let $X$ be finite $d$-dimensional $q$-thick
	finite building with Coxeter
	diagram $T$ and let $A$ be an abelian group. Let $L$
	denote the set of edge labels occurring in $T$ and put $m=\max(\{2\}\cup T)$.\footnote{
		If $q\geq 3$, then $m\leq 8$; see \cite[Chapter~9]{Abramenko_2008_Buildings}.	
	}
	\begin{enumerate}[label=(\roman*)]
		\item There exists   $\veps>0 $, depending only on $d$,
		such that $(X,\aug{A} )$ is an $\veps$-coboundary expander in dimensions
		$ -1,0,\dots,d-1 $.
		
		\item $(X,\aug{A} )$ is a $(1-\frac{\sqrt{m-2}}{\sqrt{q}-(d-1)\sqrt{m-2}})$-coboundary
		expander in dimension $0$ if $q>(d-1)^2(m-2)$, and a $1$-coboundary expander in dimension $-1$
		in general.
	\end{enumerate}
\end{thm}

\begin{proof}
	The assertions about coboundary expansion in dimension $-1$  and
	the case where $\dim X=0$ are straightforward.
	As for the rest,
	(i) is proved
	in   \cite{Lubotzky_2016_expansion_of_buildings}
	for $A=\F_2$ and in
	\cite{Kaufman_2018_cosystolic_expanders}, 
	for general $A$,
	and  (ii) is \cite[Corollary~7.4]{First_2021_weighted_mixing_lemmas_preprint}.
\end{proof}

\begin{thm}[{\cite[Corollary~7.6]{First_2021_weighted_mixing_lemmas_preprint}}]
	\label{TH:cbe-buildings-quotient-sheaves}
	Let $d,q,X,m,A$ be as in Theorem~\ref{TH:cbe-for-buildings}
	and assume that $q>(d-1)^2(m-2)$.
	Let $\{A_x\}_{x\in X(0)}$ be subgroups of $A$
	such that for every subset $S\subseteq X(0)$ with $|S|\leq \ceil{\frac{2}{3}|X(0)|}$,
	the summation map $\bigoplus_{x\in S}A_x\to A$ is injective.
	Define a subsheaf $\calC$ of $\aug{A}$ by 
	$\calC(y)=\sum_{v\in y}A_{\{v\}}$.
	Then $(X,\aug{A}/\calC,{\|\cdot\|}_{\supp})$ is
	a $\veps$-coboudary expander in dimension $0$ for
	\[
	\veps = \frac{2d}{5d+2}-\frac{(4d^3+4d)\sqrt{m-2}}{(5d+2)(\sqrt{q}-(d-1)\sqrt{m-2})}
	-\frac{14d^2+4d}{(5d+2)(q+d-1)}=\frac{2d}{5d+2}-O_{d,m}(q^{-1/2}).
	\]
\end{thm}

The theorems we just recalled
concern with coboundary expansion with respect
to the weighted support norm (Example~\ref{EX:weighted-support}),
but they  can be adapted to the Hamming norms 
of Examples~\ref{EX:non-weighted-support}
and~\ref{EX:Hamming-norm} by means of Proposition~\ref{PR:Hamming-and-support-ratio} and Remark~\ref{RM:properties-of-cbe}(ii).

\section{Locally Minimal Cochains}
\label{sec:locally-min}

Locally minimal  cochains were
introduced in 
\cite{Kaufman_2016_isoperimetic_inequalities}
and
\cite{Evra_2016_cosystolic_expanders_arxiv_version}
for the augmented sheaf $\aug{(\F_2)}$,
and in
\cite{Kaufman_2018_cosystolic_expanders}
for   general constant augmented sheaves 
as a mean to establish cosystolic expansion. 
In this section, we 
extend this notion to all sheaves, and explain how to derive lower bounds  on
the
cosystolic expansion of a   sheaved $d$-complex $(X,\calF)$
from lower bounds on the expansion of locally minimal cochains.

We work exclusively with the weighted support norm
(Example~\ref{EX:weighted-support}), which we denote
by  $\ws{\cdot}$. 

\subsection{Minimal and Locally Minimal Cochains}
\label{subsec:loc-min}

Let $(X,\calF)$ be a sheaved $d$-complex and let $k\in\{0,\dots,d\}$.
A cochain $f\in C^k(X,\calF)$ is called \emph{minimal}
if $\ws{f}_{C^k}=\|f+B^i(X,\calF)\|_{C^k/B^k}$ (see~\S\ref{subsec:coboundary-exp}).
Given $z\in X$ of dimension $i\in\{{0,\dots,k-1}\}$, we say that
$f$ is 
\emph{locally minimal at $z$} if for every $g\in C^{k-i-2}(X_z,\calF_z)$,
we have $\|f+d_{k-1}(g^z)\|\geq \|f\|$,
where   $g^z$ is defined as in \S\ref{subsec:sheaf-at-link} 
and  the vertices of $z$ are given some   ordering (the ordering has no effect
as $g$ can vary).
We say that $f$ is \emph{locally minimal} 
if it is locally
minimal at every $z\in X(0)\cup\dots\cup X(k-1)$.

Clearly, every  minimal cochain is locally minimal.
Also, vacuously, all $0$-cochains are   locally minimal.

\begin{prp}\label{PR:locally-minimal-props}
	Let $(X,\calF)$ be a sheaved $d$-complex, let $k\in \{0,\dots,d\}$
	and
	let $f\in C^k(X,\calF)$. Then:
	\begin{enumerate}[label=(\roman*)]
		\item   $f$  is locally
		minimal at   
		$z\in X_{\ord}(i)$ ($0\leq i<k$)
		if and only if $f_z\in C^{k-i-1}(X_z,\calF_z)$
		is minimal.
		
		\item If $f$ is locally
		minimal at some $z\in X(0)\cup\dots\cup X(k-1)$,
		then $f$ is locally minimal at every $w\in X(0)\cup\dots\cup X(k-1)$
		containing $z$. In particular, $f$ is locally minimal if and only if it is locally minimal at
		every $0$-face of $X$.
	\end{enumerate}	 
\end{prp}

\begin{proof}
(i) 
Let $g\in C^{k-i-2}(X_z,\calF_z)$.
The equivalence follows readily once noting that
\[
\|f-d_{k-1}(g^z)\|=\ws{f-(f_z)^z)}+\ws{(f_z)^z-d_{k-1}(g^z)}
=\ws{f-(f_z)^z)}+\schoose{k+1}{i+1}w_X(z)\ws{f_z-d_{k-i-2}g}_{X_z},
\]
where here, $\ws{\cdot}_{X_z}$ is the weighted support norm
of $\calF_z$ and the second equality follows from
\eqref{EQ:weight-in-link}

(ii)  
Choose orderings on $z$ and $w$
such that $w= uz$ for some $u\in X_{\ord}(j-i-1)$.
Then, regarding $X_w$ as the link of $u$ in $X_z$,
we have $(f_z)_u= f_w$. By (i), $f_z$ is minimal, hence
locally minimal at $u$. Applying (i) again, we see
that $f_w=(f_z)_u$ is minimal, so $f$ is locally minimal at $w$.
\end{proof}

Given a minimal cochain, we can produce more minimal cochains
by annihilating some of its entries.

\begin{lem}\label{LM:restrict-min-coch}
	Let $(X,\calF)$ be a sheaved $d$-complex, let $k\in\{0,\dots,d\}$ and
	let $f,g\in C^k(X,\calF)$.
	Assume $f$ is minimal. If $\supp g\subset \supp f$ and $g(x)=f(x)$
	for all $x\in X_{\ord}(k)$ with $x\in \supp g$, then
	$g$ is minimal.
\end{lem}

\begin{proof}
	Let $b\in B^k(X,\calF)$. We need
	to prove that $\|g\|\leq \|g-b\| $.
	Our assumptions on $g$ imply that $\|g\|=\|f\|-\|f-g\|$.
	As $f$ is minimal,
	$\|f\|-\|f-g\|\leq \|f-b\|-\|f-g\|\leq \|g-b\|$,
	hence the lemma.
\end{proof}


The following lemma
shows that, under mild assumptions on $X$,
every
$(k+1)$-cochain $f$ is equivalent modulo $B^k(X,\calF)$
to a locally minimal cochain $f':=f-d_k g$,
and moreover, that the $k$-cochain $g$
used for ``correcting'' $f$ can chosen  
so that its norm   is proportional to that of $f$.

\begin{lem}\label{LM:locally-min-approximation}
	Let $(X,\calF)$ be a sheaved $d$-complex
	of degree $Q=D(X)$ (see \S\ref{subsec:complexes}).
	Let $k\in\{-1,\dots,d-1\}$ and $f\in C^{k+1}(X,\calF)$.
	Then there exists $g\in C^{k }(X,\calF)$ such that:
	\begin{enumerate}
		\item[(i)] $f-d_{k }g$ is locally minimal,
		\item[(ii)] $\ws{g}\leq \frac{(k+1)Q}{d+1}{d+1 \choose k+2}\ws{f}$,
		\item[(iii)] $\ws{f-d_{k }g}\leq \ws{f}$.
	\end{enumerate}
\end{lem}

\begin{proof}
We define   sequences $f_0,\dots,f_r\in C^{k+1}(X,\calF)$
and $g_0, \dots,g_r\in C^{k}(X,\calF)$ by induction as follows:
Take $f_0=f$ and $g_0=0$.
Assume that $f_n$ and $g_n$ have been defined.
If $f_n$ is locally minimal, we stop and let $r=n$.
Otherwise, $k\geq 0$ and by
Proposition~\ref{PR:locally-minimal-props}(ii), there exist 
$x\in X(0)$ and $g\in C^{k}(X_x,\calF_x)$
such that $\|f_n-d_{k}(g^x)\|<\|f_n\|$. Take
$f_{n+1}=f_n-d_{k-1}(g^x)$ and $g_{n+1}=g^x$.

We claim that $g:=g_0+\dots+g_r$ satisfies the requirements.
Indeed, by construction, $f-d_{k}g=f_r$ is locally minimal
and satisfies $\|f-d_{k}g\|=\|f_r\|<\dots<\|f_0\|=\|f\|$.
Furthermore, since the norm of any $(k+1)$-cochain
in $C^{k+1}(X,\calF)$ is an integral multiple of ${{d+1}\choose {k+2}}^{-1}|X(d)|^{-1}$
(see Example~\ref{EX:weighted-support}),
we have $r\leq {{d+1}\choose{k+2}}|X(d)|\cdot\|f\|$.
On the other hand, each $g_n$ is supported on the $k$-faces containing a particular  vertex  
$v\in X(0)$
and therefore satisfies 
$\|g_n\|\leq w(X(k)_{\supseteq v})= \schoose{k+1}{1}w(v)
\leq \schoose{k+1}{1} |X(d)|^{-1}{{d+1}\choose 1}^{-1} Q=|X(d)|^{-1}\frac{(k+1)Q}{d+1}$
(the first  equality is \eqref{EQ:weight-of-containing-cells}).
Consequently,
$\|g\|\leq r|X(d)|^{-1}\frac{(k+1)Q}{d+1}\leq  \frac{(k+1)Q}{d+1} {{d+1}\choose{k+2}} \|f\|$.
\end{proof}

\subsection{Expansion of Small Locally Minimal Cochains}
\label{subsec:expansion-loc-min}

We continue to assume that $(X,\calF)$ is a sheaved $d$-complex.
Let $k\in\{0,\dots,d-1\}$
and  $\alpha,\beta\in\R_+$.
We say that $(X,\calF)$   \emph{$\beta$-expands $\alpha$-small locally minimal $k$-cochains}
if for every locally minimal $f\in C^k(X,\calF)$ such that $\ws{f}< \alpha$,
we have $\ws{d_k f}\geq \beta \ws{f}$.
We say that $(X,\calF)$    $\beta$-expands $\alpha$-small locally minimal $k$-cocycles 
if this condition holds   for all locally minimal $f\in Z^k(X,\calF)$.

\begin{prp}\label{PR:small-set-to-cse}
	Let $(X,\calF)$ be a sheaved $d$-complex of
	degree $Q=D(X)$ (see \S\ref{subsec:complexes}), 
	let $k\in\{0,\dots,d-2\}$,
	and let $\alpha,\alpha',\beta,\beta'\in\R_+$.
	Suppose that 
	\begin{enumerate}[label=(\arabic*)]
		\item $\calF(x)\neq 0$ for all $x\in X(k)\cup X(k+1)\cup X(k+2)$,
		\item $(X,\calF)$   
	$\beta$-expands $\alpha$-small locally minimal $k$-cocycles, and
		\item $(X,\calF)$   
	 	$\beta'$-expands $\alpha'$-small locally minimal $(k+1)$-cocycles.
	\end{enumerate}
	Then $(X,\calF)$ is a $(\min\{\alpha',\frac{d+1}{(k+1)Q}\schoose{d+1}{k+2}^{-1}\},\alpha)$-cosystolic expander in dimension $k$.
	If only (1) and (2) are assumed, then  
	$\ws{f}\geq \alpha$ for every $f\in Z^k(X,\calF)-B^k(X,\calF)$.
\end{prp}

\begin{proof}
	We need to verify conditions (C1) and (C2) of \S\ref{subsec:coboundary-exp}
	for $\delta=\alpha$ and
	$\veps=\min\{\alpha',\frac{d+1}{(k+1)Q}\schoose{d+1}{k+2}^{-1}\}$. 
	Note that $m(k)=m(k+1)=m(k+2)=1$ by condition (1).
	
	We begin with (C2). Suppose that $f\in Z^k(X,\calF)-B^k(X,\calF)$.
	We need to show that $\ws{f+B^k(X,\calF)}_{C^k/B^k}\geq \alpha$.
	Choose a minimal $f'\in f+B^k(X,\calF)$. Then $f'$ is locally minimal
	and nonzero.
	If $\ws{f'}<\alpha$, then by (2), we would have $0=\ws{d_k f'}\geq \beta \ws{f'}>0$,
	a contradiction. Thus, $\ws{f+B^k(X,\calF)}_{C^k/B^k}=\ws{f'}\geq \alpha$.
	
	We turn to (C1). Let $f\in C^k(X,\calF)$.
	If $\ws{d_0f}\geq \alpha'$, then $\ws{d_0 f}\geq \alpha'\ws{f}\geq \veps \ws{f}$,
	as required. Otherwise, $\ws{d_0 f}<\alpha'$. We apply Lemma~\ref{LM:locally-min-approximation}
	to $d_0f$ to get $g\in C^k(X,\calF)$ such that $d_kf-d_k g$ is a locally minimal
	$(k+1)$-cochain, 
	$\ws{g}\leq \frac{(k+1)Q}{d+1}\schoose{d+1}{k+2}\ws{d_k f}$ and
	$\ws{d_kf-d_kg}\leq \ws{d_kf}<\alpha'$.
	The latter and   (3) imply that $0=\ws{d_{k+1}(d_kf-d_kg)}
	\geq \beta'\ws{d_k f-d_kg}$, so $d_k (f-g)=0$,
	or rather, $g\in f+Z^k(X,\calF)$.
	This means that
	$\ws{f+Z^k(X,\calF)}_{C^k/Z^k}\leq \ws{g}\leq  \frac{(k+1)Q}{d+1}\schoose{d+1}{k+2}\ws{d_k f}$,
	and
	by rearranging, we get 
	\[\ws{d_kf}\geq  
	\frac{d+1}{(k+1)Q}\schoose{d+1}{k+2}^{-1}\ws{f+Z^k(X,\calF)}_{C^k/Z^k}\geq \veps \ws{f+Z^k(X,\calF)}_{C^k/Z^k},\] 
	which is what we want.
\end{proof}

	We will give a  criterion for sufficiently small locally minimal
	cochains to expand 
	in Section~\ref{sec:cse-from-links}. 

\section{Locally Testable Codes and Quantum CSS Codes Arising from Sheaves}
\label{sec:ltc-and-css}

We now explain how sheaved complexes which are good cosystolic expanders
give rise to locally testable codes and quantum CSS codes. 
We further show that
if the sheaved complex in question expands small locally minimal cochains,
then
there is an efficient decoding algorithm.

\subsection{Conventions}

As usual, a code of length $n$
on a finite alphabet $\Sigma$  consists of a pair
$C=(C,\Sigma^n)$ such that $C$ is  a subset of $\Sigma^n$;
we often simply say that $C$ is a code inside $\Sigma^n$.
Given $f,g\in\Sigma^n$, the Hamming distance and the noramlized
Hamming distance of $f$ from $g$ are 
\[
D_{\Ham}(f,g)=\#\{i\in\{1,\dots,n\}\suchthat f_i\neq g_i\}
\qquad\text{and}\qquad
d_{\Ham}(f,g)=\frac{1}{n}D_{\Ham}(f,g),
\]
respectively.
The \emph{distance} of $C$  is $ \Delta(C):=\max\{D_{\Ham}(f,g)\where f,g\in C,\, f\neq g\}$
and the \emph{relative distance} of $C$ is $\delta(C):=\frac{1}{n}\Delta(C)=
\max\{d_{\Ham}(f,g)\where f,g\in C,\, f\neq g\}$.
The \emph{message length} of
$C$ is   $\log_{|\Sigma|}|C|$ and its \emph{rate}
is the message length divided by $n$, i.e., $\log_{|\Sigma^n|}|C|$.

If $\Sigma$ is a finite field $\F$ (resp.\ an abelian group), then the code
$C$ is said to be linear (resp.\ abelian) if $C$ is an $\F$-subspace (resp.\ subgroup) 
of $\Sigma^n$. In this case, $\delta(C)=\min\{\|f\|_{\Ham}\where f\in C-\{0\}\}$,
where $\|f\|_{\Ham}=d_{\Ham}(f,0)$ is normalized Hamming norm of $f\in\Sigma^n$. 
In the linear case, the message length of $C$ is $\dim_{\F} C$.

A family of codes $\{(C_i,\Sigma^{n_i})\}_{i\in I}$ is said to be \emph{good}
there are $\rho,\delta\in (0,1]$ such that
each $C_i$ has rate $\geq \rho$ and relative distance   $\geq \delta$.
When the latter holds, we also say that the codes  $\{(C_i,\Sigma^{n_i})\}_{i\in I}$ have
linear distance (indeed, $\Delta(C_i)\geq \delta n_i$ for all $i\in I$).

Let $(C,\Sigma^n)$ be a code
and let $\eta\in [0,1]$.
A  \emph{decoding algorithm} for  
$(C,\Sigma^n)$     able to to correct up to $\eta n$
errors, or an $\eta$-fraction of errors, is an algorithm
which takes a word $f\in \Sigma^n$ with $d_{\Ham}(f,C)< \eta$
and outputs some $g\in C$ with $d_{\Ham}(f,g)<\eta$.
If $\eta \leq  \frac{1}{2}\delta(C)$, then $g$ is uniquely determined
by $f$.

\subsection{Cocycle Codes}
\label{subsec:cocycle-codes}

We use the following general notation  throughout the rest 
of this section:
\begin{itemize}
	\item $X$ is a simplicial complex of dimension $d$,
	\item $\F$ is a finite field with $q$ elements,
	\item $\calF$ is an $\F$-sheaf on $X$,
	\item $B$ is an $\F$-basis of $\calF$, i.e.,
	a collection $B=\{B(x)\}_{x\in X-\{\emptyset\}}$
	such that $B(x)$ is an $\F$-basis of $\calF(x)$.
	\item $\ws{\cdot}$, $\nws{\cdot}$
	and $\|\cdot\|_B $ denote the weighted support norm,
	the normalized Hamming norm, and the (non-normalized)
	Hamming norm of $\calF$ w.r.t.\ $B$, respectively 
	(see Examples~\ref{EX:weighted-support},
	\ref{EX:non-weighted-support}, \ref{EX:Hamming-norm}).
\end{itemize}
We associate the following parameters to with the above data:
\begin{itemize}
	\item $Q=D(X)$ and $P=D_{k,d}(X)$ (see \S\ref{subsec:complexes}),
	\item $M_k=M_k(\calF)=\max\{\dim \calF(x)\where x\in X(k)\}$,
	and $M=M(\calF)=\max\{M_0,\dots,M_d\}$.
\end{itemize}
We fix a linear ordering   on $V(X)$ and use it to identify
$C^k:=C^k(X,\calF)$ with $\prod_{x\in X(k)} \calF(x)$; see Remark~\ref{RM:ordered-cohomology}.
We abbreviate $ Z^k(X,\calF)$ to $Z^k$ and $ B^k(X,\calF)$ to $B^k$.

\medskip

Let $k\in\{0,\dots,d\}$. 
We use the data of $X,\calF,B,k$ to construct a linear code as follows:
The ambient space of the code will be $C^k=\prod_{x\in X(k)} \calF(x)$,
which identify with $\F^{\dim C^k}$ using the basis
$\bigsqcup_{x\in X(k)} B(x)$,
and the set of code words will be $Z^k$.
Thus, $(Z^k,C^k\cong \F^{\dim C^k})$ is a linear code with alphabet
$\F$.

\begin{dfn}
	The linear code $(Z^k(X,\calF),C^k(X,\calF)\cong \F^{\dim C^k})$
	is  the  \emph{linear $k$-cocycle code} of $(X,\calF,B)$.
\end{dfn}

If there exists   $m\in \N$
such that $\dim \calF(x)=m$
for all $x\in X(k)$, then we may   identify
$\calF(x)$ with $\Sigma:=\F^m$ for every $x\in X(k)$, so that
$C^k=\Sigma^{X(k)}$.
This allows us to view $Z^k$ as an abelian code
inside $\Sigma^{X(k)}$ (rather than $\F^{\dim C^k}$), 
the alphabet being $\Sigma$.

\begin{dfn}
	The abelian code $(Z^k(X,\calF),C^k(X,\calF)\cong \Sigma^{X(k)})$
	is  the \emph{$k$-cocycle code} of $(X,\calF )$.
\end{dfn}

Henceforth, whenever we refer to the $k$-cocycle code of $(X,\calF)$,
we tacitly assume that there exists $m\in\N$ such that $\dim\calF(x)=m$
for all $x\in X(k)$.

The rate of the $k$-cocycle code
$(Z^k, C^k)$ is $\dim Z^k/\dim C^k$;
this is independent of whether we view $(Z^k, C^k)$ 
as a linear code with alphabet $\F$, or an abelian code with alphabet
$\Sigma$. Since $B^k$ typically contains $k$-cochains
with small support, the distance of $(Z^k,C^k)$ is  poor unless $B^k=0$,
e.g., if $k=0$, or $\calF(x)=0$ for all $x\in X(k-1)$.  

\subsection{Locally Testable Codes}
\label{subsec:ltc}

Let $C=(C,\Sigma^n)$ be a code on  a finite alphabet $\Sigma$.
Recall that  a  randomized algorithm $\Phi$ which takes a word $f\in \Sigma^n$
and decides whether to accept or reject $f$
is called a \emph{$t$-query $\mu$-tester} ($t\in\N$, $\mu\in \R_+$) if $\Phi$
queries
up to $t$ letters from from $f$,
accepts all words $f\in C$,
and the  probability of rejecting a word $g\in \Sigma^n-C$ is at least
$\mu d_{\Ham}(g,C)$. 
In this case, we call  
$(C,\Sigma^n,\Phi)$  a \emph{$t$-query $\mu$-testable} code.
We also say that $(C,\Sigma^n,\Phi)$ is a \emph{code with a tester}
if we wish to make no reference to $t$ and $\mu$.

A family of codes with testers $(C_i,\Sigma^{n_i},\Phi_i)_{i\in I}$
is a family of \emph{locally testable codes} (LTCs) if there are $t\in\N$ 
and $\mu\in \R_+$ such that each $(C_i,\Sigma^{n_i},\Phi_i)$ is a $t$-query $\mu$-testable code.

\begin{remark}
The quality of a tester $\Phi$ for a code $(C,\Sigma^n)$
can also be measured by means of \emph{soundness} and \emph{tolerance}.
Recall that $\Phi$, or the triple
$(C,\Sigma^n,\Phi)$, is said to have \emph{$c$-soundness} at least $\veps$ ($c\in [0,\frac{1}{2}]$,
$\veps\in[0,1]$) if
$\Phi$ rejects an $f\in\Sigma^n$ 
satisfying $d_\Ham(f,C)>  c\operatorname{dist}(C)$
with probability at least $\veps$. 
It has \emph{$c$-tolerance} at least $\veps$ if $\Phi$ accepts an $f\in\Sigma^n$ 
satisfying $d_\Ham(f,C)\leq  c\operatorname{dist}(C)$
with probability 
at least $\veps$.
Thus, a $t$-query $\mu$-testable
code of relative distance $\delta$  has $c$-soundness $\geq c\delta\mu$
and $0$-tolerance $1$ (i.e., the tester is perfect).
\end{remark}

Keeping the notation of \S\ref{subsec:cocycle-codes},
we fix  $k\in\{0,\dots,d-1\}$,
and  assume throughout  that there exists $m\in\N$
such that $\dim\calF(x)=m$ for all $x\in X(k)$.
We further let $\Sigma=\F^m$.

The   $k$-cocycle code of $(X,\calF)$,
namely  $(Z^k,C^k\cong\Sigma^{X(k)}) $,  
admits a natural   $(k+2)$-query tester $\Phi=\Phi(X,\calF,k)$:
given $f\in C^k$, choose $y\in X(k+1)$ uniformly at random\footnote{
	For the applications considered in this
	paper, it is better to choose $y\in X(k+1)$ according to the distribution 
	induced by the canonical weight function of $X$ (\S\ref{subsec:weights}).
	Using this distribution improves
	the testability   by a factor of $Q=D(X)$ in Proposition~\ref{PR:cbe-vs-ltc-weighted-supp}, 
	but makes statements more cumbersome elsewhere,
	so we stick to the uniform distribution.
} and accept $f$ if $\sum_{x\in X(k)_{\subseteq y}} \res_{y\from x}f(x)=0$.

\begin{dfn}
The triple  $(Z^k,C^k,\Phi )$ is
called  the \emph{$k$-cocycle code-with-tester}
of $(X,\calF )$. 
\end{dfn}

When there is no risk of confusion,
we will refer to $(Z^k,C^k,\Phi)$ 
simply as the $k$-cocycle code $(X,\calF)$.

\begin{remark}
The tester $\Phi$ can also be
viewed as a tester for the linear $k$-cocycle code
of $(X,\calF,B)$. From this point of view,  $\Phi$
queries $(k+2)M_k$ letters.
One can use this observation to adapt
the   following discussion  to linear $k$-cocycle codes.
\end{remark}

The distance and the testability of $(Z^k,C^k,\Phi)$ are tightly related
to the coboundary expansion of $(X,\calF,\nws{\cdot} )$, where
$\nws{\cdot}$ is the normalized Hamming norm of $\calF$ (Example~\ref{EX:non-weighted-support}). 
This is expressed
in the following proposition, which is immediate   from the definitions, see \S\ref{subsec:normed-sheaves}--\ref{subsec:coboundary-exp}.

\begin{prp}\label{PR:cbe-vs-ltc}
	With notation as in \S\ref{subsec:cocycle-codes}, suppose that $\dim\calF(x)=m$ for all $x\in X(k)$. 
	Let $(Z^k,C^k,\Phi)$
	be the $k$-cocycle code on the alphabet $\Sigma=\F^m$  associated to $(X,\calF)$,
	and let $\veps,\delta\in \R_+$. Suppose that
	$B^k=0$ (e.g., if $k=0$) and $\calF(x)\neq 0$ for all $x\in X(k)\cup X(k+1)$.
	Then the following conditions are equivalent:
	\begin{enumerate}[label=(\alph*)]
		\item $(X,\calF,\nws{\cdot})$ is an $(\veps,\delta)$-coboundary expander in dimension $k$;
		\item $(Z^k,C^k,\Phi)$ is a $(k+2)$-query $\veps$-testable code with relative distance $\geq \delta$.
	\end{enumerate}
\end{prp}

We can also relate the distance and   testability
of $(Z^k,C^k,\Phi)$ to the coboundary expansion of $(X,\calF)$   relative to the weighted
support norm $\|\cdot\|$. 
We loose a factor of $P=D_{k,d}(X)$ in the process.

\begin{prp}\label{PR:cbe-vs-ltc-weighted-supp}
	Keep the   assumptions of Proposition~\ref{PR:cbe-vs-ltc}
	and suppose further that $X$ is a $d$-complex.
	If
	$(X,\calF)$ is an $(\veps,\delta)$-coboundary expander in dimension $k$,
	then $(Z^k,C^k,\Phi)$ a $(k+1)$-query $\frac{\veps}{P^2}$-testable code
	with  relative distance $\geq \frac{\delta}{P }$.
	Conversely, if  $(Z^k,C^k,\Phi)$ a $(k+1)$-query $\veps$-testable code
	with   relative distance $\delta$, then $(X,\calF)$
	is an $(\frac{\veps}{P^2},\frac{\delta}{P})$-coboundary expander.
\end{prp}

\begin{proof}
	This follows from Remark~\ref{RM:properties-of-cbe}(ii),
	Proposition~\ref{PR:Hamming-and-support-ratio}(i) and
	Proposition~\ref{PR:cbe-vs-ltc}.
\end{proof}

We have seen in Proposition~\ref{PR:small-set-to-cse} that if $(X,\calF)$ expands
small locally minimal $k$-cocycles and $(k+1)$-cocycles, then $(X,\calF)$ is a good
coboundary expander in dimension $k$,  and so the associated
$k$-cocycle code   $(Z^k,C^k,\Phi)$ is  an LTC. 
We now show that these  stronger assumptions also
guarantee that $(Z^k,C^k,\Phi)$ has a linear-time decoding algorithm.
The weaker assumption that  small locally minimal $k$-cocycles
expand is enough to bound the distance of $(Z^k,C^k,\Phi)$ from below.

\begin{prp}\label{PR:decoding}
	With notation as in~\S\ref{subsec:cocycle-codes},
	suppose that $X$ is a $d$-complex, $\dim \calF(x)=m$ for all $x\in X(k)$
	and $\calF(x)\neq 0$ for all $x\in X(k)\cup X(k+1)\cup X(k+2)$.
	Let $(Z^k,C^k,\Phi)$ be the $k$-cocycle code of $(X,\calF)$,
	put $n=|X(k)|$ (the length of the code)
	and let $\beta,\beta',\gamma,\gamma'\in \R_+$.
	\begin{enumerate}[label=(\roman*)]
		\item If $B^k=0$ and $(X,\calF)$ $\beta$-expands
		$\gamma$-small locally minimal $k$-cocycles,
		then $\delta(Z^k)\geq \frac{\gamma}{P}$.
		\item If, in addition, $(X,\calF)$ $\beta'$-expands
		$\gamma'$-small locally minimal $(k+1)$-cocycles,
		then $(Z^k,C^k,\Phi)$ 
		is $\frac{1}{P^2}\min\{ \frac{d+1}{(k+1)Q}\schoose{d+1}{k+2}^{-1}   ,\gamma'\}$-testable		
		and has a decoding algorithm
		able to correct up to
		$\frac{1}{(k+2)P }
	\min\{(\frac{(k+1)Q}{d+1}\schoose{d+1}{k+2} +1)^{-1}\gamma,\gamma'\}$-fraction
		of errors
		in
		$O(2^{{d+1 \choose k+2}Q} Q^4 M_{k+1}^2 m \cdot n )=O_{M,Q,d,k}(n)$ operations.
	\end{enumerate}
\end{prp}

Here and elsewhere in this work, we assume that operations
in $\F$ are performed in time $O(1)$.
Elements of $\calF(x)$ are represented as vectors in $\F^{B(x)}$ via the basis $B(x)$,
and the restriction maps   $\res^{\calF}_{y\from  x}:\calF(x)\to \calF(y)$ are represented by 
the  corresponding $B(y)\times B(x)$-indexed matrix.
The complexity of evaluating 
$\res^{\calF}_{y\from x}$ therefore depends on the number of nonzero entries in this
matrix, which is $O(M_{\dim x}M_{\dim y})$. Recall that our assumptions imply   $M_k=m$.

\begin{proof}
	(i) 
	Let $f\in Z^k-\{0\}=Z^k-B^k$.
	By Proposition~\ref{PR:small-set-to-cse},
	$\ws{f}\geq \gamma$, and by  
	Proposition~\ref{PR:Hamming-and-support-ratio}(i),
	this means that $\nws{f}\geq \schoose{d+1}{k+1}\frac{|X(d)|}{|X(k)|}P^{-1}\gamma
	\geq \frac{\gamma}{P}$.

	(ii)
	By Proposition~\ref{PR:small-set-to-cse}(i),
	$(X,\calF)$ is a $(\min\{ \frac{d+1}{(k+1)Q}\schoose{d+1}{k+2}^{-1}   ,\gamma'\},
	\gamma)$-cosystolic expander. The assertion about the testability
	is therefore a consequence of Proposition~\ref{PR:cbe-vs-ltc-weighted-supp}.
	It remains to show that $(Z^k,C^k )$ has a decoding algorithm
	as claimed.
	 

	Write $\eta:=\frac{1}{(k+2)P }
	\min\{(\frac{(k+1)Q}{d+1}\schoose{d+1}{k+2} +1)^{-1}\gamma,\gamma'\}$.
	Let $f\in C^k$ and assume that there is 
	$f_0\in Z^k$ with $\nws{f-f_0}<\eta  $.
	We claim that the output of the algorithm in Figure~\ref{FG:decoding-algo}
	is $f_0$. 
	
\begin{figure}[h]
\caption{Decoding Algorithm}
{\tt
\begin{enumerate}
\item $f'\leftarrow f$ 
\item  $L \leftarrow$ empty queue
\item  $B \leftarrow $ boolean array indexed
by $X(0)$ 
\item For each $z\in X(0)$: 
	\begin{enumerate}
	\item[(4a)] $L$.push($z$)
	\item[(4b)] $B[z]\leftarrow$ True // $z$ is in $L$ 
	\end{enumerate}
\item While $L$ is not empty:
	\begin{enumerate}
	\item[(5a)] $z \leftarrow L$.pop(), order the vertices of $z$ arbitrarily
	\item[(5b)] $B[z]\leftarrow$ False // $z$ is not in $L$ 
	\item[(5c)] 
		Search for  $h\in C^{k-1}(X_z,\calF_z)$ with $\ws{(d_k f')_z-d_{k-1} h}<\ws{(d_k f')_z}$; set $h=0$ if there is no such $h$.
	\item[(5d)] If $h\neq 0$:
		\begin{enumerate}
		\item $f' \leftarrow f' - h^z$.
		\item For every  $z'\in X(0)$ adjacent to $z$ with $B[z']=\mathtt{False}$: 
			\begin{enumerate}
			\item $L.$push($z'$) 
			\item $B[z']\leftarrow$ True // $z'$ is in $L$ 
			\end{enumerate}
		\end{enumerate}
	\end{enumerate}
\item Return $f'$.
\end{enumerate}
}
\label{FG:decoding-algo}
\end{figure}

	To see this, observe that by Proposition~\ref{PR:Hamming-and-support-ratio}(i),
	$\ws{f-f_0}<\schoose{d+1}{k+1}^{-1}P\frac{|X(k)|}{|X(d)|}\eta
	\leq P\eta$.
	Since $f_0\in Z^k$, this means that
	\begin{align*}
		\ws{d_kf}&=\ws{d_k(f-f_0)}
		\leq \sum_{x\in \supp (f-f_0)}w(X(k+1)_{\supseteq x})
		=\sum_{x\in \supp(f-f_0)} (k+2)w(x)\\
		&=(k+2)\ws{f-f_0}<(k+2) P \eta\leq \gamma',
	\end{align*}
	where in the second equality we used \eqref{EQ:weight-of-containing-cells}.
	By applying Lemma~\ref{LM:locally-min-approximation} to $d_kf$,
	we get  $g\in C^k$ such that $d_k f-d_kg$ is locally minimal,
	$\ws{g}\leq \frac{(k+1)Q}{d+1}\schoose{d+1}{k+2}\ws{d_k f}<
	\frac{(k+1)Q}{d+1}\schoose{d+1}{k+2} (k+2)P \eta$ 
	and $\|d_kf -d_kg\|\leq \|d_k f\|<\gamma'$.
	Moreover, by comparing the proof of Lemma~\ref{LM:locally-min-approximation}
	with the algorithm in Figure~\ref{FG:decoding-algo}, we see that the output $f'$
	is in fact $f-g$ (for a suitable choice of vertices in the proof).
	Since $d_kf-d_kg$ is $\gamma'$-small and locally
	minimal, we have $0=\|d_k(d_kf-d_kg)\|\geq \beta' \|d_kf-d_kg\|$, so $d_kf-d_kg=0$,
	and it follows that
	$f'=f-g\in Z^k$.
	Now, 
	\[\ws{f'-f_0}\leq \ws{f-f'}+\ws{f-f_0}=\ws{g}+\ws{f-f_0}<
	(k+2)P \eta\cdot(\frac{(k+1)Q}{d+1}\schoose{d+1}{k+2} +1)\leq \gamma .\]
	Since   $f'-f_0\in Z^k$ (because $f',f_0\in Z_k$)
	and $(X,\calF)$ $\beta$-expands
	$\gamma$-small cochains, 
	Proposition~\ref{PR:small-set-to-cse} tells us that $f'-f_0\in B^k=0$,
	so $f'=f_0$,
	as required.

	We proceed with analyzing the time 
	complexity of the algorithm in Figure~\ref{FG:decoding-algo}. 
	The proof of Lemma~\ref{LM:locally-min-approximation}
tells us that the loop
{\tt (5)} cannot be executed more than ${d+1 \choose k+2} |X(d)|=O(|X(k)|P)=O(Q n)$
times (see~\eqref{EQ:degree-ratio}).  
In order to perform the instruction {\tt (5c)}, we have to enumerate on  $O(q^{Qm})$ possible
$h$-s. For each $h$, the 
computation of $\ws{(d_k f')_z-d_{k-1}h}$ takes $O(Q m M_{k+1})$ operations,
so naively, {\tt (5c)} requires $O(q^{Q m}Qm M_{k+1})$ operations. However, it is better
to enumerate on subsets $E\subseteq X_z(k)$ and look for some $h$ 
such that $(d_k f')_z-d_{k-1}h$ vanishes on $E$  by solving the linear system
of equations $\{(d_{k-1}h)(x)=(d_k f')(x\cup z)\}_{x\in E}$.
This allows us to perform {\tt (5c)} by solving at most $2^{|X_z(k)|}\leq 2^{D_{1,k+1}(X)}
\leq 2^{{d+1 \choose k+2}Q}$ 
systems of at
most $M_{k+1}|X_z(k)|=O(M_{k+1}Q)$  linear equations in at most
$m|X_z(k-1)|=O(mQ)$ variables, which amounts to $O(2^{{d+1 \choose k+2}Q} Q^3 M_{k+1}^2 m)$ operations.
The remaining actions inside the loop {\tt (5)} are negligible by comparison, so
the instructions {\tt (5a)}--{\tt (5d)}
require   $O(2^{{d+1 \choose k+2}Q} Q^3 M_{k+1}^2 m )$ operations. 
The total time complexity of the algorithm is therefore
$O(2^{{d+1 \choose k+2}Q} Q^4 M_{k+1}^2 m \cdot n )$.
\end{proof}

\subsection{Quantum CSS codes}
\label{subsec:css}

We proceed with explaining how sheaved complexes
give rise to    quantum CSS codes.
We refer the reader to \cite[\S1, \S2.2, \S2.3, Lemma~13]{Leverrier_2021_LT_quantum_codes_preprint}
for a survey of these codes and their  significance to quantum computing. 

For every $m\in\N$, we endow $\F^m$ with the standard symmetric bilinear form
$\Trings{f,g}=\sum_{i=1}^m f_ig_i$, where $f_i$ is the $i$-th coordinate of $f\in \F^m$.
Given  $A\subseteq \F^m$, we write $A^\perp=\{f\in \F^m\suchthat \Trings{f,A}=0\}$.

For the purposes of this work it is convenient
to define a \emph{quantum CSS code} as a quintet
$C=(C_X,C_Z,\F^n,\Phi_X,\Phi_Z)$ such that $C_X$
and $C_Z$ are subspaces of $\F^n$,   $\Phi_X$ is a set of vectors generating $C_X^\perp$,
$\Phi_Z$ is a set of vectors generating $C_Z^\perp$, and $C_X^\perp\subseteq C_Z$ (equivalently,
$C_Z^\perp\subseteq C_X$).
In particular, $C_X$ and $C_Z$ are linear codes inside $\F^n$,
the alphabet being $\F$.
The rate of $C$ is $\frac{1}{n}(\dim C_X-\dim C_Z^\perp)=\frac{1}{n}(\dim C_Z-\dim C_X^\perp)$
and its relative distance is $\min\{d_X,d_{Z}\}$, where $d_X=\min \{\|w\|_{\Ham}\where w\in C_X-C_Z^\perp\}$
and $d_{Z}=\min \{\|w\|_{\Ham}\where w\in C_Z-C_X^\perp\}$;
we call $d_X$ and $d_Z$ the relative $X$- and $Z$-distance, respectively.
(The   distance and message length of $C$ are 
obtained from the relative distance and rate by multiplying by $n$, respectively.)

Given $\eta\in [0,1]$, a  decoding algorithm for the $X$-side of $C=(C_X,C_Z,\F^n,\Phi_X,\Phi_Z)$ 
able to correct up to an $\eta$-fraction
of errors is a
decoding for   $(C_X,\F^n)$ able to correct
to up to $\eta$-fraction of errors. Note that if $f\in \F^n$
satisfies $d_\Ham(f,C_X)<\eta$,
then there could be numerous $x\in C_X$ with $d_{\Ham}(f,x)<\eta$,
because   $C_Z^\perp$ may (and often does) contain short vectors. 
However, if $2\eta $ is smaller than the relative $X$-distance,
then the coset $x+C_Z^\perp$ is uniquely determined.

We use the generating set  $\Phi_X$ to define a natural tester for the linear code
$C_X\subseteq \F^n$: given $f\in C_X$, choose $\phi\in\Phi_X$ uniformly at random, and
accept $f$ if $\Trings{f,\phi}=0$. Abusing the notation, we denote this tester
by $\Phi_X$. Likewise, we use the set $\Phi_Z$ to define a tester for $C_Z\subseteq \F^n$.
Let $q\in\N$ and $\mu\in \R_+$.
We say that   $C=(C_X,C_Z,\F^n,\Phi_X,\Phi_Z)$
is a \emph{$q$-query  $\mu$-testable} quantum CSS code, if this holds
for both linear codes  $(C_X,\F^n,\Phi_X)$ and $(C_Z,\F^n,\Phi_Z)$.
If this holds only for $(C_X,\F^n,\Phi_X)$, we say that $C$
is \emph{one-sided} $q$-query $\mu$-testable quantum CSS code.


\medskip

Let $X,\calF,B,d$ be as in \S\ref{subsec:cocycle-codes}, and let
$k\in\{1,\dots,d-1\}$.
Write $n=\dim C^k$ and identify $C^k=C^k(X,\calF)$ with $\F^{n}$
via the basis $\bigsqcup_{x\in X(k)}B(x)$.
Then $Z^k=Z^k(X,\calF)$ is a linear code inside $\F^n$.
We enrich $Z^k$ into a quantum CSS code as follows.

Put $C_X=Z^k$.
For every $x\in X-\{\emptyset\}$, we 
identify $\calF(x)$ with $\F^{B(x)}$ via the basis $B(x)$. Under this
identification, the standard bilinear form on $\F^{B(x)}$
corresponds to a nondegnerate   bilinear form on $\calF(x)\times\calF(x)\to\F$, denoted $\Trings{\cdot,\cdot}_x$.
The standard bilinear form on $C^k=\F^n$ can now be written
as $\Trings{f,g}=\sum_{x\in X(k)}\Trings{f(x),g(x)}_x$ for $f,g\in C^k=\prod_{x\in X(k)}\calF(x)$.
For every $\emptyset\neq x\subsetneq y\in X$, define $\res'_{x\from y}:\calF(y)\to \calF(x)$
to be the dual of $\res_{y\from x}:\calF(x)\to\calF(y)$ relative to the bilinear pairings
on $\calF(x)$ and $\calF(y)$, that is,
$\res'_{y\from x}$ is determined by the condition
$\Trings{\res'_{y\from x}f,g}_y=\Trings{f,\res_{x\from y} g}_x$
for all $f\in \calF(x)$, $g\in \calF(y)$. 
For $i\in\{1,\dots,d\}$, the $i$-th \emph{boundary} map $\partial_i:C^i\to C^{i-1}$
is defined by
\[
(\partial_i f)(y)=\sum_{x\in X(i)_{\supseteq y}}[x:y]_L\res'_{y\from x} f,
\]
for all $f\in C^k$ and $y\in X(i-1)$;
the coefficient $[x:y]_L$ is as in Remark~\ref{RM:ordered-cohomology}.\footnote{
	If one does not wish to choose a linear ordering on $V(X)$ 
	and identify $C^i(X,\calF)$ with $\prod_{x\in X(i)}\calF(x)$ as in Remark~\ref{RM:ordered-cohomology},
	then the formula is given by
	$(\partial_i f)(y)=\sum_{z\in X(i)_y}\res'_{y\from x}f(yz)$, where $y\in X_{\ord}(i-1)$.
} One readily checks
that $\partial_i$ is the dual of $d_{i-1}:C^{i-1}\to C^i$ relative to the bilinear pairings on these
vector spaces, i.e.,
\begin{equation}\label{EQ:boundary-is-dual}
\trings{\partial_i f,g}=\trings{f,d_{i-1} g}
\end{equation} 
for all $f\in C^i$, $g\in C^{i-1}$.
Since $d_{i} d_{i-1}=0$, we have $\partial_{i-1}\partial_i=0$, with the convention that $\partial_0=0$.
As expected, the \emph{$k$-cycles} and \emph{$k$-boundaries} with coefficients in $\calF$
are defined to be the subspaces of $C^i$ given by 
\[
Z_i=Z_i(X,\calF;B)=\ker d_{i}\qquad\text{and}\qquad B_i=B_i(X,\calF;B)=\im d_{i-1},
\]
respectively. Set $C_Z=Z_k$. 

We now define subsets $\Phi_X\subseteq (Z^k)^\perp$ and $\Phi_Z\subseteq (Z_k)^\perp$
as follows. For every $y\in X(k+1)$ and $b\in B(y)$, define $\phi_{y,b}\in C^k$
to be the unique vector for which $\Trings{f,\phi_{y,b}}=\Trings{d_k f(y),b} $ for all
$f\in C^k$,
and let $\Phi_X$ be the set of all $\phi_{y,b}$. Similarly, for every
$z\in X(k-1)$ and $b\in B(z)$, define $\phi'_{z,b}\in C^k$
to be the unique vector for which $\Trings{f,\phi'_{z,b}}=\Trings{\partial_k f(z),b} $
for all $f\in C^k$, and let $\Phi_Z$ be the set of all $\phi'_{z,b}$. It is clear that $\Phi_X$ generates
$(Z^k)^\perp$ and $\Phi_Z$ generates $(Z_k)^\perp$.
The following lemma says that $C_X^\perp\subseteq C_Z$, and thus 
$C:=(C_X,C_Z,C^k,\Phi_X,\Phi_Z)$ is a quantum CSS code.

\begin{lem}\label{LM:orth-relations}
	In the previous notation, $(Z^k)^\perp=B_k$ and $(Z_k)^\perp =B^k$.
\end{lem}

\begin{proof}
	It is enough to prove that $B_k^\perp=Z^k$ and $(B^k)^\perp=Z_k$.
	We have $f\in B_k^\perp$ if and only if $\Trings{f,\partial_{k+1}g}=0$
	for all $g\in C^{k+1}$, or equivalently $\Trings{d_kf,g}=0$ for all $g\in C^{k+1}$.
	Since the bilinear form on $C^{k+1}$ is nondegenerate, the latter is equivalent to $d_kf=0$.
	This proves that $B_k^\perp=Z^k$. The equality $(B^k)^\perp=Z_k$ is shown similarly.
\end{proof}

\begin{dfn}
	We call $(C_X=Z^k,C_Z=Z_k,C^k,\Phi_X,\Phi_Z)$ defined above the \emph{$k$-cocycle quantum CSS code} 
	associated to $(X,\calF,B)$.
\end{dfn}

At this point, it is convenient to introduce an analogue
of cosystolic expansion which uses   boundary maps  instead of   coboundary maps.
Let $ \veps,\delta\in\R_+$ and let $\|\cdot\|$ be a norm on $\calF$
with mass function $m$ (see~\S\ref{subsec:normed-sheaves}).
We say that $(X,\calF,\|\cdot\|)$
is an  \emph{$(\veps,\delta)$-systolic expander}, if
\begin{enumerate}[label=(S\arabic*)]
	\item  $\|\partial_kf\|  m(k)\geq \veps \|f+Z_k(X,\calF)\|_{ C^k/Z_k} m(k-1)$ 
	for all $f\in C^k(X,\calF)$, and
	\item  $\|f\| \geq \delta   m(k)$ for all $f\in Z_k(X,\calF)-B_k(X,\calF)$.
\end{enumerate}

The following proposition relates the cosystolic and systolic expansion
of $(X,\calF,\|\cdot\|_B)$ to the code-theoretic properties
of the $k$-cocycle quantum CSS code associated to $(X,\calF,B)$. 
Adapting the result to use the weight-support norm instead of $\|\cdot\|_B$
can be done using Proposition~\ref{PR:Hamming-and-support-ratio}(ii)
and Remark~\ref{RM:properties-of-cbe}(ii),
and is left to the reader.

\begin{prp}\label{PR:cbe-to-quantum-CSS}
	Let $X,\calF,B,d,Q,P $ be as in \S\ref{subsec:cocycle-codes},
	let $k\in\{1,\dots,d-1\}$ and let
	$C:=(C_X,C_Z,C^k,\Phi_X,\Phi_Z)$ be the $k$-cocycle quantum CSS code associated to $(X,\calF,B)$.
	Then:
	\begin{enumerate}[label=(\roman*)]
		\item $C$ is a  quantum  CSS code, its rate is $\frac{1}{n}\dim\HH^k(X,\calF)$,
		and the testers $\Phi_X$ and $\Phi_Z$
		query  
		$(k+2)M_k$ and $ D_{k-1,k}(X)(d+1-k)M_k$  letters, respectively.
		\item    $(X,\calF,\|\cdot\|_B)$ is an $(\veps,\delta)$-cosystolic
		expander and an $(\veps,\delta)$-systolic expander in dimension $k$
		if and only if
		$C$ has relative distance $\geq \delta$ and is $\veps$-testable.
		\item  $(X,\calF,\|\cdot\|_B)$ is an $(\veps,\delta)$-cosystolic
		expander in dimension $k$ if and only if $C$ has relative $X$-distance $\geq \delta$
		and is one-sided $\veps$-locally testable.
		\item Suppose
		that $\calF(x)\neq 0$ for all $x\in X(k)\cup X(k+1)\cup X(k+2)$
		and  there are $\beta,\beta',\gamma,\gamma'\in\R_+$ 
		such that $(X,\calF)$ $\beta$-expands $\gamma$-small 
		$k$-cocycles and $\beta'$-expands $\gamma'$-small
		$(k+1)$-cocycles. 
		Then  
		the $X$-side of $C$ admits an error correcting algorithm able
		to correct up to $\frac{1}{(k+2)P M_k}
	\min\{(\frac{(k+1)Q }{d+1}\schoose{d+1}{k+2} +1)^{-1}\gamma,\gamma'\}$-fraction
		of errors
		in $O_{M,Q,d,k}(n)$ operations.
	\end{enumerate}
\end{prp}

\begin{proof}
	(i) 
	It is straightforward to see
	that every $\phi_{y,b}\in \Phi_X$ is supported on at most $(k+2)M_k$ coordinates.
	Let $z\in X(k-1)$. Then $z$ is contained in at most $D_{k-1,k}(X)$ $d$-faces.
	This means
	that every   $\phi'_{z,b}\in \Phi_Z$ is supported at most $D_{k-1,k}(X)M_k$ coordinates.
	The assertion about the rate follows from Lemma~\ref{LM:orth-relations}.

	(ii) and (iii) are immediate from the definitions.
	
	(iv) is shown as in Proposition~\ref{PR:decoding} with two differences.
	First,  we use part (ii) of
	Proposition~\ref{PR:Hamming-and-support-ratio}  instead of part (i).
	In particular, 
	writing $\eta:=	\frac{1}{(k+2)P M_k}
	\min\{(\frac{(k+1)Q }{d+1}\schoose{d+1}{k+2} +1)^{-1}\gamma,\gamma'\}$,
	the assumption $\|f-f_0\|_B\leq \eta n\leq \eta M_k |X(k)| $ gives
	$\|f-f_0\|\leq \schoose{d+1}{k+1}^{-1}\frac{|X(k)|}{|X(d)|}PM_k\eta\leq PM_k \eta$,
	and   extra $M_k$ is carried throughout the computations.
	Second, 
	instead of asserting that 
	$f'=f_0$, we only conclude that $f_0-f'\in B^k\subseteq (Z_k)^\perp=C_Z^\perp$.
\end{proof}

The subject matter of Section~\ref{sec:cse-from-links} is the construction  
of  cosystolic expanders, which in turn give rise to LTCs.
Unfortunately, we do not know of analogous results for \emph{systolic expansion}, so
our results   only give rise to quantum CSS codes whose $X$-side is an LTC with linear distance,
and {\it a priori} no information on the $Z$-side.

\chapter{The Tower Paradigm}
\label{chap:tower}

	In this chapter we provide a method for constructing good infinite
	families of LTCs from a single sheaved complex $(X,\calF)$, called
	the \emph{tower paradigm}. 
	Broadly, the idea is to assume that $X$ admits an infinite tower of 
	double coverings
	$\dots\to X_2\to X_1\to X_0=X$, and take the cocycle codes associated to the
	sheaved complexes $\{(X_r,u_r^*\calF)\}_{r=0}^{\infty}$, where $u_r:X_r\to X$ is the composition
	$X_r\to X_{r-1}\to\dots\to X_0=X$. We show that if $(X,\calF)$ satisfies
	a list of conditions, then this family of codes is a good family of LTCs.
	Specifically, the main result of 
	Section~\ref{sec:cse-from-links} --- a local-to-global principle for cosystolic expansion
	of sheaves --- provides the conditions on $(X,\calF)$ that would secure linear distance
	and   testability for the codes, and the main result of Section~\ref{sec:rate}
	--- rate conservation --- gives    conditions on $(X,\calF)$ 
	that are sufficient 
	for the rate of the codes to be constant. We pack these results  together in Section~\ref{sec:recap}
	to give the tower paradigm. The intermediate Section~\ref{sec:quo-aff-buildings}
	gives   examples of cocycle codes to which the local-to-global principle
	of Section~\ref{sec:cse-from-links} can be applied.

\section{A Local-to-Global Principle for Cosystolic Expansion}
\label{sec:cse-from-links}

Let $(X,\calF)$ be a sheaved connected $d$-complex.
The purpose of this section is to prove the following theorem,
which provides \emph{local}   condition on $(X,\calF)$ 
guaranteeing that $(X,\calF)$ is a good cosystolic expander in dimension $k$
(see \S\ref{subsec:coboundary-exp}),
and moreover, expands small locally minimal $k$-cochains and $(k+1)$-cochains
(see~\S\ref{subsec:expansion-loc-min}).
By saying that the conditions are local we mean that 
they involve only the links $(X_z,\calF_z)$
with $z\in X-\{\emptyset\}$.

Following the convention set in \S\ref{subsec:sheaf-at-link}, 
we say that $(X,\calF)$ is an \emph{$i$-local $\veps$-coboundary expander} in dimension $k$
($-1\leq i\leq k$) if $(X_z,\calF_z)$ is an $\veps$-coboundary expander   in dimension
$k-i-1$ for all $z\in X(i)$.
Recall also (\S\ref{subsec:skeleton}) that $X$ is said to be an $i$-local $[-1,\lambda]$-spectral
expander ($-1\leq i\leq d-2$)  if the underlying weighted graph of $X_z$ is a $[-1,\lambda]$-spectral
expander for all $z\in X(i)$. Likewise for $\alpha$-skeleton expansion.

\begin{thm}\label{TH:cosystolic-exp-from-links}
	Let  
	$k\in\N\cup \{0\}$,
$Q\in \N$, $\veps_0 ,\dots,\veps_k ,\veps'_0 ,\dots,\veps'_{k+1}$ and $\lambda \in \R_+$.
Put
\begin{align*}
\veps &=\min\left\{\frac{(k+2)\veps_i}{k+1-i}\bigg| i\in\{0,\dots,k\}\right\}
\quad\text{and}\quad
\veps'  =\min\left\{\frac{(k+3)\veps'_i}{k+2-i}\bigg| i\in\{0,\dots,k+1\}\right\},
\end{align*}
and suppose that
\[
\lambda \leq 
{\textstyle\frac{1 }{d}}\min\{({\textstyle\frac{\veps }{(k+1)^2 2^{2k+6}}})^{2^{k }},
(\textstyle{\frac{\veps'}{(k+2)^22^{2k+8}})^{2^{k+1}}},1\}
\]
Let $(X,\calF)$ be a  sheaved  \emph{strongly connected} $d$-complex with $d\geq k+2$  such that:
	\begin{enumerate}
	\item[(1)]  $(X ,\calF )$ is an $i$-local $\veps_{i}$-coboundary
	expander in dimension $k$ for all $i\in\{0,\dots,k\}$. 
	\item[(2)]  $(X,\calF)$ is an $i$-local $\veps'_{i}$-coboundary
	expander in dimension  $k+1$ for 
	all $i\in\{0,\dots,k+1\}$.
	\item[(3)] 
	$X$ is a $(d-2)$-local $[-1,\lambda]$-spectral expander.
	\item[(4)] $D(X)\leq Q$, i.e., every vertex of $X$ is belongs to  at most $Q$ $d$-faces,  
	\item[(5)] $\calF(x)\neq 0$ for all $x\in X(k)\cup X(k+1)\cup X(k+2)$.
	\end{enumerate}
	Then:
	\begin{enumerate}[label=(\roman*)]
		\item
		$(X,\calF)$ $\frac{\veps }{2}$-expands $(\frac{\veps }{(k+1)^22^{2k+6}})^{2^{k+1}-1}$-small locally minimal 
		$k$-cochains.
		\item  
		$(X,\calF)$ $\frac{\veps'}{2}$-expands $(\frac{\veps'}{(k+2)^22^{2k+8}})^{2^{k+2}-1}$-small locally minimal 
		$(k+1)$-cochains.
		\item  $(X,\calF )$ is a $(\min\{(\frac{\veps'}{(k+2)^22^{2k+8}})^{2^{k+2}-1},
	\frac{d+1}{(k+1)Q}\schoose{d+1}{k+2}^{-1}\} ,(\frac{\veps}{(k+1)^2 2^{2k+6}})^{2^{k+1}-1} )$-cosystolic expander in dimension $k$.
	\end{enumerate}	 
\end{thm}

This theorem is an immediate consequence of the following theorem and Proposition~\ref{PR:small-set-to-cse}.

\begin{thm}\label{TH:expansion-of-small-sets}
	Let $k\in\N\cup\{0\}$, let  $\veps_0,\dots,\veps_k\in \R_+$,
	and set
	\[
	\veps:=\min\left\{\frac{(k+2)\veps_i}{k+1-i}\bigg| i\in\{0,\dots,k\}\right\}
	\qquad
	\text{and}
	\qquad
	\lambda=\frac{1}{d}\min\left\{\left(\frac{\veps}{(k+1)^2 2^{2k+6}}\right)^{2^{k}},1\right\}.
	\]
	Let $(X,\calF)$ be a sheaved \emph{strongly connected} $d$-complex
	with $d\geq k+1$ such that:
	\begin{enumerate}[label=(\arabic*)]
		\item 
		$(X,\calF)$ is an $i$-local $\veps_i$-coboundary expander
		in dimension $k$ for all $i\in \{0,\dots,k\}$.		
	\item  
	$X$ is a $(d-2)$-local $[-1,\lambda]$-spectral expander.
	\item $\calF(x)\neq 0$ for all $x\in X(k+1)$.
	\end{enumerate}
	Then $(X,\calF)$ $\frac{\veps}{2}$-expands $(\frac{\veps}{(k+1)^2 2^{2k+6}})^{2^{k+1}-1}$-small
	$k$-cochains.
\end{thm}

Before setting to prove  Theorem~\ref{TH:expansion-of-small-sets}, a few remarks are in order.

\begin{remark}\label{RM:comments-on-local-to-global}
(i) Theorems~\ref{TH:expansion-of-small-sets} and~\ref{TH:cosystolic-exp-from-links}
were proved for the constant   sheaf $\F_2$
(Example~\ref{EX:sheaves-basic-examples}(ii)) 
in \cite{Evra_2016_cosystolic_expanders_arxiv_version}, with  different constants,
and this was extended to all constant   sheaves in \cite{Kaufman_2018_cosystolic_expanders}.
As in these sources, the proof of Theorem~\ref{TH:expansion-of-small-sets} uses
machinery of \emph{heavy faces}\footnote{
	Called \emph{fat faces}
	in \cite{Evra_2016_cosystolic_expanders_arxiv_version}
	and \cite{Kaufman_2018_cosystolic_expanders}.}. 
However, we use a different summation argument that makes
milder assumptions and gives explicit
and asymptotically better expansion constants.

(ii) By Lemma~\ref{LM:max-cbe}, the  value of $\veps$ (resp.\ $\veps'$)
in   Theorems~\ref{TH:cosystolic-exp-from-links}
and~\ref{TH:expansion-of-small-sets} cannot exceed $k+2$ (resp.\ $k+3$).
We do not know if there exist examples approaching this upper bound.

(iii) In Theorem~\ref{TH:expansion-of-small-sets},
condition (2)  
and the assumption that $X$ is connected 
can be   replaced with the milder assumption
that $X$ is an $i$-local $\alpha_i$-skeleton expander
for for all $i\in\{-1,\dots,k-1\}$ and $\alpha_i\leq \Theta_k(\veps^{2^{k-1-i}})$.
Condition (3) of Theorem~\ref{TH:cosystolic-exp-from-links} can be similarly relaxed.
See Theorem~\ref{TH:expansion-of-small-sets-finer} and the following corollaries 
for the precise requirements
on the $\alpha_i$, and Corollary~\ref{CR:cse-from-links-dim-2} for a version
of Theorem~\ref{TH:cosystolic-exp-from-links} where the upper bounds on the $\alpha_i$
are optimized for $k=0$.
Our   methods cannot increase the  
order magnitude of the required $i$-local skeleton expansion 
of $X$ (as a function of $\veps_0,\dots,\veps_k$) 
beyond
$\Theta_k(\veps^{2^{k-1-i}})$; see Remark~\ref{RM:constant-analysis-general}.

(iv) 
Let $(X,\calF)$ be a sheaved strongly connected $d$-complex and let $p:Y\to X$ be a covering map
such that $Y$ is connected (and hence strongly connected).
Recall (\S\ref{subsec:pushforward}) that $p^*\calF$ denotes the pullback of $\calF$ to $Y$.
Then, for every $y\in Y-\{\emptyset\}$, the map $p$ restricts
to an isomorphism $Y_y\cong X_{f(y)}$, and under this isomorphism 
we have $(p^*\calF)_y = \calF_{f(y)}$.
This means that the assumptions of 
Theorem~\ref{TH:cosystolic-exp-from-links}
(resp.\ Theorem~\ref{TH:expansion-of-small-sets})
hold for $(X,\calF)$ if and only if they hold for $(Y,p^*\calF)$.
\end{remark}

The proof of Theorem~\ref{TH:expansion-of-small-sets} will be given 
in \S\ref{subsec:proof-small-set-exp}, after some 
after some preliminary results have been established in \S\ref{subsec:heavy}.
Examples of sheaved complexes satisfying the assumptions of 
Theorem~\ref{TH:cosystolic-exp-from-links} are given 
in Section~\ref{sec:quo-aff-buildings}, \gap{}and further examples will be given in Section~\ref{sec:candidates}.

\subsection{Heavy Faces}
\label{subsec:heavy}

Fix a sheaved $d$-complex $(X,\calF)$.
Unless indicated otherwise,
$k\in\{0,\dots,d\}$, 
$f\in C^k(X,\calF)$ and $\vec{h}=(h_{-1},h_0,\dots,h_{k-1})\in (0,1]^{\{-1,\dots,k-1\}}$. 
Recall from \S\ref{subsec:weights}
that $w=w_X:X\to \R_+$ denotes the canonical weight function of $X$.

Generalizing \cite[\S3.3]{Kaufman_2018_cosystolic_expanders} 
and \cite[\S3.2]{Evra_2016_cosystolic_expanders_arxiv_version}, we define 
for every $i\in\{ -1,\dots, k\}$  a set $A_k(f,\vec{h})\subseteq X(k)$ as follows:
Set $A_k(f,\vec{h})=\supp f$.
Assuming $A_i(f,\vec{h})$ was defined,
let $A_{i-1}(f,\vec{h})$ consist of the faces
$x\in X(i-1)$
such that 
\[w(  A_i(f,\vec{h})_{\supseteq x})\geq  
h_{i-1} w(X(i)_{\supseteq x}).\]
In other words, $x\in A_{i-1}(f,\vec{h})$ if at least $h_{i-1}$-fraction of the $i$-faces containing $x$
(counted by weight)
are in $A_i(f,\vec{h})$.
Elements of $A_i(f,\vec{h})$ are called \emph{$(f,\vec{h})$-heavy} $i$-faces, or  just \emph{heavy} 
$i$-faces for short.

\begin{lem}\label{LM:convenient-lem}
Let $X$ be a $d$-complex,
let $-1\leq i\leq k\leq d$ and let $A\subseteq X(k)$. Then
$
\sum_{z\in X(i)}w(A_{\supseteq z})={\textstyle{ {k+1} \choose {i+1}}} w(A)
$.
\end{lem}

\begin{proof} We have
\begin{align*}
\sum_{z\in X(i)}w(A_{\supseteq z})  &=
\sum_{z\in X(i)}\sum_{x\in A:z\subseteq x} w(x)=
\sum_{x\in A}\sum_{z\in X(i):z\subseteq x} w(x)
=
\sum_{x\in A} {\textstyle{ {k+1} \choose {i+1}}}w(x)=
{\textstyle{ {k+1} \choose {i+1}}} w(A).\qedhere
\end{align*}
\end{proof}

\begin{lem}\label{LM:bound-heavy-set}
Let $f\in C^k(X,\calF)$, $\vec{h}\in (0,1]^{\{-1,0,\dots,k-1\}}$
and $i\in\{-1,0,\dots,k\}$.
Then $w(A_i(f,\vec{h}))\leq (\prod_{i\leq j<k}h_j^{-1})\ws{f}$.
\end{lem}

\begin{proof}
	This is clear if $i=k$, so assume $i<k$. In this case, using 
	\eqref{EQ:weight-of-containing-cells}, the definition of heaviness,
	and
	Lemma~\ref{LM:convenient-lem}, we see that
	\begin{align*}
	w(A_i(f, \vec{h}))
	& = \sum_{x\in A_i(f,\vec{h})} w(x)
	=\sum_{x\in A_i(f,\vec{h})} (i+2)^{-1}w(X(i+1)_{\supseteq x})\\
	&\leq 
	h_i^{-1} {(i+2)}^{-1} \sum_{x\in A_i(f,\vec{h})} 
	w(A_{i+1}(f,\vec{h})_{\supseteq x})
	\leq h_i^{-1}w(A_{i+1}(f,\vec{h})).
	\end{align*}
	Iterating, we find that
	\[
	w(A_i(f,\vec{h}))\leq h_i^{-1}h_{i+1}^{-1}\cdots h_{k-1}^{-1}  w(A_k(f,h))=
	(\prod_{i\leq j<k}h_j^{-1})\ws{f}. \qedhere
	\]
\end{proof}

It can happen that the intersection of two $(f,\vec{h})$-heavy $i$-faces ($0\leq i\leq k$)
is an $(i-1)$-face  which is  not  $(f,\vec{h})$-heavy. We call such pairs \emph{$(f,\vec{h})$-bad}, or just \emph{bad}.
Provided $k<\dim X$, we also say that a $(k+1)$-face is $(f,\vec{h} )$-bad if it contains a bad pair
of faces. 
The set of $(f,\vec{h} )$-bad $(k+1)$-faces
is denoted $\Upsilon (f,h)$.

\begin{lem}\label{LM:bad-faces-bound}
Let $f\in C^k(X,\calF)$
with $k\in\{0,\dots,d-1\}$, 
and let $\vec{h},\vec{\alpha}  \in (0,1]^{\{-1, \dots,k-1\}}$.
Suppose  that 
$X$ is an $i$-local $\alpha_i$-skeleton expander for all $i\in\{-1,\dots,k-1\}$ (in particular,
$X=X_{\emptyset}$ is an $\alpha_{-1}$-skeleton expander).
Then
\[
w(\Upsilon (f,\vec{h}))\leq  \sum_{i=0}^{k}   {\textstyle {{k+2}\choose {i+2}}  }(i+1)  (\alpha_{i-1}+h_{i-1}) h_i^{-1}\cdots h_{k-1}^{-1}\ws{f}.
\]
\end{lem}

\begin{proof}
	Fix $i\in\{0,\dots,k\}$ and $z\in X(i-1)$. We call an $(i+1)$-face $e$
	$z$-bad if $e\supseteq z$ and the two $i$-faces lying between $z$ and $e$
	form a bad pair; denote by $B(z)$ the set of $z$-bad faces.
	Let $e\in B(z)$ and let    $x,y$ be the $i$-faces between $z$
	and $e$. Then $e-z$ is an edge connecting the $0$-faces $x-z$ and $y-z$ in the link $X_z$.
	Since both $x-z,y-z\in A_i(f,h)_z$ (because $x$ and $y$ are heavy),  
	our assumption that $X_z$ is an $\alpha_{i-1}$-skeleton expander
	implies that
	\[
	w_{ z}(B(z)_z)\leq  
	(w_z(A_i(f,h)_z)+\alpha_{i-1})w_z(A_i(f,h)_z),
	\]	
	where $w_z=w_{X_z}$.
	Since $z$ is not heavy ($x,y$ is a bad pair),
	$w(A_i(f,h)_{\supseteq z})\leq h_{i-1} w(X(i)_{\supseteq z})$,
	which means that 
	$w_z(A_i(f,h)_z)\leq h_{i-1}w_z(X(i)_z)=h_{i-1}$, by \eqref{EQ:weight-in-link}.
	Thus,
	\[
	w_z(B(z)_z)\leq  (\alpha_{i-1}+h_{i-1})w_{ z}(A_i(f,h)_z).
	\]
	Scaling both sides using \eqref{EQ:weight-in-link},
	we get
	\[
	w(B(z))\leq (\alpha_{i-1}+h_{i-1})w (A_i(f,h)_{\supseteq z} ).
	\]
	
	Now, since every face in $\Upsilon(f,\vec{h})$ contains a face in $B(z)$
	for some $z$, we have
\begin{align*}
	w(&\Upsilon (f,\vec{h})) 
	 \leq  \sum_{i=0}^{k}\sum_{z\in X(i-1)}\sum_{e\in B(z)} w(X(k+1)_{\supseteq e}).
\end{align*}
	Using \eqref{EQ:weight-of-containing-cells}, Lemma~\ref{LM:convenient-lem} and Lemma~\ref{LM:bound-heavy-set}, the right hand side evaluates to
	\begin{align*}
	\sum_{i=j}^{k}\sum_{z\in X(i-1)}\sum_{e\in B(z)} &{\textstyle {{k+2}\choose {i+2}}}w(e) 
	 = \sum_{i=0}^{k}\sum_{z\in X(i-1)} {\textstyle {{k+2}\choose {i+2}}} w(B(z))\\
	&\leq 
	\sum_{i=0}^{k}
	\sum_{z\in X(i-1)} {\textstyle {{k+2}\choose {i+2}}} (\alpha_{i-1}+h_{i-1}) w (A_i(f,h)_{\supseteq z} )
	\\
	& =\sum_{i=0}^{k}     {\textstyle {{k+2}\choose {i+2}}  } (i+1) (\alpha_{i-1}+h_{i-1}) w (A_i(f,h) ) \\
	&\leq  \sum_{i=0}^{k}   {\textstyle {{k+2}\choose {i+2}}  }(i+1)  (\alpha_{i-1}+h_{i-1}) h_i^{-1}\cdots h_{k-1}^{-1}\ws{f}. \qedhere
	\end{align*}
\end{proof}

We continue to assume that  $f\in C^k(X,\calF)$ and $\vec{h}\in (0,1]^{\{-1,\dots,k-1\}}$.
Given two heavy faces $x,y$ with $x\subseteq y$, we say that $y$ \emph{$(f,\vec{h})$-descends}
to $x$, or $x$ \emph{$(f,\vec{h})$-descends} from $y$,
if there exists a sequence
$x=x_i\subseteq x_{i+1}\subseteq\dots\subseteq x_{\ell} =y$ with $x_j\in A_j(f,\vec{h})$
for all $j\in\{i,i+1,\dots,\ell\}$. We will simply say that $y$ descends to $x$ if there is no
risk of confusion.
We say that a $(k+1)$-face $y$   descends to a 
heavy face $x$ if $y$ contains a heavy $k$-face descending to $x$.
A  face will be called \emph{$(f,\vec{h})$-terminal},
or just terminal, if it is heavy and does not descend to any of its proper faces. 
It is clear that every heavy face descends to some terminal face. Beware that a terminal face 
may contain another terminal face.

\begin{lem}\label{LM:unique-terminal}
	Let $f\in C^k(X,\calF)$ and $\vec{h}\in (0,1]^{\{-1,\dots,k-1\}}$.
	Let $y\in X(k+1)-\Upsilon(f,\vec{h})$ and let
	$D(y)$ denote the set of heavy faces which $(f,\vec{h})$-descend from $y$.
	If $D(y)\neq\emptyset$, then there exists exactly one terminal face $z$
	descending from $y$.
	Moreover, every face in $D(y)$ descends to $z$. 
\end{lem}

\begin{proof}
	Since $D(y)\neq\emptyset$, the face $y$ must descend to some terminal face $z$.
	In  order to prove the lemma, it is enough to show that
	every $z'\in D(y)$  descends to $z$.

	By definition, there
	are   sequences  $z=x_r\subseteq \dots \subseteq x_k \subseteq y$
	and $z'=x'_s\subseteq \dots\subseteq x'_k\subseteq y$ such
	that $x_j$ and $x'_j$ are in $A_j(f,\vec{h})$ for all $j$. 
	Set $x_{k+1}=x'_{k+1}=y$.
	We claim that $x_i\cap x'_j$ is heavy for all $i\in\{r,\dots,k,k+1\}$
	and $j\in\{s,\dots,k,k+1\}$. We show this   by decreasing induction
	on $i$ and $j$. The claim is clear if $i=k+1$ or $j=k+1$, so assume
	$i,j\leq k$.
	By the induction hypothesis,   $x_{i+1}\cap x'_j$
	and $x_i\cap x'_{j+1}$ are both heavy.
	If $x_{i+1}\cap x'_j=x_i\cap x'_j$, or $x_i\cap x'_{j+1}=x_i\cap x'_j$,
	then $x_i\cap x'_j$ is also heavy.
	Otherwise, the dimension of both $x_{i+1}\cap x'_j$
	and $x_i\cap x'_{j+1}$ is $\dim (x_i\cap x'_j)+1$,
	and $(x_{i+1}\cap x'_j)\cap (x_i\cap x'_{j+1})=x_i\cap x'_j$.
	Since $y\notin \Upsilon(f,\vec{h})$,
	the face  $x_i\cap x'_j$ must be    heavy as well, hence our claim.
	
	To finish, 
	consider the sequence of heavy faces $z=z\cap x'_{k+1}\supseteq z\cap x'_k\supseteq \dots
	\supseteq z\cap x'_s=z\cap z'$.
	The difference between
	the dimensions of every two consecutive faces in this sequence is either $0$ or $1$,
	so $z$ descends to $z\cap z'$. Since $z$ is terminal, we must have $z=z\cap z'$,
	or rather, $z\subseteq z'$. By a similar argument, $z'$ descends to $z'\cap z=z$,
	which is what we want.
\end{proof}

\begin{lem}\label{LM:terminal-ineq}
	Let $f\in C^k(X,\calF)$ with $k\in\{0,\dots,d-1\}$, and let $\vec{h}\in (0,1]^{\{-1,\dots,k-1\}}$.
	Let $z$ be an $(f,\vec{h})$-terminal   face,   and let     
	\begin{align*}
	L(z)&= \{x\in X(k)\suchthat\text{$x$ $(f,\vec{h})$-descends to $z$}\},\\
	L'(z)&= \{y\in X(k+1)\suchthat\text{$y$ $(f,\vec{h})$-descends to $z$}\}.
	\end{align*}
	Suppose
	that $f$ is locally minimal at $z$ (see \S\ref{subsec:loc-min}),   $(X_z,\calF_z)$ is an $\veps$-coboundary
	expander in dimension $k-\dim z-1$ and $\calF(y)\neq 0$ for all $y\in X(k+1)_{\supseteq z}$.
	Then
	\[
	\frac{ (k+2)\veps}{k+1-\dim z}w(L(z))\leq 
	w([\supp (d_0 f)\cup \Upsilon(f,\vec{h})]\cap L'(z))
	\]
\end{lem}

\begin{proof}
	Write $i=\dim z$ and fix some ordering on the vertices of $z$. Define $g\in C^k(X,\calF)$ by
	\[
g(x)=\left\{\begin{array}{ll}
f(x) & x\in L(z) \\
0 & \text{otherwise},
\end{array}\right. 
\]
where  $x\in X(k)_{\ord}$. Then $(g_z)^z=g$ (notation as in \S\ref{subsec:sheaf-at-link}). 
By Proposition~\ref{PR:locally-minimal-props}(i),
$f_z\in C^{k-i-1}(X_z,\calF_z)$ is minimal, and by Lemma~\ref{LM:restrict-min-coch},
so is $g_z$. 
Write $\|\cdot\|_{z}$ for the weighted support norm on $\calF_z$ and $m_z$
for its associated mass function. Since $\calF(y)\neq 0$ for all $y\in X(k+1)_{\supseteq z}$,
we have $m_z(k-i)=1$.
Our assumption that $(X_z,\calF_z)$ is an $\veps$-coboundary expander in dimension $k-i-1$,
therefore implies that
\[
\|d_{k-i-1} g_z\|_z\geq \|d_{k-i-1} g_z\|_{z}m_z(k-i-1)\geq \veps\|g_z\|_{z}m_z(k-i)=\veps\|g_z\|_z
.\]
By Lemma~\ref{LM:f-eta-basic-props},
$(d_{k-1-1} (g_z))^z=d_k ((g_z)^z)=d_k g$.
Using this and~\eqref{EQ:weight-in-link}, we find that
\begin{align}\label{EQ:g-norm-bound}
\ws{d_k g}&\geq \schoose{k+2}{i+1} \schoose{k+1}{i+1}^{-1}\veps \ws{g}=
\frac{ (k+2)\veps}{k+1-i}w(L(z)).
\end{align}

Let $y\in \supp( d_k g)$. By the definition of $g$, the face $y$ descends
to $z$, that is, $y\in L'(z)$.
If $y\notin \Upsilon(f,\vec{h})$,
then by Lemma~\ref{LM:unique-terminal}, every face descended from
$y$ also descends to $z$. In particular, every $x\in (\supp f)\cap X(k)_{\subseteq y}$
descends to $z$, and thus belongs to $L(z)=\supp g$. It follows
that $(d_kf)(y)=(d_kg)(y)\neq 0$, so $y\in \supp (d_k f)$.
This shows that
\[
\supp (d_k g)\subseteq [\supp (d_k f)\cup \Upsilon (f,\vec{h})]\cap L'(z).
\]
Combining this with \eqref{EQ:g-norm-bound} gives the lemma.
\end{proof}

\begin{notation}\label{NT:Upsilon-crit-value}
	We call a collection of subsets $E\subseteq P(\{1,\dots,n\})$
	an $n$-\emph{vine} if:
	\begin{enumerate}[label=(\arabic*)]
		\item $\{1,\dots,n\}\in E$,
		\item Every $s\in E$ admits a sequence 
		$s=s_i\subseteq s_{i+1}\subseteq\dots\subseteq s_{n}=\{1,\dots,n\}$
		such that $s_j\in E$ and $|s_j|=j$ for all $j$.
	\end{enumerate}
	We say that $s\in E$ is terminal if no maximal subset of $s$ is in $E$.
	(It is possible for non-maximal subsets of $s$ to be in $E$.)
	Denote by $T(E)$ the terminal subsets in $E$.
	Finally, set
	\[
	U(n)=\max\{\#T(E)\where \text{$E $ is an $n$-vine}\}.
	\]
	Direct computation shows that $U(1)=1$, $U(2)=2$ and $U(3)=3$.
	In general, we have $ \schoose{n}{\floor{n/2}}\leq U(n)\leq 2^{n }-1$.\footnote{
		The number $U(n)$ is larger than $\schoose{n}{\floor{n/2}}$ for large $n$.
		Indeed, assuming $n=4k$, consider the $n$-vine $E=\{s\subseteq\{1,\dots,n\}\suchthat
		|s|\geq 2k\}\cup\{s\subseteq \{1,\dots,2k\}\suchthat |s|\geq k\}$.
		It routine to check that $T(E)=\schoose{n}{2k}-(2k)^2-1+\schoose{2k}{k}$,
		which is larger than $\schoose{n}{\floor{n/2}}$ as soon as $n\geq 16$.
	}
\end{notation}

\begin{lem}\label{LM:locally-min-norm-bound}
	Let $f\in C^k(X,\calF)$ be locally minimal (see~\S\ref{subsec:loc-min}), 
	let $\vec{h}\in (0,1]^{\{-1,\dots,k-1\}}$
	and let $\veps_{0},\dots,\veps_{k }\in\R_+$.
	Suppose that $\calF(x)\neq 0$ for all $x\in X(k+1)$,
	and that 
	$(X,\calF)$ is an $i$-local $\veps_i$-coboundary expander in dimension
	$k$ for every $i\in\{0,\dots,k\}$.	
	If the empty face of $X$ is not $(f,\vec{h})$-heavy, then
	\[
	\min\left\{\frac{(k+2)\veps_i}{k+1-i}\,\bigg|\, i\in\{0,\dots,k\}\right\}\ws{f}
	\leq \ws{d_0 f}+U(k+2)w(\Upsilon(f,\vec{h})).
	\]
\end{lem}

\begin{proof}
	Denote by $T$ the set of  $(f,\vec{h})$-terminal faces.
	Given $z\in T$, define $L(z)$ and $L'(z)$
	as in Lemma~\ref{LM:terminal-ineq}. 
	We abbreviate 	$\Upsilon(f,\vec{h})$ to $\Upsilon$.
	
	Let $z\in T$.
	By assumption,
	$z\neq \emptyset$, so   $f$ is locally minimal at $z$ and
	$(X_z,\calF_z)$ is an $\veps_{\dim z}$-coboundary expander in dimension $k-\dim z-1$. 
	Lemma~\ref{LM:terminal-ineq} now tells us that 	
	\begin{align*}
		\frac{ (k+2)\veps}{k+1-\dim z}w(L(z))
		& \leq  
	w([\supp (d_0 f)\cup \Upsilon ]\cap L'(z))\\
		&= w([\supp(d_0 f)-\Upsilon]\cap L'(z))+
		w(\Upsilon\cap L'(z)).
	\end{align*}
	Summing over all $z\in T$, we get
	\begin{align}\label{EQ:sum-over-terminal-faces}
		\sum_{z\in T} \frac{ (k+2)\veps}{k+1-\dim z}w(L(z))
		\leq 
		\sum_{z\in T}
		w([\supp (d_0 f)-\Upsilon ]\cap L'(z))+
		\sum_{z\in T}
		w(\Upsilon\cap L'(z)).
	\end{align}
	
	Since every face in $\supp f$ descends to some terminal face,
	the left hand side of \eqref{EQ:sum-over-terminal-faces}
	is at least
	\[
	\min\left\{\frac{(k+2)\veps_i}{k+1-i}\,\bigg|\, i\in\{0,\dots,k\}\right\}\|f\|.
	\]
	As for the right hand side of \eqref{EQ:sum-over-terminal-faces},
	by Lemma~\ref{LM:unique-terminal}, every face in $y\in \supp(d_0 f)-\Upsilon$
	descends to a unique terminal face. Thus,
	\[
	\sum_{z\in T}
		w([\supp (d_0 f)-\Upsilon ]\cap L'(z))= w(\supp d_0 f-\Upsilon)\leq \|d_0 f\|.
	\]
	If $y\in \Upsilon$, then upon identifying
	$y$ with $\{1,\dots,k+2\}$, the set of   faces to which
	$y$ descends is a $(k+2)$-vine in the sense Notation~\ref{NT:Upsilon-crit-value}.
	Thus, the number of terminal faces to which $y$ descends is at most $U(k+2)$,
	meaning that
	\[
	\sum_{z\in T}
		w(\Upsilon\cap L'(z))\leq U(k+2)w(\Upsilon).
	\]
	Plugging these observations into \eqref{EQ:sum-over-terminal-faces}
	gives the lemma.
\end{proof}

\subsection{Proof of Theorem~\ref{TH:expansion-of-small-sets}}
\label{subsec:proof-small-set-exp}

We will deduce Theorem~\ref{TH:expansion-of-small-sets}
from the following more general theorem.

\begin{thm}\label{TH:expansion-of-small-sets-finer}
	Let   $k\in\N\cup\{0\}$,
	$\alpha_0,\dots,\alpha_{k-1},\veps_0,\dots,\veps_{k}\in \R_+$,
	and put
	\[
	\veps:=\min\left\{\frac{(k+2)\veps_i}{k+1-i}\,\bigg|\, i\in\{0,\dots,k\}\right\}.
	\]
	Suppose that   there are $h_{-1},\dots,h_{k-1}\in (0,1]$ such that:
	\begin{equation}\label{EQ:base-ineq}
	U(k+2)
	\sum_{i=0}^{k}   {\textstyle {{k+2}\choose {i+2}}  }(i+1) \frac{\alpha_{i-1}+h_{i-1}}{h_i\cdots h_{k-1}}
	<\veps,
	\end{equation}
	where $U(k+2)$ is as in Notation~\ref{NT:Upsilon-crit-value}.
	Then there exist  $\beta,\gamma\in\R_+$ such that the following hold:
	Let $(X,\calF)$ be a sheaved $d$-complex, where $d\geq k+1$. Assume that
	\begin{enumerate}
	\item[(1)] $(X,\calF)$ is an $i$-local $\veps_{i}$-coboundary 	expander
	in dimension $k$ for all $i\in\{0,\dots,k\}$.
	\item[(2)] 
	$X$ is an $i$-local $\alpha_i$-skeleton expander for all $i\in\{-1,\dots,k-1\}$.
	\item[(3)] $\calF(x)\neq 0$ for all $x\in X(k+1)$.
	\end{enumerate}
	Then $(X,\calF)$ $\beta$-expands $\gamma$-small
	locally minimal $k$-cochains.
	In  fact, one can take $\gamma=h_{-1}\cdots h_{k-1}$ and 
	$\beta$ to be the difference between the right hand side
	and the left hand side of \eqref{EQ:base-ineq}.
\end{thm}

\begin{proof}
	Put $\vec{h}=(h_{-1},\dots,h_{k-1})$ and define $\beta$ 
	and $\gamma$ as in the theorem.
	Let $f\in C^k(X,\calF)$ be a locally minimal $k$-cochain such that
	$\ws{f}<\gamma$. We need to prove that $\|d_0f\|\geq \beta\|f\|$.
	
	We first claim that the empty face is not $(f,\vec{h})$-heavy.
	Indeed, by Lemma~\ref{LM:bound-heavy-set},
	$w(A_{-1}(f,\vec{h}))\leq (h_{-1}h_0\cdots h_{k-1})^{-1}\|f\|<
	(h_{-1}h_0\cdots h_{k-1})^{-1}\gamma=1$.
	Since the empty face has weight $1$, this means that $A_{-1}(f,\vec{h})=\emptyset$,
	so the empty face is not heavy.
	
	We may now apply Lemma~\ref{LM:locally-min-norm-bound}, which tells us that
	\[
	\veps\|f\|\leq \|d_0f\|+U(k+2)w(\Upsilon(f,\vec{h})).
	\]
	By Lemma~\ref{LM:bad-faces-bound}, this means that
	\[
	\veps\|f\|\leq \|d_0 f\|+U(k+2)\sum_{i=0}^{k}   {\textstyle {{k+2}\choose {i+2}}  }(i+1) \frac{\alpha_{i-1}+h_{i-1}}{h_i\cdots h_{k-1}}\|f\|,
	\]
	and by rearranging, we get
	$
	\beta \|f\|\leq \|d_0 f\| 
	$.
\end{proof}

\begin{proof}[Proof of Theorem~\ref{TH:expansion-of-small-sets}]
	Assumption (2) and Oppenheim's Trickling Down Theorem
	\cite[Theorem~1.4]{Oppenheim_2015_vanishing_of_cohomology} imply that
	for every $z\in X$ of dimension $i\in\{-1,\dots,d-2\}$, the weighted
	underlying graph of $X_z$
	is a $[-1,\frac{\lambda}{1-(d-2-i)\lambda}]$-spectral expander, and thus $X_z$
	is a $\frac{\lambda}{1-(d-2-i)\lambda}$-skeleton expander (see \S\ref{subsec:skeleton}).
	By the assumptions on $\lambda$, we have
	\[
	\frac{\lambda}{1-(d-2-i)\lambda}\leq \frac{\lambda}{1-(d-2+1)\frac{1}{d}}=
	d\lambda \leq  \left(\frac{\veps}{ (k+1)^2 2^{2k+6}}\right)^{2^{k }}.
	\]
	Setting $\alpha_i:=(\frac{\veps}{ (k+1)^2 2^{2k+6}})^{2^{k-1-i}}$ for $i\in\{-1,\dots,k-1\}$,
	we conclude that $X $ is an $i$-local $\alpha_i$-skeleton expander
	for all $i\in\{-1,\dots,k-1\}$.

	We now apply Theorem~\ref{TH:expansion-of-small-sets}
	with  
	$h_i= \alpha_i$. 
	To see that the inequality \eqref{EQ:base-ineq} holds, note
	that for all $i\in\{0,\dots,k\}$, we have 
	\[\frac{\alpha_{i-1}+h_{i-1}}{h_i\cdots h_{k-1}}
	=2\left(\frac{\veps}{ (k+1)^22^{2k+6}}\right)^{2^{k-i}-2^{k-1-i}-\dots-2^0}=\frac{ \veps}{ (k+1)^22^{2k+5}}.
	\]
	Thus,
	\begin{align*}
	&U(k+2)
	\sum_{i=0}^{k}   {\textstyle {{k+2}\choose {i+2}}  }(i+1) \frac{\alpha_{i-1}+h_{i-1}}{h_i\cdots h_{k-1}}
	\leq  
	2^{k+2}
	\sum_{i=0}^{k}2^{k+2}(k+1)\frac{  \veps}{ (k+1)^22^{2k+5}}=\frac{\veps}{2}<\veps.
	\end{align*}
	It also follows that the constants
	$\beta$ and $\gamma$ of Theorem~\ref{TH:expansion-of-small-sets} satisfy $\beta\geq \veps- \frac{\veps}{2}=\frac{\veps}{2}$
	and   $\gamma=h_{-1} h_0 \cdots h_{k-1} =
	(\frac{\veps}{(k+1)^22^{2k+6}})^{2^k+2^{k-1}+\dots+2^0}=(\frac{\veps}{ (k+1)^22^{2k+6}})^{2^{k+1}-1}$.
	We conclude
	that $(X,\calF)$ $\frac{\veps}{2}$-expands
	$(\frac{\veps}{ (k+1)^22^{2k+6}})^{2^{k+1}-1}$-small $k$-cochains.
\end{proof}

In the remainder of this subsection, we analyze
the solubility of the inequality \eqref{EQ:base-ineq} in the cases $k=0$ and $k=1$,
deriving specialized versions of 
Theorem~\ref{TH:expansion-of-small-sets-finer}.
We also address the asymptotic behavior of the
general case. To begin, we make the following remark.

\begin{remark}
	When solving \eqref{EQ:base-ineq}, we may assume that 
	$h_{-1},\dots,h_{k-1}$ live in $\R_+$, rather than $(0,1]$,
	because  
	\eqref{EQ:base-ineq} and
	assumption (1) of Theorem~\ref{TH:expansion-of-small-sets-finer} force  $h_{-1},\dots,h_{k-1}\leq 1$.
	This can be seen by decreasing induction on $i$. 
	For $i=k-1$, the inequality \eqref{EQ:base-ineq} 
	and Remark~\ref{RM:comments-on-local-to-global}(ii)	
	imply that $2(k+1)h_{k-1}\leq U(k+2)\schoose{k+2}{k+2}(k+1)h_{k-1}<\veps\leq k+2$,
	so $h_{k-1}\leq 1$. Assuming $h_{i+1},\dots,h_{k-1}\leq 1$ for $-1\leq i<k-1$,
	the the same reasoning shows that
	$(k+2)h_{i } \leq \schoose{k+2}{i+3}h_{i }<\veps\leq k+2$, so $h_{i }\leq 1$.
	(In fact, the assumption that $h_i\leq 1$ for all $i$ was never used in
	the proof of Theorem~\ref{TH:expansion-of-small-sets-finer}.)
\end{remark}

\begin{cor} 
	\label{CR:expansion-of-small-zero-ch}
	Let $(X,\calF)$ be a sheaved $d$-complex ($d\geq 1$),   let $\alpha ,\veps \in\R_+$ 
	be numbers such that $\alpha <\veps $, and let $h\in [0,\veps -\alpha ]$.
	Suppose that (1) $(X_v,\calF_v)$
	is an $\veps $-coboundary expander in dimension $-1$ for every   $v\in X(0)$,
	(2) $X$ is an $\alpha $-skeleton expander, and (3) $\calF(e)\neq 0$
	for all   $e\in X(1)$. Then   $(X,\calF)$ $2(\veps -\alpha -h )$-expands
	$h $-small $0$-cochains.
\end{cor}

\begin{proof}
	When $k=0$, the inequality
	\eqref{EQ:base-ineq} becomes
	$2(\alpha_{-1}+h_{-1})<2\veps_0$. Setting $h_{-1}=h$, $\veps_0=\veps$
	and $\alpha_{-1}=\alpha$,
	the statement follows from Theorem~\ref{TH:expansion-of-small-sets-finer}.
	(Note that every $0$-cochain is locally minimal.)
\end{proof}

\begin{cor} 
	\label{CR:expansion-of-small-one-ch}
	Let  $\alpha_{-1},\alpha_0,\veps_0,\veps_1\in\R_+$ 
	be numbers such that 
	\[
	\alpha_0<\min\{\frac{\veps_0}{4},\frac{\veps_1}{2}\}
	\qquad
	\text{and}
	\qquad
	\alpha_{-1}<\frac{1}{6}(\min\{\frac{\veps_0}{4},\frac{\veps_1}{2}\}-\alpha_0)^2.
	\]
	Then there exist $\beta,\gamma\in \R_+$ such that the following hold:
	Let $(X,\calF)$ be a sheaved $d$-complex ($d\geq 2$)
	such that   (1) $(X,\calF)$ is an $i$-local $\veps_i$-coboundary expander
	in dimension $1$ for $i\in\{0,1\}$,
	(2) $X$ is an $i$-local $\alpha_{i}$-skeleton expander for $i\in\{-1,0\}$, and (3) $\calF(x)\neq 0$
	for all $x\in  X(2)$. Then 
	$(X,\calF)$ $\gamma$-expands
	$\beta $-small locally minimal $1$-cochains.
\end{cor}

\begin{proof} 
	When $k=1$, the inequality \eqref{EQ:base-ineq} becomes
	\[
	9\cdot \frac{\alpha_{-1} +h_{-1}}{h_0}+6(\alpha_0+h_0)<
	\min\{\frac{3}{2}\veps_0,3\veps_1\}. 
	\]	
	By treating this as a quadratic inequality in $h_0$, one finds
	that it is solvable for   $h_0,h_{-1}\in\R_+$ if and only if the inequalities
	in the corollary are satisfied. The corollary is therefore
	a special case of Theorem~\ref{TH:expansion-of-small-sets-finer}.
\end{proof}

\begin{cor}\label{CR:cse-from-links-dim-2}
	Let $\veps_0,\veps'_0,\veps'_1\in\R_+$
	and $\alpha_{-1},\alpha_0\in[0,1]$ be numbers such that
	\[
	\alpha_0<\min\{\frac{\veps'_0}{4},\frac{\veps'_1}{2}\}
	\qquad
	\text{and}
	\qquad
	\alpha_{-1}<\min\{\veps_0,\frac{1}{6}(\min\{\frac{\veps'_0}{4},\frac{\veps'_1}{2}\}-\alpha_0)^2\},
	\] 
	and let $Q,d\in \N$ be integers with $d\geq 2$.
	Then there exist $ \beta,\beta',\gamma,\gamma' \in\R_+$, depending on
	$\veps_0,\veps'_0,\veps'_1,\alpha_{-1},\alpha_0$,
	and $\delta,\veps\in\R_+$,
	depending on $\veps_0,\veps'_0,\veps'_1,\alpha_{-1},\alpha_0,d,Q$,
	such that the following hold:
	If $(X,\calF)$ is a sheaved $d$-complex such that 
	\begin{enumerate}[label=(\arabic*)]
		\item 
		$(X,\calF)$ is a  $0$-local $\veps_0$-coboundary expander
		in dimension $0$.
		\item
		$(X,\calF)$ is a $i$-local $\veps'_i$-coboundary expander
		in dimension $1$ for $i\in\{0,1\}$. 
		\item $X_v$ is an $\alpha_0$-skeleton expander for all $v\in X(0)$ and
			$X$ is an $\alpha_{-1}$-skeleton expander,
		\item $D(X)\leq Q$, i.e., every vertex of $X$ belongs to at most $Q$ $d$-faces, and
		\item $\calF(x)\neq 0$ for all $x\in X(0)\cup X(1)\cup X(2)$,
	\end{enumerate}
	then $(X,\calF)$ is an $(\veps,\delta)$-cosystolic expander in dimension $0$,
	$\beta$-expands $\gamma$-small $0$-cochains 
	and $\beta'$-expands $\gamma'$-small locally minimal $1$-cochains.
\end{cor}

\begin{proof}
	By Corollaries~\ref{CR:expansion-of-small-zero-ch}
	and~\ref{CR:expansion-of-small-one-ch}, there
	are $\beta',\beta,\gamma',\gamma'\in\R_+$,
	depending  on $\veps_0,\veps'_0,\veps'_1,\alpha_{-1},\alpha_0$,
	such that $(X,\calF)$
	$\beta'$-expands $\gamma'$-small $0$-cochains and $\beta$-expands
	$\gamma$-small locally minimal $1$-cochains.
	The existence of $\gamma$ and $\delta$ is now a consequence of 
	Proposition~\ref{PR:small-set-to-cse}.
\end{proof}

\begin{remark}\label{RM:main-ineqs-for-2dim-case-with-Oppenheim}
	As in the proof of Theorem~\ref{TH:cosystolic-exp-from-links},
	we can use Oppenheim's Trickling Down Theorem 
	\cite[Theorem~1.4]{Oppenheim_2015_vanishing_of_cohomology}	
	to replace condition (3)
	of Corollary~\ref{CR:cse-from-links-dim-2} with
	\begin{enumerate}
		\item[(3$'$)] $X$ is  connected
		and, for all $v\in X(0)$, the underlying weighted graph
		of $X_v$ is a  $[-1,\lambda]$-spectral expander,
	\end{enumerate}
	where $\lambda\in\R_+$ is required to satisfy the inequalities
	\[
	\lambda<\min\{\frac{\veps'_0}{4},\frac{\veps'_1}{2}\}
	\qquad
	\text{and}
	\qquad
	\frac{\lambda}{1-\lambda}<
	\min\{\veps_0,\frac{1}{6}(\min\{\frac{\veps'_0}{4},\frac{\veps'_1}{2}\}-\lambda)^2\}.
	\]
\end{remark}

\begin{remark}\label{RM:constant-analysis-general}
	We use the notation of Theorem~\ref{TH:expansion-of-small-sets-finer}.
	It was demonstrated in the proof of Theorem~\ref{TH:expansion-of-small-sets}
	that the inequality \eqref{EQ:base-ineq} is solvable when
	\[
	\alpha_i\leq \left(\frac{\veps}{(k+1)^22^{2k+6}}\right)^{2^{k-1-i}}
	\]
	for all $i\in\{-1,\dots,k-1\}$.
	The order of magnitude of this upper bound 
	on the $i$-local skeleton expansion of $X$
	(as a function $\veps_0,\dots,\veps_{k}$)
	cannot be increased with our present methods. More precisely,
	if \eqref{EQ:base-ineq} is satisfied, then 
	\[\alpha_i<\left(\frac{\veps}{U(k+2)}\right)^{2^{k-1-i}}\] 
	for all $i\in\{-1,\dots,k-1\}$, so we must have $\alpha_i=O(\veps^{2^{k-1-i}})$
	in order to apply Theorem~\ref{TH:expansion-of-small-sets-finer}.
	To see this, note that if \eqref{EQ:base-ineq} holds
	for some $h_{-1},\dots,h_{k-1}\in (0,1]$, then for all $i\in\{0,\dots,k\}$, we have
	\[
	\veps>U(k+2)
	\sum_{i=0}^{k}   {\textstyle {{k+2}\choose {i+2}}  }(i+1) \frac{\alpha_{i-1}+h_{i-1}}{h_i\cdots h_{k-1}}
	\geq U(k+2)\cdot \frac{\alpha_{i-1}+h_{i-1}}{h_i\cdots h_{k-1}}.
	\]
	As a result,
	\begin{align*}
	\max\{\alpha_{k-1},h_{k-1}\} & <  U(k+2)^{-1}\veps  , \\
	\max\{\alpha_{k-2},h_{k-2}\} & <   U(k+2)^{-1} h_{k-1} \veps , \\
	& \vdots \\
	\max\{\alpha_{-1},h_{-1}\} & < U(k+2)^{-1} h_{-1}h_0\cdots h_{k-1} \veps 	
	\end{align*}
	These inequalities imply readily that  
	$h_i< (\frac{\veps }{U(k+2)})^{2^{k-1-i}}$ for all $i$.
	Plugging this in the right hand side of the inequalities gives
	$\alpha_i <(\frac{\veps }{U(k+2)})^{2^{k-1-i}}$.
	It also follows that $\gamma=h_{-1}\cdots h_{k-1}$
	(the   smallness of locally minimal $k$-cochains which are guaranteed
	to $\beta$-expand)
	is smaller than $(\frac{\veps }{U(k+2)})^{2^{k+1}-1}$.
\end{remark}

\section{Examples of Cocycle Codes}
\label{sec:quo-aff-buildings}

In this section, we give   examples of sheaved $d$-complexes
to which   Theorems~\ref{TH:cosystolic-exp-from-links}
and~\ref{TH:expansion-of-small-sets} can be applied, and analyze
the properties of the associated cocycle codes.

Some of the examples make use of simplicial complexes
covered affine buildings, recalled in \S\ref{subsec:buildings}.

\subsection{$0$-Cocycle Codes of Sheaves on Graphs}
\label{subsec:0cocycle-graphs}

We begin by revisiting an example from the Overview section.
Fix some $m,k\in \N$ with $\frac{k}{2}<m\leq k$,
let $X$ be a $k$-regular graph,   and let $\F$ be
a finite field. Given $v\in X(0)$, write $E(v)$ for $X(1)_{\supseteq v}$
and choose    
an injective $\F$-linear map $T_v:\F^m\to \F^{E(v)}\cong \F^k$.
We think of $C_v:=\im T_v$ as a code inside $\F^{E(v)}$ with alphabet $\F$
and denote its relative distance by $\delta(C_v)$.
In \S\ref{subsec:intro-sheaves},
we defined  a sheaf $\calF$ on $X$ by setting $\calF(v)=\F^m$ and $\calF(e)=\F$
for all $v\in X(0)$, $e\in X(1)$, and
$\res^{\calF}_{e\from v}=\mathrm{Proj}_e \circ T_v$ ---
where $\mathrm{Proj}_e:\F^{E(v)}\to \F$ is projection onto the $e$-component ---
whenever $v\subseteq e$.
Putting $\Sigma:=\F^m$, we   form the $0$-cocycle
code $Z^0(X,\calF)$ inside $C^0(X,\calF)=\Sigma^{X(0)}$
as in \S\ref{subsec:ltc}.

\begin{prp}
	With notation as above, suppose that $X$ is an $\alpha$-skeleton expander
	($\alpha\in\R_+$) and $\veps:=\min\{\delta(C_v)\where v\in X(0)\}>\alpha$.
	Then the $0$-cocycle code $Z^0(X,\calF)\subseteq \Sigma^{X(0)}$ has 
	rate $\geq 1-\frac{k}{2m}$ and
	relative distance $\geq \veps-\alpha$. 
\end{prp}

\begin{proof}
We observed in \S\ref{subsec:expanding-sheaves}
that $(X_v,\calF_v)$ is a $\delta(C_v)$-coboundary expander
in dimension $-1$. The claim about the relative distance is
therefore a consequence of Corollary~\ref{CR:expansion-of-small-zero-ch}
and Proposition~\ref{PR:decoding}(i).
Dimension count implies that $\dim_{\F} Z^0(X,\calF)\geq m|X(0)|-|X(1)|=
|X(0)|(m-\frac{k}{2})$, hence the lower bound on the rate.
\end{proof}

\subsection{Cocycle Codes of Sheaves on Complexes Covered by Affine Buildings}

In the following examples we put into use the fact that 
constant sheaves on finite spherical
buildings are good coboundary expanders.

\begin{thm}\label{TH:sheaves-on-quo-of-aff-buildings}
	For every $d\in \N-\{1\} $,
	there exists  $q  \in\N$ for which the following hold:
	Let $Y$ be a $q$-thick affine building,  let $k\in \{0,\dots,d-2\}$,
	let $X$ be a (finite) simplicial complex   covered by $Y$,
	and let $\calF$ be a nonzero locally constant sheaf on $X$.
	Then:
	\begin{enumerate}[label=(\roman*)]
		\item There
		are
		$\veps_0 ,\dots,\veps_k ,\veps'_0 ,\dots,\veps'_{k+1},\lambda \in \R_+$,
		depending only on $k$ and $d$,  
		and $Q\in \N$, 
		depending only on $Y$,
		such that
		the assumptions   of Theorem~\ref{TH:cosystolic-exp-from-links}
		hold for $(X,\calF)$. 
		\item 
		There are 	$\beta,\beta',\gamma,\gamma'\in\R_+$,
		depending  only on $k$ and $d$,
		and $\delta,\veps\in \R_+$, depending only on $Y$,
		such that	
		$(X,\calF)$ 
		$\beta$-expands $\gamma$-small locally minimal $k$-cochains,
		$\beta'$-expands $\gamma'$-small locally minimal $(k+1)$-cochains		
		and is an $(\veps,\delta)$-coboundary expander in dimension $k$.
	\end{enumerate}
\end{thm}

\begin{proof}
	Part (ii) follows from (i) and Theorem~\ref{TH:cosystolic-exp-from-links}.
	We turn to prove (i).

	We claim that for every $i\in\{0,\dots,k\}$,
	there is $\veps_i>0$,
	depending  only on $d$,
	such that $(X,\calF)$
	is an $i$-local $\veps_i$-coboundary expander in dimension $k$.
	Indeed, let $z\in X(i)$.
	Since $\calF$ is locally constant,
	$\calF_z$ is a constant sheaf on  $X_z$.	
	The link $X_z$ is isomorphic to a proper link of $Y$, so it is a spherical
	building of dimension $d-i-1$.
	Thus, by Theorem~\ref{TH:cbe-for-buildings}(i),
	there exists   $\veps_i>0$ (depending only on $\dim X_z$) such
	that $(X_z,\calF_z)$ is an $\veps_i$-coboundary expander in dimensions $k-i-1$.
	
	A similar argument shows that there are 
	$\veps'_0,\dots,\veps'_{k+1}\in\R_+$,
	depending only on $d$,
	such that  $(X,\calF)$
	is an $i$-local $\veps'_i$-coboundary expander in dimension $k+1$
	for all $i\in\{0,\dots,k+1\}$.
	
	Take $\lambda$ to be the maximal number
	for which the inequality in Theorem~\ref{TH:cosystolic-exp-from-links}
	holds. Let $Q=D(Y)$; it is finite
	because $Y$ admits a strongly transitive action
	(see~\S\ref{subsec:buildings}).
	Finally, set $q=\ceil{\frac{16}{\lambda^2}}$.
	
	We claim that 
	assumptions (1)--(5) of Theorem~\ref{TH:cosystolic-exp-from-links}
	hold for $(X,\calF)$
	with the parameters we have chosen, provided that  $Y$ is $q$-thick.
	Indeed, assumptions (1) and (2) are immediate.
	Assumption (4) holds because $D(X)= D(Y)$ (since $Y$ covers $X$),
	and  
	(5) holds because $\calF$ is locally constant and nonzero.
	To see that   (3) holds,
	let $z\in X(d-2)$. Then $X_z$ is isomorphic to a $1$-dimensional
	link of $Y$  and  is therefore a spherical building of dimension $1$. 
	By Theorem~\ref{TH:skeleton-exp-building}(i),
	$X_z$ is a $[-1,\frac{4}{\sqrt{q}}]$-spectral expander,
	and  $\frac{4}{\sqrt{q}}\leq \lambda$
	by our choice of $q$.
\end{proof}

\begin{cor}\label{CR:building-quotient-loc-const-codes}
	For every $d\in \N-\{1\} $,
	there exists  $q  \in\N$ for which the following hold:
	Let $Y$ be a $q$-thick affine building,  let $k\in \{0,\dots,d-2\}$,
	let $X$ be a (finite) simplicial complex   covered by $Y$,
	let $\F$ be a finite field,
	let $\calF$ be a nonzero locally constant $\F$-sheaf  of dimension $m$ on $X$
	and let $B$ be an $\F$-basis of $\calF$ (see Example~\ref{EX:Hamming-norm}).
	\begin{enumerate}[label=(\roman*)]
		\item 
			If $k=0$,
			then there are  
			 $\delta,\veps,\eta \in\R_+$, depending only on $Y$,
			such that  the $0$-cocycle code $(Z^0,C^0,\Phi)$ associated to $(X,\calF)$
			(see \S\ref{subsec:ltc}; the alphabet is $\F^{m}$ and the length is $|X(0)|$),
			is $2$-query $\veps$-locally testable of relative distance $\geq \delta$.
			Furthermore, it admits a decoding algorithm able to correct
			an $\eta$-fraction of errors  in $O_{|\F|,
			\dim\calF}(|X(0)|)$
			operations.
		\item 
			If $k>0$, then there are $\delta ,\veps,\eta \in\R_+$ and $r\in \N$,
			depending only on $Y$, $k$ and $\dim\calF$,
			such that  the $X$-side
			of the $k$-cocycle quantum CSS  code $C:=(Z^k,Z_k,C^k,\Phi_X,\Phi_Z)$
			associated to $(X,\calF,B)$ (see \S\ref{subsec:css}; the alphabet is $\F$)
			has relative distance $\geq \delta$, is $r$-query $\veps$-testable,
			and  admits a decoding algorithm able to correct an $\eta $-fraction
			of errors in $O_{|\F|,\dim\calF}(\dim C^k)$
			operations.
	\end{enumerate}
\end{cor}

\begin{proof}
	This follows from Theorem~\ref{TH:sheaves-on-quo-of-aff-buildings}(ii)
	together with Propositions~\ref{PR:decoding}
	and~\ref{PR:cbe-to-quantum-CSS}. 
	Use Proposition~\ref{PR:Hamming-and-support-ratio}(ii) and Remark~\ref{RM:properties-of-cbe}(ii) in order
	to replace the weighted support norm   $\|\cdot\|_{\mathrm{ws}}$ with $\|\cdot\|_B$.
\end{proof}

For every $q,d\in \N$, there are 
$q$-thick $d$-dimensional affine buildings $Y$ which cover arbitrarily large finite 
$d$-complexes $X$ (see \S\ref{subsec:quotients-of-buildings}, for instance).
Each of these quotients $X$ admits an $m$-dimensional locally
constant $\F$-sheaf $\calF$, e.g., the constant sheaf $\F^m$ on $X$. 
Choosing $q$ large enough in advance and fixing   $Y$ and $m$, Corollary~\ref{CR:building-quotient-loc-const-codes}(i)
says that the $0$-cocycle codes of the form $(Z^0(X,\calF),C^0(X,\calF)\cong (\F^m)^{X(0)},\Phi)$
are an infinite family of   $2$-query LTCs with linear distance on the alphabet $\Sigma=\F^m$.
Unfortunately, the rate of these codes is very poor --- at most $\frac{1}{|X(0)|}$ ---,
because  $\dim_\F \HH^0(X,\calF)\leq \dim \calF=m$ by Lemma~\ref{LM:coh-loc-const}.

If, instead of considering $0$-cocycle codes, 
we fix $k\in\{1,\dots,d-2\}$ and $m:=\dim \calF$, and  look at the $k$-cocycle quantum CSS codes associated 
to  $(X,\calF,B)$, with $B$ being some $\F$-basis of $\calF$, 
then, by Corollary~\ref{CR:building-quotient-loc-const-codes}(ii), 
we get an infinite family of quantum CSS codes    whose $X$-side
is locally testable and has linear distance. The rate of these quantum CSS codes
is $\frac{1}{|X(k)|}\dim_{\F}\HH^k(X,\calF)$.
Very little is known about $\dim_\F \HH^k(X,\calF)$,
but experts expect that it is polylogarithmic in $|X(k)|$ and linear in 
the fixed parameter $m=\dim \calF$.

\medskip

Returning to the case of $0$-cocycle codes, 
as demonstrated in \S\ref{subsec:0cocycle-graphs}, it is possible to obtain larger
rates by considering sheaves that are not locally constant. 
We now give such an example.

\begin{construction}\label{CN:modification-prototype}
	Let  $X$ be a $d$-complex ($d\geq 1$),
	let $  \calF$ be a locally constant   sheaf on $X$,
	and let $E\subseteq C^1(X, \calF)$ be an abelian subgroup.
	For every edge $e\in X_{\ord}(1)$,
	let $E(e)$ be the image of $E$ under the projection from $ \prod_{x\in X_{\ord}(1)} \calF(x)$
	to $\calF(e)$. The abelian group $E(e)\subseteq \calF(e)$ is independent of the ordering on $e$,
	so it makes sense to discuss $E(e)$ for unordered edges $e\in X(1)$. We  define a subsheaf
	$\calC_E$ of $ \calF$ by letting
	\[
	\calC_E(x)=\sum_{e\in X(1)_{\subseteq x}}\res_{x\from e}E(e).
	\]
	for all $x\in X$.
\end{construction} 

Note that $\calC_E(v)=0$ for all $v\in X(0)$, because $v$ contains no edges.
The subsheaf $\calC_E$ can be characterized as 
the smallest subsheaf of $\calF$ for which $E\subseteq C^1(X,\calC)$.

We will be interested in the quotient sheaf $\quo{\calF}:=\calF/\calC_E$
when $\calF$ is a locally constant $\F$-sheaf of dimension $m$
and $E$ is an $\F$-subspace of $C^1(X,\calF)$.
In this case, $\calC_E$ and $\quo{\calF}$ are   $\F$-sheaves.
For every $v\in X(0)$, we have
$\quo{\calF}(v)=\calF(v)/0\cong \F^m$, so we may
consider $Z^0(X,\calF)$ as a code inside $C^0(X,\calF)=(\F^m)^{X(0)}$,
the    alphabet being $ \F^m$. 
As we now show, when $X$ is covered by a sufficiently thick affine building, 
and $E$ is small and in general position, the $0$-cocycle
code of $(X,\quo{\calF})$ is  locally testable and has linear distance.
The rate of this code depends  on 
the choice of $E$ and will be studied in Chapter~\ref{chap:initial-data}.

\begin{thm}\label{TH:sheaves-on-quo-of-aff-buildings-quotients}
	Let $d\in \N-\{1\}$.
	There exists $q\in \N$
	such that, for every $Y$ and $X$ as in
	Theorem~\ref{TH:sheaves-on-quo-of-aff-buildings}
	(resp.\ Corollary~\ref{CR:building-quotient-loc-const-codes}),
	the conclusions  of
	Theorem~\ref{TH:sheaves-on-quo-of-aff-buildings} 
	(resp.\ Corollary~\ref{CR:building-quotient-loc-const-codes}(i))
	continue  to hold  with $k=0$ (but with possibly
	different expansion constants) if the sheaf  $\calF$
	is a replaced by any sheaf of the form $\quo{\calF}=\calF/\calC_E$,
	where $\calC_E$ is as in 
	Construction~\ref{CN:modification-prototype},
	and 	the subgroup (resp.\ $\F$-subspace)
	$E\subseteq C^1(X,\calF)$ satisfies the following conditions:
	\begin{enumerate}[label=(a\arabic*)]
		\item \label{item:TH:F-E:linear-disjointness}
		For every $v\in X(0)$, the map $\sum_{e} \res_{e\from v}^{-1}:\bigoplus_e E(e)\to
		{\calF}(v)$, with $e$ ranging over $X(1)_{\supseteq v}$,
		is injective.
		\item \label{item:TH:F-E:linear-depedence}
		For every triangle $t\in X(2)$ with edges $e,e',e''$, we have $E(e)|_t\subseteq E(e')|_t+E(e'')|_t$.
	\end{enumerate}
\end{thm}

\begin{example}\label{EX:cocycle-E-is-okay}
	Condition (a2) of Theorem~\ref{TH:sheaves-on-quo-of-aff-buildings}
	holds if $E\subseteq Z^1(X,\calF)$.
	Indeed, let $t\in X_{\ord}(2)$ and let
	$e ,e' ,e'' $ denote  the ordered edges, obtained
	by removing the $0$-th, $1$-st and $2$-nd vertex of $t$, respectively.
	Then for every $f\in E(e)$, there is 
	$\hat{f}\in E$ such that $\hat{f}(e)=f$. Since $E\subseteq Z^1(X,\calF)$,
	we have $d_1\hat f(t)=0$,
	which means that $f|_t= \hat f(e)|_t = \hat{f}(e')|_t-\hat{f}(e'')|_t\in E(e')|_t+E(e'')|_t$.
	This shows that
	$E(e)|_t\subseteq E(e')|_t+E(e'')|_t$.
	
	Condition (a1) typically holds if $\dim E\cdot D_{0,1}(X)\leq \dim \calF$ is $E$
	is chosen uniformly at random. We make this precise in Proposition~\ref{PR:securing-expansion-for-quotients}(ii) below.
\end{example}

We first prove the following lemma:

\begin{lem}\label{LM:quotient-sheaf-identification}
	Let $X$, $ \calF$, $E$ and $\calC:=\calC_E$
	be as in Construction~\ref{CN:modification-prototype}
	and assume that conditions (a1) and (a2) of
	Theorem~\ref{TH:sheaves-on-quo-of-aff-buildings-quotients} hold.
	Let $v\in X(0)$ and write $A= \calF(v)$.
	For every $u\in X(1)_v$, put	
	$A_u=\res^{-1}_{u\cup v\from v}(E(u\cup v))\subseteq A$,
	and for every $x\in X_v$, define $\calC'(x)=\sum_{u\in X_v(0)_{\subseteq x}} A_u$.
	Then:
	\begin{enumerate}[label=(\roman*)]
		\item $\calC'$ is a subsheaf of the augmented sheaf $\aug{A}$ on $X_v$,
		and the summation map $\bigoplus_{u\in X_v(0)} A_u\to A$ is injective.
		\item $(\calF/\calC )_v\cong \aug{A}/\calC'$ as sheaves on $X_v$.
	\end{enumerate}	 
\end{lem}

\begin{proof}
	(i) That $\calC'$ is a subsheaf of $\aug{A}$ is straightforward,
	and the injectivity of
	$\bigoplus_{u\in X_v(0)} A_u\to A$ is a direct consequence of    (a1).

	(ii) Write $\quo\calF =\calF/\calC$. Then $\quo{\calF}_v=\calF_v/\calC_v$.
	For every $x\in X_v$,
	we have 
	\begin{align*}
	\res_{x\cup v\from v}(\calC'(x))
	&=
	\sum_{y\in X_v(0)_{\subseteq x}} \res_{x\cup v\from v}(A_y)
	=
	\sum_{y\in X_v(0)_{\subseteq x}}\res_{x\cup v\from v}\res^{-1}_{y\cup 	v\from v}(E(y\cup v))
	\\	
	&=
	\sum_{y\in X_v(0)_{\subseteq x}}\res_{x\cup v\from y\cup v}(E(y\cup v))
	=
	\sum_{e\in X(1): v\subseteq e\subseteq x\cup v }\res_{x\cup v\from e}(E(e))
	\\
	&
	\subseteq \calC(x\cup v)=\calC_v(x).
	\end{align*}
	This allows us to define $\vphi_x: \aug{A}(x)/\calC'(x)\to \calF'_v(x)=\calF_v(x)/\calC_v(x)$ 
	by $\vphi_x(f+\calC'(x))=\res_{x\cup v\from v}(f)+\calC_v(x)$ for all $f\in A$.
	It is routine to check that 
	$\vphi:=(\vphi_x)_{x\in X_v}:\aug{A}/\calC'\to \calF'_v$ is a morphism of sheaves.
	It remains to prove that each $\vphi_x$ is bijective, or
	equivalently,
	that $\res_{x\cup v\from v}(\calC'(x))=\calC_v(x)$.
	We already observed that the left hand side is contained in the right hand side.
	Proving the reverse inclusion amounts to showing that
	for every $x\in X_{\supseteq v}$ and $e\in X(1)_{\subseteq x}$,
	we have
	$\res_{x\from e} E(e) \subseteq \sum_{y\in X(1):v\subseteq y\subseteq x}\res_{x\from y}E(y)$.
	
	Fix   $x\in X_{\supseteq v}$, $e\in X(1)_{\subseteq x}$ and $f\in \res_{x\from e}E(e)$.
	Then there is $g\in E(e)$ such that $f=\res_{x\from e}(g)$.
	If $v\subseteq e$, then $f\in\sum_{y\in X(1):v\subseteq y\subseteq x}\res_{x\from y}E(y)$.
	Otherwise, $t:=e\cup v\in X(2)$. Let $e'$ and $e''$ be the edges of $t$ different from $e$. 
	Then $v\subseteq e'$ and $v\subseteq e''$.
	By (a2), there are $f'\in E(e')$ and $f''\in E(e'')$
	such that $g|_t=f'|_t+f''|_t$.
	This means that $f=g|_x=f'|_x+f''|_x\in \sum_{y\in X(1):v\subseteq y\subseteq x}\res_{x\from y}E(y)$,
	which is what we want.
\end{proof}

\begin{proof}[Proof of Theorem~\ref{TH:sheaves-on-quo-of-aff-buildings-quotients}]
	Write $\calC=\calC_E$. The argument is similar to the proof of Theorem~\ref{TH:sheaves-on-quo-of-aff-buildings}.
	
	We first show that if $q$ is sufficiently
	large, then there exists $\veps'_0>0$, not depending on $q$,
	such that $(X,\quo\calF)$ is a $0$-local $\veps'_0$-coboundary expander in dimension $1$.
	Let $v\in X(0)$.
	Then $X_v$ is a $q$-thick spherical building of dimension $d-1$,
	and 
	Lemma~\ref{LM:quotient-sheaf-identification}  and conditions (a1), (a2)  
	say  that $\quo{\calF}_v$ is isomorphic to a sheaf
	as in Theorem~\ref{TH:cbe-buildings-quotient-sheaves}.
	Thus, $(X_v,\quo{\calF}_v)$ is a $\veps' $-coboundary
	expander in dimension $0$
	for $\veps' =\frac{2(d-1)}{5(d-1)+2}-O_d(\frac{1}{\sqrt{q}})$.
	Taking $q$ large enough in advance, we get that 
	$(X_v,\quo{\calF}_v)$ is a 
	$\frac{1}{4}$-coboundary expander in dimension $0$, so $\veps'_0=\frac{1}{4}$
	suffices.

	Next, we claim that 
	$(X,\quo\calF)$ is a $1$-local $\veps'_1$-coboundary expander in dimension $1$
	for $\veps'_1=\frac{1}{2}$.
	Let $e\in X_{\ord}(1)$ and let $f\in \quo{\calF}_e(\emptyset)=\quo{\calF}(e)= \calF(e)/E(e)$;
	we shall freely regard $f$ as a member of $C^{-1}(X_e,\quo{\calF}_e)$.
	Fix some $0$-face $v$ of $e$.
	Then, for every $t\in X(2)_{\supseteq e}$, we have $f|_t = f+\calC(t)$.
	Thanks to (a2), we have $\calC(t)=E(e)|_t+E(v\cup(t-e))|_t$.
	Thus, by condition (a1) and the assumption
	that $\calF$ is locally constant, $f|_t=0$ if and only if $f\in (E(e)+E(v\cup (t-e)))/E(e)$.
	Condition (a1) also means that $(E(e)+E(v\cup (t-e)))\cap (E(e)+E(v\cup (t'-e)))=E(e)$
	for every $t'\in X(2)_{\supseteq e}$ different from $t$, so we can have $f|_t=0$
	for at most one $t\in X(2)_{\supseteq e}$.
	This means that $\supp (d_{-1} f )\supseteq X_e-\{t-e\}$
	for some $t\in X(2)_{\supseteq e}$, so $(X_e,\calF_e)$
	is a $(1-\xi)$-coboundary expander for
	$\xi =\max\{w_{X_e}(u)\where u\in X_e(0)\}$.
	Writing $z=e-v$, equation~\eqref{EQ:weight-in-link}
	gives $w_{X_v}(u\cup z)=2 w_{X_v}(z)w_{X_e}(u)$
	for every $u\in X_e$.
	Thus, $\xi=\max\{\frac{1}{2}w_{X_v}(u\cup z)w_{X_v}(z)^{-1}\where u\in X_e(0)\}$.
	Since $X_v$ is a $q$-thick spherical building of dimension $d-1$, 
	\cite[Lemma~7.5]{First_2021_weighted_mixing_lemmas_preprint}
	says that $\xi\leq \frac{1}{2}\cdot\frac{2}{q+d-1}=\frac{1}{q+d-1}$.
	This means that 
	$(X,\quo\calF)$ is a $1$-local $(1-\frac{1}{q+d-1})$-coboundary expander in dimension $1$,
	and the claim follows since $d\geq 2$ and $q\geq 1$.
	
	As similar argument shows that if $q$ is large enough,
	then there is    $\veps_0>0$, not depending on $q$,
	such that $(X,\quo{\calF})$ is a $0$-local $\veps_0$-coboundary expander
	in dimension $0$. 
	Briefly, one similarly finds that this holds for $\veps_0=1-\zeta$
	with $\zeta = \max\{w_{X_v}(z)\where v\in X(0), z\in X_v(0)\}$,
	and by \cite[Lemma~7.5]{First_2021_weighted_mixing_lemmas_preprint},
	$\zeta\leq \frac{2}{q+d-1}$, because $X$ is $q$-thick.
	
	Now that $\veps'_0,\veps'_1,\veps_0$ have been determined,
	define $\lambda$ to be the largest real number for 
	which the inequality of Theorem~\ref{TH:cosystolic-exp-from-links} holds,
	and proceed as in the proof of Theorem~\ref{TH:sheaves-on-quo-of-aff-buildings}.	
\end{proof}

\begin{remark}\label{RM:necessary-thickness}
	In Theorems~\ref{TH:sheaves-on-quo-of-aff-buildings}
	and~\ref{TH:sheaves-on-quo-of-aff-buildings-quotients},
	the lower bound on the thickness of the building $Y$ can be lowered
	by using part (ii)
	of Theorem~\ref{TH:cbe-for-buildings}  instead
	of part (i) whenever possible, and by using Corollary~\ref{CR:cse-from-links-dim-2}
	instead of Theorem~\ref{TH:cosystolic-exp-from-links} when $k=0$.
	If $X$ is moreover assumed to be a $2$-dimensional Ramanujan complex 
	(in the 
	sense of \cite{Carwright_2003_Ramanujan_geometries_An},
	\cite{Lubotzky_2005_explicit_constructions_of_Ramanujan_complexes}),
	then the lower bound on the required thickness can be further lowered 
	by using  Proposition~\ref{PR:skeleton-exp-Ramanujan}
	instead of Theorem~\ref{TH:skeleton-exp-building}.
\end{remark}

\subsection{Good $1$-Cocycle Codes}
\label{subsec:higher-cocycle-codes}

We finish this section by giving examples of   $1$-cocycle codes with linear distance and constant
rate.
We do not know if these codes are locally testable relative to their natural $3$-tester.

\begin{construction}\label{CN:one-cocycle-code}
	Let $X$ be a connected $d$-complex ($d\geq 2$) and assume that there are
	$Q\in \N$, $\lambda\in [-1,1]$ and $\kappa\in [1,2]$ such that 
	\begin{enumerate}[label=(\arabic*)]
		\item every $0$-face of $X$
	is contained in exactly $Q$ edges,  
		\item the weighted graph underlying $X_v$ is a $[-1,\lambda]$-spectral
		expander for every $v\in X(0)$,
		\item $\kappa^{-1} w(e)\leq w(e')\leq \kappa w(e)$ for every $e,e'\in X(1)$ sharing a vertex.
	\end{enumerate}
	Fix integers $0< r < m $ such that $\frac{\kappa}{Q}<\frac{r}{m}<\frac{2}{Q}$.
	Let   $\F$ be a finite field and let $\calF$ be a locally constant $\F$-sheaf of dimension $m$.
	Suppose that, for every edge $e\in X(1)$, we are given an $(m-r)$-dimensional subspace
	$U(e)\subseteq \calF(e)$, and that these subspaces are in general position in the following sense:
	\begin{enumerate}[resume, label=(\arabic*)]
		\item \label{item:U-general-pos}
		For every $v\in X(0)$ and any $S\subseteq X_{\supseteq v}$ with $|S|r\geq m$,
		we have 
		$ \bigcap_{e\in S} \res^{-1}_{e\from v}U(e) =0$.
	\end{enumerate}
	Define a subsheaf $\calG$ of $\calF$   by setting 
	\begin{itemize}
		\item $\calG(v)=0$ for every $v\in X(0)$,
		\item $\calG(e)=U(e)$ for every $e\in X(1)$,
		\item $\calG(x)=\calF(x)$ for every $x\in X$ of dimension $>1$.
	\end{itemize}
	Since $\dim_{\F}\calG(e)=m-r$ for all $e\in X(1)$,
	we may form the $1$-cocycle code of $(X,\calG)$ as in \S\ref{subsec:ltc};
	its alphabet is $\Sigma:=\F^{m-r}$.
\end{construction}

\begin{remark}\label{RM:one-cocycle-code-U-cond}
	(i) If the subspaces $U(e)$ are chosen uniformly at random, then
	the probability that condition \ref{item:U-general-pos} will hold for  particular $v$
	and $S$ is at least $\prod_{i=1}^m(1-|\F|^{-i})>1-2|\F|^{-1}$. In particular, the probability
	that this holds for all $e$ and $S$ is at last $1-|X(0)|2^{Q+1}|\F|^{-1}$.
	This means  that we can find subspaces $\{U(e)\}_{e\in X(1)}$ as in (3)
	if $|\F|>|X(0)|2^Q$. (This bound can be  improved to $|\F|>\mathrm{poly}(Q)$ with a little
	more work.)
	
	(ii) Conditions (1)--(4) of Construction~\ref{CN:one-cocycle-code} are local  in the sense that they involve
	only the proper links of $X$.  
	As a result, if
	$p:Y\to X$ is a covering and   $\calF$ and $\{U(e)\}_{e\in X(1)}$
	are as in Construction~\ref{CN:one-cocycle-code}, then 
	(1)--(4) hold for $Y$, the sheaf  	
	$p^*\calF$ (see~\S\ref{subsec:pushforward})
	and the subspaces $\{U(e')\}_{e'\in Y(1)}$ defined by
	$U(e')=U(p(e))$.
 \end{remark}

\begin{thm}\label{TH:one-cocycle-code}
	With notation as in Construction~\ref{CN:one-cocycle-code}, suppose
	that
	\[
	\lambda\leq \min\{
	\frac{1}{6}(\frac{1}{4}(1-\lambda)(1-\frac{\kappa/Q}{r/m})-\lambda)^2(1-\lambda),
	\frac{1}{4}(1-\lambda)(1-\frac{\kappa/Q}{r/m})\}.
	\]
	Then there exists $\delta>0$, depending  on $Q,\lambda,r,m,\kappa$ and $D(X)$,
	such that the  $1$-cocycle code associated to $(X,\calG)$  
	has relative distance $\geq \delta$.
	The rate of this code is  $\geq  \frac{2}{Q}-\frac{r}{m}-|X(1)|^{-1}$.
\end{thm}

\begin{proof}
	Note that $B^1(X,\calG)=0$. 
	Thus, by Proposition~\ref{PR:decoding}(i), in order to prove the lower
	bound on the relative distance, it is enough to show that $(X,\calG)$
	$\beta$-expands $\gamma$-small locally minimal $1$-cochains for  some $\beta,\gamma\in\R_+$
	depending only on $Q,\lambda,r,m,\kappa$.
	To that end, we apply Corollary~\ref{CR:expansion-of-small-one-ch} 	
	to $(X,\calG)$.
	
	The fact that $\calG$ is a subsheaf of the locally constant sheaf $\calF$
	implies that all of the restriction maps $\res^{\calG}_{y\from x}$ are injective,
	so $(X,\calG)$ is a $1$-local $1$-coboundary expander in dimension $1$.
	
	We claim that $(X,\calG)$ is a $0$-local $\veps_0$-coboundary expander in dimension $1$
	for $\veps_0=(1-\lambda)(1-\frac{\kappa/Q}{r/m})$.
	To see this, fix $v\in X(0)$. Then $\calG_v$ is a subsheaf of $\calF_v$. 
	Since $\calF$ is locally constant of dimension $m$, we have $\calF_v\cong \aug{(\F^m)}$,
	so we may assume that $\calF_v=\aug{(\F^m)}$. 
	In particular, we view $U(v\cup u)$ as a subspace of $\F^m$
	for every $u\in X_v$. Let $f\in C^0(X_v,\calG_v)$. We need
	to show that $\|d_0 f\|\geq \veps_0\|f+B^0(X_v,\calG_v)\|=\veps_0\|f\|$.
	If $\|d_0 f\|\geq \veps_0$, then this is clear, so assume $\|d_0 f\|<\veps_0$.
	Assumption (2) and Theorem~\ref{TH:cbe-of-spectral-exps} imply
	that $(X_v,\calF_v)$ is a $(1-\lambda)$-coboundary expander in dimension $0$,
	so there is $g\in B^0(X_v,\calF_v)=B^0(X_v,\aug{(\F^m)})$ such that $\|f-g\|\leq  (1-\lambda)^{-1}\|d_0 f\|<
	\veps_0(1-\lambda)^{-1} $.
	There is $g_0\in \F^m$ such that $g(u)=g_0$ for all $u\in X_v(0)$.
	Put $S=\{u\in X_v(0)\suchthat f(u)=g_0\}$. Then $w_{X_v}(S)> 1-(1-\lambda)^{-1}\veps_0 =
	\frac{\kappa/Q}{r/m}$.
	Now, assumption (3)  implies that $|S|> \frac{m}{r}$,
	so by (4), we have $\bigcap_{u\in S}U(v\cup u)=0$.
	Since $g_0\in  \bigcap_{u\in S}U(v\cup u)$, we must have $g=0$,
	and $\|f\|=\|f-g\|\leq  (1-\lambda)^{-1}\|d_0 f\| $.
	This means that $\veps_0\|f\|\leq (1-\lambda)\|f\|\leq \|d_0 f\|$, as required.
	
	Next, assumption (2) implies that $X$ is a $1$-local  $\lambda$-skeleton expander
	(see~\S\ref{subsec:skeleton}). Combining (2) with Oppenheim's Trickling Down Theorem
	\cite[Theorem~1.4]{Oppenheim_2015_vanishing_of_cohomology},
	we see that the underlying weighted
	graph of $X$ is a $[-1,\frac{\lambda}{1-\lambda}]$-spectral expander,
	so $X$ is a $\frac{\lambda}{1-\lambda}$-skeleton expander.
	
	Plugging everything into Corollary~\ref{CR:expansion-of-small-one-ch} now gives
	the existence of $\gamma$ and $\beta$; the inequalities in the corollary
	hold by our assumptions on $\lambda$.
	
	To finish, we show that the rate of $Z^0(X,\calG)$ is at least
	$ \frac{2}{Q}-\frac{r}{m}-|X(1)|^{-1}$. This is equivalent to 
	showing that $\dim_\F Z^0(X,\calG)\geq |X(1)|m\cdot(\frac{2}{Q}-\frac{r}{m}-\frac{1}{|X(1)|})$.
	Observe that 
	\[
	\dim B^1(X,\calF)=\dim C^0(X,\calF)-\dim Z^0(X,\calF)\geq |X(0)|m-m=|X(1)|m\cdot\frac{2}{Q}-m,
	\]
	where the inequality follows from Lemma~\ref{LM:coh-loc-const} and the last
	equality follows from assumption (1). On the other hand,
	$\dim C^1(X,\calG)=|X(1)|(m-r)=|X(1)|m\cdot (1-\frac{r}{m})$.
	Since $Z^1(X,\calG)$ contains $C^1(X,\calG)\cap B^1(X,\calF)$
	and the intersection takes place in the ambient space $C^1(X,\calF)$, of dimension
	$|X(1)|m$, it follows
	that
	\[
	\dim Z^1(X,\calG)\geq |X(1)|m\cdot((\frac{2}{Q}-\frac{1}{|X(1)|})+ (1-\frac{r}{m})-1)
	=|X(1)|m\cdot(\frac{2}{Q}-\frac{r}{m}-\frac{1}{|X(1)|}).\qedhere
	\]
\end{proof}

\begin{remark}
	The pair $(X,\calG)$ does not satisfy the assumptions of Theorem~\ref{TH:cosystolic-exp-from-links}
	with $k=1$, so we cannot assert that  $(X,\calG)$ is an $(\veps,\delta)$-cosystolic expander
	in dimension $1$ for some $\veps,\delta>0$. Indeed, $(X,\calG)$ is \emph{not} a
	$1$-local $\veps$-coboundary expander
	in dimension $2$ for every $\veps>0$, because $\HH^0(X_e,\calG_e)\cong \calF(e)/U(e)\neq 0$ for all $e\in X(1)$.
	Consequently, we cannot assert that the $1$-cocycle code of $(X,\calG)$ (with its natural $3$-tester)
	is $\mu$-testable for $\mu>0$ independent of $(X,\calG)$. We do not know if such codes are good LTCs
	in general.
%
\end{remark}

We finish with explaining how to get an infinite family of good $1$-cocycle codes using 
Theorem~\ref{TH:one-cocycle-code}.

\begin{example}
	Let $Y$ be the affine building of $\nSL{F}{3}$, where $F$ is a local field
	with residue field of   $q$ elements; it is a
	$2$-dimensional building  of type $\tilde{A}_2$. It is well-known that one can find
	a sequence of simplicial complexes $\{X_s\}_{s\in \N\cup \{0\}}$ covered by $Y$
	and such $|X_s|$ tends to $\infty$ with $s$ and such that each $X_s$ covers $X_0$; see \cite{Lubotzky_2005_explicit_constructions_of_Ramanujan_complexes}
	for explicit constructions, or \S\ref{subsec:quotients-of-buildings} below.
	For every $v\in Y$, the link $Y_v$ is isomorphic to the spherical building $A_2(\F_q)$
	of Example~\ref{EX:An-building}. This implies readily that,
	for every $s\in\N\cup\{0\}$,
	every $0$-face of $X_s$ is contained in exactly $ 2(q^2+q+1)$ edges,
	every edge of $X_s$ is contained in exactly $q+1$ triangles,
	and $(X_s)_v$ is the incidence graph of the projective plane over $\F_q$ for every $v\in X_s(0)$.
	Thus, conditions (1)--(3) of Construction~\ref{CN:one-cocycle-code}
	hold for $X=X_s$ with $Q= 2(q^2+q+1)$, $\lambda=\frac{\sqrt{q}}{q+1}$
	and $\kappa=1$.
	Let $m=\floor{\frac{3}{4}Q}$ and $r=1$, and let $\calF_s$ denote the constant sheaf $\F^m$
	on $X_s$.
	By Remark~\ref{RM:one-cocycle-code-U-cond}(i), for a sufficiently
	large finite field $\F$, there is a choice of subspaces $\{U(e)\}_{e\in X_0(1)}$
	for which condition (4) holds with $(X,\calF)=(X_0,\calF_0)$. Part (ii)
	of that remark then implies that a choice of $U(e)$-s satisfying (4) exists
	for every $(X_s,\calF_s)$; let $\calG_s$ denote the sheaf from Construction~\ref{CN:one-cocycle-code}
	constructed using this data.
	Note that $\lambda$ tends to $0$ as $q$ tends to $\infty$.
	Thus, if $q$ is sufficiently large, then
	Theorem~\ref{TH:one-cocycle-code} says that the family of $1$-cocycle
	codes $\{Z^1(X,\calG_r)\subseteq C^1(X,\calG_r)\}_{r\in \N}$ on the alphabet $\Sigma:=\F^{m-1}$ 
	has linear distance
	and constant rate.
\end{example}

\section{Rate Conservation}
\label{sec:rate}

Throughout this section, $\F$ denotes a finite field of characteristic $p>0$.
Let $X$ be a $d$-complex, let $k\in\{0,\dots,d-1\}$, let $\calF$ be an $\F$-sheaf on $X$,
and let $B$ be a $\F$-basis of $\calF$
(see Example~\ref{EX:Hamming-norm}). 
Recall from Section~\ref{sec:ltc-and-css} that if $B^k(X,\calF)=0$  and $\dim \calF(x)=m$ for all $x\in X(k)$,
then $(X,\calF)$ gives rise to a $k$-cocycle code $Z^k(X,\calF)\subseteq C^k(X,\calF)\cong \Sigma^{X(k)}$
with alphabet $\Sigma=\F^m$.
If $k>0$ and $B^k(X,\calF)\neq 0$, then $(X,\calF,B)$
gives rise to $k$-cocycle quantum CSS code with alphabet $\F$.
In either case, the rate of the code is the ratio $\dim\HH^k(X,\calF)/\dim C^k(X,\calF)$.

Let $u:Y\to X$ be a covering.  Then the pullback sheaf $u^*\calF$ similarly gives
rise to a $k$-cocycle code-with-tester or a $k$-cocycle quantum CSS code. 
(Note that $B^k(X,\calF)=0$ implies $B^k(Y,u^*\calF)=0$.) 
In general, there is no
relation between the rates of the $k$-cocycle codes associated to $(X,\calF)$ and $(Y,u^*\calF)$.
However, in this section, we will show that under some assumptions on $u$ and $(X,\calF)$, we can guarantee
that the rate  $\dim\HH^k(Y,u^*\calF)/\dim C^k(Y,u^*\calF)$ is bounded from below by a constant
depending only on $X$ and $\calF$. This principle, formalized as Theorem~\ref{TH:rate-conservation},
will be called   \emph{rate conservation} in the sequel.

\medskip

We begin with two  lemmas.
The cyclic group of order $n$ is denoted $C_n$.
We will make extensive use of $C_n$-Galois coverings in the sense
\S\ref{subsec:coverings}. For example,
every double covering is a $C_2$-Galois covering and vice versa (Example~\ref{EX:coverings-basic}(i)).
Recall (\S\ref{subsec:sheaf-coh}) that if $\calF$ is an $\F$-sheaf on   $X$,
then $h^k(\calF)$  or $h^k(X,\calF)$, denotes $\dim_{\F}\HH^k(X,\calF)$.

\begin{lem}\label{LM:filtration}
	Let $p$ be a prime number, let $u:Y\to X$ be a $C_p$-Galois covering of simplicial complexes,
	let $\F$ be  a field of characteristic $p$,
	and let  $\calF$ be an $\F$-sheaf.
	Then there exists a sequence of subsheaves
	\[
	0=\calF_0\subseteq \calF_1\subseteq \dots\subseteq \calF_p=u_*u^*\calF
	\]
	such that $\calF_i/\calF_{i-1}\cong \calF$ for all $i\in\{1,\dots,p\}$.
\end{lem}

The proof of the lemma is shorter and more elementary when $p=2$.
To help the reader, we decided to address this special case   
before proving the lemma in general.

\begin{proof}[Proof when $p=2$.]
	Put $\calF_2=u_*u^*\calF$ and let $\calF_0$ be the zero subsheaf of $\calF_2$.
	For every $x\in X-\{\emptyset\}$,
	we have $\calF_2(x)=u_*u^*\calF(x)=\prod_{y\in u^{-1}(x)} u^*\calF(y)=\prod_{y\in u^{-1}(x)}\calF(x)\cong \calF(x)\times \calF(x)$.
	Fix  an isomorphism $\calF_2(x)\cong \calF(x)\times \calF(x)$ for every $x$. 
	If $y\in X_{\supseteq x}-\{x\}$,
	then the restriction map $\res^{\calF_2}_{y\from x}:\calF(x)\times \calF(x)\to \calF(y)\times \calF(y)$
	is either $\res^{\calF}_{y\from x}\times  \res^{\calF}_{y\from x}$,
	or $\res^{\calF}_{y\from x}\times  \res^{\calF}_{y\from x}$ followed by swapping the two copies of $\calF(y)$.
	This observation allows us to define a subsheaf $\calF_1$  of $\calF_2$   by setting
	$\calF_1(x)=\{(f,f)\where f\in \calF(x)\}$ for all $x\in X-\{\emptyset\}$.
	
	For every $x\in X-\{\emptyset\}$, 
	define  $\vphi_x:\calF(x)\to \calF_1(x)$
	and $\psi_x:\calF_2(x)/\calF_1(x)\to \calF(x)$ by
	$\vphi_x(f)=(f,f)$ and $\psi_x((f,g)+\calF_1(x))=f-g$.
	Using the assumption $\Char \F=2$, it is straightforward to check that $\vphi=(\vphi_x)_{x\in X-\{\emptyset\}}$
	and $\psi = (\psi_x)_{x\in X-\{\emptyset\}}$ determine isomorphisms
	of sheaves $\vphi:\calF\to \calF_1=\calF_1/\calF_0$
	and $\psi:\calF_2/\calF_1\to\calF$, hence the lemma.
\end{proof}

\begin{proof}[Proof for general $p$.]
	Let $g$ denote a generator of the group $C_p$ and let $\F C_p=\F[g\where g^p=1]$
	denote the group algebra of $C_p$. Let $I$ be the augmentation ideal of $\F C_p$,
	i.e., $I=(g-1)\F C_p$. 
	Since $\Char \F=p$, we have $I^p=0$ (because $(g-1)^p=g^p-1^p=0$).
	In fact, one readily checks that 
	$\dim_{\F} I^n=p-n$ for all $n\in\{0,\dots,p\}$. The 
	quotient $\F C_p$-module $I^n/I^{n+1}$
	is spanned as an $\F$-vector space by $(g-1)^n$,
	and $g$ acts trivially
	on $I^n/I^{n+1}$ because, for every $a\in I^n$, we have $ga-a=(g-1)a\in (g-1)I^n= I^{n+1}$.
	In what follows, all tensor products are over $\F$.

	For every $x\in X-\{\emptyset\}$, choose some $\hat{x}\in Y$ with $u(\hat{x})=x$.
	Since $u:Y\to X$ is a $C_p$-Galois covering, $u^{-1}(x)=\{\tau \hat{x}\where \tau\in C_p\}$.
	As a result, for every $y\in X_{\supseteq x}$, there is a unique element $c_{y,x}\in C_p$
	such that $c_{y,x}\hat{x}\subseteq \hat{y}$. 
	This also allows us to identify $u_*u^*\calF(x)=\prod_{x'\in u^{-1}(x)}\calF(x)$ with
	$\F C_p\otimes_\F \calF(x) $ via sending 
	$(f_{x'})_{x'\in u^{-1}(x)}$
	to $\sum_{\tau \in C_p} \tau\otimes f_{\tau \hat{x}}$.
 	The restriction map $\res^{u_*u^*\calF}_{y\from x}:\F C_p \otimes \calF(x) \to \F C_p \otimes \calF(y) $
	is then given by $a\otimes f\mapsto a c_{y,x}^{-1} \otimes (\res^{\calF}_{y\from x}f)$.
	
	For $n\in\{0,\dots,p\}$, let $\calF_n$ denote the subsheaf of $u_* u^*\calF$
	determined by $\calF_n(x)=I^{p-n}\otimes \calF(x)$.
	Fix $n\in\{1,\dots,p\}$.
	For every $x\in X-\{\emptyset\}$,
	define  $\vphi_x:\calF(x)\to \calF_n(x)/\calF_{n-1}(x)\cong (I^{p-n}/I^{p-n+1})\otimes \calF(x)$
	by $\vphi_x(f)=((g-1)^{p-n}+I^{p-n+1})\otimes f$. This is an $\F$-vector space isomorphism.
	Moreover, 
	since $g$, and thus all elements of $C_p$, act  trivially on $I^{p-n}/I^{p-n+1}$,
	we have $\vphi_y\circ \res^{\calF}_{y\from x}=\res^{\calF_n/\calF_{n-1}}_{y\from x} \circ \vphi_x$
	whenever $\emptyset\neq x\subsetneq y\in X$.
	This means that $\vphi=(\vphi_x)_{x\in X-\{\emptyset\}}:\calF\to \calF_n/\calF_{n-1}$ is a
	an $\F$-sheaf isomorphism, and the lemma follows.
\end{proof}

\begin{lem}\label{LM:zero-coh-trick}
	Let $p$ be a prime number, let $u:Y\to X$ be a $C_p$-Galois covering of simplicial complexes,
	let $\F$ be a field of characteristic $p$,
	let $\calF$ be an $\F $-sheaf, and let $k\in\N\cup\{0\}$.
	Suppose that $\HH^{k-1}(X,\calF)=0$ (this always holds for $k=0$). Then
	\begin{enumerate}[label=(\roman*)]
		\item  $\HH^{k-1}(Y,u^*\calF)=0$, and
		\item
		$h^{k}(Y,u^*\calF)-h^{k+1}(Y,u^*\calF) \geq p (h^k(X, \calF)-
		h^{k+1}(X,\calF))$.
	\end{enumerate}
\end{lem}

\begin{proof}
	Let $\{\calF_i\}_{i=0}^p$ be the sequence
	of sheaves from Lemma~\ref{LM:filtration},
	and put 
	$N_i=\dim \HH^{k}(X,\calF_i)-\dim \HH^{k+1}(X,\calF_i)$.
	We will show by increasing induction on $i\in\{1,\dots,p\}$
	that $\HH^{k-1}(X,\calF_i)=0$ and $N_i\geq iN_1$.
	Provided this holds, taking $i=p$ and applying Lemma~\ref{LM:Shapiro}
	to $\calF_p=u_*u^*\calF$ gives (i) and (ii).	
	
	The case $i=1$ follows from the assumptions of the lemma,
	because  $\calF_1\cong \calF$. 
	
	Suppose that $i>1$ and we have shown that $\HH^{k-1}(X,\calF_{i-1})=0$
	and $N_{i-1}\geq (i-1)N_1$.
	The inclusion $\calF_{i-1}\subseteq \calF_i$
	gives rise to a short exact sequence
	$0\to \calF_{i-1}\to \calF_i\to\calF_i/\calF_{i-1}\to 0$,
	and thus to a long exact sequence of cohomology
	groups (see \S\ref{subsec:sheaf-coh}):
	\begin{align*}
		\cdots\to &\HH^{k-1}(X,\calF_{i-1})\to \HH^{k-1}(X,\calF_{i})\to \HH^{k-1}(X,\calF) \\
		\to & \HH^{k }(X,\calF_{i-1})\to \HH^{k }(X,\calF_i)\to \HH^{k }(X,\calF)\\
		\to & \HH^{k+1}(X,\calF_{i-1})\to \HH^{k+1}(X,\calF_i)\to \HH^{k+1}(X,\calF)\to \cdots
	\end{align*}
	Here, we substituted $\calF_i/\calF_{i-1}$ with the isomorphic sheaf
	$\calF$.
	Since both $\HH^{k-1}(X,\calF_{i-1})$
	and $\HH^{k-1}(X,\calF)$ are $0$, so is $\HH^{k-1}(X,\calF_i)$.
	Write $V= \coker (\HH^{k+1}(X,\calF_i)\to \HH^{k+1}(X,\calF))$.
	Then we have a $7$-term exact sequence
	\begin{align*}
	0&\to \HH^{k }(X,\calF_{i-1})\to \HH^{k }(X,\calF_i)\to \HH^{k }(X,\calF)\\
	&\to \HH^{k+1}(X,\calF_{i-1})\to \HH^{k+1}(X,\calF_i)\to \HH^{k+1}(X,\calF)\to 
	V\to 0
	\end{align*}
	This
	means that
	\[
	h^k(\calF_{i-1})-h^k(\calF_i)+h^k(\calF)-h^{k+1}(\calF_{i-1})+
	h^{k+1}(\calF_i)-h^{k+1}(\calF)+\dim V=0,
	\]
	and by rearranging, we get
	\[
	N_i=N_{i-1}+N_1+\dim V\geq (i-1)N_1 + N_1 =iN_i. \qedhere
	\]
\end{proof}

\begin{thm}[Rate Conservation]\label{TH:rate-conservation}
	Let $p$ be a prime number and let $\F$ be a   field
	of characteristic $p$.
	Let $X$ be a simplicial complex of dimension $d$, let $k\in\{0,\dots,d\}$
	and let $\calF$ be an $\F$-sheaf on $X$
	such that
	\[
	h^{k-1}(X,\calF)=0
	\qquad\text{and}\qquad
	h^{k }(X,\calF)> h^{k+1}(X,\calF).
	\]
	Put 
	\[\rho = (h^{k }(X,\calF)-h^{k+1}(X,\calF))/
	\dim C^k(X,\calF).
	\]
	Let $u:Y\to X$ be a covering map of degree $p^r$
	such that $u$ factors as 
	$Y=X_r\to X_{r-1}\to \dots\to X_0=X$
	and each map $X_i\to X_{i-1}$ is a $C_p$-Galois covering.
	Then:
	\begin{enumerate}[label=(\roman*)]
		\item $h^k(Y,u^*\calF)-h^{k+1}(Y,u^*\calF)\geq \rho \dim C^k(Y,u^*\calF)$.
		\item The rate of the $k$-cocycle code
		of $(Y,u^*\calF)$ (resp.\ the   $k$-cocycle quantum CSS
		code associated to $(Y,u^*\calF)$ and some basis of $u^*\calF$)
		is   at least $\rho$.
	\end{enumerate}
\end{thm}

\begin{proof}
	Let $u_i$ denote the composition $X_i\to \dots\to X_0=X$.
	Applying Lemma~\ref{LM:zero-coh-trick} 
	to the covering $X_i\to X_{i-1}$ and the sheaf $u_{i-1}^*\calF$
	with $i$ ranging from $1$ to $r$ shows that
	$h^k(Y,u^*\calF)-h^{k+1}(Y,u^*\calF)
	\geq p^r (h^k(X,\calF)-h^{k+1}(X,\calF))\geq p^r \rho  \dim C^k(X,\calF)$.
	Since $\dim C^k(Y,u^*\calF)=p^r \dim C^k(X,\calF)$,
	this means that $h^k(Y,u^*\calF)-h^{k+1}(Y,u^*\calF)\geq \rho \dim C^k(Y,u^*\calF)$. 
	This proves (i), and (ii) follows because
	$\dim Z^k(Y,u^*\calF)\geq 
	h^k( u^*\calF)\geq h^k(u^*\calF)-h^{k+1}(u^*\calF)$.
\end{proof}

\section{What Is Required  to Construct an Infinite Family of LTCs?}
\label{sec:recap}

We now put together the local-to-global principle
for cosystolic expansion (Theorem~\ref{TH:cosystolic-exp-from-links})
and the Rate Conservation Theorem (Theorem~\ref{TH:rate-conservation}) 
to give a recipe for constructing infinite families of low-query LTCs
with linear distance and constant rate. This is the \emph{tower paradigm} outlined
in Section~\ref{sec:overview}. 

\begin{thm}[Tower Paradigm]\label{TH:tower-paradigm}
	Let $p$ be a prime number and let $\F$ be a finite field of
	characteristic $p$.	
	Let $X $ be a  strongly connected  $d$-complex,
	let $k\in\{0,\dots,d-2\}$ and let $\calF$ be an $\F$-sheaf on $X$.
	Suppose that there is $m\in\N$ such that $\calF(x)=\F^m$ 
	for all $x\in X(k)$, that $B^{k-1}(X,\calF)=0$ (e.g., if $k=0$),
	that $\calF(x)\neq 0$ for all $x\in   X(k+1)\cup X(k+2)$, and the following  
	conditions are met:
	\begin{enumerate}[label=(t\arabic*)]
		\item \label{item:TH:tower-paradigm:tower}
			$X$ admits an infinite tower of connected  $C_p$-Galois coverings $\cdots \to X_2\to X_1\to X_0=X$ (i.e.,   $X_r\to X_{r-1}$ is a $C_p$-Galois covering
			for all $r\in\N$).
		\item \label{item:TH:tower-paradigm:LTC}
			There exist numbers $\veps_0,\dots,\veps_k,\veps'_0,\dots,\veps'_{k+1},\lambda\in \R_+$
			satisfying the inequality 
			of Theorem~\ref{TH:cosystolic-exp-from-links} and such that
			\begin{enumerate}[label=(t2-\alph*)]
				\item $(X,\calF)$ is an $i$-local $\veps_i$-coboundary
				expander in dimension $k$ for all $i\in\{0,\dots,k\}$,
				\item $(X,\calF)$ is an $i$-local $\veps'_i$-coboundary
				expander in dimension $k+1$ for all $i\in\{0,\dots,k+1\}$, and
				\item $X$ is a $(d-2)$-local $[-1,\lambda]$-spectral expander.
			\end{enumerate}
		\item \label{item:TH:tower-paradigm:rate}
			$\dim \HH^0(X,\calF)>\dim \HH^1(X,\calF)$.
	\end{enumerate}
	Write $u_r$ for the composition $X_r\to X_{r-1}\to \dots\to X_0=X$,
	put $\calF_r=u_r^*\calF$ and $\Sigma=\F^m$, and let 
	\[(Z^k(X_r,\calF_r), C^k(X_r,\calF_r)=\Sigma^{X_r(k)},\Phi_r)\] denote the $k$-cocycle
	code of $(X_r,\calF_r)$ with its natural tester
	(\S\ref{subsec:ltc}). 
	Then   $\{(Z^k(X_r,\calF_r),  \Sigma^{X_r(k)},\Phi_r) \}_{r\geq 0}$
	is a family of $(k+2)$-query LTCs with linear distance and constant rate.
	Moreover, there is $\eta>0$
	such that every code  in the family admits   a linear-time decoding
	algorithm able to correct  up to an $\eta $-fraction of errors.

	If only conditions \ref{item:TH:tower-paradigm:tower} and \ref{item:TH:tower-paradigm:LTC}
	are met, then $\{(Z^k(X_r,\calF_r),  \Sigma^{X_r(k)},\Phi_r) \}_{r\geq 0}$ is a family
	of LTCs with linear distance and the codes admit a decoding algorithm as above.
	If only conditions  \ref{item:TH:tower-paradigm:tower} and \ref{item:TH:tower-paradigm:rate}
	are met, then the codes $\{(Z^k(X_r,\calF_r),  \Sigma^{X_r(k)}) \}_{r\geq 0}$ 
	have constant rate.
\end{thm}

By default, we will assume that $k=0$ when talking about the tower paradigm.
In this special case, Theorem~\ref{TH:tower-paradigm}
gives a recipe for getting an infinite family of $2$-query
LTCs with linear distance and constant rate.

\begin{proof}
	Condition \ref{item:TH:tower-paradigm:rate} and
	Theorem~\ref{TH:rate-conservation} imply
	that the rate of the codes
	$\{(Z^k(X_r,\calF_r),  \Sigma^{X_r(k)}) \}_{r\geq 0}$ is
	bounded from below by some $\rho>0$.
	The remaining assertions    follow from 
	Proposition~\ref{PR:decoding}
	and Theorem~\ref{TH:cosystolic-exp-from-links},
	provided   that   conditions (1)--(5) of
	Theorem~\ref{TH:cosystolic-exp-from-links} hold
	for every $(X_r,\calF_r)$ with $\veps_0,\dots,\veps_k,\veps'_0,\dots,\veps'_{k+1},\lambda$ as in \ref{item:TH:tower-paradigm:LTC}
	and $Q=D(X)$.
	These conditions hold for $(X,\calF)$ by our assumptions,
	so they also hold for $(X_r,\calF_r)$
	by Remark~\ref{RM:comments-on-local-to-global}(iv). 
\end{proof}

\begin{remark}\label{RM:alternative-three-conds}
	In Theorem~\ref{TH:tower-paradigm}, it is possible
	to replace  \ref{item:TH:tower-paradigm:LTC}
	and the connectivity assumption in~\ref{item:TH:tower-paradigm:tower}
	with   milder assumptions  
	by using Theorem~\ref{TH:expansion-of-small-sets-finer}
	and its corollaries instead of Theorem~\ref{TH:cosystolic-exp-from-links}
	in the proof. Specifically, in the case $k=0$, if
	we use Corollary~\ref{CR:cse-from-links-dim-2} instead of Theorem~\ref{TH:cosystolic-exp-from-links}, then we can replace \ref{item:TH:tower-paradigm:tower}
	and~\ref{item:TH:tower-paradigm:LTC} with the following:
	There are $\alpha_{-1},\alpha_0,\veps_0,\veps'_0,\veps'_1\in \R_+$
	satisfying the inequalities Corollary~\ref{CR:cse-from-links-dim-2} such that
	\begin{enumerate}[label=(t\arabic*$'$)]
		\item  
			$X$ admits an infinite tower of $C_p$-Galois coverings $\cdots \to X_2\to X_1\to X_0=X$,
			and each $X_r$ is an $\alpha_{-1}$-skeleton expander.
		\item  
			$(X,\calF)$ is a $0$-local $\veps_0$-coboundary
			expander in dimension $0$, an $i$-local $\veps'_i$-coboundary
			expander in dimension $1$ for $i\in\{0,1\}$
			and a $0$-local $\alpha_0$-skeleton expander.
	\end{enumerate}
\end{remark}

Thanks to Theorem~\ref{TH:tower-paradigm}, the problem
of constructing an infinite family of good $2$-query LTCs reduces
to the following question:

\begin{que}\label{QE:initial-data-for-tower-paradigm}
	Is there a sheaved $d$-complex $(X,\calF)$  
	satisfying conditions \ref{item:TH:tower-paradigm:tower}--\ref{item:TH:tower-paradigm:rate}
	of Theorem~\ref{TH:tower-paradigm} with $k=0$?
\end{que}

Note that we only need a  \emph{single} pair $(X,\calF)$ satisfying
\ref{item:TH:tower-paradigm:tower}--\ref{item:TH:tower-paradigm:rate} with $k=0$;
we will refer to such pairs as \emph{initial data} for the tower paradigm.
Note also that once a candidate $(X,\calF)$ is presented,
conditions \ref{item:TH:tower-paradigm:LTC} and
\ref{item:TH:tower-paradigm:rate} can be checked by computation, and only condition
\ref{item:TH:tower-paradigm:tower} needs a theoretical proof.

Finding  initial data for the tower paradigm
is the subject matter of Chapter~\ref{chap:initial-data},
where we reduce the problem to an experiment-supported
conjecture and the existence of certain arithmetic groups.
We also show in  \S\ref{subsec:three-cond-examples}
that   any two of the conditions
\ref{item:TH:tower-paradigm:tower}--\ref{item:TH:tower-paradigm:rate} are
fulfilled for some pair  $(X,\calF)$.
Alas, 
Question~\ref{QE:initial-data-for-tower-paradigm} 
remains open.

\medskip

We finish this section by explaining why
simplicial complexes with a locally constant sheaf
(see \S\ref{subsec:loc-const})
cannot serve as initial data for the tower paradigm.
Note first that if $\calF$ is a nonzero locally constant sheaf on $X$, then
the assumption $B^k(X,\calF)=0$
of Theorem~\ref{TH:tower-paradigm} is satisfied only if $k=0$.
The following proposition says that in this case,  conditions \ref{item:TH:tower-paradigm:tower} and 
\ref{item:TH:tower-paradigm:rate} cannot hold simultaneously.

\begin{prp}\label{PR:loc-constant-failure}
	Let $X$ be a $d$-complex, let $\F$ be a field
	of characteristic $p>0$,
	and let $\calF$ be a locally
	constant $\F$-sheaf on $X$. Suppose
	that $X$ admits an infinite tower
	of connected $C_p$-Galois coverings $\cdots\to X_2\to X_1\to X_0=X$.
	Then $\dim \HH^0(X,\calF)\leq \dim \HH^1(X,\calF)$.
\end{prp}

\begin{proof}
	Write $u_r$ for the map $X_r\to\dots\to X_1\to X$ and put $\calF_r=u_r^*\calF$.
	Then $\calF_r$ is locally constant of dimension $\dim\calF$.
	Since $X_n$ is connected, Lemma~\ref{LM:coh-loc-const}
	tells us that $h^0(X_r,\calF_r)\leq \dim\calF_r=\dim\calF$.
	Now, if it were the case that $h^0(X,\calF)> h^1(X,\calF)$,
	then Theorem~\ref{TH:rate-conservation}
	would imply that $ h^0(X_r,\calF_r)$ tends to $\infty$
	as $r\to\infty$, which contradicts our previous conclusion
	that $h^0(X_r,\calF_r)\leq \dim\calF$ for all $r$.
	Thus, we must have $h^0(X,\calF)\leq h^1(X,\calF)$.
\end{proof}

\begin{remark}
	There is an analogue of  Theorem~\ref{TH:tower-paradigm} 
	for quantum CSS codes. 
	That is, we can impose conditions similar
	to \ref{item:TH:tower-paradigm:tower}--\ref{item:TH:tower-paradigm:rate} 
	on a sheaved $d$-complex
	$(X,\calF)$ that would give rise to an infinite family of 
	one-sided locally
	testable quantum CSS
	codes that have constant rate, linear $X$-distance, and 
	whose  
	$X$-side has
	a linear-time decoding algorithm able to correct a constant-fraction of errors.
	Simply assume $k>0$, drop the assumption
	$B^k(X,\calF)=0$, and replace the use of Proposition~\ref{PR:decoding}
	with Proposition~\ref{PR:cbe-to-quantum-CSS}.
	
	Unlike the case of LTCs, it is seemingly 
	possible for $d$-complexes with locally constant sheaves
	to satisfy the required conditions.
\end{remark}

\chapter{Toward Initial Data for The Tower Paradigm}
\label{chap:initial-data}

Having the tower paradigm (Theorem~\ref{TH:tower-paradigm} with $k=0$)
at our disposal to produce good $2$-query LTCs, we now set to look
for sheaved complexes $(X,\calF)$ satisfying   its three prerequisites  
\ref{item:TH:tower-paradigm:tower}--\ref{item:TH:tower-paradigm:rate} (with $k=0$). 
We only need
one such pair.
The purpose of this chapter is to construct sheaved complexes which
satisfy 
\ref{item:TH:tower-paradigm:tower} (existence of an infinite tower),
\ref{item:TH:tower-paradigm:LTC} (local expansion conditions) 
and conjecturally also 
\ref{item:TH:tower-paradigm:rate} (rate conservation).
While \ref{item:TH:tower-paradigm:rate} could be checked by computation,
such a computation is beyond
the reach of present computers because of the sheer size of $X$ and $\calF$.

In more detail, our approach to constructing initial data for the tower paradigm  
starts 
with a $d$-complex $X$ and a locally constant $\F$-sheaf $\calF$ on $X$ of a  
large (but ultimately fixed) dimension.\footnote{
	Recall that the tower paradigm \emph{cannot} work for locally constant sheaves;
	see Proposition~\ref{PR:loc-constant-failure}.
} As in \S\ref{subsec:sheaf-coh}, we abbreviate
\[
h^i(\calF)=\dim  \HH^i(X,\calF). 
\]
In Section~\ref{sec:modifying}, we present an iterative process taking $\calF$
and producing a subsheaf $\calC$ of $\calF$
such that  
$\quo{\calF}:= \calF/\calC$ is a (non-locally constant) sheaf with
$h^0(\quo{\calF})>h^1(\quo{\calF})  $, that is,
it satisfies the requirement \ref{item:TH:tower-paradigm:rate}.
We show that if the resulting
subsheaf $\calC$ of $\calF$ --- which
grows with each iteration of the process --- 
is ``small'' with respect to $\calF$,
then $(X,\quo{\calF})$ will satisfy the local   expansion conditions
in \ref{item:TH:tower-paradigm:LTC}  when $(X,\calF)$ satisfies them.
(We could also terminate the process while $\calC$ is still ``small''
to secure \ref{item:TH:tower-paradigm:LTC}, and hope that it suffices
to get \ref{item:TH:tower-paradigm:rate}.)
Choosing $X$ in advance so that it has an infinite tower of connected  double coverings
would   secure
\ref{item:TH:tower-paradigm:tower}.

Broadly speaking, we expect the process to converge quickly enough
when $h^1(\calF)$ is small with respect to $\dim \calF$.
We conjecture
that this is indeed the case when 
$X$ is covered by a sufficiently thick affine building (\gap{}Conjecture~\ref{CJ:process-for-buildings}).
We show in Theorems~\ref{TH:sheaves-with-small-coh}
and~\ref{TH:sheaves-with-small-coh-Ram} that there are simplicial complexes
covered by affine buildings which admit $\F_2$-sheaves $\calF$
of arbitrarily large dimension such that $h^1(\calF)=o(\dim \calF)$ (even
$h^1(\calF)=O(1)$ or $h^1(\calF)=0$, in some cases).
Thus, if our conjecture  holds for just one such $X$, 
then there is an $\F_2$-sheaf $\calF$ on $X$ such that
if we feed it into our process to produce $\quo{\calF}=\calF/\calC$, 
then   $(X,\quo{\calF})$ satisfies the requirements
\ref{item:tower-intro:first}--\ref{item:tower-intro:last} of
the tower paradigm.
In particular, $(X,\quo{\calF})$ gives rise to an infinite
family of  $2$-query LTCs with constant rate and linear distance.

We also give strong evidence that if
the kernel of the 
cup product $\cupp:\HH^1(X,\F)\otimes_{\F} \HH^1(X,\calF)\to\HH^2(X,\calF)$
(see \S\ref{subsec:cup-prod}) is of dimension smaller than $h^0(\calF)$,
then the iterative process stops quickly enough (after one step, in fact).
Assuming this, we show  that if a sheaf satisfying the said condition exists
over a simplicial complex $X$   covered by a sufficiently   
thick affine building, then there is a locally constant sheaf $\calF'$
on $\calF$ such that $(X,\quo{\calF'})$ satisfies the requirements
of the tower paradigm.

The iterative process is presented and discussed in Section~\ref{sec:modifying}. 
In Section~\ref{sec:candidates}, we show that
there exist simplicial complexes $X$ covered by arbitrarily
thick affine buildings that admit (1) an infinite tower
of double covering, and (2) $\F_2$-sheaves $\calF$
of arbitrarily large dimension such that $h^1(\calF)=o(\dim\calF)$.
We also demonstrate that any pair of   the conditions 
\ref{item:tower-intro:first}--\ref{item:tower-intro:last} 
of the tower paradigm are
satisfied by some sheaved complex.
The intermediate Section~\ref{sec:congruence} establishes the existence of another tower of coverings
(not the tower   required for the tower power paradigm)
with some special properties that  is needed
for the construction of the desired sheaves in Section~\ref{sec:candidates}.
This makes use of deep results about arithmetic groups.

\section{Modifying Sheaves to Get Rate Conservation}
\label{sec:modifying}

Throughout this section, $\F$ is a field (of any characteristic) and
$X$ is a   $d$-complex with $d\geq 2$. We fix a linear ordering  on the vertices
of $X$ and use it to identify $C^i(X,\calF)$ with $\prod_{x\in X(i)}\calF(x)$   
for every sheaf $\calF$ on $X$ and $i\in\{0,\dots,d\}$, see Remark~\ref{RM:ordered-cohomology}.
We let $ {\calF}$ denote a locally constant $\F$-sheaf of dimension $m$.

\subsection{An Iterative Modification Process}
\label{subsec:quotient}

Recall from Construction~\ref{CN:modification-prototype} 
that if $E$ is an $\F$-subspace of $C^1(X,{\calF})$,
then we can form a subsheaf $\calC_E$  of ${\calF}$ by setting:
\[
\calC_E(x) = \sum_{e\in X(1)_{\subseteq x}} \res_{x\from e}\mathrm{Proj}_e(E)
\]
for all $x\in X$, where $\mathrm{Proj}_e:C^1(X, {\calF})\to  {\calF}(e)$
is the projection $f\mapsto f(e ):C^1(X,\calF)\to \calF(e)$. 
Otherwise stated, $\calC_E$ is the smallest
subsheaf of ${\calF}$ such that $E\subseteq C^1(X,\calC_E)$.
We will be interested in quotients sheaves of the form 
\[\calF_E:={\calF}/\calC_E.\]
Since $\dim\calC_E(v)=0$ for all $v\in X(0)$,
we have $\calF_E(v)=\calF(v)$ for all $v\in X(0)$.
However, if $E\neq 0$,
then $\calC_E\neq 0$, and as a result, $\calF_E$ is not locally
constant.
When $\dim E\ll \dim\calF$, the sheaf $\calF_E$ may be regarded
as being ``close'' to $\calF$ because
$
\dim \calC(x)\leq \schoose{i}{2} \dim E \ll \dim\calF
$
for every $i$-face $x\in X$.

With the tower paradigm in mind, our purpose will be to find a (typically
small) subspace $E$ of $C^1(X,\calF)$
such that $(X,\calF_E)$ satisfies conditions \ref{item:TH:tower-paradigm:LTC}
and \ref{item:TH:tower-paradigm:rate} of Theorem~\ref{TH:tower-paradigm}
(always with $k=0$).
We first focus on \ref{item:TH:tower-paradigm:rate}, which says that
$
h^1(\calF_E) < h^0(\calF_E)$.
%
%
%

The effect of $E$ on $\dim\HH^i(X,\calF_E)$ 
for $i=0,1$ is (crudely) described  in
the following proposition.

\begin{prp}\label{PR:effect-of-E}
	In the previous notation, let $B_E= E\cap B^1(X,\calF )$, $Z_E=E\cap Z^1(X,\calF )$
	and let $H_E\cong Z_E/B_E$ be the image of $Z_E$ under the quotient map
	$Z^1(X, \calF)\to \HH^1(X, \calF)$.
	Denote the natural map $\HH^2(X,\calC_E)\to \HH^2(X,{\calF})$ by $\omega_E$. 
	Then:
	\begin{enumerate}[label=(\roman*)]
		\item $h^0(\calF_E)\geq h^0({\calF})+\dim B_E$.
		\item $h^1( \calF_E)\leq h^1( {\calF})-\dim H_E+\dim \im(\omega_E)$.
	\end{enumerate}
\end{prp}

\begin{proof}
Recall from \S\ref{subsec:sheaf-coh}
that the short exact sequence $0\to\calC_E\to  \calF\to \calF_E\to 0$
gives rise to a long cohomology exact sequence:
\begin{align*}
	0 = & \HH^0(X,\calC_E)\to \HH^0(X,\calF)\to \HH^0(X,\calF_E) \to\\
	&\HH^1(X,\calC_E)\xrightarrow{\alpha} \HH^1(X, \calF)\to \HH^1(X,\calF_E)\to \\
	&\HH^2(X,\calC_E)\xrightarrow{\omega} \HH^2(X,\calF)
\end{align*}
Since $C^0(X,\calC_E)=0$, we have $\HH^0(X,\calC_E)=0$
and $\HH^1(X,\calC_E)=Z^1(X,\calC_E)$. Noting that $Z_E\subseteq Z^1(X,\calC_E)$
and $B_E= Z_E\cap B^1(X, \calF)\subseteq Z^1(X,\calC_E)\cap B^1(X, {\calF})$, 
we see that $B_E\subseteq \ker\alpha$
and $H_E\subseteq \im \alpha$. The proposition
now follows readily from the exactness.
\end{proof}

If we (incorrectly)  ignore the   factor $\dim \im(\omega_E)$ in
Proposition~\ref{PR:effect-of-E}
(say, if $\omega_E=0$), 
then
the affect of replacing $E$ with $E+\F f$ for some $f\in C^1(X, {\calF})-E$
can be summarized as follows: 
\begin{itemize}
	\item if $f\in B^1(X, {\calF})$,
then adding $f$ to $E$ increases $  h^0( \calF_E)$
by $1$, 
	\item if $f\in Z^1(X, {\calF})-B^1(X, {\calF})$,
then adding $f$ to $E$ decreases $  h^1( \calF_E)$ by $1$ and leaves
$h_0(\calF_E)$ unchanged, and
	\item   if 
$f\in C^1(X,\calF)-Z^1(X,\calF)$, then adding $f$ to $E$ \emph{seemingly} has no affect
on $h^0( \calF_E)$ and $h^1(\calF_E)$. 
\end{itemize}
Recall that our goal is to choose $E$ so that $h^0( \calF_E)>h^1( \calF_E)$.
Taking these thumb rules as facts, we could attempt to achieve this 
by simply taking $E$ to be a subspace of $Z^1(X,\calF)$ of dimension $>h^1(\calF)-h^0(\calF)$.
Indeed, we can decompose $E$ as $E_1\oplus E_2$ with $E_1\subseteq B^1(X,\calF)$
and $E_2\cap B^1(X,\calF)=0$. We expect to have $h^0(\calF_E)\geq h^0(\calF)+\dim E_1$
and $h^1(\calF)\leq h^1(\calF_E)-\dim E_2$, which together gives
$h^0(\calF_E)-h^1(\calF_E)>0$.

Unfortunately,   
$\ker(\omega_E:\HH^2(X,\calC_E)\to \HH^2(X,\calF))$ 
may be nonzero, and its effect on $h^1(\calF_E)$ should  be taken into account.
However, if   $\ker\omega_E\neq 0$, then we can attempt
to eliminate $\ker \omega_E$  by enlarging $E$ with more $1$-cochains.
Specifically, choose a subspace 
$K\subseteq C^2(X,\calC_E)$
mapping bijectively onto $\ker\omega_E\subseteq \HH^2(X,\calC_E)$ 
and   subspace  $E'\subseteq C^1(X, \calF)$ such that $d_1$
restricts to a bijection $E'\to K$ (it exists because
$K\subseteq B^2(X, \calF)$), and replace $E_0:=E$
with $E_1:=E+E'$. 
We show in Proposition~\ref{PR:securing-expansion-for-quotients}(i) below
that replacing $E$ by $E+E'$ does not affect $Z_E=E\cap Z^1(X,\calF)$
and $B_E=E\cap B^1(X,\calF)$.
At the same time, replacing $E$ by $E+E'$  trivializes the   
cohomology classes in  $\ker(\omega_{E_0}:\HH^2(X,\calC_{E_0})\to \HH^2(X,\calF)) $,
because 
the natural map
$\HH^2(X,\calC_{E_0})\to \HH^2(X,\calC_{E_1})$ 
vanishes on on $\ker \omega_{E_0}$.
The replacement of $E$ by $E+E'$ may result in new cohomology
classes in $\ker\omega_E$, so we can repeat this process until $\ker \omega_E =0$.
We therefore arrive at the following iterative process:

\begin{construction}\label{CN:modification-process}
Let $\calF$ be a locally constant $\F$-sheaf on a $d$-complex $X$ such that $h^1(\calF)\geq h^0(\calF)$.
\begin{enumerate}[label=(\arabic*)]
	\item Set $E_0$ to be the zero subspace of $C^0(X,\calF)$. 
	\item Let $E'_0$ be a subspace of $Z^1(X,\calF)$ of dimension $h^1(\calF)-h^0(\calF)+1$.\footnote{
		It is also possible to take a   subspace of larger dimension.	
	}
	\item Set $r=1$ and $E_1=E'_0$
	\item While $\dim E'_{r-1}>0$:
	\begin{enumerate}[label=(\alph*)]
		\item Choose
		a subspace  $E'_r\subseteq C^1(X, \calF)$ such that $d_1(E'_r)\subseteq Z^2(X,\calC_{E_r})$
		and the composition $ E'_r\xrightarrow{d_1} Z^2(X,\calC_{E_r})\to\HH^2(X,\calC_{E_r})$
		maps $E'_r$ bijectively onto  $\ker(\omega_{E_r}:\HH^2(X,\calC_{E_r})\to \HH^2(X, \calF))$.
		\item Set $E_{r+1}=E_r+E'_r$  and increase $r$ by $1$.
	\end{enumerate}
	\item Set $E=E_r$.
\end{enumerate}
We say that the iterative process stops or converges after $n$ steps if $E=E_{n }$.
If $\F$ is finite, and
if not indicated otherwise,   we assume that the spaces
$E'_0,E'_1,\dots$ defined in (2) and (a)  are chosen uniformly
at random among all eligible subspaces of $C^1(X,\calF)$.
\end{construction}

In what follows, we abbreviate $\omega_{E_r}$ to $\omega_r$
and  $\calC_{E_r}$ to $\calC_r$,
so that
\[
0=E_0\subseteq E_1\subseteq E_2\subseteq\dots\subseteq C^1(X,\calF)
\qquad\text{and}\qquad
\calC_0\subseteq \calC_1\subseteq\calC_2\subseteq \dots
\subseteq \calF.\]
Observe that $\dim E'_r$ is determined   by $E_r$. 
We shall see in the  the following proposition
that  $E'_r\cap E_r=0$,
and hence
$\dim E_{r+1}=\dim E_r+\dim E'_r$. Thus, $\dim E_{r+1}$
determined by $E_r$. However, the choice of $E'_r$ may  affect
$\dim E_s$ for $s>r+1$.
 
\medskip 
 
Part (ii) of the following proposition says that 
$\calF_E$   satisfies condition \ref{item:TH:tower-paradigm:rate}
when   $E$ is the subspace constructed in the iterative process
of Construction~\ref{CN:modification-process}.

\begin{prp}
	\label{PR:process-stops}
	With notation as above:
	\begin{enumerate}[label=(\roman*)]
		\item $E_r\cap E'_r=0$, $E_r\cap Z^1(X,\calF) =E_1\cap Z^1(X,\calF)$
		and $E_r\cap B^1(X,\calF) =E_1\cap B^1(X,\calF)$
		for every $r\in \N $.
		\item The process in Construction~\ref{CN:modification-process} stops, and the resulting subspace
	$E$ satisfies
	$h^0(\calF_E)>h^1(\calF_E)$.
	\end{enumerate}
\end{prp}

\begin{proof}
	(i) By construction, $E'_r\cap d_1^{-1}(B^2(X,\calC_r))=0$. 
	Since $d_1$ maps  $E_r\subseteq C^1(X,\calC_r)$
	into $B^2(X,\calC_r)$, this means that   
	$E'_r\cap E_r=0$. 
	
	Next, if $r>1$ and $f\in E_r\cap Z^1(X,\calF)=(E_{r-1}+E'_{r-1})\cap Z^1(X,\calF)$,
	then we can write $f=g+g'$ with $g\in E_{r-1}$ and $g'\in E'_{r-1}$.
	Since $d_1f=0$, we have $d_1g'=-d_1g\in B^2(X,\calC_{r-1})$,
	so $g'\in d_1^{-1}(B_2(X,\calC_{r-1}))\cap E'_{r-1}=0$. This means that $g'=0$,
	hence $f=g\in E_{r-1}\cap Z^1(X,\calF)$.
	By induction on $r$, it follows that $E_r\cap Z^1(X,\calF) =E_1\cap Z^1(X,\calF)$.
	Since $B^1(X,\calF)\subseteq Z^1(X,\calF)$,
	this means that $E_r\cap B^1(X,\calF) =E_1\cap B^1(X,\calF)$.

	(ii) Since $\dim E_r$ is bounded from above by $\dim C^1(X,\calF)$,
	in order to prove that the process stops, it
	is enough to show that   $\dim E_{r+1}>\dim E_r$ for all $r\geq 0$
	such that $\ker \omega_r \neq 0$.
	By (i), $\dim E_{r+1}=\dim E'_r+\dim E_r= \dim E_r+\dim \ker\omega_r$, hence our claim.
	
	Suppose now that the process stopped after $r$ steps.
	Then $E=E_r$ and $\ker\omega_r = \ker \omega_E= 0$.
	Write $V_1=E'_0\cap B^1(X,\calF)$ and
	choose a subspace $V_2$ such that $E'_0=V_1\oplus V_2$.
	By (i),  $\dim (E\cap B^1(X,\calF)) = \dim (E_1\cap B^1(X,\calF))=\dim V_1$
	and $\dim(E  \cap Z^1(X,\calF))= \dim(E_1\cap Z^1(X,\calF))=\dim V_2$.
	Now, applying Proposition~\ref{PR:effect-of-E} to $E=E_r$,
	we get 
	$h^0( \calF_E)-h^1( \calF_E)\geq  
	(h^0( \calF)+\dim V_1) - (h^1( \calF)- \dim V_2 +
	\dim\ker\omega_r)=\dim E'_0-(h^1(\calF)-h^0(\calF))>0$.
\end{proof}

Now that we know that iterative process of Construction~\ref{CN:modification-process}
outputs a subspace $E\subseteq C^1(X,\calF)$ such that $\calF_E$ satisfies
condition \ref{item:TH:tower-paradigm:rate} of Theorem~\ref{TH:tower-paradigm},
we turn to check whether $\calF_E$ also satisfies the local expansion conditions
in \ref{item:TH:tower-paradigm:LTC}.
This is {\it a priori} not true in general.
Indeed,   the process might stop only when 
$\calC_E(x)=\calF(x)$ for all $x\in X-X(0)-X(-1)$,
in which case $\calF_E$ will be isomorphic to the sheaf obtained from $\calF$
by setting $\calF_E(x)=\calF(x)\cong\F^m$ if $x\in X(0)$ and $\calF_E(x)=0$ otherwise.
The pair  $(X_v, (\calF_E)_v)$   is a poor coboundary expander for every $v\in X(0)$, so condition
\ref{item:TH:tower-paradigm:LTC} of the tower paradigm  
will not hold for $(X,\calF_E)$.
Nevertheless, we will now show that if $\dim E\ll \dim\calF$
and $X$ is covered by a sufficiently thick affine building, then,
with high probability, $\calF_E$ satisfies \ref{item:TH:tower-paradigm:LTC}.

To that end,
we would like to apply Theorem~\ref{TH:sheaves-on-quo-of-aff-buildings-quotients}.
Recall that in order to use this theorem, $E$ must satisfy the following to conditions:
\begin{enumerate}[label=(a\arabic*)]
	\item \label{item:F-E:linear-disjointness}
	For every $v\in X(0)$, the map $\sum_{e} \res_{e\from v}^{-1}:\bigoplus_e \calC_E(e)\to
	{\calF}(v)$, with $e$ ranging over $X(1)_{\supseteq v}$,
	is injective.
	\item \label{item:F-E:linear-depedence}
	For every $t\in X(2)$ with edges $e,e',e''$, we have $\calC_E(e)|_t\subseteq \calC_E(e')|_t+\calC_E(e'')|_t$.
\end{enumerate}
We would therefore like to choose the spaces   $E_0,E_1,\dots$ of Construction~\ref{CN:modification-process}
such that they all satisfy (a1) and (a2).

\begin{prp}\label{PR:securing-expansion-for-quotients}
	With notation as in Construction~\ref{CN:modification-process},
	suppose that $\F$ is a finite field and
	let $Q=D_{0,1}(X)=\max\{\#X(1)_{\supseteq v}\where v\in X(0)\}$. Then:
	\begin{enumerate}[label=(\roman*)]
		\item For all $r\in\N\cup\{0\}$, condition \ref{item:F-E:linear-depedence} holds for   $E_r$.

		\item For every $r\in \N\cup\{0\}$, 
		if condition \ref{item:F-E:linear-disjointness} 
		holds  for $E_r$ and 
		\[\dim E_{r+1}
		\leq \frac{\dim \calF - \log_{|\F|}|X(0)|}{Q}, \]
		then $E_{r+1}$
		can be chosen to satisfies
		\ref{item:F-E:linear-disjointness}.
		More  precisely,
		if $E'_r$ is chosen uniformly at random,
		then \ref{item:F-E:linear-disjointness}
		is satisfied
		with probability $>1-|X(0)||\F|^{Q \dim E_{r+1} -\dim \calF}$.

	\end{enumerate}
	The non-probabilistic 
	assertions of (ii) also hold if $\F$ is infinite upon replacing $\log_{|\F|}|X(0)|$ with $0$.
\end{prp}

\begin{proof}
	(i) We use induction on $r$.
	The case $r=0$ is clear, and the case $r=1$
	follows from Example~\ref{EX:cocycle-E-is-okay}.
	
	Suppose that $r>1$ and  we proved that \ref{item:F-E:linear-depedence}
	holds for $E_{r-1}$. 
	Let $t\in X(2)$ be a triangle with edges $e,e',e''$,
	and let $f_e\in \calC_r(e)$. 
	We need to show that $f_e|_t\in \res_{t\from e'}\calC_r(e')+
	\res_{t\from e''}\calC_r(e'')$.
	There is $f\in E_r$ such that $f_e=f(e)$ (recall that we have fixed
	a linear ordering on $X(0)$ and used it to identify
	$C^i(X,\calF)$ with $\prod_{x\in X(i)}\calF(x)$).	
	By construction, $E_r=E_{r-1}+E'_{r-1}$
	with $d_1(E'_{r-1})\subseteq C^1(X,\calC_{r-1})$, so
	$d_1 f\in C^1(X,\calC_{r-1})$.
	Thus, for some choice of signs, $f(e)|_t \in \pm f(e')|_t\pm 	f(e'')|_t+\calC_{r-1}(t)$.
	By the definition $\calC_{r-1}$
	and the induction hypothesis, $\calC_{r-1}(t)=
	\calC_{r-1}(e)|_t + \calC_{r-1}(e')|_t + \calC_{r-1}(e'')|_t \subseteq 
	\calC_{r-1}(e')|_t + \calC_{r-1}(e'')|_t$.
	It follows that $f(e)|_t\in \pm f(e')|_t\pm 	f(e'')|_t+
	\calC_{r-1}(e')|_t + \calC_{r-1}(e'')|_t
	\subseteq \calC_{r }(e')|_t + \calC_{r }(e'')|_t$, which is what we want.
	
	(ii) Suppose first that $r>0$.
	Fix a subspace $E' \subseteq C^1(X,\calF)$   such that
	$d_1E' \subseteq Z^2(X,\calC_{r  })$
	and the composition $E' \xrightarrow{d_1}  Z^2(X,\calC_{r })\to \HH^2(X,\calC_{r })$
	is injective with image $\ker\omega_r$ and
	let $h'_1,\dots,h'_t$ be a basis of $E' $
	(so $t=\dim\ker\omega_r$).
	In order to choose $E'_r$   uniformly at random in Construction~\ref{CN:modification-process},
	we can
	choose   $h_i\in h'_i+Z^1(X, \calF)$   uniformly at random
	for each $i\in\{1,\dots,t\}$ and take $E'_r=\F h_1+\dots+\F  h_t$.
	
	Fix some $v\in X(0)$ and $i\in \{1,\dots,t\}$.
	Abbreviate $N(v)=X(1)_{\supseteq v}$.
	We claim that the collection $\{(\res_{e\from v})^{-1}( h_i(e))\}_{e\in N(v)}$
	distributes uniformly in $ \calF(v)^{N(v)}$. 
	To see this, write $z_i = h_i-h'_i$. It
	is enough to show that $\{(\res_{e\from v})^{-1}( z_i(e))\}_{e\in N(v)}$
	distributes uniformly in $ \calF(v)^{N(v)}$,
	which, in turn, will follow if we show that
	the linear transformation $T:Z^1(X, \calF)\to \calF(v)^{N(v)}$ 
	given by $T(f)=(\res_{e\from v}^{-1} f(e))_{e\in N(v)}$ is onto. 
	Given $(f_e)_{e\in N(v)}\in  \calF(v)^{N(v)}$, define $g\in C^0(X,\calF)$
	by
	\[
	g(x) = \left\{\begin{array}{ll}
	-[x\cup v:x]\res_{x\cup v\from x}^{-1}\res_{x\cup v\from v}  f_{x\cup v}  & x\in X(1)_v\\
	0 & x\notin X(1)_v,
	\end{array}\right.
	\]
	where $[x\cup v:x]$ is $1$ if $x<v$ relative to the ordering on $V(X)$, and $-1$ otherwise.
	It is straightforward to check that $T(d_0 g)=(f_e)_{e\in N(v)}$, so $T$ is onto, and our claim follows.

	Since $h_1,\dots,h_t$
	are chosen independently, the previous paragraph implies that, for every $v\in X(0)$, 
	the collection
	$\{(\res_{y\from v})^{-1}( h_i(e))\}_{e\in N(v),i\in\{1,\dots,t\}}$
	distributes uniformly in $\calF(v)^{E(v)\times\{1,\dots,t\}}$.
	Let $R(v)=\sum_{e\in N(v)}\res_{e\from v}^{-1} \calC_{r }(e)$;
	it is a subspace of $\calF(v)$.
	Condition  \ref{item:F-E:linear-disjointness} for $E_r$ implies
	that $R(v)=\bigoplus_{e\in N(v)}\res_{e\from v}^{-1} \calC_{r }(e)$.
	This means that
	\[
	\dim R(v)=\sum_{e\in N(v)}\dim \calC_{r }(e)\leq \sum_{e\in N(v)}\dim E_r\leq Q\dim E_r.
	\]
	Now, setting $m=\dim \calF$, the probability that 
	$\{(\res_{e\from v})^{-1}( h_i(e))\}_{e\in N(v),i\in\{1,\dots,t\}}$
	are linearly independent and span a subspace of $\calF(v)$ meeting $R(v)$
	only at $0$ is  
	$\prod_{j=0}^{tQ -1}(1-|\F|^{j+\dim R(v)-m})> 1-|\F|^{Qt +Q\dim E_r-m}=
	1-|\F|^{Q\dim_{E_{r+1}}-m}$.
	Consequently, the
	probability 
	that for every $v\in X(0)$,
	the collection $\{(\res_{e\from v})^{-1}( h_i(e))\}_{e\in N(v) ,i\in\{1,\dots,t\}}$
	is linearly independent in $\calF(v)$ and its span meets $R(v)$ only at $0$ 
	is greater than
	$1-|X(0)||\F|^{Q\dim E_{r+1} -m}$. This number is non-negative by our assumption on $\dim E_{r+1}$,
	so we   may choose $h_1,\dots,h_t$
	to satisfy the last condition, which is easily seen to imply
	that \ref{item:F-E:linear-disjointness} 
	holds for $E_{r+1}$.
	
	The same argument we used for $r>0$ also works for   
	$r=0$.
	The only difference is that one starts with some subspace $E'$
	of $Z^1(X,\calF)$ of the same dimension as that of $E'_0$.

	Finally, when $\F$ is infinite, an adaptation of the argument
	shows that if we
	write $h_i=z_i+h'_i$ with $z_i\in Z^1(X,\calF)$,  
	then condition \ref{item:F-E:linear-disjointness} 
	is met if $z_1,\dots,z_t$
	are chosen outside of the zero locus of some nonzero multivariate polynomial
	on $Z^1(X,\calF)^t\cong \F^N$. Such a choice is possible because $\F$ is infinite.
\end{proof}

\begin{cor}\label{CR:securing-expansion-for-quotients}
	Fix $d\in\N-\{1\}$ and let $q=q(d)$
	be as in Theorem~\ref{TH:sheaves-on-quo-of-aff-buildings-quotients}.
	Let $X$ be a $d$-complex covered by a $q$-thick affine building,
	let $\F$ be a finite field, and let $\calF$ be a locally
	constant $\F$-sheaf on $X$ with $h^1(\calF)\geq h^0(\calF)$.
	Apply the iterative process
	of  Construction~\ref{CN:modification-process} to $(X,\calF)$ and
	suppose that it stopped after $n$ steps.
	Let $Q=D_{0,1}(X)$ and let
	$r$ denote the maximal member of $\{0,1,\dots,n\}$ such
	that 
	\[\dim E_r\leq Q^{-1}(\dim\calF-\log_{|\F|}\dim\calF -\log_{|\F|}|X(0)|).\]
	Then, with probability greater than $1-\frac{r}{\dim\calF}
	\geq 1-\frac{\dim E_r}{\dim\calF}\geq 1-\frac{1}{Q}$,
	the sheaf 
	$\calF_{E_r}$ satisfies condition \ref{item:TH:tower-paradigm:LTC}
	of the tower paradigm (i.e., Theorem~\ref{TH:tower-paradigm} with $k=0$).
\end{cor}

\begin{proof}
	Proposition~\ref{PR:securing-expansion-for-quotients}(ii)
	and our choice of $r$ imply that for all $0\leq s\leq r-1$,
	the probability that $E_{s+1}$ satisfies \ref{item:F-E:linear-disjointness}
	provided that $E_s$ satisfies it is greater than 
	$1-|X(0)||\F|^{Q\dim E_{s+1}-\dim \calF}\geq 1-|X(0)||\F|^{Q\dim E_r-\dim\calF}$.
	Our assumption on $\dim E_r$
	says that the latter quantity is at least
	$1-|X(0)||\F|^{-\log_{|\F|}\dim\calF -\log_{|\F|}|X(0)|}=
	1-\frac{1}{\dim\calF}$.
	Since $E_0$ satisfies \ref{item:F-E:linear-disjointness},
	this means that 
	the probability that $E_r$ satisfies \ref{item:F-E:linear-disjointness} is greater
	than $1-\frac{r}{\dim\calF}$. (We have $1-\frac{r}{\dim\calF}\geq
	1-\frac{\dim E_r}{\dim\calF}$ because $\dim E_r>\dim E_{r-1}
	>\dots>\dim E_0=0$, by Proposition~\ref{PR:process-stops}(i).)
	Proposition~\ref{PR:securing-expansion-for-quotients}(i) also
	tells us that $E_r$ satisfies \ref{item:F-E:linear-depedence}.
	Applying 
	Theorem~\ref{TH:sheaves-on-quo-of-aff-buildings-quotients}
	to the sheaf $\calF$ and the space $E_r$ completes the proof. 
\end{proof}

We conclude from Proposition~\ref{PR:process-stops}(ii)
and Corollary~\ref{CR:securing-expansion-for-quotients}
that if $X$ is covered
by a sufficiently thick affine building, and if 
$\calF$ is a locally constant $\F$-sheaf for which 
the iterative process of Construction~\ref{CN:modification-process}
outputs a subspace $E\subseteq Z^1(X,\calF)$
with $\dim E\ll \dim\calF$ with high probability,
then $(X,\calF_E)$ satisfies conditions \ref{item:TH:tower-paradigm:LTC}
and~\ref{item:TH:tower-paradigm:rate} with high probability.
Since $\dim E_1=h^1(\calF)-h^0(\calF)+1$,
in order to have $\dim E\ll \dim\calF$, we should start
with a sheaf $\calF$ such that $h^1(\calF)\ll \dim \calF$;
we will show that such exist in \S\ref{sec:candidates}.
Taking this for granted, we turn to discuss the growth of $\dim E_r$
in the iterative process.


\subsection{The Effect of The Cup Product on The Modification Process}
\label{subsec:growth}

As in \S\ref{subsec:quotient}, let $\calF$ be a locally
constant $\F$-sheaf on $X$ such that $h^1(\calF)\geq h^0(\calF)$.
We assume that $\F$ is finite, and apply the notation introduced in Construction~\ref{CN:modification-process}.

We conducted a number of computer simulations\footnote{
	The {\tt Python} code of the simulations was written
	by the first named author and is attached to the {\tt arXiv} version of the paper.
} where we applied Construction~\ref{CN:modification-process}
to a variety of complexes\footnote{
	So far, we checked   
different triangulations of a $3$-dimensional torus  and a $3$-thick
$2$-dimensional Ramanujan complex with $273$ vertices. The latter is
a quotient of the explicit example in \cite[\S10]{Lubotzky_2005_explicit_constructions_of_Ramanujan_complexes} by a Borel subgroup of $\nGL{\F_{16}}{3}$. 
} and sheaves. These simulations suggest that the typical behavior of the iterative process
can be predicted by means of the cup product (see \S\ref{subsec:cup-prod}), as we now explain.

Let $r\in \N$ and suppose that $E_r=E_{r-1}\oplus E'_{r-1}$ of 
Construction~\ref{CN:modification-process} has just been defined.
We can attempt to construct elements in the kernel of $\omega_r:\HH^2(X,\calC_r)\to \HH^2(X,\calF)$
as follows.
Let $f_1,\dots,f_t$ be an $\F$-basis for $E_{r-1}$ and let $f'_1,\dots,f'_s$ be
an $\F$-basis for $E'_{r-1}$.
Suppose that there are 
$\alpha_1,\dots,\alpha_t,\alpha'_1,\dots,\alpha'_s\in C^1(X,\F)$
such that
\begin{align}\label{EQ:cup-product-effect-general}
g:=\sum_{i=1}^t\alpha_i\cupp f_i + \sum_{j=1}^s \alpha'_j\cupp f'_j\in B^2(X,\calF).
\end{align}
Then $g\in B^2(X,\calF)\cap C^2(X,\calC_r)\subseteq Z^2(X,\calC_r)$,
which means that the cohomology class $[g]_{\calC_r}$ is in $\ker  \omega_r$.
Denote by $V_r$ the subspace of $C^1(X,\F)^{t+s}$ consisting of tuples
$(\alpha_1,\dots,\alpha_t,\alpha'_t,\dots,\alpha'_s)$ for which
\eqref{EQ:cup-product-effect-general} holds.

In general, not all elements of $V_r$ give rise to a nonzero class
in $\ker\omega_r$. This happens in particular when $\alpha'_1,\dots,\alpha'_s\in B^1(X,\F)$.
Indeed, in this case, there are $\beta'_1,\dots,\beta'_s\in C^0(X,\F)$
such that $\alpha'_j=d_0\beta'_j$ for all $j\in\{1,\dots,s\}$. By Proposition~\ref{PR:cup-product-basic-properties},
this means that
\begin{align*}
g &=
\sum_{i=1}^t\alpha_i\cupp f_i + \sum_{j=1}^s d_0\beta'_j\cupp f'_j  
= \sum_{i=1}^t\alpha_i \cupp f_i-\sum_{j=1}^s\beta'_j\cupp d_1f'_j +\sum_{j=1}^s d_1(\beta'_j\cupp f'_j).
\end{align*}
If $r=1$, then $t=0$ and $d_1 f'_j=0$ for all $j$,
so $g= \sum_{j=1}^s d_1(\beta'_j\cupp f'_j)\in B^2(X,\calC_1)$
and $[g]_{\calC_1}=0$. 
If $r>1$, then by the construction of $E'_r$, we have $d_1f'_j\in C^2(X,\calC_{r-1})$, so
$\tilde g:=g-\sum_{j=1}^s d_1(\beta'_j\cupp f'_j)\in C^2(X,\calC_{r-1})$.
Since $[\tilde{g}]_{\calF}=[g]_{\calF}=0$, we have $[\tilde{g}]_{\calF}\in\ker\omega_{r-1}$,
so,  by the construction of $E'_{r-1}$, there
is $h\in E'_{r-1}$ such that $d_1h=\tilde{g}$. Consequently,
$g=d_1h+ \sum_{j=1}^s d_1(\beta'_j\cupp f'_j)\in B^2(X,\calC_r)$,
and $[g]_{\calC_r}=0$.

Similarly, if $\alpha'_1,\dots,\alpha'_s$ are in the left radical
of the pairing $\cupp :C^1(X,\F)\times C^1(X,\calF)\to C^2(X,\calF)$,
denoted $L(\calF)$, then $[g]_{\calC_r}=0$.

Let $U_r$ denote the sum of $L(\calF)$
and the subspace of $V_r$ consisting of 
tuples
$(\alpha_1,\dots,\alpha_t,\alpha'_1,\dots,\alpha'_s)\in V_r$
with $\alpha'_1,\dots,\alpha'_s\in B^1(X,\F)$.
Our simulations suggest that if $\dim\calF$ is big
enough with respect to $\dim E_r$, then with high probability, 
\begin{enumerate}[label=(\roman*)]
\item all elements
in $\ker\omega_r$ are obtained from elements of $V_r$
as in \eqref{EQ:cup-product-effect-general}, and 
\item the elements of $V_r$
which give rise to the zero element in $\ker\omega_r$
are precisely the subspace $U_r$.
\end{enumerate}
Consequently, $\dim E'_r=\dim V_r-\dim U_r$.
Informally, this means that if $\dim E_r\ll \dim \calF$, then 
almost surely, relations coming from
the cup product  are the only explanation to elements in $\ker\omega_r$.

We now analyse heuristically  how big should $\dim\calF$
be with respect to $\dim E_r$ in order to make the above estimations
valid. Classes in $\ker\omega_r$ which are not explained
by the cup product may occur if $\dim \HH^2(X,\calC_r)>\dim\HH^2(X,\calF)$,
so we need to require that $\dim \HH^2(X,\calC_r)\leq \dim\HH^2(X,\calF)$.
The iterative process does not use information from faces
of dimension $\geq 3$, so we may assume that $\dim X=2$.
Now,
since \ref{item:F-E:linear-depedence}
holds for $E_r$ (Proposition~\ref{PR:securing-expansion-for-quotients}(i)),
we typically have $\dim\calC_{E_r}(x)=2\dim E_r$ for $x\in X(2)$,
so we expect that $\dim C^2(X,\calC_r)=2|X(2)|\dim E_r$.
The kernel of $d_1:C^1(X,\calC_r)\to C^2(X,\calC_r)$
contains $E_r\cap Z^1(X,\calF)=E_1$ (Proposition~\ref{PR:process-stops}(i)),
so if $\dim X=2$, then we have $\dim \HH^2(X,\calC_r)\geq 
\dim C^2(X,\calC_r)-\dim C^1(X,\calC_r)+\dim E_1=
(2|X(2)|-|X(1)|)\dim E_r+(h^1(\calF)-h^0(\calF)+1)$.
Our simulations suggest that equality   holds with high probability.
A similar computation shows that when $\dim X=2$,
we have $\dim\HH^2(X,\calF)\geq (|X(2)|-|X(1)|+|X(0)|-1)\dim\calF$.
The requirement 
$\dim \HH^2(X,\calC_r)\leq \dim\HH^2(X,\calF)$
is therefore likely to hold if
\begin{equation}\label{EQ:Er-dimension}
\dim E_{r }\leq \frac{(|X(2)|-|X(1)|+|X(0)|-1)\dim\calF - (h^1(\calF)-h^0(\calF)+1)}{
2|X(2)|-|X(1)|.}
\end{equation} 
The right hand side is roughly $\frac{1}{2}\dim\calF$ if $|X(2)|$
is large w.r.t.\ $|X(1)|$ and $|X(0)|$.


We summarize our observations  in the following conjecture,
in which we let $|\F|$ or $\dim\calF$  grow. It is supported
by all of our simulations.

\begin{cnj}\label{CJ:cup-product-explains-everything}
	With notation as in Construction~\ref{CN:modification-process},
	suppose that $E_r$ has just been constructed, and thus $\dim E'_{r}=
	\dim\ker\omega_r$
	is  determined. 
	Define $V_r$ and $U_r$ as above. Then:
	\begin{enumerate}[label=(\roman*)]
		\item If \eqref{EQ:Er-dimension} holds, then 
		$\dim E'_r=\dim V_r-\dim U_r$ with probability $1-o(1)$
		as a function of $|\F|$.
		\item If \eqref{EQ:Er-dimension} holds 
		and $M$ is the difference between the right hand
		side and the left hand side of \eqref{EQ:Er-dimension},
		then $\dim E'_r=\dim V_r-\dim U_r$
		with probability $1-o(1)$
		as a function of $M$.
	\end{enumerate}	  
	In particular, if $V_r=U_r$, then the iterative process will stop
	at the $r$-th step with probability $1-o(1)$ (in the sense of (i) or (ii)).
\end{cnj}

Informally, the conjecture means that  the ``most likely''
value  of $\dim E_r$ can be predicted purely by means of the cup product action
of $C^i(X,\F)$ on $C^j(X,\calF)$ for $i,j\in\{0,1\}$.
(We moreover expect that it is  determined by the homotopy type
of the differential graded module $C^*(X,\calF)$
over the  differential graded algebra $C^*(X,\F)$.)

\begin{example}\label{EX:analysis-of-dim-E2}
	Let us use Conjecture~\ref{CJ:cup-product-explains-everything} to
	predict what $\dim E_2$ will typically be.
	Recall that $E_1=E_0\oplus E'_0$ with $E_0=0$
	and $E'_0$  a subspace of $Z^1(X,\calF)$
	of dimension $s:=h^1(\calF)-h^0(\calF)-1$. (In fact, the analysis that
	we carry applies to any small subspace of $Z^1(X,\calF)$.)
	We assume that \eqref{EQ:Er-dimension}
	holds for $E_1$, or equivalently,
	that $h^1(\calF)-h^0(\calF)+1\leq \frac{
	|X(2)|-|X(1)|+|X(0)|-1}{
	2|X(2)|-|X(1)|+1}\dim\calF$.
	
	Let $f'_1,\dots,f'_s$ be an $\F$-basis for $E'_0$. Then $V_1$
	is the space of $(\alpha'_1,\dots,\alpha'_s)\subseteq C^1(X,\F)^s$
	such that
	\begin{align}\label{EQ:condition-for-E1}
	\sum_{j=1}^s\alpha'_j\cupp f'_j\in B^2(X,\calF),
	\end{align}
	and $U_1=V_1\cap (B^1(X,\F)+L(\calF))^s$. 
	Since the $f'_1,\dots,f'_s$
	live in $Z^1(X,\calF)$, and since the cup product of cocycles is a cocycle,
	it is reasonable to expect that, modulo $L(\calF)$, \eqref{EQ:condition-for-E1}
	will hold only if $\alpha'_1,\dots,\alpha'_s\in Z^1(X,\F)$; this heuristic
	is confirmed by our simulations.
	Replacing $V_1$ with $V_1\cap Z^1(X,\calF)^s$
	and $U_1$ with $B^1(X,\calF)^s$,
	so that $\alpha'_1,\dots,\alpha'_s\in Z^1(X,\F)$,
	condition \eqref{EQ:condition-for-E1} is equivalent to having
	\[
	\sum_{j=1}^s[\alpha'_j]\cupp [f'_j]_{\calF}=0
	\]
	in $\HH^2(X,\calF)$. Since $\dim V_1-\dim U_1$ equals
	the image of $V_1$ in $\HH^1(X,\F)^s$, it follows that
	\[
	\dim V_1-\dim U_1 =
	\ker ([\alpha]\otimes f\mapsto [\alpha\cupp f]:
	\HH^1(X,\F)\otimes_{\F} E'_0\to \HH^2(X,\calF)).
	\]
	This leads to   Conjecture~\ref{CJ:dimension-of-E2} below, which is again supported
	by our simulations.
	
	The conjectural formula for $\dim E'_1=\dim V_1-\dim U_1$ 
	also demonstrates how the choice of $E_1=E'_0$ might affect $\dim E_2$.
	For example, if $E'_0$ is taken to be a subspace of $B^1(X,\calF)$,
	then $[f]_{\calF}=0$ for every $f\in E'_0$, and we find that, heuristically,
	\[
	\dim E'_1=\dim V_1-\dim U_1=s\cdot h^1(X,\F).
	\]
	On the other hand, if $\dim E'_0$ is chosen such that $E'_0\oplus B^1(X,\calF)=Z^1(X,\calF)$,
	i.e., we are trying to eliminate \emph{all} the cohomology classes in $\HH^1(X,\calF)$
	by passing to $\calF_{E_1}$, then the map $E'_0\to\HH^1(X,\calF)$ is a bijection,
	and we get
	\[
	\dim E'_1=\ker ([\alpha]\otimes [f]\mapsto [\alpha\cupp f]:
	\HH^1(X,\F)\otimes_{\F} \HH^1(X,\calF)\to \HH^2(X,\calF)).
	\]
\end{example}

\begin{cnj}\label{CJ:dimension-of-E2}
	With notation as in Construction~\ref{CN:modification-process},
	suppose that $M:
	=\frac{
	|X(2)|-|X(1)|+|X(0)|-1}{
	2|X(2)|-|X(1)|+1}\dim\calF- (h^1(\calF)-h^0(\calF)+1)\geq 0$.
	Then 
	\[
	\dim E'_1 = \ker ([\alpha]\otimes f\mapsto [\alpha\cupp f]:
	\HH^1(X,\F)\otimes_{\F} E'_0\to \HH^2(X,\calF))
	\]
	with probability $1-o(1)$ as a function of $M$
	(resp.\ $|\F|$).
\end{cnj}

It possible to continue
the analysis of 
Example~\ref{EX:analysis-of-dim-E2}
in order to predict  the dimension of $\dim E_r$ for larger  values
of $r$. 
This is manageable for $r=3$,
or if one assumes that $h^1(X,\calF)=0$,
but
the general case    becomes intractable   very quickly. We omit the details.

If $X$ has significantly more $2$-cells
than $1$-cells, then dimension considerations suggest that 
the equations defining $V_r$
will be less and less likely to have nontrivial solutions as $r$ grows. 
Thus, it may be the case
that the iterative process of Construction~\ref{CN:modification-process}
stops after a fixed number of steps if $X$ is covered by a sufficiently thick affine building. We 
pose it as a conjecture, although we have
no computational evidence.

\begin{cnj}\label{CJ:process-for-buildings}
	There $q,d\in \N$  ($d\geq 2$) 
	and a function $f:\N \cup\{0\}\to \N$
	such that if $X$ is covered by a $q$-thick $d$-dimensional
	affine building, then, when the iterative
	process of Construction~\ref{CN:modification-process} stops, 
	we have $\dim E\leq f(h^1(\calF))$ with probability $1-o(1)$
	as $|\F|\to\infty$ or $\dim \calF\to\infty$.
\end{cnj}

We will see in the sequel how to find infinite families of sheaves on $d$-complexes
covered by $q$-thick buildings such that $h^1(\calF)=O(1)$ as a function of $\dim \calF$,
so any function $f:\N\cup\{0\}\to \N$ will do.

\subsection{Candidates for Infinite Families of Good $2$-Query LTCs}
\label{subsec:candidates}

We conclude this section with showing how a positive
answer to the conjectures
raised in \S\ref{subsec:growth} would lead to the existence
of an infinite family of good $2$-query LTCs.
To that end, we use the following theorem, that will be
proved in Section~\ref{sec:candidates}.

\begin{thm}\label{TH:sheaves-with-small-coh}
	For every $q,d\in \N$ with $d\geq 2$,
	there exists a $d$-complex $X$ that is covered by a $q$-thick affine building
	and a nonzero locally constant $\F_2$-sheaf   $\calG$
	on $X$ such that $h^0(\calG)=h^1(\calG)=0$.
	Moreover $X$ admits an infinite
	tower of double coverings $\dots\to X_2\to X_1\to X_0\to X$.
\end{thm}

Based on this, we show:

\begin{thm}\label{TH:main-candidates}
	Let $d\in\N-\{1\}$, let $q=q(d)$
	be as in Theorem~\ref{TH:sheaves-on-quo-of-aff-buildings-quotients},
	and let $X$ be a $d$-complex covered
	by a $q$-thick affine building as in Theorem~\ref{TH:sheaves-with-small-coh}.
	Let $\F$ be a finite field
	of characteristic $2$ and suppose that one of the following holds:
	\begin{enumerate}[label=(\arabic*)]
		\item Conjecture~\ref{CJ:dimension-of-E2} holds for $X$, and 
		there is a nonzero locally
		constant $\F$-sheaf
		$\calF_0$ on $X$
		and a subspace $E\subseteq \HH^1(X,\calF_0)$ 
		of dimension $  h^1(\calF_0)-h^0(\calF_0)+1$
		such that
		$\cupp:\HH^1(X,\F)\otimes_{\F}
		E\to \HH^2(X,\calF_0)$ is injective.
		\item Conjecture~\ref{CJ:process-for-buildings} is true 
		for  $X$.
	\end{enumerate}		
	Then there exists 
	an $\F$-sheaf $\calF$ on $X$ such 
	that, if we apply the iterative process
	of Construction~\ref{CN:modification-process}
	to $\calF$, then the pair $(X,\calF_E)$
	satisfies
	conditions 
	\ref{item:TH:tower-paradigm:tower},
	\ref{item:TH:tower-paradigm:LTC}
	and \ref{item:TH:tower-paradigm:rate}
	of the tower paradigm
	(Theorem~\ref{TH:tower-paradigm} with $k=0$) with probability $>0$.
	Consequently, there exists initial data for the tower
	paradigm, and as a result,
	an infinite family of 
	$2$-query LTCs with linear distance and constant rate.
\end{thm}

Note that $  h^1(\calF_0)-h^0(\calF_0)+1\geq 0$ by Proposition~\ref{PR:loc-constant-failure}.

Writing $\Gamma=\pi_1(X)$, the existence of $\calF_0$ in (1) is equivalent
to the existence of a nonzero representation  $\rho:\Gamma\to \nGL{\F_2}{m}$
and a subspace $E\subseteq \HH^1(\Gamma,\rho)$
of dimension $\HH^1(\Gamma,\rho)-\HH^0(\Gamma,\rho)+1$
such that  $
\cupp:\HH^1(\Gamma,\F_2)\otimes E\to\HH^2(\Gamma,\rho) $
is injective.
There  are representations $\rho$ of arbitrarily
large \emph{finite}
groups having this property, see the  {\tt MathOverflow} answer \cite{SashaP_2022_MO_answer}.

\begin{proof}
	Let $\calG$ be the sheaf promised
	by Theorem~\ref{TH:sheaves-with-small-coh}.
	We replace $\calG$ with its base-change from $\F_2$ to $\F$
	to assume that $\calG$ is an $\F$-sheaf, see Lemma~\ref{LM:field-ext}.
	
	Suppose that (1) holds. 
	For every $s\in \N$, put $\calF_s=\calF_0\times\calG^s$.
	By
	our assumptions on $\calG$, the natural map $\calF_0\cong \calF_0\times 0\to \calF_0\times\calG^s=\calF_s$
	induces maps
	$\HH^i(X,\calF_0)\to\HH^i(X,\calF_s)$
	which are bijective for 
	$i\in\{0,1\}$ and injections for
	$i\geq 2$. The map 
	$\HH^i(X,\calF_0)\to\HH^i(X,\calF_s)$ is compatible with
	the cup product, so, 
	writing $V_s$ for the image of $E$ in $\HH^1(X,\calF_s)$,
	the map
	$\cupp :\HH^1(X,\F)\otimes_{\F} V_s\to \HH^2(X,\calF_s)$ is injective.
	
	Let $E_1$ be a subspace of $Z^1(X,\calF_s)$ of dimension
	$h^1(\calF_s)-h^0(\calF_s)+1=\dim V_s$, chosen uniformly
	at random. As $s$ grows, the probability that
	the image of $E$ in $\HH^1(X,\calF_s)$ is $V_s$
	approaches some $p>0$.
	Consequently, the probability that 
	$\cupp:\HH^1(X,\F)\otimes_\F E_1\to \HH^2(X,\calF_s)$
	is injective is bounded from below by some $p'>0$.
	It now follows from Conjecture~\ref{CJ:dimension-of-E2}
	that for all $s$ large enough, the iterative process
	of Construction~\ref{CN:modification-process}
	stops for $\calF_s$ after $1$ step with probability $p'>0$.
	When this happens, the  output of the process is the subspace
	$E_1$, so $\dim E=h^1(\calF_0)-h^0(\calF_0)+1$
	is independent of $s$. Thus, 
	by Corollary~\ref{CR:securing-expansion-for-quotients},
	for all $s$ large enough, 
	$(X,(\calF_s)_E)$ satisfies \ref{item:TH:tower-paradigm:LTC}
	with probability $p''>0$.
	By Proposition~\ref{PR:process-stops}(ii),
	$(X,(\calF_s)_E)$ satisfies \ref{item:TH:tower-paradigm:rate},
	and   \ref{item:TH:tower-paradigm:tower} holds by 
	the choice of  $X$ in Theorem~\ref{TH:sheaves-with-small-coh}.
	To conclude, we can take $\calF=\calF_s$ for any $s$ large enough.

	The case where (2) holds is handled similarly but with the following differences:
	One can start
	with any locally constant sheaf $\calF_0$
	on $X$, e.g., the zero sheaf,
	and one uses 
	Conjecture~\ref{CJ:process-for-buildings} to bound
	$\dim E$ from above by $f(h^1(\calF_0))$.
\end{proof}

\begin{remark}\label{RM:main-candidates}
	Suppose that in Theorem~\ref{TH:main-candidates}   we take
	$E$ be $E_r$ of Corollary~\ref{CR:securing-expansion-for-quotients}
	instead of   the output of Construction~\ref{CN:modification-process}
	(i.e., we terminate the iterative process of the construction when $\dim E_r$
	is small w.r.t.\ to $\dim \calF$).
	The same argument as in the proof of the theorem then
	shows that for all $s$ large enough there exists
	a subspace
	$E\subseteq Z^1(X,\calF_s)$
	such that 
	$(X,(\calF_s)_E)$ satisfies
	\ref{item:TH:tower-paradigm:tower}
	and~\ref{item:TH:tower-paradigm:LTC}  
	unconditionally, and also  
	\ref{item:TH:tower-paradigm:rate}
	provided that Conjecture~\ref{CJ:process-for-buildings} holds.
\end{remark}

\section{Arithmetic Groups and Simplicial Complexes Covered by Affine Buildings}
\label{sec:congruence}

The purpose of this section is to prove the following theorem, which will be used in
the next section to prove Theorem~\ref{TH:sheaves-with-small-coh} and for a few other purposes.

\begin{thm}\label{TH:good-quotient}
	Let $q,d\in \N$ and assume that $d\geq 3$.
	There exist  a (finite) 
	simplicial complex $X$ covered by   a $q$-think $d$-dimensional affine building,
	a tower of (proper) connected coverings
	$\cdots \to X'_2\to X'_1\to X'_0=X$  and
	a constant  $C \in\R_+$  such that the following hold:
	\begin{enumerate}[label=(\roman*)]
		\item Every connected covering of $X$ admits an infinite tower of connected double coverings.
		\item  $[X'_r:X]$ is odd
		and $\dim \HH^1(X'_r,\F_2)\leq C$ for all $r\in \N\cup\{0\}$.
	\end{enumerate}
	Here, $[X'_r:X]$ denotes the degree of $X'_r\to X$.
\end{thm}

The constructions we provide are particular simplicial complexes, which are described
in \S\ref{subsec:construction-of-quotients}. 
The theorem is also true for $d=2$, but we omit most of the details;
see Remark~\ref{RM:good-quotient-in-dim-2}.

The proof of Theorem~\ref{TH:good-quotient} will make extensive use    of
arithmetic groups --- particularly congruence subgroups --- and their actions on affine buildings.
While we   briefly recall the   definitions and facts that we need in \S\ref{subsec:arith-grps-prelim},  
some knowledge of algebraic number theory is nevertheless assumed.
Familiarity with linear  algebraic groups
and affine group schemes is also recommended. 
We refer the reader to \cite{Platonov_1994_algebraic_groups_and_number_theory} for
further  details and an extensive discussion of these subjects.
A gentle introduction to group schemes is  
\cite[Chapter~1]{Waterhouse_1979_intro_affine_group_schemes}.

Readers who wish to skip the proof of Theorem~\ref{TH:good-quotient} should proceed
to Section~\ref{sec:candidates}.

\subsection{Preliminaries}
\label{subsec:arith-grps-prelim}

We begin with recalling necessary facts about algebraic groups,  setting
notation along the way.

\medskip

Let $R$ be any commutative ring,
let $\Comm{R}$ denote the category of commutative $R$-algebras
and let $\catGrp$ denote the category of groups. By a \emph{group scheme}\footnote{
	In this paper, all group schemes are affine and of finite presentation.
} over $R$
we mean a functor $G$ from $\Comm{R}$ to $\catGrp$ for which
there is a set of multivariate polynomials $f_1,\dots,f_t\in R[x_1,x_2,\dots,x_n]$ 
such that, for every $S\in\Comm{R}$,
the set $G(S)$ is in a \emph{natural} bijection with    the solutions of the equations
$f_1=\dots=f_t=0$ in $S^n$. The actual polynomials $f_1,\dots,f_t$
will rarely matter, and we would only care that   they exist. 
When $R$ is a field, group schemes over $R$ are also called (linear)
\emph{algebraic groups} over $R$.

We will only need the following  examples of group schemes.

\begin{example}\label{EX:group-schemes}
	(i) The functor $S\mapsto \nSL{S}{m}:\Comm{R}\to \catGrp$
is a group scheme, denoted $\uSL_{m}(R)$. (Formally,
$(\uSL_m(R))(S)=\nSL{S}{m}$ for all $S\in \Comm{R}$.) 
	Indeed, the functoriality
is clear, and $\nSL{S}{n}$ can be naturally identified with the zeroes
of the polynomial $\det(x_{ij})-1$ in $m^2$ indeterminates. 

	(ii) The functor $\nGm{R} :\Comm{R}\to \catGrp$ sending an $R$-algebra
	$S$ to its group of invertible elements $\units{S}$ is a group scheme.
	To see this, note that the map $s\mapsto (s,s^{-1}):\units{S}\to S^2$
	identifies $\units{S}$ with the solutions of the equation $x_1x_2-1=0$
	in $S^2$ for every $S\in\Comm{R}$.

	
	(iii) Given a commutative ring $R$,   put $R[i]=R[i\where i^2=-1]$. Elements
	of $R[i]$ are formal sums $\alpha+\beta i $ with $\alpha,\beta\in R$, and the product
	in $R[i]$ is determined by the rule $i^2=-1$.
	Write $\sigma_R:R[i]\to R[i]$ for the automorphism 
	sending $\alpha+\beta i$ to $\alpha-\beta i$. For example, if $R=\R$,
	then $R[i]$ is just $\C$ and $\sigma_\R$ is complex conjugation. 
	Given a matrix $a=(\alpha_{ij})_{i,j}\in \nMat{R[i]}{n}$, we write $a^*$
	for the matrix $(\sigma_R(\alpha_{ji}))_{i,j}$.	
	
	Define  $\nSU{R[i]/R}{m}=\{a\in\nMat{R[i]}{m}\suchthat a^*a=1~\text{and}~ \det(a)=1_{R[i]}\}$;
	it is a subgroup of $\nGL{R[i]}{m}$. 
	The functor $\uSU_{m}(R[i]/R):\Comm{R}\to \catGrp$ sending a  commutative $R$-algebra
	$S$ to $\nSU{S[i]/S}{m}$ is a group scheme.
	Indeed, we can   identify $\nMat{S[i]}{m}$ with $S^{2m^2}$ by sending
	a matrix $a = (\alpha_{j\ell}+i\beta_{j\ell})_{j,\ell}$ to the vector
	$(\alpha_{11},\alpha_{12},\dots,\alpha_{mm},\beta_{11},\beta_{12},\dots,\beta_{mm})\in S^{2m^2}$.
	The condition  $\det(a)=1_{S[i]}$ can now be rewritten as \emph{two} polynomial equations 
	with coefficients coming from $R$ (or even $\Z$), and the condition $a^*a=1$ can be rewritten
	as $2m^2$ polynomial equations.
	
	(iv) We can generalize (iii) by   fixing   $r_0,r_1\in R$
	and replacing $R[i]$ with $\hat{R}:= R[x\where x^2=r_1x+r_0]$;
	the $R$-automorphism 	
	$\sigma_R:\hat R\to \hat R$ then sends $x$ to $r_1-x$.
	Moreover, instead of considering matrices $a\in \nMat{\hat{S}}{m}$
	with $a^*a=1$, we could fix a   matrix $M\in \nGL{\hat{R}}{m}$
	with $M^*=M$ and consider the group of matrices $a\in \nMat{S'}{m}$
	satisfying $a^* Ma=M$ and $\det(a)=1$. This group is denoted $\mathrm{SU}(f_S)$,
	where $f_S:{\hat S}^m\times {\hat S}^m\to {\hat S}$
	is the $\sigma_S$-hermitian form corresponding to $M$, i.e., $f_S(x,y)=x^* M y$
	for $x,y\in {\hat S}^n$ (regarded as column vectors).
	The functor $S\mapsto \mathrm{SU}(f_S):\Comm{R}\to \catGrp$
	is a group scheme denoted $\uSU(f)$. 
\end{example}

Suppose that $G$ and $H$ are group schemes over $R$.
A morphism from $G$ to $H$  is a natural
transformation $f$ from $G$ to $H$. In particular, the data
of $f$ consists of  
a group homomorphism $f_S:G(S)\to H(S)$ for every $S\in \Comm{R}$.
We say that $f$ is a monomorphism if $f_S:G(S)\to H(S)$
is one-to-one for all $S\in\Comm{R}$.

If $R'\in \Comm{R}$, then 
every $R'$ may be regarded as an $R$-algebra.
This defines a functor $\Comm{R'}\to \Comm{R}$,
and its composition with  $G:\Comm{R}\to\catGrp$
is a group scheme over $R'$,  denoted $G_{R'}$. (The polynomial
equations defining $G_{R'}$ are the same as those defining $G$, but we think of
the coefficients as living in $R'$ instead of $R$.)

Let  $I$ be an ideal of $R$ (written $I\idealof R$). Then
$R/I$ is an $R$-algebra, and thus the quotient map $q:R\to R/I$ gives
rise to a group homomorphism $Gq:G(R)\to G(R/I)$. We define
\[
G(R;I)=\ker (G(R)\to G(R/I))
\]
and call $G(R;I)$ a \emph{principal congruence subgroup}
of $G$ (or $G(R)$, if $G$ is clear from the context). 
A subgroup of $G(R)$ containing a principal congruence subgroup  is called a \emph{congruence subgroup}
of $G$ (or $G(R)$).\footnote{
	This definition of congruence subgroups is different from the one used
	in \cite{Platonov_1994_algebraic_groups_and_number_theory}.
	The definitions are nevertheless 
	equivalent by \cite{First_2021_congruence_subgroups_via_group_schemes_preprint}.
}

\begin{example}
	Taking $R=\Z$, $I=\ell\Z$ and $G=\uSL_{n}(\Z)$, the group $\nSL{\Z;I}{m}:=G(\Z;I)$
	is just the group of $m\times m$ integral matrices which have determinant $1$
	and are congruent to the identity matrix modulo $\ell$.
\end{example}

The group scheme $G$ is called \emph{absolutely almost simple (and) simply connected}
if there is a faithfully flat commutative
$R$-algebra $R'$ such that, up to isomorphism,
$G_{R'}$ is in the list of \emph{split} absolutely almost simple
simply connected   group schemes  over $R'$ (also 
called the absolutely almost simple simply connected \emph{Chevalley groups} over $R'$).
When $R$ is a domain, this list  consists
of   $4$ infinite families and $5$ exceptional groups,
denoted $A_m$ ($m\geq 1$), $B_m$ ($m\geq 2$), $C_m$ ($m\geq 3$), $D_m$ ($m\geq 4$), $E_6$, $E_7$, $E_8$, $F_4$, $G_2$. 
The description of these
group schemes will not matter to us except for the fact that $A_m$ is the group
scheme $\uSL_{m+1}(R')$.\footnote{
	The groups schemes
	$B_m$, $C_m$, $D_m$ are also not
	difficult to describe and are $ \mathbf{Spin}_{2m+1 }(R')$, $ \uSp_{2m }(R')$ and  $ \mathbf{Spin}_{2m }(R')$, respectively. 
}
The \emph{type} of 
$G$ is the symbol ($A_m$, $B_m$, $C_m$, $D_m$, $E_6$, $E_7$, $E_8$, $F_4$ or $G_2$)
used to denote $G_{R'}$. For example,
$G$ is absolutely almost simple   simply connected  of  
type $A_m$ if $G_{R'}\cong \uSL_{m+1}(R')$ for some faithfully
flat $R'\in \Comm{R}$.
When $R$ is a field, the   absolutely almost simple simply connected group
schemes of a given type further break into  two kinds: \emph{inner} and \emph{outer}.

Suppose that $R$ is a field $K$ and $G$ is absolutely almost simple  simply connected.
The group scheme $G$ is called \emph{isotropic} if there is a monomorphism
from $\nGm{K}$ to $G$, and \emph{anisotropic} otherwise. The largest
$r\in \N\cup\{0\}$ for which there is a monomorphism $f:(\nGm{K})^r\to G$
is called the \emph{rank} of $G$, and denoted $\rank  G$.
More generally, given a field extension $L$ of $K$, we say that $G$ is $L$-anisotropic
(resp.\ $L$-isotropic) if $G_L$ is anisotropic (resp.\ isotropic),
and define the $L$-rank of $G$ as $\rank_LG:=\rank (G_L)$.

\begin{example}\label{EX:simply-conn-groups}
	(i) For all $m>1$, the group $\uSL_{m }(R)$ is absolutely almost simple  simply connected 
	of type $A_{m-1}$.
	If $R$ is a field $K$, then $\uSL_{m}(K)$ is isotropic of rank $m-1$. A monomorphism
	$f:(\nGm{K})^{m-1}\to \uSL_{m}(K)$ is given by $f_S(s_1,\dots,s_{m-1})=
	\mathrm{diag}(s_1,\dots,s_{m-1},s_1^{-1}\cdots s_{m-1}^{-1})$ for all $S\in \Comm{K}$.

	
	(ii) 
	Suppose that $2\in\units{R}$.
	The group scheme $\uSU_m(R[i]/R)$ of Example~\ref{EX:group-schemes}(iv) 
	is 
	absolutely almost simple  simply connected of  type $A_{m-1}$.
	To see this, suppose first that $R$  
	contains an element $\veps\in R$ with $\veps^2=-1$.
	Using this element, one can define an isomorphism of $R$-algebras
	$\alpha+\beta i\mapsto (\alpha+\veps\beta,\alpha-\veps\beta):R[i]\to R\times R$ 
	(this is bijective because $2\in\units{R}$). 	
	Under this isomrophism,
	$\sigma_R$ corresponds to $\sigma'_R:R\times R\to R\times R$ given by $\sigma'_R(x,y)=(y,x)$.
	Now, a routine computation shows that   the induced 
	$R$-algebra isomorphism $\nMat{R[i]}{m}\to \nMat{R\times R}{m}\cong \nMat{R}{m}\times \nMat{R}{m}$
	maps $\nSU{R[i]/R}{m}$ to the pairs of matrices $(a,b)\in \nMat{R}{m}\times \nMat{R}{m}$
	with $\det(a)=\det(b)=1_R$ and $ab=1$, namely, onto $\{(a,a^{-1})\where a\in \nSL{R}{m}\}$.
	Since a similar computation applies over any commutative $R$-algebra,
	we have constructed an isomorphism from $G:=\uSU_m(R[i]/R)$ to $\uSL_{m}(R)$.
	If $R$ does not contain a root of $-1$, then we can simply adjoin one, setting $R'=R[i]$,
	and get that $G_{R'}=\uSU_m(R'[i]/R')\cong \uSL_{m}(R')$.	
	
	When $R$ is a field, the algebraic group
	$\uSU_m(R[i]/R)$ is inner if $R$ contains a square root of $-1$ and outer otherwise.

	(iii) The group scheme $\uSU(f)$ of Example~\ref{EX:group-schemes}(v) 
	is 
	absolutely almost simple  simply connected of type $A_{m-1}$ 
	if $r_1^2+4r_0 \in\units{R}$.  
	Assuming this and that   $R$ is a field $K$, it is inner if and only if $x^2-r_1x-r_0$
	has a root in $K$.
\end{example}

In the remainder of this section we    will use the following general notation:

\begin{notation}\label{NT:general-num-thy-notation}
\phantom{something}
\begin{itemize}
	\item $K$ is a global field, e.g.,\ $\Q$ or $\F_p(t)$.
	\item $\calV$ is the set of places of $K$ and $\calV_{\infty}$ is the subset of 
	archimedean places.
	\item $K_\rho$ is the completion of $K$ at $\rho\in\calV$. 
\end{itemize}
If $\rho\in\calV$ is a non-archimedean place, we also use $\rho$ to denote the
corresponding additive valuation $\rho:K_\rho \to \Z\cup\{\infty\}$
and set 
\begin{itemize}
	\item $\calO_\rho=\{x\in K_\nu\suchthat \nu(x)\geq 0\}$,
	\item $\frakm_\rho =\{x\in K_\nu\suchthat \nu(x)> 0\}$ (the maximal ideal of $\calO_\rho$),
	\item $k(\rho)=\calO_\rho/\frakm_\rho$ (the residue field at $\rho$),
	\item $P_\rho=\calO\cap \frakm_\rho=\{x\in\calO\suchthat \rho(x)>0\}$ (the prime ideal of $\calO$
	corresponding to $\rho$).
\end{itemize}	
We further fix the following data:
\begin{itemize}
	\item $S$ is a nonempty subset of $\calV$ containing $\calV_\infty$.
	\item $\calO=\calO^S$ is the ring of $S$-integers in $K$, namely,
	$\{x\in K\suchthat \text{$\rho(x)\geq 0$ for all $\rho\in \calV-S$}\} 
	$.
	The fraction field of $\calO$ is $K$.
	\item $\nu$ is a fixed non-archimedean place in $S$.
	\item $\bfG$ is a simply connected absolutely almost simple algebraic group over $K$.
	\item $\calG$ is a group scheme over $\calO$ such that $\calG_K=\bfG$.
	\item $G=\bfG(K_\nu)=\calG(K_\nu)$. 
	\item $Y$ is the affine building attached to $\bfG_{K_\nu}$.
	Its dimension is $\rank_{K_\nu} \bfG$ \cite{Tits_1979_reductive_groups_over_local_fields}.
\end{itemize}
Given an ideal $I\idealof \calO$ and $\rho\in \calV-S$, we let 
\begin{itemize}
	\item $I_\rho = I\cdot\calO_\rho$.
	\item $\rho(I)=\min\{\nu(x)\where x\in I\}$; if $I\neq 0$,
	then
	this is also the unique   $n\in\N\cup \{0\}$ such that $I_\rho=\frakm_\rho^n$.
\end{itemize}
\end{notation}

For every $\rho\in \calV$, the group $\bfG(K_\rho)$
inherits a topology from $K_\rho$, and if $\rho\notin S$, then
$\calG(\calO_\rho)$ is a compact   open
subgroup of $\bfG(K_\rho)$.\footnote{
	Indeed,  $\bfG(K_\rho)$ (resp.\ $\calG(\calO_\rho)$) may be understood as the solution  set of some polynomial
	equations $f_1=\dots=f_r=0$ in $K_\rho^n$ (resp.\ $\calO_\rho^n$). 
	We give $\bfG(K_\rho)$ (resp.\ $\calG(\calO_\rho)$) the topology induced from $K_\rho^n$ ($\calO_\rho^n$).
	This is independent of how $\bfG$ (resp.\ $\calG$) 
	is realized as the solutions of polynomial equations.
} 
A theorem of Bruhat, Tits and Rousseau (see \cite{Prasad_1982_elementary_proof_BTR}, for instance)
states that $\bfG$ is $K_\rho$-anisotropic if and only if $\bfG(K_\rho)$ is compact.
We shall also need the following facts.

\begin{prp}\label{PR:density-in-arith-grp}
	With notation as in Notation~\ref{NT:general-num-thy-notation},
	if there is $\theta\in S$ such that $\bfG$ is $K_\theta$-isotropic, 
	then   $\calG(\calO;I)$ is dense 
	$\prod_{\rho\in\calV-S} \calG(\calO_\rho;I_\rho)$.	
\end{prp}

\begin{proof}
	Let $\bbA^S=\prod'_{\rho\in \calV-S}K_\rho$ denote the ad\'el\`es away from $S$.
	We embed $K$ diagonally in   $\bbA^S$.
%
	By the Strong Approximation Theorem (\cite{Prasad_1977_strong_approximation},
	\cite{Margulis_1977_strong_apporximation}), $\bfG(K)$ is dense $\bfG(\bbA^S)$.
	Since $U:=\prod_{\rho\in\calV-S} \calG(\calO_\rho;I_\rho)$ is open in $\bfG(\bbA^S)$,
	the set   $\bfG(K)\cap U$ is dense in $U$.
	As $K\cap \prod_{\rho\in\calV-S} I_\rho=I$ inside $\bbA^S$, 
	we have $\bfG(K)\cap U=\calG(\calO;I)$,
	and the proposition follows. 
%
\end{proof}

\begin{prp}\label{PR:elementary-p-group}
	With notation as in Notation~\ref{NT:general-num-thy-notation}, there is $D\in\N$
	such that the following hold: Let
	$R$ be a commutative $\calO$-algebra with  trivial $\calO$-torsion 
	and let $I\subseteq J$ be ideals of $R$.
	Then $\calG(R;I)/\calG(R;IJ)$ is isomorphic to a subgroup of
	the additive group $(I/IJ)^D$. 
\end{prp}

\begin{proof}
	We view both $R$ and $K$ as a subrings of $L:=R\otimes_\calO K$.
	By  \cite{First_2021_congruence_subgroups_via_group_schemes_preprint},
	there is a monomorphism of algebraic groups $f:\bfG\to \uSL_{m}(K)$
	such that $f(\calG(R;N))=f(\bfG(L))\cap \nSL{R;N}{m}$
	for all $N\idealof R$. Thus, $f$ induces
	a  one-to-one group homomorphism $\quo{f}:\calG(R;I)/\calG(R;IJ)\to 
	\nSL{R;I}{m}/ \nSL{R;IJ}{m}$.
	Since $I\subseteq J$, we can define a map  
	$g:\nSL{R;I}{m}/ \nSL{R;IJ}{m}\to \nMat{I/IJ}{m}$ by
	sending a matrix $x\in 1+\nMat{I }{m} $ to the image of $x-1$ in 
	$\nMat{I/IJ}{m}$. It is straightforward to check that
	$g$ is a one-to-one group homomorphism (use the fact that $I^2\subseteq IJ$).
	Taking $D=m^2$,
	the lemma follows.
\end{proof}

\subsection{Finite Quotients of Buildings}
\label{subsec:quotients-of-buildings}

Keeping Notation~\ref{NT:general-num-thy-notation},
recall that $G=\bfG(K_\nu)$ acts on $Y$ via simplicial automorphisms. 
In particular, 
for any $I\idealof \calO$, the
principal congruence subgroup $\Gamma=\calG(\calO;I)$ also acts on $Y$.
In this subsection we will be concerned with determining when is $\Gamma\leftmod Y$
a \emph{finite} simplicial complex covered by $Y$, and, provided this is so, when does it admit
an infinite tower of connected double coverings.

Beware that in general $\Gamma\leftmod Y$ is only a partially ordered set  
relative to the face-inclusion ordering it inherits from   $Y$. When $\Gamma\leftmod Y$
is isomorphic to a simplicial complex as a partially ordered set, 
we will say that $\Gamma\leftmod Y$
is a simplicial complex and treat it as one for all purposes. However, even when
$\Gamma\leftmod Y$ is a simplicial complex, the quotient
map $Y\to \Gamma\leftmod Y$ may not be a covering map.

\begin{prp}\label{PR:finite-quotients}
	With notation as in Notation~\ref{NT:general-num-thy-notation}, 
	suppose that $ {\bfG}$ is $K$-anisotropic,  $K_\rho$-anisotropic for every
	$\rho\in S-\{\nu\}$, and $K_\nu$-isotropic. Then:
	\begin{enumerate}[label=(\roman*)]
		\item $ \calG(\calO)$ is a discrete subgroup of $ G  $
		and $\calG(\calO)\leftmod G$ is compact.
		\item There is a finite subset $U\subseteq \calG(\calO)-\{1_G\}$
		such that if $I\idealof \calO$, $\calG(\calO;I)\cap U=\emptyset$
		and $\Gamma$ is a finite-index subgroup of $\calG(\calO;I)$, then
		$\Gamma$ acts freely on $Y$, the quotient $\Gamma\leftmod Y$
		is a finite simplicial complex,
		$Y\to \Gamma\leftmod Y$
		is a covering map, and $\pi_1(\Gamma\leftmod Y)\cong \Gamma$.
	\end{enumerate}
\end{prp}


\begin{proof}
	(i) 
	As in the proof of Proposition~\ref{PR:density-in-arith-grp},	
	let $\bbA$ denote the ad\'el\`es ring of $K$,
	let $\bbA^S$ be the ad\'el\`es away from $S$ and set $\bbA_S=\prod_{\rho\in S}K_\rho$.
	Since $ \bfG$
	is $K$-anisotropic, 
	the quotient $ {\bfG}(K)\leftmod  {\bfG}(\bbA)$ is compact, see 
	\cite[Theorem~5.5]{Platonov_1994_algebraic_groups_and_number_theory}
	and
	\cite[Corollary~2.2.7]{Harder_1969_minkowskische_reduktionstheorie}.
	Embedding $K$ diagonally in $\bbA^S\times \bbA_S$,
	we have $\calO=K\cap ([\prod_{\rho\notin S}\calO_\rho]\times \bbA_S)$.
	Thus, $\calG(\calO)=\bfG(K)\cap (U\times  {\bfG}(\bbA_S))$,
	where $U= \prod_{\rho\notin S}\calG(\calO_\rho) $ and the intersection is taken in $ {\bfG}(\bbA)$.
	It is now routine to check that the map 
	$\calG(\calO)\leftmod  {\bfG}(\bbA_S)\to  {\bfG}(K)\leftmod  {\bfG}(\bbA)/(U\times \{1_{ {\bfG}(\bbA_S)}\})$ given
	by $\calG(\calO)x \mapsto  {\bfG}(k)(1_{ {\bfG}(\bbA^S)},x)(U\times \{1_{ {\bfG}(\bbA_S)}\})$ is injective.
	By the Strong Approximation Theorem (\cite{Prasad_1977_strong_approximation},
	\cite{Margulis_1977_strong_apporximation}),
	this map is  also onto
	(here we need $\bfG$ to be simply-connected and $K_\nu$-isotropic).
	As it is open as well,
	it is a homeomorphism and we conclude that $\calG(\calO)\leftmod  {\bfG}(\bbA_S)$
	is compact. Since $\calG(\calO)\leftmod G $ the image of $\calG(\calO)\leftmod  {\bfG}(\bbA_S)$
	under a continuous map, $\calG(\calO)\leftmod G$ is also compact.
	Finally, note that $\calO $ is discrete in $\bbA_S$,
	and therefore $\calG(\calO )$ is a discrete subgroup of $ {\bfG}(\bbA_S)=\prod_{\rho\in S} {\bfG}(k_\rho)$.
	Since $ {\bfG}(k_\rho)$ is compact for all $\rho\in S-\{\nu\}$,
	the image of $ \calG(\calO )$ in $G=\bfG(K_\rho)$ is also discrete.
	
	(ii) 
	The building $Y$ attached to  $\bfG_{K_\nu}$ is constructed so that  
	the stabilizer of every nonempty face in $Y$  is compact and open in $G$.
	Let $y_1,\dots,y_t$ be representatives for the $G$-orbits in $Y$ and let $K_i=\{g\in G\suchthat gy_i=y_i\}$.
	Then the $G$-set $\bigsqcup_{i=1}^t G/K_i$ can be identified with $Y$ by mapping $gK_i$ to $gy_i$ 
	for all $g\in G$ and $i\in\{1,\dots,t\}$.
	Consequently, $\calG(\calO)\leftmod Y$ is in bijection with $\bigsqcup_{i=1}^t\calG(\calO)\leftmod G/K_i$,
	which is compact by (i).
	Since $\bigsqcup_{i=1}^t\calG(\calO)\leftmod G/K_i$ is also discrete (because each $K_i$ is open in $G$),
	it must be finite, and it follows that $\calG(\calO)\leftmod Y$ is finite.	
	Applying \cite[Corollaries 3.11, 3.12]{First_2016_Ramanujan_propery_preprint_v3} 
	(here we need the fact the $G(\calO)$ is discrete in $G(K_\nu)$)
	now gives the
	set $U$ and the desired
	conclusions. 
\end{proof}

\begin{cor}\label{CR:finite-quotients}
	Suppose that the assumptions of Proposition~\ref{PR:finite-quotients}
	hold. 
	Let $\{I_m\}_{m\in \N}$ be a decreasing
	sequence  of ideals of $\calO$ such
	that $\bigcap_{m\in\N} I_m=\{0\}$,
	let $\rho\in\calV-S$ and let $p=\Char k(\rho)$.
	Then:
	\begin{enumerate}[label=(\roman*)]
		\item There exists $m_0\in\N$ such that for every finite-index subgroup $\Gamma$ of $\calG(\calO;I_{m_0})$, 
		the action of
		$\Gamma$ on $Y$ is free, the quotient $\Gamma\leftmod Y$
		is a finite simplicial complex,
		$Y\to \Gamma\leftmod Y$
		is a covering map, and $\pi_1(\Gamma\leftmod Y)\cong \Gamma$.  
		\item 
		If moreover   $I_{m_0}\subseteq P_\rho$,
		then, for every  $\Gamma$ as in (i),
		the complex $\Gamma\leftmod Y$ has an infinite tower
		of connected $C_p$-Galois  coverings ($C_p$ is the cyclic group order $p$).
	\end{enumerate}
\end{cor}

\begin{proof}
	(i) Let $U\subseteq \calG(\calO)$ be the subset from Proposition~\ref{PR:finite-quotients}(ii).
	Our assumptions on the sequence $\{I_m\}_{m\in \N}$
	imply that   $\bigcap_{m\in\N}\calG(\calO;I_m)=\{1_G\}$
	and $\calG(\calO;I_1)\supseteq \calG(\calO;I_2)\supseteq \dots$.
	Thus,
	there exists $m_0\in\N$
	such that $\calG(\calO;I_{m_0})\cap U=\emptyset $,
	and the conclusion follows from
	Proposition~\ref{PR:finite-quotients}(ii).
	
	(ii) 
	For every $i\geq 0$, put $\Gamma_i=\Gamma\cap \calG(\calO;I_{m_0}P_\rho^i)$.
	Then $\Gamma_i/\Gamma_{i+1}$ is isomorphic
	to a subgroup of $\calG(\calO;I_{m_0}P_\rho^i)/ \calG(\calO;I_{m_0}P_\rho^{i+1})$,
	which is an elementary abelian $p$-group by Proposition~\ref{PR:elementary-p-group}.
	This means that there are normal subgroups 
	\[\Gamma_i=\Gamma_{i,1}\supseteq \dots\supseteq \Gamma_{i,t(i)}\supseteq \Gamma_{i,t(i)+1}
	=\Gamma_{i+1}\]
	such that $|\Gamma_{i,k}/\Gamma_{i,k+1}|=p$ for all $k\in\{1,\dots,t(i)\}$.
	Put $X_{i,k}=\Gamma_{i,k}\leftmod Y$.
	Then 
	\[
	\cdots \to X_{2,2}\to X_{2,1}\cdots \to X_{1,2}\to X_{1,1}\to\dots\to X_{0,2}\to X_{0,1}=\Gamma\leftmod Y
	\]
	is the required tower of $C_p$-Galois coverings.
\end{proof}

\begin{remark}
	Suppose we are given a global field $K$, a finite place $\nu$, and an absolutely almost simple
	simply connected isotropic algebraic group $\bfH$ over $ K_\nu$, and we wish to complete
	this data to the setting of Notation~\ref{NT:general-num-thy-notation} in such a way
	that $\bfH=\bfG_{K_\nu}$ and the assumptions of Proposition~\ref{PR:finite-quotients} hold.
	This is known to be possible if 
	$\Char K=0$, and thus the affine building $Y$ of $\bfH$ covers 
	infinitely many finite similicial complexes.
	On the other hand, if $\Char K>0$ and $\rank \bfH>1$, 
	then completing the data in this manner
	is possible only if $\bfH$ is of type $A_m$.
\end{remark}

\subsection{The   Congruence Subgroup Property}
\label{subsec:csp}

Keep Notation~\ref{NT:general-num-thy-notation}.
We proceed by recalling the \emph{congruence subgroup property} and using 
it to bound the   the number of group homomorphisms from a principal congruence subgroup
$\calG(\calO;I)$ to the additive group of $\F_p$.

Let $\what{\calG}_K=\invlim \calG(K)/U$, where the limit ranges over the finite
index subgroups $U$ of $\calG(\calO)$, and let $\quo{\calG}_K =\invlim \calG(K)/\calG(\calO;I)$,
where the limit ranges over the nonzero ideals $I\idealof\calO$.
While $\calG(K)/U$, resp.\ $\calG(K)/\calG(\calO;I)$, are not groups, $\what{\calG}_K$
and $\quo{\calG}_K$ are groups, which we topologize by giving $\calG(K)/U$, resp.\ $\calG(K)/\calG(\calO;I)$,
the discrete topology and taking the limit topology.
There is an evident surjective group homomorphism $\what{\calG}_K\to \quo{\calG}_K$.
The kernel of this map, denoted $C^S(\bfG)$, is called the \emph{congruence kernel} of $(\bfG,S)$,
and $(\bfG,S)$ is said to satisfy the \emph{congruence subgroup property} (CSP)
if $C^S(\bfG)$ is finite. For example,   $C^S(\bfG)$ is trivial if and only if any finite index subgroup
of $\calG(\calO)$ contains a principal congruence subgroup. 
The question of which pairs $(\bfG,S)$ have CSP has a long and rich history; we refer 
the reader to \cite{Prasad_2010_developments_on_CSP}
for a survey, and to   
\cite[\S9.5]{Platonov_1994_algebraic_groups_and_number_theory} for an extensive discussion.
The main conjecture in this field is due to Serre:

\begin{cnj}[Serre] \label{CJ:csp}
	With notation as in Notation~\ref{NT:general-num-thy-notation}, 
	$(\bfG,S)$ has CSP if 
	$\bfG$ is $K_\rho$-isotropic for every $\rho\in S-\calV_\infty$
	and $\sum_{\rho\in S}  \rank_{K_\rho} G >1$.
	The pair $(\bfG,S)$ does not have CSP if $\sum_{\rho\in S}  \rank_{K_\rho} G =1$.
\end{cnj} 

The conjecture is known to hold in many cases.
The following two theorems are a culmination of many results due to  Borovoi, P.~Gille, Platonov, G.~Prasad, Raghunathan, A.~Rapinchuk, Segev, Seitz, Tomanov
and others. See \cite{Prasad_2010_developments_on_CSP}, \cite[\S9.5]{Platonov_1994_algebraic_groups_and_number_theory} and
the references therein.

\begin{thm}\label{TH:csp}
	With notation as in Notation~\ref{NT:general-num-thy-notation},
	suppose that
	\begin{enumerate}[label=(\arabic*)]
		\item $\bfG$ is $K_\rho$-isotropic for all   $\rho\in S-\calV_\infty$, 
		\item $\sum_{\rho\in S}\rank_{K_\rho}\bfG \geq 2$, and
		\item $\bfG$ is $K$-isotropic, or
		of the types $B_n$ ($n\geq 2$), $C_n$ ($n\geq 3$), $D_n$ ($n\geq 5$),
		$E_7$, $E_8$, $F_4$, $G_2$, 
		or $\bfG=\uSU(f)$
		where $f$ is a nondegenerate hermitian form of dimension $\geq 4$ over a quadratic 
		Galois extension of $K$
		(cf.\ Example~\ref{EX:group-schemes}(iv)).
	\end{enumerate}
	Then $(\bfG,S)$ has CSP.
\end{thm}

\begin{proof}
	Theorem~2 in \cite{Prasad_2010_developments_on_CSP} states that if   (1)
	and (2) hold, then CSP for $(\bfG,S)$ 
	follows from the centrality
	of   $C^S(\bfG)$  in $\what{\calG}_K$. The group
	$C^S(\bfG)$ is known to be central in $\what{\calG}_K$ when (3) holds;
	see 
	\cite{Raghunathan_1976_congruence_subgroup_problem}
	and \cite{Raghunathan_1986_congruence_subgroup_problem_II} for the 
	case where $\bfG$ is  isotropic
	(the assumption $G(k)=(k)^+$ for groups of type $  E_6$
	that was unknown at the time and was established in \cite{Gille_2009_Kneser_Tits_problem}),
	and  \cite[Theorems~9.23, 9.24]{Platonov_1994_algebraic_groups_and_number_theory}
	for the anisotropic cases.
\end{proof}

\begin{thm}\label{TH:trivial-cong-kernel}
	With notation as in Notation~\ref{NT:general-num-thy-notation},
	assume that conditions (1) and (2) of \ref{TH:csp} hold
	that $\bfG$ has CSP (this follows from (1) and (2) if Conjecture~\ref{CJ:csp} holds)
	and that 
	\begin{enumerate}[label=(\arabic*)]
		\item[(3$\,'$)] $\bfG$ is $K$-isotropic, or
		of the types $B_n$ ($n\geq 2$), $C_n$ ($n\geq 3$), $D_n$ ($n\geq 4$,
		excluding ${}^{3,6}D_4$),
		$E_7$, $E_8$, $F_4$, $G_2$, or 
		inner of type $A_n$,   		
		or $\bfG=\uSU(f)$
		where $f$ is a nondegenerate hermitian form of dimension $ \geq 3$ over a quadratic 
		Galois extension of $K$
		(cf.\ Example~\ref{EX:group-schemes}(iv)).
	\end{enumerate}
	Then $C^S(\bfG)$ is isomorphic to a subgroup of $\mu(K)$,
	the group of roots of unity in $K$.
	If moreover $S$ contains a non-archimedean place, the $C^S(\bfG)$ is trivial.
\end{thm}

\begin{proof}
	Condition (3$'$) lists the cases where the \emph{Margulis--Platonov conjecture} is
	known to hold,
	see \cite{Rapinchuk_2006_Margulis_Platonov_conj_An_inner} (inner type $A_n$),
	\cite{Gille_2009_Kneser_Tits_problem} (isotropic groups)
	and
	\cite[Appendix~A]{Rapinchuk_2001_valuation_like_maps} (all other cases).
	By \cite[Theorem~2]{Prasad_2010_developments_on_CSP},
	if the  Margulis--Platonov   conjecture  holds for 
	$\bfG$, and $(\bfG,S)$ has CSP, then $C^S(\bfG)$ is isomorphic to
	the \emph{metaplectic kernel} $M(S,\bfG)$. 
	G.~Prasad and A.~Rapinchuk
	\cite{Prasad_1996_metaplectic_kernel} (or
	\cite[Theorem~3]{Prasad_2010_developments_on_CSP}) 
	showed that the latter is isomorphic
	to  a  subgroup of $\mu(K)$, and it is moreover trivial
	if there is $\rho\in S-\calV_\infty$ such that $\bfG$
	is $K_\rho$-isotropic.
\end{proof}

%


In the following theorem, we use CSP in order to bound the number of group homomorphism
from a principal congruence subgroup to $\F_p$. 
If $A$ and $B$ are topological groups, we write $\Hom_c(A,B)$ for the set of continuous 
group homomorphisms from $A$ to $B$. We give $\F_p$ the discrete topology.

\begin{thm}\label{TH:no-of-subgroups}
	With notation as in Notation~\ref{NT:general-num-thy-notation},
	suppose that 
	$\sum_{\rho\in S}\rank_{K_\rho}\bfG\geq 2$,
	that $\bfG$ is $K_\rho$-isotropic for every $\rho\in S-\calV_\infty$
	and that $(\bfG,S)$ has CSP (this is superfluous if Conjecture~\ref{CJ:csp} holds).
	Let $p\in\N$ be a prime number, and let $I\idealof \calO$.
	Define the following sets of places:
	\begin{itemize}
		\item $T_1$ is the set of $\rho\in \calV-S$
		such that $\Char k(\rho)\neq p$ and $I_\rho\neq \calO_\rho$.
		\item $T_2$ is the set of $\rho\in \calV-S$ such that 
		$\Char k(\rho)\neq p$,
		$I_\rho=\calO_\rho$,		
		$\calG_{\calO_\rho}$ is absolutely almost simple simply connected
	and $|k(\rho)|\geq 4$.
		\item $T_3$ is the set of $\rho\in \calV-S$
		such that    $\Char k(\rho)=p$ and  $\calG_{\calO_\rho}$ is a \emph{split} absolutely almost simple simply connected. If  $\calG_{\calO_\rho}$ is of type
		$C_2$ or $G_2$, we further require that  $|k(\rho)|>2$.
		\item $T_4=\calV-S-T_1-T_2-T_3$.
	\end{itemize}		
	Then there are $C,M\in\N\cup\{0\}$, depending only on $\calG$ and $K$, such 
	that if   $\Char K=0$, then	
	\[
	\dim_{\F_p}\Hom(\calG(\calO;I),\F_p)\leq C |T_3|+\sum_{\rho\in T_4}
	\dim_{\F_p}\Hom_c(\calG(\calO_\rho;I_\rho),\F_p)+
	M.
	\]	
	and if $\Char K>0$, then
	\begin{align*}
	\dim_{\F_p}\Hom(\calG(\calO;I),\F_p)&\leq C \sum_{\rho\in T_3}\rho(I)\dim_{\F_p}k(\rho) +\sum_{\rho\in T_4}
	\dim_{\F_p}\Hom_c(\calG(\calO_\rho;I_\rho),\F_p)
	+
	M.
	\end{align*}
	Each term  $ 
	\dim_{\F_p}\Hom_c(\calG(\calO_\rho;I_\rho),\F_p)$
	is finite. Moreover, if $p\nmid |C^S(\bfG)|$,
	then we can take $M=0$.
\end{thm}

The group scheme $\calG_{\calO_\rho}$ is   absolutely almost simple simply connected
for all but finitely many $\rho\in \calV-S$, see
\cite[Lemma~1.9]{Raghunathan_1976_congruence_subgroup_problem}.
Thus, if $\Char K=0$, then $\calV-S-T_2$ is finite.

\begin{proof}
	Put $\Gamma=\calG(\calO;I)$, let $\what{\Gamma}$ be the profinite
	completion of $\Gamma$, and let $\quo{\Gamma}=\invlim \calG(\calO;I)/\calG(\calO;I')$
	where $I'$ ranges over the nonzero ideals of $\calO$ contained in $I$.
	Then $\what{\Gamma}$ and $\quo{\Gamma}$ are  
	subgroups of $\what{\calG}_K$ and $\quo{\calG}_K$, respectively,
	and the natural
	map $\what{\calG}_K \to \quo{\calG}_K $ restricts to
	a surjective map $\what{\Gamma}\to \quo{\Gamma}$.
	The kernel $H:=\ker(\what{\Gamma}\to \quo{\Gamma})$
	is a subgroup of $C^S(\bfG)$, hence finite by our assumptions.
	
	There is a one-to-one correspondence between
	group homomorphisms $\vphi:\Gamma\to \F_p$ and continuous 
	group homomorphisms $\hat{\vphi}:\what{\Gamma}\to\F_p$  
	($\F_p$ is regarded as a discrete tolopgical space).
	Thus, it is enough to bound the $\F_p$-dimension of
	$\Hom_c(\what{\Gamma},\F_p)$.
	The short exact sequence $1\to H\to \what{\Gamma}\to \quo{\Gamma}\to 1$
	gives rise to an exact sequence
	\[
	0\to \Hom_c(\quo{\Gamma},\F_p)\to \Hom_c(\what{\Gamma},\F_p)\to \Hom_c(H,\F_p),
	\] 
	so
	\[
	\dim \Hom(\Gamma,\F_p)\leq \dim \Hom_c(\quo{\Gamma},\F_p)+\dim \Hom (H,\F_p).
	\]
	We take $M=\max_{H'\leq C^S(\bfG)} \dim_{\F_p}\Hom(H',\F_p)$,
	so that $\dim_{\F_p}\Hom(H,\F_p)\leq M$.
	By strong approximation, the group $\quo{\Gamma}$
	is the profinite group $\prod_{\rho\in \calV-S}\calG(\calO_\rho;I_\rho)$,
	cf.\ Proposition~\ref{PR:density-in-arith-grp}.
	Since the kernel of every 
	continuous homomorphism  $\prod_{\rho\in \calV-S}\calG(\calO_\rho;I_\rho)\to \F_p$
	is open, 
	we have
	\[
	\Hom_c(\prod_{\rho\in \calV-S}\calG(\calO_\rho;I_\rho),\F_p)\cong
	\bigoplus_{\rho\in\calV-S}\Hom_c(\calG(\calO_\rho;I_\rho),\F_p).
	\]
	Putting everything together, we get
	\begin{align*}
	\dim \Hom(\Gamma,\F_p)
	&
	\leq \sum_{\rho\in \calV-S} \dim\Hom_c(\calG(\calO_\rho;I_\rho),\F_p) +M.
	\end{align*}
	Now, in order to prove the theorem, it remains to bound
	$\dim\Hom_c(\calG(\calO_\rho;I_\rho),\F_p)$  according to   whether $\rho$
	is in $T_1$, $T_2$, $T_3$ or $T_4$. We split into cases.
	
\smallskip	
	
	Suppose first that $\rho\in T_1$. 
	Then $I_\rho\subseteq\frakm_\rho$ and
	$\ell:=\Char k(\rho)\neq p$. 
	We claim that  $\Hom_c(\calG(\calO_\rho ;I_\rho),\F_p)=0$.
	To see this, 
	note that there is some $n\in \N $
	such that $I_\rho = \frakm_\rho^n$.
	Let $\vphi:\calG(\calO_\rho ;\frakm_\rho^n)\to\F_p$ be a continuous homomorphism.
	Then $\ker\vphi$ is open and thus contains $\calG(\calO;\frakm_\rho^m)$
	for some $m\geq n$.
	By Proposition~\ref{PR:elementary-p-group},
	$\calG(\calO_\rho,\frakm_\rho^i)/\calG(\calO_\rho,\frakm_\rho^{i+1})$
	is an elementary abelian $\ell$-group for all $\ell$ and $i\geq 1$.
	Since $\ell\neq p$,
	this forces $\vphi$ to be $0$.

\smallskip		
	
	Suppose next that $\rho\in T_2$.
	Then $I_\rho=\calO_\rho$, $\ell:=\Char k(\rho)\neq p$,
	$\calG_{\calO_\rho}$ is absolutely almost simple simply connected 
	and $|k(\rho)|\geq 4$.
	We need to show that   $\Hom_c(\calG(\calO_\rho ;I_\rho),\F_p)=\Hom_c(\calG(\calO_\rho),\F_p)=0$.
	Let $\vphi:\calG(\calO_\rho)\to \F_p$ be a continuous homomorphism.
	By the previous paragraph, 
	$\vphi$  vanishes on
	$\calG(\calO_\rho;\frakm_\rho)$. 
	Since $\calG_{\calO_\rho}$ is absolutely almost simple simply connected,
	the same applies to  $\calG_{k(\rho)}$.
	Now, theorems of Chevalley, Steinberg and Tits 
	\cite[Proposition~7.5]{Platonov_1994_algebraic_groups_and_number_theory}
	tell us that $\calG(k(\rho))$ is a perfect group
	(here we need $|k(\rho)|\geq 4$). 
	Since $\calG_{\calO_\rho}\to \Spec \calO_\rho$	is smooth, $\calG(\calO_\rho)/\calG(\calO_\rho;\frakm_{\rho})
	\cong \calG(k(\rho))$.
	Since $\F_p$ is abelian,
	this forces $\vphi$ to be $0$, as claimed.
	
\smallskip		
	
	Suppose now that $\rho\in T_3$.
	Then $\Char k(\rho)=p$ and $\calG_{\calO_\rho}$ is a 
	split almost simple  simply connected
	group scheme over $\calO_\rho$.
	Write $R=\calO_\rho$, $\frakm=\frakm_{\rho}$ and $J=I_\rho$.
	We will make use of the (absolute) \emph{elementary subgroups} $E(R;J)$ of $\calG(R;J)$;
	see \cite[\S2]{Hazrat_2013_relative_comm_Chevalley_grps}
	and the references therein for their definition. In fact, in our case, $E(R;J)=\calG(R;J)$ by
	\cite[Propositions~2.3, 2.4]{Abe_1976_normal_subgroups_Chevalley_groups}.
	It follows from the definition of the elementary subgroups that
	$E(R;J)^p\supseteq E(R;pJ)$, hence
	$\calG(R;J)^p\supseteq \calG(R;pJ)$.
	Moreover, by 
	\cite[Lemma~17]{Hazrat_2013_relative_comm_Chevalley_grps} 
	and the comment following Lemma~18 in that source,
	we have $[E(R;J),E(R;J)]\supseteq   E(R;J^3)$ (in fact, we can replace
	$E(R;J^3)$ with $E(R;J^2)$ if $\calG$ is not of type $C_\ell$),
	and likewise for $\calG(R;J)$.
	Consequently, any group homomorphism $\vphi:\calG(R;J)\to \F_p$
	must vanish on $\calG(R;pJ)$ and $\calG(R;J^3)$.
	If $J=\calO_\rho$ (i.e.\ $\rho(I)=0$), then $\vphi$ must be zero, so assume $J\neq \calO_\rho$.
	In particular, $pR\subseteq \frakm $.
	We now break into subcases.
	
	Suppose first that $\Char K=0$. Then $pJ\neq 0$.
	Let $\quo{\vphi}$ denote the induced map $\calG(R;J)/\calG(R;pJ)\to \F_p$.
	We need to bound the number of these maps by $p^C$ for   $C$ depending only on
	$\calG$ and $K$.
	Write $t=\rho(p)>0$, $s=\rho(J)$ and $f=\dim_{\F_p}k(\rho)$.
	Proposition~\ref{PR:elementary-p-group} 
	tells us that there is a constant $D$ such that  
	$\calG(R;\frakm^i)/\calG(R;\frakm^{i+1})$ is an elementary abelian $p$-group of rank $\leq D\cdot f$.
	Thus, $\calG(R;J)/\calG(R;pJ)=\calG(R;\frakm^s)/\calG(R;\frakm^{t+s})$
	is a $p$-group with at most $p^{Dtf}$ elements, hence it admits at most $p^{Dtf}$
	homomorphisms into $\F_p$.
	Our assumption that $\Char K=0$ implies that only finitely many places $\rho\in\calV-S$
	divide $p$, so there is a constant $C_1\in \N$ such that $tf\leq C_1$ for all $\rho\in T_3$.
	We can take $C=DC_1$.

	Now suppose that $\Char K>0$.
	Then   $\Char K=\Char k(\rho)=p$.
	Let 
	$\quo{\vphi}$ denote the induced map $\calG(R;J)/\calG(R;J^3)\to \F_p$,
	and
	write $J= \frakm^s$ (so that $s=\rho(I)$).
	An argument similar to the previous paragraph shows that
	the dimension  of the $\F_p$-vector space of such $\quo{\vphi}$ is at most 
	$2sD\dim_{\F_p}k(\rho) =2D\rho(I)\dim_{\F_p}k(\rho)$.	
	Taking $C=2D$ completes the case $\rho\in T_3$.
	
	%
	
	Finally,  we need
	to show that $\Hom_c(\calG(\calO_\rho;I_\rho),\F_p)$ is finite for all $\rho\in \calV-S$.	
	To that end, it is enough to prove that $\calG(\calO_\rho;I_\rho)$
	is finitely generated as a profinite group.
	By Proposition~\ref{PR:density-in-arith-grp}, $\calG(\calO;I)$ is dense in $\calG(\calO_\rho;I_\rho)$.	
	Thus, it is enough to show that $\calG(\calO;I)$ is finitely generated.
	This holds by
	\cite[Theorem~5.11]{Platonov_1994_algebraic_groups_and_number_theory}
	if $\Char K=0$
	and    \cite{Behr_1969_finite_generation_of_arith_grps_func_field}
	(see also \cite{Behr_1987_finite_present_of_arith_grps_func_field}) if $\Char K>0$.
\end{proof}

\subsection{Proof of Theorem~\ref{TH:good-quotient}}
\label{subsec:construction-of-quotients}

We prove Theorem~\ref{TH:good-quotient} by exhibiting an  example  of a simpicial
complex and showing that it satisfies (i) and (ii). The example   can be   generalized, of course.
To that end, we would like  to apply   Corollary~\ref{CR:finite-quotients}
and Theorem~\ref{TH:no-of-subgroups} together. Their joint assumptions 
force the set of places $S$ to be   $\calV_\infty\cup\{\nu\}$, and moreover,
$\bfG$ has to be   $K_\nu$-isotropic and $K_\rho$-anisotropic for every $\rho\in\calV_\infty$.

\medskip

We first recall the following well-known fact from algebraic topology.

\begin{lem}\label{LM:fund-group-double-coverings}
	Let $X$ be a connected (finite) simplicial complex with fundamental group $\Gamma$
	and let $p$ be a prime number.
	Then $\HH^1(X,\F_p)\cong \Hom(\Gamma,\F_p)$ as $\F_p$-vector spaces.
\end{lem}

\begin{construction}\label{CN:good-quotient-not-Ram}
	We specialize Notation~\ref{NT:general-num-thy-notation} as follows. Take
	\begin{itemize}
		\item $K=\Q$.
	\end{itemize}
	The set of places $\calV$ can be identified with the set of prime numbers
	together with $\infty$.
	Fix a   prime number $p$ with $p\equiv 1\bmod 4$ and take:
	\begin{itemize}
		\item $S=\{p,\infty\}$,
		\item $\nu=p$ ($\nu$ will also denote the $p$-adic valuation),
	\end{itemize}
	Thus, $\calO=\Z[\frac{1}{p}]$. Fix $d\geq 3$ and let
	\begin{itemize}
		\item $\calG=\uSU_{d+1}(\calO[i]/\calO)$, 
		\item $\bfG=\uSU_{d+1}(\Q[i]/\Q)$,
	\end{itemize}
	where the notation is as in    Example~\ref{EX:group-schemes}(iii).
	Thus, $ K_\nu=\Q_p$, and $Y$ is the affine building attached to $\bfG_{\Q_p}$.
	Our choice of $p$ implies that $\Q_p$ contains a square root of $-1$, so by
	Example~\ref{EX:simply-conn-groups}(ii), $\bfG_{\Q_p}\cong \uSL_{d+1}(\Q_p)$.
	Thus,
	$Y$ is the the familiar affine building of $\uSL_{d+1}(\Q_p)$. In particular, $\dim Y=d$
	and $Y$ is $(p+1)$-thick.
	
	Observe that $\bfG(\R)$ is nothing but the group of $(d+1)\times (d+1)$
	complex unitary matrices, so it is compact. Thus, $\bfG$ is $\R$-anisotropic, and therefore
	$\Q$-anisotropic. On the other hand, $\bfG$ is $\Q_p$-isotropic
	by Example~\ref{EX:simply-conn-groups}(i).
	This allows us to apply 
	Corollary~\ref{CR:finite-quotients} with   the sequence of ideals
	$I_m=2^{m }3\calO$; set $\Gamma=\calG(\calO;I_{m_0})=\calG(\calO;2^{m_0}3\calO)$,
	where $m_0\in\N$ is in the corollary.
	(Explicitly, $\Gamma$ is   the set of matrices in $\nSU{\Z[\frac{1}{p}][i]/\Z[\frac{1}{p} ]}{d+1}$
	which are congruent to the identity matrix modulo $ 2^{m_0}\cdot 3$.)
	Put $X:=\Gamma\leftmod Y$ and $X'_r=\calG(\calO;2^{m_0}3^{r+1} \calO)\leftmod Y$ for all $r\in\N\cup\{0\}$.
	Then $X=X'_0,X'_1,X'_2,\dots$ are connected simplicial complexes covered by $Y$,
	and the evident quotient maps give rise to a tower of coverings $\dots X'_2\to X'_1\to X'_0=X$.
\end{construction}

\begin{prp}\label{PR:good-quotient-not-Ram}
	The complex   $X$ and the tower $\dots X'_2\to X'_1\to X'_0=X$ 
	of Construction~\ref{CN:good-quotient-not-Ram} satisfy
	conditions (i) and (ii) of Theorem~\ref{TH:good-quotient}.
\end{prp}

\begin{proof}
	Condition (i) follows from Corollary~\ref{CR:finite-quotients}(ii).
	
	To show (ii), observe that Proposition~\ref{PR:elementary-p-group}
	implies that $\calG(\calO;2^{m_0}3^{r+1} \calO)/\calG(\calO;2^{m_0}3^{r+2} \calO)$
	is an elementary abelian $3$-group for all $r\geq 1$. Thus,
	$[X'_r:X]=[\calG(\calO;2^{m_0}3 \calO): \calG(\calO;2^{m_0}3^{r+1}   \calO)]$
	is a power of $3$, and in particular odd. (In the equality we  
	used the fact that the $\calG(\calO;2^{m_0}3   \calO)$ acts freely on $Y$,
	which we know by Corollary~\ref{CR:finite-quotients}(i).)
	
	Next, Lemma~\ref{LM:fund-group-double-coverings}
	and Corollary~\ref{CR:finite-quotients}(i)
	tell us that
	\[\dim\HH^1(X'_r,\F_2)=\dim_{\F_2} \Hom (\pi_1(X'_r),\F_2)=
	\dim_{\F_2}\Hom (\calG(\calO;2^{m_0}3^{r+1}\calO),\F_2).\]
	We bound the right hand side using Theorem~\ref{TH:no-of-subgroups} with $p=2$
	and $I=2^{m_0}3^{r+1}\calO$.
	The assumptions
	of the theorem hold because $\rank_{\R}\bfG=0$
	and $\rank_{\Q_p}\bfG=\rank \uSL_{d+1}(\Q_p)=d>1$ (Example~\ref{EX:simply-conn-groups}(i)),
	and $(\bfG,S)$ has CSP holds by Theorem~\ref{TH:csp} (here we need $d\geq 3$).
	In the notation of Theorem~\ref{TH:no-of-subgroups},
	we have $T_1=\{3\}$, $T_2=\calV-\{2,3,p,\infty\}$, $T_3=\emptyset$ and $T_4=\{2\}$,
	so the theorem implies that  
	$\dim_{\F_2}\Hom (\calG(\calO;2^{m_0}3^{r+1}\calO),\F_2)$
	is bounded by a constant $C$ independent of $r$.
\end{proof}

\begin{proof}[Proof of Theorem~\ref{TH:good-quotient}]
The theorem follows from  
Proposition~\ref{PR:good-quotient-not-Ram}. 
The thickness requirement on the building $Y$
can be taken care of by choosing $p$ in Construction~\ref{CN:good-quotient-not-Ram}
large enough in advance.
\end{proof} 

\begin{remark}\label{RM:good-quotient-in-dim-2}
	If Conjecture~\ref{CJ:csp} holds, then the proof of Theorem~\ref{TH:good-quotient}
	also applies with $d=2$. 
	In fact,
	the theorem holds when $d=2$ even without
	assuming Conjecture~\ref{CJ:csp}.
	The idea is to use the affine buildings of the (unique)
	simply connected algebraic group of type $G_2$ over $\Q_p$.
%
%
	To that end,
	one could  
	take $\calG$ to be the automorphism group scheme of a $\Z$-order
	in the standard \emph{octonion algebra} over $\Q$ and argue
	as in the proof of Proposition~\ref{PR:good-quotient-not-Ram}.
	We omit the details.
\end{remark}

For later use, we also introduce a variation of Construction~\ref{CN:good-quotient-not-Ram}.

\begin{construction}\label{CN:good-quotient-no-double-covs}
	Let $K=\Q$ and $\calV$ be as in Construction~\ref{CN:good-quotient-not-Ram}.
	Let $E=\Q[\sqrt{-7}]$ and $A=\Z[\frac{\sqrt{-7}+1}{2}]\cong \Z[x\where x^2=x-2]$,
	and let $p$ be an odd prime number such
	that $x^2-x+2$ factors into two distinct factors
	modulo $p$, or equivalently,   $p$ is congruent to $1$, $2$, or $4$ modulo $7$.
	Let 
	\begin{itemize}
		\item $S=\{p,\infty\}$ and
		\item $\nu=p$ ($\nu$ will also denote the $p$-adic valuation).
	\end{itemize}
	Then $\calO=\Z[\frac{1}{p}]$. 
	
	Fix $d\geq 3$ and 	
	let $f:A^{d+1}\times A^{d+1}\to A$ denote the hermitian
	form $f((x_i),(y_i))=\sum_{i=1}^{d+1} \quo{x_i} y_i$, where $\quo{x_i}$
	is the complex conjugate of $x_i$. Now, using the notation
	of Example~\ref{EX:group-schemes}(iv) (with $R=\Z$, $r_0=-2$, $r_1=1$), define
	\begin{itemize}
		\item $\calG=\uSU(f)$, 
		\item $\bfG=\uSU(f_\Q)$.  
	\end{itemize}
	Since the polynomial $x^2-x+2$ factors into two distinct factors
	modulo $p$,
	we have $E\otimes_\Z \Q_p \cong \Q_p\times \Q_p$
	and under  isomorphism,   complex conjugation becomes   swapping
	the coordinates. Thus, as in Example~\ref{EX:simply-conn-groups}(ii),
	we see that 	
	$\bfG_{K_\nu}\cong \uSL_{d+1}(\Q_{p})$,
	so
	$Y$ is the the affine building of $\uSL_{d+1}(\Q_{p})$.  
	
	Since $E\otimes_{\Q} \R=\C$,
	we again see  that $\bfG(\R)$ is    the group of $(d+1)\times (d+1)$
	complex unitary matrices, so again, 
	$\bfG$ is $\R$-anisotropic, and therefore
	$\Q$-anisotropic. On the other hand, $\bfG$ is $\Q_p$-isotropic
	by Example~\ref{EX:simply-conn-groups}(i).
	This allows us to apply 
	Corollary~\ref{CR:finite-quotients} with   the sequence of ideals
	$I_m=7^m 3\calO$; set $\Gamma=\calG(\calO;I_{m_0})=\calG(\calO;7^{m_0} 3\calO)$,
	where $m_0\in\N$ is in the corollary.
	Put $X:=\Gamma\leftmod Y$.
	Then $X$ is a connected simplicial complexes covered by $Y$.
\end{construction}

\begin{prp}\label{PR:good-quotient-no-double-covs}
	The complex   $X$  
	of Construction~\ref{CN:good-quotient-no-double-covs}  
	satisfies
	$\HH^1(X,\F_2)=0$.
\end{prp}

\begin{proof}
	As in the proof of Proposition~\ref{PR:good-quotient-not-Ram},
	we reduce into showing   $\Hom (\calG(\Z;7^{m_0}3\Z),\F_2)=0$.
	We show this by  applying Theorem~\ref{TH:no-of-subgroups} with $p=2$.
	Note that $x^2-x+2$ splits into two distinct factors
	modulo $2$, and thus $A\otimes_{\Z} \Z_2\cong \Z_2\times \Z_2$,
	so $\calG_{\Z_2}\cong\uSL_{d+1}(\Z_2)$ (cf.\ Example~\ref{EX:simply-conn-groups}(ii)).
	This means that, $T_1=\{7,3\}$, $T_2=\calV-S-\{2,3,7\}$,
	$T_3=\{2\}$ and $T_4=\emptyset$.
	By Theorem~\ref{TH:trivial-cong-kernel}, $C^S(\bfG)=1$,
	so we conclude that $\Hom(\calG(\Z;7^{m_0}3\Z),\F_2)=0$.
\end{proof}

\section{Sheaves of Large Dimension with Small Cohomology}
\label{sec:candidates}

This final section has to purposes.
First, we construct  examples of locally cosntant $\F$-sheaves
$\calF$ of arbitrarily large dimension
such that $h^1(\calF)\ll\dim \calF$.
In particular, we shall prove Theorem~\ref{TH:sheaves-with-small-coh} by
constructing examples with $h^0(\calF)=h^1(\calF)=0$ and $\dim\calF\to\infty$;
this    uses Theorem~\ref{TH:good-quotient} as a black box.
Sheaves with $h^1(\calF)\ll\dim \calF$ are natural candidates for
the iterative modification process discussed in Section~\ref{sec:modifying}.

Second, we will   show that every two of the conditions 
\ref{item:TH:tower-paradigm:tower}--\ref{item:TH:tower-paradigm:rate} of
the tower paradigm (Theorem~\ref{TH:tower-paradigm} with $k=0$) are
satisfied for some sheaved complex $(X,\calF)$. Otherwise stated, no two
of the prerequisites of the tower paradigm  are contradictory.

%
%
%

\subsection{Proof of Theorem~\ref{TH:sheaves-with-small-coh}}

We begin with proving the following lemma.
Recall that $\F_X$  denotes the constant sheaf $\F$ on a simplicial complex $X$.

\begin{lem}\label{LM:ker-of-counit}
	Let $\F$ be a field, let $u:Y\to X$ be a degree-$m$ covering of connected simplicial complexes,
	and suppose that $\Char \F\nmid m$.
	Define a sheaf morphism $\vphi:u_*\F_{ Y}\to \F_{ X}$
	by  $\vphi_x(\alpha_1,\dots,\alpha_m)=\sum_{i}\alpha_i$ for all $x\in X-\{\emptyset\}$,
	and
	put $\calG=\ker \vphi$. Then $h^0(\calG)=0$ and $h^1(\calG)=\dim\HH^1(Y,\F)-\dim\HH^1(X,\F)$.
\end{lem}

\begin{proof}
	Note that $u_*\F_Y(x)=\prod_{y\in u^{-1}(x)}\F$  for all $x\in X-\{\emptyset\}$ and we implicitly
	identified $\prod_{y\in u^{-1}(x)}\F$ with $\F^m$ by numbering the faces in $ u^{-1}(x)$.
	Define $\psi:\F_{ X}\to u_*\F_{ Y}$
	by $\psi_x(\alpha)=(m^{-1}\alpha,\dots,m^{-1}\alpha)$ ($m$ times) for all $x\in X-\{\emptyset\}$
	(here we used the assumption $\Char \F\nmid m$).
	It is routine to check that $\psi$ is indeed a morphism of sheaves, and moreover,
	$\vphi\circ \psi= \id_{\F_{ X}}$.
	This means that  $u_*\F_{ Y}$ breaks as a product
	of the sheaves
	$\im \psi\cong \F_{ X}$ and $\ker \vphi=\calG$.
	Consequently, $h^i(\calG)=h^i(u_*\F_{ Y})-h^i( \F_{ X})$
	for all $i\geq 0$.
	By Lemma~\ref{LM:Shapiro}, we have $h^i(u_*\F_{ Y})=h^i(\F_{ Y})$.
	Since $X$ and $Y$ are connected
	$h^0(\F_X)=h^0(\F_Y)=1$, and lemma follows.
\end{proof}

\begin{proof}[Proof of Theorem~\ref{TH:sheaves-with-small-coh}]
Recall that we are given $q,d\in\N$ with $d\geq 2$,
and we need to construct a
$q$-thick $d$-dimensional affine building $Y$ covering   
a finite simplicial complex $X$ and a nonzero locally constant $\F_2$-sheaf 
$\calG$ such that $X$ admits an infinite tower of connected
double coverings and $h^0(\calG)=h^1(\calG)=0$.

By  Theorem~\ref{TH:good-quotient},
there exist  a $q$-thick affine
building $Y$ covering a simlicial complex $X'_0$
and  a tower of connected coverings $\dots\to X'_2\to X'_1\to X'_0$ such that every connected covering
of $X'_0$ admits an infinite
tower of connected double coverings, and $[ X'_r:X'_0]$ is odd and bounded by some $C\in \N$ for all $r\in\N\cup\{0\}$.
The latter implies that there are $t>s\geq 0$ with
	$\dim\HH^1(X'_t,\F_2)= \dim\HH^2(X'_s,\F_2)$.
	Let $u$ denote the covering map $X'_t\to X'_s$.
	The degree of $u$ --- call it $m$ --- is odd and greater than $1$, so Lemma~\ref{LM:ker-of-counit}
	provides us with a locally
	constant $\F_2$-sheaf $\calG$ on $X'_s$ of dimension $m-1>0$
	such that $h^0(\calG)=0$ and $h^1(\calG)=\dim\HH^1(X'_t,\F_2)- \dim\HH^2(X'_s,\F_2)=0$.
	Taking $X=X'_s$, we have obtained the desired
	sheaved $d$-complex $(X,\calG)$.
	Alternatively, writing $p$ for the covering map $X'_s\to X'_0$,
	we can also take $(X'_0,p_*\calG)$ thanks to Lemma~\ref{LM:Shapiro}.
%
%
\end{proof}

\begin{remark}
	If $X$ admits a tower of coverings
	$\dots\to X'_2\to X'_1\to X'_0 =X$
	such that $\dim\HH^1(X'_r,\F_2) \ll [X'_r:X]$
	as $r$ grows, then by setting $\calF_r = (u_r)_*(\F_2)_{X'_r}$,
	where $u_r$ is the map $X'_r\to X$,
	we get a family $\{\calF_r\}_{r\in \N}$
	of $\F_2$-sheaves on $X$ with $\dim \calF_r \to \infty$
	and $h^1(\calF_r)\ll \dim \calF_r$
	as $r\to\infty$.
	
	Applying this approach to the Ramanujan complexes
	of \cite{Lubotzky_2005_explicit_constructions_of_Ramanujan_complexes}
	to construct $\F_2$-sheaves
	and then applying   the modification process
	of Construction~\ref{CN:modification-process}  to these sheaves
	gives   the explicit example considered in  
	\S\ref{subsec:explicit-example}.
%
\end{remark}

\subsection{Satisfying Every Two of The Three Prerequisites of  The Tower Paradigm}
\label{subsec:three-cond-examples}

We finish     by demonstrating that every
two of the conditions
conditions \ref{item:TH:tower-paradigm:tower}--\ref{item:TH:tower-paradigm:rate} 
of Theorem~\ref{TH:tower-paradigm} with $k=0$ are
met for some sheaved complex $(X,\calF)$.
Theorem~\ref{TH:good-quotient} will play a  
role in all of the constructions.

\begin{example}[Conditions \ref{item:TH:tower-paradigm:tower} and 
\ref{item:TH:tower-paradigm:LTC} of Theorem~\ref{TH:tower-paradigm} can be met]
	\label{EX:tower-LTC}
	Fix $q,d\geq 3$ and let $Y$, $X$	
	be simplicial complexes satisfying condition  (i)  of Theorem~\ref{TH:good-quotient}.
	Then $Y$ is a $q$-thick $d$-dimensional affine building covering $X$
	and  $X$ has an infinite tower of connected double coverings.
	The latter implies   \ref{item:TH:tower-paradigm:tower}.
	Let $\calF$ denote the constant sheaf $\F_2$ on $X$.
	Choosing $q$ sufficiently large in advance allows 
	us to apply Theorem~\ref{TH:sheaves-on-quo-of-aff-buildings}(i) to $(X,\calF)$,
	thus establishing \ref{item:TH:tower-paradigm:LTC}.
	
	We observed in Proposition~\ref{PR:loc-constant-failure} that \ref{item:TH:tower-paradigm:rate}
	does not hold for this choice of $(X,\calF)$.
\end{example}

\begin{example}[Conditions \ref{item:TH:tower-paradigm:tower} and 
\ref{item:TH:tower-paradigm:rate} of Theorem~\ref{TH:tower-paradigm} can be met]
	\label{EX:tower-and-rate}
	Again, fix $d\geq 3$
	and let $X$ be a $d$-complex satisfying condition (i)
	of Theorem~\ref{TH:good-quotient}.
	Then $X$ has an infinite tower of double coverings, hence \ref{item:TH:tower-paradigm:tower} holds.
	
	Define a sheaf $\calF$ on $X$ by 
	setting $\calF(v)=\F_2^{k}$ for all $v\in X(0)$,
	$\calF(y)=\F_2$ for all $y\in X-X(0)-X(-1)$, and setting all the restriction maps
	$\res^\calF_{y\from x}$ to be $0$.
	Then $\dim \HH^0(X,\calF)=k\cdot |X(0)|$  while $\dim \HH^1(X,\calF)=|X(1)|$,
	so if we choose $k> \frac{|X(1)|}{|X(0)|}$, then condition \ref{item:TH:tower-paradigm:rate} is  satisfied.

	Of course, $(X_z,\calF_z)$ is a   poor coboundary expander,
	in all dimension, for all $z\in X-\{\emptyset\}$,
	so condition \ref{item:TH:tower-paradigm:LTC} does not hold for $(X,\calF)$. More generally,
	this highlights the difficulty in securing \ref{item:TH:tower-paradigm:LTC} when the dimensions of
	the spaces $\{\calF(y)\}_{y\in X(1)}$
	are significantly smaller than those of $\{\calF(x)\}_{x\in X(0)}$, which is the naive
	approach to making $h^0(\calF)$   large.
\end{example}

\begin{example}[Conditions \ref{item:LTC-cond-into} and 
\ref{item:TH:tower-paradigm:rate} of Theorem~\ref{TH:tower-paradigm} can be met]
	\label{EX:LTC-and-rate}
	Fix $d\geq 3$ and a prime number $p $ that is congruent
	to $1$, $2$ or $4$ modulo $7$.
	We apply Construction~\ref{CN:good-quotient-no-double-covs}
	with $p$ to get a $d$-dimensional affine building
	$Y$ covering a simplicial complex $X$.
	Let $\calF$ be the constant sheaf $\F_2$ on $X$.
	By Proposition~\ref{PR:good-quotient-no-double-covs},
	$\dim_{\F_2} \HH^1(X,\calF)=0$, while $\dim_{\F_2}\HH^0(X,\calF)=1$,
	so \ref{item:TH:tower-paradigm:rate} holds for $(X,\calF)$.
	In addition, choosing $p$ sufficiently large in advance allows
	us to apply Theorem~\ref{TH:sheaves-on-quo-of-aff-buildings}(i),
	which tells us that \ref{item:TH:tower-paradigm:LTC} holds for $(X,\calF)$.
	
	However, \ref{item:TH:tower-paradigm:tower} fails in this case
	because $X$ has no double coverings. Indeed, it is well-known
	that the double coverings of $X$
	are classified by $\HH^1(X,\F_2)$, which is $0$ in our case.
\end{example}

\appendix

\chapter*{Appendices}
\addcontentsline{toc}{chapter}{Appendices}

\renewcommand\thesection{\Alph{section}} 

\section{Sheaves on Simplicial Comlexes versus Sheaves on Topological Spaces}

\label{sec:comparison}

In this appendix we explain the relation between the sheaves on simplicial
complexes defined in this paper (Section~\ref{sec:sheaves})
and the well-known sheaves on topological spaces.
Notably, we will show that sheaves on simplicial complexes
can be realized as sheaves on certain topological spaces in such a way 
that the   cohomologies agree.
The comparison will lead to a definition of the pushforward of a sheaf
along an arbitrary morphism of simplicial complex, extending the
definition given in \S\ref{subsec:pushforward} for dimension preserving maps.

We have made the first two subsections of this appendix accessible
to readers with no prior knowledge of sheaves. However, the more advanced
topics considered in the remaining sections require some familiarity
with pushforward, pullback and sheaf cohomology; the relevant background material can be found in
\cite{Iversen_1986_cohomology_of_sheaves}, for instance.

Throughout, simplicial complexes are allowed to be infinite.
We denote the category of sheaves on a simplicial
complex $X$ by $\Sh(X)$ (cf.\ Remark~\ref{RM:sheaf-category}).

\subsection{Sheaves on Topological Spaces: a Quick Introduction}

Let $Y$ be a topological space.
Recall that a  sheaf (of abelian groups) $\calF$ on  $Y$ consists of 
\begin{enumerate}[label=(\arabic*)]
	\item an abelian group $\calF(U)$ for every open subset $U\subseteq Y$ and
	\item a group homomorphism $\res^{\calF}_{V\from U}:\calF(U)\to \calF(V)$
	for every $V\subseteq U$ open in $Y$ 	
\end{enumerate} 
such that the following conditions are met:
\begin{enumerate}[label=(S\arabic*)]
	\item $\res^{\calF}_{U\from U} = \id_{\calF(U)}$ for every open
	$U\subseteq Y$.
	\item $\res^{\calF}_{W\from V}\circ \res^{\calF}_{V\from U}=\res^{\calF}_{W\from U}$
	for all $W\subseteq V\subseteq U$ open in $Y$.
	\item \label{item:sheaf-topl:gluing} Given open subsets $\{U_i\}_{i\in I}$  of $Y$ 
	and elements $f_i\in \calF(U_i)$ for all $i\in I$ such that,   for all $i,j\in I$,
	we have $\res^{\calF}_{U_i\cap U_j\from U_i} f_i = 
	\res^{\calF}_{U_i\cap U_j\from U_j} f_j$ in $\calF(U_i\cap U_j)$, there exists a unique
	$f\in \calF(U)$, where $U=\bigcup_{i\in I}U_i$, such that $f_i=\res^{\calF}_{U_i\from U}f$ for all $i\in I$.
\end{enumerate}
The maps $\res^{\calF}_{V\from U}$ are called   \emph{restriction maps} 
and elements of $\calF(U)$ are called \emph{$U$-sections}, or just \emph{sections}.
Elements of $\calF(X)$ are called \emph{global sections}.
It is common to abbreviate $\res^{\calF}_{V\from U}f$ to $f|_{U\to V}$ or $f|_V$.
The abelian group $\calF(U)$ is also written $\Gamma(U,\calF)$.

If $\F$ is a field, then a sheaf of $\F$-vector spaces on $Y$
is defined similarly, by requiring each $\calF(U)$ to be an $\F$-vector space
and each restriction map to be an $\F$-linear map.
In the same manner, one can define sheaves of groups, $R$-modules, sets (the restriction maps
are arbitrary functions), and so on.

\begin{remark}\label{RM:sheaf-on-emptyset}
	Condition (S3) is also required to hold with $I=\emptyset$,
	in which case $\{U_i\}_{i\in I}$ is an empty
	collection and $U$ must be $\emptyset$. This choice of $\{U_i\}_{i\in I}$
	tells us that $\calF(\emptyset)$ is the trivial group.
\end{remark}

The most fundamental example of a sheaf on $Y$ is obtained by setting
\[
\calF(U) = \{f:U\to \R\suchthat\text{$f$ is continuous}\},
\]
with $\res^{\calF}_{U\from V}$ being given by $\res^{\calF}_{U\from V} f= f|_V$
(the right hand side is the restriction of $f$ to a function from $V$ to $\R$).
The addition law in $\calF(U)$ is point-wise addition.
Notice that  in this case,
$\res^{\calF}_{U\from V}$ is literally the restriction-of-domain
operation. Conditions (S1)--(S3) now become to the following 
simple facts:
\begin{enumerate}[label=(S\arabic*$'$)]
	\item If $U $ is open in $Y$ and $f:U\to\R$ is continuous, then $f|_U=f$.
	\item If $W\subseteq V\subseteq U$ are open in $Y$
	and $f:U\to \R$ is continuous, then $(f|_V)|_W=f|_W$.
	\item  Given open subsets $\{U_i\}_{i\in I}$  of $Y$
	and, for each $i\in I$, a
	continuous function $f_i:U_i\to \R$ 
	such that $f_i|_{U_i\cap U_j}=f_j|_{U_i\cap U_j}$
	for all $i,j\in I$,
	then all the $f_i$ glue  uniquely  to a
	continuous function $f:\bigcup_{i\in I}U_i\to \R$
	such that $f|_{U_i}=f_i$ for all $i\in I$.
\end{enumerate}
The sheaf $\calF$ is actually a sheaf of $\R$-vector spaces.

Similarly, given an abelian group $A$, we could define a sheaf $\calF_A$ on $Y$
by setting $\calF_A(U)$ to be the abelian group of \emph{all} functions from $U$ to $A$,
and again define the restriction maps by restriction of the domain.
If $A$ were a topological group, we could replace ``all'' with ``continuous'' and get
a sheaf as well; the example from the previous paragraph is the special case $A=\R$.

In light of the previous examples, the concept of a sheaf on $Y$
can be seen as axiomatizing an ensemble of ``good'' (e.g.\ continuous) functions
from open subsets of $Y$ to some fixed target space, but without specifying what ``good'' means, or what
is the target. 

The following example is one reason why elements 
of $\calF(U)$ are called sections.

\begin{example}
	Let $X$ be another topological space and let $p:X\to Y$
	be a continuous function. Recall that a (continuous) section
	of $p$ is a continuous function $f:Y\to X$ such that $p\circ f=\id_Y$.
	More generally, given an open subset $U\subseteq Y$,
	we say that a continuous function $f:U\to Y$ is a section of $p$
	if $p\circ f=\id_U$. Denote by
	$\calF_p(U)$   the set of sections $f:U\to Y$ of $p$.
	Then $\calF_p$ defines a set-sheaf on $Y$ by setting $\res^{\calF_p}_{V\from U} f=f|_V$.
	Moreover, the $U$-sections of $\calF_p$ are the exactly the sections of $p$
	defined on $U$. 
\end{example}

If $\calF$ and $\calG$ are sheaves on $Y$, then
a \emph{morphism}  $\vphi$ from $\calF$ to $\calG$ consists
of a group homomorphism $\vphi_U:\calF(U)\to \calG(U)$
for every open $U\subseteq Y$
such that 
\[\vphi_V\circ \res^{\calF}_{V\from U}= \res^{\calG}_{V\from U}\circ \vphi_U\]
for all $V\subseteq U$ open in $Y$. If $\calF$ and $\calG$ are sheaves
of $\F$-vector spaces (resp.\ rings, sets, etc.), then 
we instead require that each $\vphi_U$ is a linear
transformation (resp.\ ring homomorphism, any function, etc.).
The sheaves on $Y$ and the morphisms between them form a category denoted
$\Sh(Y)$.

\subsection{Sheaves on Simplicial Complexes as Sheaves on Topological Spaces}
\label{subsec:sheaves-as-ordinary-sheaves}

Let $X$ be a simplicial complex.
We say that a subset $U\subseteq X-\{\emptyset\}$ is \emph{simplicially open} (in $X$)
if $x\in U$ implies that $X_{\supseteq x}\subseteq U$.
Informally, the set $X_{\supseteq x}$ may be regarded
as the smallest simplicial neighborhood  of $x$ in $X$. A subset of $U$
of $X-\{\emptyset\}$ is therefore simplicially open if and only if it contains a   simplicial
neighborhood of every face in $U$. 
The collection of simplicially open sets forms a topology on $X-\{\emptyset\}$,
and we denote by $\topl{X}$ the resulting topological space. By design, the subcollection
$\{X_{\supseteq x}\where x\in X-\{\emptyset\}\}$ is a basis of $\topl{X}$.

Let $\calG$ be a sheaf on $\topl{X}$ and let $U\subseteq \topl{X}$ be an open subset.
Condition \ref{item:sheaf-topl:gluing} in the definition of a sheaf on a topological space
implies that we can recover
$\calG(U)$ (up to isomorphism) by knowing the groups
$\{\calG(X_{\supseteq x})\where x\in \topl{X}\}$
and the restriction maps between them.
More precisely, $\calG(U)$ may be naturally
identified with the set of ensembles
$(g_x)_{x\in U}$ where $g_x\in \calG(X_{\supseteq x})$ for all $x\in U$ and such
that $g_x|_{X_{\supseteq x}\cap X_{\supseteq y}} = g_y|_{X_{\supseteq x}\cap X_{\supseteq y}}$
for all $x,y\in U$. 
Indeed,   such a collection $(g_x)_{x\in U}$ determines
a  unique $g\in \calF(U)$ with $g|_{X_{\supseteq x}}=g_x$ for all $x\in U$. 
Note also that   $X_{\supseteq x}\cap X_{\supseteq y}$ is $X_{\supseteq x\cup y}$ if $x\cup y$ is a face
of $X$,
and $\emptyset$ otherwise. Thus, the condition 
on the $(g_x)_{x\in X}$ is equivalent to having
$g_x|_{X_{\supseteq z}}=g_y|_{X_{\supseteq z}}$ whenever $x,y,z\in U$ and $x,y\subseteq z$.
Taking $y=z$ or $x=z$, this is in turn  equivalent to having
$g_x|_{X_{\supseteq y}}=g_y$ for all $x,y\in U$ with $x\subsetneq y$. 
Now, abbreviating $\calG(X_{\supseteq x})$ to $\smp{\calG}(x)$ and $\res^{\calG}_{X_{\supseteq y}
\from X_{\supseteq x}}$ to $\res^{\smp{\calG}}_{y\from x}$ for every $x,y\in U$ with $x\subsetneq y$,
we find that $\smp{\calG}$ is a sheaf on $X$ in the sense of \S\ref{subsec:sheaves},
and we can recover $\calG$ (up to isomorphism) from $\smp{\calG}$ via
\[
\calG(U)\cong \{(g_x)_{x\in U}\in \prod_{x\in U}\smp{\calG}(x)
\suchthat 
\text{$\res^{\smp{\calG}}_{y
\from x} g_x =g_y$ for all $x,y\in U$ with $x\subsetneq y$}\},
\]
where the isomorphism is given by $g\mapsto (g|_{X_{\supseteq x}})_{x\in U}$.
To conclude, each sheaf $\calG$ on $\topl{X}$ determines a sheaf $\smp{\calG}$ on $X$,
and we can recover $\calG$ from $\smp{\calG}$.

Conversely, we may start with a sheaf $\calF$ on $X$ in the sense of \S\ref{subsec:sheaves}
and construct a sheaf $\topl{\calF}$ on $\topl{X}$ 
as follows:
Given an open subset $U\subseteq \topl{X}$, let $\topl{\calF}(U)$ denote
the set of $(f_x)_{x\in U}\in \prod_{x\in U}\calF(x)$ such that
$\res^{\calF}_{y\from x} f_x= f_y$ for all $x,y\in U$ with $x\subsetneq y$.
Then, given $V\subseteq U$ open in $\topl{X}$, define $\res^{\topl{\calF}}_{V\from U}:\topl{\calF}(U)
\to \topl{\calF}(V)$ by $(f_x)_{x\in U}\mapsto (f_x)_{x\in V}$. It is routine to check
that this defines a sheaf on $\topl{X}$.   

As we shall now see,
up to sheaf isomorphism, the constructions $\calG\mapsto \smp{\calG}$
and $\calF\mapsto \topl{\calF}$ are   inverse to each other.
Thus, sheaves on the topological space $\topl{X}$ and sheaves
on the simplicial complex $X$ are essentially the same thing.
Here is a precise statement:

\begin{thm}\label{TH:sheaf-category-equiv}
	The assignment $\calG\mapsto \smp{\calG}$
	extends naturally to a functor $\Sh(\topl{X})\to \Sh(X)$,
	and the assignment $\calF\mapsto\topl{\calF}$
	extends naturally to a functor $\Sh(X)\to \Sh(\topl{X})$.
	These functors are mutual  inverses, up to natural isomorphism.
\end{thm}

\begin{proof}[Proof (sketch)]
	The extension of $\calG\mapsto \smp{\calG}$ (resp.\ $\calF\mapsto\topl{\calF}$)
	to a functor is straightforward, but we include it for the sake of completeness.
	Given a morphism $\vphi:\calG_1\to\calG_2$ between two sheaves
	on $\topl{X}$, define $\smp{\vphi}:\smp{\calG_1}\to \smp{\calG_2}$
	by $\smp{\vphi}_x=\vphi_{X_{\supseteq x}}$ for all $x\in X-\{\emptyset\}$.
	Given a morphism $\psi:\calF_1\to\calF_2$ between two sheaves
	on $X$, define $\topl{\psi}:\topl{\calF_1}\to\topl{\calF_2}$
	by $\topl{\psi}_U((f_x)_{x\in U})=(\psi_x f_x)_{x\in U}$ for all open
	$U\subseteq \topl{X}$.
	We leave it to the reader to check that these constructions
	determine functors $(-)\mapsto \smp{(-)}:\Sh(\topl{X})\to \Sh(X)$
	and $(-)\mapsto \topl{(-)}:\Sh(X)\to\Sh(\topl{X})$,
	and proceed with showing   that these functors   are  
	inverse to each other up to natural isomorphism.
	
	Given a sheaf $\calF$   on $X$ and $x\in X-\{\emptyset\}$, observe
	that 
	\[\smp{(\topl{\calF})}(x)=\topl{\calF}(X_{\supseteq x})=
	\{(f_y)_{y}\in \prod_{y\in X_{\supseteq x}}\calF(y) \suchthat 
	\text{$f_x|_y=f_y$ for all $y\in X_{\supseteq x}-\{x\}$}\} . \]
	Define $\psi_{\calF,x}:\smp{(\topl{\calF})}(x)\to \calF(x)$
	by $\psi_{\calF,x}((f_y)_{y\supseteq x})=f_x$.
	It is routine to check that $\psi_{\calF}:=\{\psi_{\calF,x}\}_{x\in X-\{\emptyset\}}$
	is a sheaf morphism from $\smp{(\topl{\calF})} $ to $\calF$, with inverse
	given by $\psi^{-1}_{\calF,x}(f_x)=(f_x|_y)_{y\in X_{\supseteq x}}$.
	Moreover, it is straightforward to check that $\psi_{\calF}:\smp{(\topl{\calF})}
	\to \calF$ is natural in $\calF$.
	
	Next, let $\calG$ be a sheaf on $\topl{X}$. Then for any
	open $U\subseteq \topl{X}$, we have
	\[
	\topl{(\smp{\calG})}(U)=
	\{(g_x)_x\in \prod_{x\in U}\calG(X_{\supseteq x})
	\suchthat 
	\text{$g_x|_{X_{\supseteq y}}=g_y$ for all $y\in X_{\supseteq x}$}\}.
	\]
	Using this, 
	define $\vphi_{\calG,U}:\calG(U)\to \topl{(\smp{\calG})}(U)$
	by $\vphi_{\calG,U}(f)=(f|_{X_{\supseteq x}})_{x\in U}$.
	It is straightforward to check that $\vphi_{\calG}:=\{\vphi_{\calG,U}\}_{\text{$U$ open in 
	$\topl{X}$}}$ defines a morphism of sheaves from $\calG$ to $\topl{(\smp{\calG})}$.
	Moreover,
	we observed earlier that condition \ref{item:sheaf-topl:gluing} (and the nullity
	of $\calG(\emptyset)$, see Remark~\ref{RM:sheaf-on-emptyset})
	implies that each $\psi_{\calG,U}$ is bijective, so $\psi_{\calG}$ is a sheaf
	isomorphism. Checking that $\psi_{\calG}: \calG\to \topl{(\smp{\calG})}$
	is natural in $\calG$ is routine.  This completes the proof.
\end{proof}

\begin{remark}
	A similar argument shows that categories
	of sheaves of $\F$-vector spaces (resp.\
	groups, rings, sets, etc.) over $X$ and   $\topl{X}$ are equivalent.\footnote{
		In fact, it is enough to show this for set-valued sheaves, since
		all other types of sheaves can be defined internally within
		the \emph{topoi}  of sheaves on $X$ and $\topl{X}$.	
	}
\end{remark}

\subsection{Comparing Additional Structure: Pullback, Pushforward and Cohomology}
\label{subsec:comparision-of-structure}

Keep the notation of \S\ref{subsec:sheaves-as-ordinary-sheaves}.
Having  identified sheaves on $X$ with sheaves on $\topl{X}$,
we turn to show that this identification respects pullback, pushforward and cohomology.
Thus, the theory of sheaves on simplicial complexes introduced in Section~\ref{sec:sheaves}
is really a special case of the theory of sheaves on a topological space,
which can be described in a more elementary way using the combinatorics of the simplicial complex at hand.

\medskip

We begin with the following lemma.

\begin{lem}
	Let $f:Y\to X$ be a morphism of simplicial complexes (see~\S\ref{subsec:complexes}),
	and
	let $\topl{f}$ denote
	the induced map $f:Y-\{\emptyset\}\to X-\{\emptyset\}$.
	Then $\topl{f}:\topl{Y}\to \topl{X}$
	is continuous.
\end{lem}

\begin{proof}
	Let $x\in X-\{\emptyset\}$. Then $(\topl{f})^{-1}(X_{\supseteq x})
	=\cup_{y\in f^{-1}(x)}Y_{\supseteq y}$, which is open in $\topl{Y}$.
	Since the sets $\{X_{\supseteq x}\where x\in X-\{\emptyset\}\}$
	form a basis to $\topl{X}$, this means that $\topl{f}$ 
	is continuous.  
\end{proof}

Let $Y$ and $Y'$ be topological spaces and let $u:Y'\to Y$
be a continuous map.
Given a sheaf $\calG'$ on $Y'$, recall that the \emph{pushforward} of $\calG'$
along $u$ is the sheaf $u_*\calG'$ determined by 
\[u_*\calG'(U)=\calG'(u^{-1}(U))
\qquad\text{and}\qquad
\res^{u_*\calG'}_{V\from U} = \res^{\calG'}_{u^{-1}(V)\from u^{-1}(U)}
\]
for all $V\subseteq U$ open in $Y$.\footnote{
	Recommended exercise for beginners: check that $u_*\calG'$ is a sheaf on $Y$.
}
The counterpart of this construction is the \emph{pullback},
which takes a sheaf  $\calG$ on $Y$  and produces
a sheaf $u^*\calG$ on $Y'$. In contrast with pullback of sheaves
on simplicial complexes (see \S\ref{subsec:pushforward}), the construction of $u^*\calG$ is somewhat more involved
and can be found in \cite[II.\S4]{Iversen_1986_cohomology_of_sheaves}
or \cite[Tag \href{https://stacks.math.columbia.edu/tag/008C}{008C}]{DeJong_2020_stacks_project}, for instance.
The functor   $u^*:\Sh(Y)\to \Sh(Y')$ can be implicitly
defined  as the \emph{left adjoint} of $u_*:\Sh(Y')\to\Sh(Y)$.

Under the equivalence of Theorem~\ref{TH:sheaf-category-equiv},
pushforward and pullback of sheaves on simplicial complexes
 corresponds
to  pushforward and pullback of sheaves on the associated topological spaces.
Formally:

\begin{thm}\label{TH:pullback-comparison}
	Let $u:Y\to X$ be a morphism of simplicial complexes, let $\calF$
	be a sheaf on $X$ and let $\calG$ be a sheaf on $Y$. Then:
	\begin{enumerate}[label=(\roman*)]
		\item If $u$ is dimension-preserving (see~\S\ref{subsec:complexes}),
		then there is a natural isomorphism $\topl{(u_*\calG)}\cong (\topl{u})_*\topl{\calG}$.
		\item 
		There is a natural isomorphism $\topl{(u^*\calF)}\cong (\topl{u})^*\topl{\calF}$.
	\end{enumerate}
\end{thm}

\begin{proof}
	(i) We will actually show that $\topl{(u_*\calG)}= (\topl{u})_*\topl{\calG}$.
	Let $U\subseteq \topl{X}$ be an open subset.
	Then 
	\[(\topl{u})_*\topl{\calG}(U)=\topl{\calG}(u^{-1}(U))
	=\{(g_y)_{y\in u^{-1}(U)}\in \prod_{y\in u^{-1}(U)}\calG(y) \suchthat 
	\text{$\res^{\calG}_{y'\from y} g_y=g_{y'}$ whenever $y\subsetneq y'$}\}.\]
	On the other hand, 
	\[
	\topl{(u_*\calG)}(U) = \{(\tilde{g}_x)_{x\in U}\in \prod_{x\in U} u_*\calG(x)
	\suchthat
	\text{$\res^{u_*\calG}_{x'\from x} \tilde{g}_x = \tilde{g}_{x'}$
	whenever $x\subsetneq x'$}\}.
	\]
	Recall from \S\ref{subsec:pushforward} that
	$u^*\calG(x)=\prod_{y\in u^{-1}(x)} \calG(y)$,
	so each $\tilde{g}_x$ is a collection $(g_y)_{y\in u^{-1}( x )}$
	with $g_y\in \calG(y)$ for all $y$.
	Moreover, for all $x\subsetneq x'$ in $U$, we have
	\[
	\res^{u_*\calG}_{x'\from x} \tilde{g}_x = 
	\res^{u_*\calG}_{x'\from x}((g_y)_{y\in u^{-1}( x )})
	= (\res^{\calG}_{y'\from y'(x)}g_{y'(x)})_{y'\in u^{-1}(x')},
	\]
	where $y'(x)$ denotes the unique face of $y'$ mapping to $x$ (it is unique
	because $u$ is dimension-preserving).
	Thus, the condition $\res^{u_*\calG}_{x'\from x} \tilde{g}_x = \tilde{g}_{x'}$
	is equivalent to having $\res^{\calG}_{y'\from y }g_y = g_{y'}$
	for all $y'\in u^{-1}(x')$ and $y\in u^{-1}(x)$ with $y\subsetneq y'$.
	Now, identifying $(\tilde{g}_{x})_{x\in U}=((g_y)_{y\in u^{-1}(x)})_{x\in U}$
	with $(g_y)_{y\in u^{-1}(U)}$, we see that
	$\topl{(u_*\calG)}(U) = (\topl{u})_*\topl{\calG}(U)$.
	
	A similar argument shows that for every $V\subseteq U$ open in $\topl{X}$, we have
	$\res^{\topl{(u_*\calG)}}_{V\from U} = \res^{(\topl{u})_*\topl{\calG}}_{V\from U}$,
	so $\topl{(u_*\calG)}= (\topl{u})_*\topl{\calG}$. That this isomorphism is natural 
	in $\calG$ is routine.
	
	(ii) Unfortunately, we shall need to unfold the definition of the pullback $(\topl{u})^*\topl{\calF}$.
	We use the definition
	in \cite[Tag \href{https://stacks.math.columbia.edu/tag/008C}{008C}]{DeJong_2020_stacks_project},
	which makes use of \emph{presheaves}\footnote{
		Presheaves on  a topological space are defined   like sheaves, but without the requiring condition
		\ref{item:sheaf-topl:gluing}.
	}, \emph{sheafification} and \emph{stalks}; 
	see
	\cite[Tag \href{https://stacks.math.columbia.edu/tag/006A}{006A}]{DeJong_2020_stacks_project} 
	for details.
	Recall that for a sheaf $\calH$ on $\topl{X}$,
	the pullback $(\topl{u})^*\calH$ is the sheafification of the presheaf $\calP$ on $Y$ given
	by
	$\calP(U)=\dirlim_{V\supseteq u(U)} \calH(V)$, where $V$ ranges over the open subsets of
	$\topl{X}$ containing $u(U)$. Fortunately, in our situation, every subset
	$T\subseteq \topl{X}$ admits a minimal open subset  containing it,
	namely, $T^\wedge := \bigcup_{x\in T} X_{\supseteq x}$.
	The definition of the presheaf $\calP $ therefore simplifies to 
	$\calP(U)=\calH(u(U)^\wedge)$. Taking $\calH=\topl{\calF}$, we find that
	$(\topl{u})^*\topl{\calF}$ is the sheafification of the presheaf $\calP$ on $Y$ given by
	\[
	\calP(U) =\topl{\calF}(u(U)^\wedge)
	=\{(f_x)_{x}\in \prod_{x\in u(U)^\wedge}\calF(x)
	\suchthat
	\text{$\res^{\calF}_{x'\from x} f_x =f_{x'}$ whenever $x\subsetneq x'$}\}
	\]
	and $\res^{\calP}_{V\from U} ((f_x)_{x\in u(U)^\wedge})=
	(f_x)_{x\in u(V)^\wedge}$ for all $V\subseteq U$ open in $\topl{Y}$. 
	On the other hand, 
	by unfolding the definitions, we find that
	\[
	(u^*\calF)^\circ(U)=\{(f_y)_y\in \prod_{y\in U}\calF(u(y))\suchthat
	\text{$\res^{\calF}_{u(y')\from u(y)} f_y=f_{y'}$
	whenever $y\subsetneq y'$}\}.
	\]
	Define $\vphi_U:\calP(U)\to (u^*\calF)^\circ(U)$ by 
	$\vphi_U((f_x)_{x\in u(U)^\wedge})=(f_{u(y)})_{y\in U}$.
	It is routine to check that this is well-defined and that $\vphi=(\vphi_U)_{U}$
	is a morphism of presheaves from $\calP$ to $(u^*\calF)^\circ$.
	
	By the universal property of sheafification
	\cite[Tag \href{https://stacks.math.columbia.edu/tag/0080}{0080}]{DeJong_2020_stacks_project}, $\vphi$
	determines a sheaf morphism $\vphi^{\rm a}$ from $(\topl{u})^*\topl{\calF}$
	(the sheafification of $\calP$) to $(u^*\calF)^\circ$.
	Moreover, in order to show that $\vphi^{\rm a}$ is an isomorphism,
	it is enough to check that $\vphi:\calP\to (u^*\calF)^\circ$
	induces an isomorphism at the stalks 
	\cite[Tags \href{https://stacks.math.columbia.edu/tag/007Z}{007Z},
	\href{https://stacks.math.columbia.edu/tag/007T}{007T}]{DeJong_2020_stacks_project}.
	Recall that if $\calH$ is a presheaf on $Y$, then the stalk of $\calH$
	at $y\in \calH$ is $\calH_y:=\dirlim_{V\ni y} \calH(V)$ where
	$V$ ranges over the open sets containing $y$.
	In our situation, there is a unique minimal open subset of $\topl{Y}$
	containing $y$, namely $Y_{\supseteq y}$, so the stalk $\calH_y$
	is just $\calH(Y_{\supseteq y})$. We are therefore reduced
	to showing that $\vphi_{Y_{\supseteq y}}: \calP(Y_{\supseteq y})\to (u^*\calF)^\circ(Y_{\supseteq y})$
	is an isomorphism
	for all $y\in Y$. Write $x=u(y)$. Then $u(Y_{\supseteq y})^\wedge = X_{\supseteq x}$.
	It is   routine to check that $(f_{y'})_{y'\in Y_{\supseteq y}}\mapsto 
	(\res^{\calF}_{x'\from x}f_y)_{x'\in X_{\supseteq x}}$ defines
	an inverse to $\vphi_{Y_{\supseteq y}}$. This shows that 
	$\vphi^{\rm a}: (\topl{u})^*\topl{\calF}\to (u^*\calF)^\circ$
	is a sheaf isomorphism.
	
	One readily checks that the formation of $\calP$ is funtorial in $\calF$
	and that $\vphi:\calP\to (u^*\calF)^\circ$ is natural
	in $\calF$, so $\vphi^{\rm a}: (\topl{u})^*\topl{\calF}\to (u^*\calF)^\circ$ is
	also natural in $\calF$.
\end{proof}

\begin{remark}
	Theorem~\ref{TH:pullback-comparison}(i) suggests a way
	to define the pushforward of a sheaf on a simplicial complex along an 
	\emph{arbitrary} morphism
	of simplicial complexes. Specifically, if $u:Y\to X$ is such a morphism
	and $\calG$ is a sheaf on $Y$, define $u_*\calG$ to be
	$\smp{((\topl{u})_*\topl{G})}$. (This is conceptually correct because
	$u_*:\Sh(Y)\to\Sh(X)$ is a right adjoint of $u^*:\Sh(X)\to\Sh(Y)$.) Unfolding this definition, 
	we find that  for $x\in X-\{\emptyset\}$,  we have
	\[
	u_*\calG(x)=\{(f_y)_{y}\in \prod_{y\in u^{-1}(x)^{\wedge}}  \calG(y)
	\suchthat
	\text{$\res^{\calG}_{y'\from y} f_y = f_{y'}$ whenver
	$y\subsetneq y'$}\},
	\]
	where $u^{-1}(x)^{\wedge}=\bigcup_{y\in u^{-1}(x)} Y_{\supseteq y}$,
	and the restriction maps $\res^{u_*\calG}_{x'\from x}:u_*\calG(x)\to u_*\calG(x')$
	are given by forgetting  coordinates.
	It is not difficult to check that 
	$(f_y)_{y\in u^{-1}(x)^{\wedge}}\mapsto (f_y)_{y\in u^{-1}(x)}$
	defines an isomorphism from $u_*\calG(x)$
	to 
	\[
	\{(f_y)_{y}\in \prod_{y\in u^{-1}(x) }   \calG(y)
	\suchthat
	\text{$\res^{\calG}_{y'\from y} f_y = f_{y'}$ whenver
	$y\subsetneq y'$}\},
	\]
	so we may take the latter as the definition of $u^*\calG(x)$. The restriction
	maps are then   given by $\res^{u_*\calG}_{x'\from x}(f_y)_{y\in \pi^{-1}(x)}
	=(\res_{y'\from y'(x)}f_{y'(x)})_{y'\in \pi^{-1}(x')}$,
	where $y'(x)$ is an arbitrary face of $y'$ mapping to $x$ (its choice is inconsequential).
	When $u:Y\to X$ is dimension preserving, no   face  in $u^{-1}(x)$
	contains another such face, so $u^*\calG(x)=\prod_{y\in u^{-1}(x)} \calG(y)$
	and we recover  the definition
	of the pullback given in \S\ref{subsec:pushforward}.
\end{remark}

	Recall that if $Y$ is a topological space
	and $\calG$ is a sheaf on $Y$, then we write $\Gamma(Y,\calG)$ for $Y(\calG)$,
	the group of    global sections  of $\calG$.
	Letting $\calG$ vary, the 
	assignment $\Gamma(Y,-)$ defines a left exact functor from $\Sh(Y)$ to abelian groups,
	and  its right derived functors are denoted $\{\HH^i(Y,-)\}_{i\geq 0}$.
	The group  $\HH^i(Y,\calG)$ is the \emph{$i$-th cohomology group} of the sheaf $\calG$;
	see \cite[II.\S3]{Iversen_1986_cohomology_of_sheaves} for further details.
	For example, $\HH^0(Y,\calG)$ is just $\Gamma(Y,\calG)=\calG(Y)$.

	We finish this section by showing that  the equivalence of Theorem~\ref{TH:sheaf-category-equiv} 
	is compatible with taking cohomology.
	More precisely:
	
	\begin{thm}\label{TH:comparison-of-cohs}
		Let $X$ be a simplicial complex.
		\begin{enumerate}[label=(\roman*)]
			\item For every sheaf $\calF$ on $X$ and $i\geq 0$, there is an
			isomorphism $\HH^i(X,\calF)\cong \HH^i(\topl{X},\topl{\calF})$
			natural in $\calF$.
			\item If $0\to \calF \to \calF' \to \calF''\to 0$
			is a short exact sequence of sheaves on $X$,
			then $0\to\topl{\calF}\to\topl{\calF'}\to \topl{\calF''}\to 0$
			is a short exact sequence of sheaves on $\topl{X}$, and
			there is a commutative diagram
			\[
			\xymatrix{
			\cdots \ar[r] &
			\HH^i(X,\calF) \ar[r] \ar[d]
			&
			\HH^i(X,\calF') \ar[r] \ar[d]
			&
			\HH^i(X,\calF'') \ar[r] \ar[d]
			&
			\HH^{i+1}(X,\calF) \ar[r] \ar[d]
			&
			\cdots \\
			\cdots \ar[r] &
			\HH^i(\topl{X},\topl{\calF}) \ar[r] 
			&
			\HH^i(\topl{X}, \topl{\calF'}) \ar[r]
			&
			\HH^i(\topl{X}, \topl{\calF''}) \ar[r]
			&
			\HH^{i+1}(\topl{X},\topl{\calF}) \ar[r] 
			&
			\cdots
			}
			\]		
			in which the rows are the long cohomology exact sequences associated
			to 	$0\to \calF \to \calF' \to \calF''\to 0$
			(see \S\ref{subsec:sheaf-coh}) and $0\to\topl{\calF}\to\topl{\calF'}\to \topl{\calF''}\to 0$,
			and the vertical maps are isomorphism  from (i).
		\end{enumerate}
	\end{thm}
		
	\begin{proof}
		The categories $\Sh(X)$ and $\Sh(\topl{X})$
		are abelian, so the equivalence $\calF\mapsto \topl{\calF}:\Sh(X)\to \Sh(\topl{X})$
		of Theorem~\ref{TH:sheaf-category-equiv} is necessarily exact.
		This shows that the sequence
		$0\to\topl{\calF}\to\topl{\calF'}\to \topl{\calF''}\to 0$
		in (ii) is exact.
		The equivalence  also implies that the right derived
		functors of $\calF\mapsto \Gamma(\topl{X},\topl{\calF})$
		from $\Sh(X)$ to $\mathrm{Ab}$ --- the category of abelian groups ---
		are $\{\calF\mapsto \HH^i(\topl{X},\topl{\calF})\}_{i\geq 0}$.
		We will show in Appendix~\ref{sec:derived} that the functors
		$\{\HH^i(X,-)\}_{i\geq 0}$ are    right derived functors
		of $\HH^0(X,-):\Sh(X)\to \mathrm{Ab}$. Since derived functors are unique
		up to a natural isomorphism, the theorem will follow
		if we show that $\HH^0(X,\calF)$
		is naturally isomorphic to $\Gamma(\topl{X},\topl{\calF})=\topl{\calF}(\topl{X})$.
		This is the content of the following lemma. 
	\end{proof}
	
	\begin{lem}
		Let $X$ be a simplicial complex. 
		There is a natural isomorphism
		$\HH^0(X,\calF)\cong \topl{\calF}(\topl{X})$.
	\end{lem}
	
	\begin{proof}
		Unfolding the definitions, we find that
		\[\topl{\calF}(\topl{X})=\{(f_x)_x\in \prod_{x\in X-\{\emptyset\}} \calF(x)
		\suchthat
		\text{$\res^{\calF}_{y\from x}f_x =f_y$ whenever $x\subsetneq y$}\},\]
		whereas
		\begin{align*}
		\HH^0(X,\calF) = \{(f_v)_{v\in X(0)} \in \prod_{v\in X(0)} \calF(v)
		\suchthat~ & 
		\text{$\res^{\calF}_{e\from v} f_v = \res^\calF_{e\from u} f_u$} 
		\\
		& \text{for all $u,v\in X(0)$ with $e=u\cup v\in X(1)$}\}.
		\end{align*}
		Define $\vphi_{\calF}:\topl{\calF}(\topl{X})
		\to \HH^0(X,\calF)$ by $\vphi_{\calF}((f_x)_{x\in X-\{\emptyset\}})=(f_v)_{v\in X(0)}$.
		It is routine to check that $\vphi_{\calF}$ is well-defined 
		and natural in $\calF$.
		
		To see that $\vphi_\calF$ is invertible, observe that
		if  $(f_v)_{v\in X(0)}\in \HH^0(X,\calF) $, $x\in X-\{\emptyset\}$
		and $u,v$ are two $0$-faces of $x$, then $f_v|_x= f_u|_x$.
		Indeed, this is clear if $u=v$, and otherwise, we  have
		$f_v|_{u\cup v} =f_u|_{u\cup v}$ and applying $\res^{\calF}_{x\from u\cup v}$
		to both sides gives the desired equality.
		This allows us to define $\psi_{\calF}:\HH^0(X,\calF)\to \topl{\calF}(\topl{X})$
		by $\psi_{\calF}((f_v)_{v\in X(0)})=(f_{v(x)}|_x)_{x\in X-\{\emptyset\}}$,
		where $v(x)$ is an arbitrary $0$-face of $x$. It is straightforward to
		check that $\psi_{\calF}$ is  defines an inverse to $\vphi_{\calF}$,
		so  
		$\vphi_{\calF}$ is a   isomorphism.
	\end{proof}

\subsection{Aside: Augmented Sheaves as Sheaves on Topological Spaces}

	We finish with explaining how some of the results in 
	\S\ref{subsec:sheaves-as-ordinary-sheaves} 
	and \S\ref{subsec:comparision-of-structure} may be adapted 
	to    augmented sheaves.
	
	Let $X$ be a simplicial complex.
	We let $\topl{X}_+$ denote the set $X$ together with the topology
	consisting of all simplicially
	open subsets $U\subseteq X-\{\emptyset\}$ and   the set $X$.
	As in \S\ref{subsec:sheaves-as-ordinary-sheaves}, given a sheaf
	$\calG$ on $\topl{X}_+$, we can define an \emph{augmented} sheaf
	$\smp{\calG}$ on $X$ by setting  $\smp{\calG}(x)=\calG(X_{\supseteq x})$
	and
	$
	\res^{\smp{\calG}}_{y\from x}=\res^{\calG}_{X_{\supseteq y}\from X_{\supseteq x}}
	$; mind that $x$ is allowed to be the empty face. Conversely, an augmented sheaf
	$\calF$ on $X$ gives rise to a sheaf $\topl{\calF}$ on $\topl{X}_+$
	defined using the same formulas as in the non-augmented sheaf case.
	The same argument as  in the proof of Theorem~\ref{TH:sheaf-category-equiv}
	shows that $\calG\mapsto  \smp{\calG}$ 
	defines  an equivalence of categories  
	from $\Sh(\topl{X}_+)$ to the category of augmented sheaves on $X$,
	and   $\calF\mapsto \topl{\calF}$ is its inverse up to natural isomorphism.
	Thus, augmented sheaves on the simplicial complex $X$ are essentially the same thing as sheaves
	on the topological space $\topl{X}_+$.
	
	However, in contrast with the non-augmented sheaf case,
	the equivalence     between augmented sheaves on $X$
	and sheaves on $\topl{X}_+$ does not respect cohomology.
	Rather, the dimensions are shifted by $1$, i.e., 
	there is a natural isomorphism $\HH^{i-1}(X,\calF)\cong \HH^{i}(\topl{X}_+,\topl{\calF})$
	for every augmented sheaf $\calF$ and every $i\in\N\cup \{0\}$.
	This can be shown as in the proof of Theorem~\ref{TH:comparison-of-cohs},
	except now one has to establish a natural isomorphism
	$\HH^{-1}(X,\calF)\cong \Gamma(\topl{X}_+,\topl{\calF})$
	and show that $\HH^i(X,-)$ is the $(i+1)$-th right derived functor of $\HH^{-1}(X,-)$.
	
	Finally, while we have not defined in \S\ref{subsec:pushforward} the pushforward and pullback
	of augmented sheaves on simplicial complexes, the equivalence with $\Sh(\topl{X}_+)$
	suggests a way one might   define them. That is, given a morphism of simplicial
	complexes $u:Y\to X$, an augmented sheaf $\calF$ on $X$ and an augmented sheaf
	$\calG$ on $Y$, let $u^*\calF=\smp{((\topl{u}_+)^*\topl{\calF})}$
	and $u_*\calG=\smp{((\topl{u}_+)_*\topl{\calG})}$, where $\topl{u}_+$
	is just $u$ viewed as a continuous function from $\topl{Y}_+ $
	to $\topl{X}_+ $. We leave it to the reader to work out what $u^*\calF$
	and $u_*\calG$ turn out to be. Beware, however,  that
	these constructions may   present exceptional behavior over the empty face.
	For example, if $u:Y\to X$ is a covering of degree $n$, 
	and $\aug{\F}$ denotes the constant
	augmented sheaf on $Y$ associated to a field $\F$, then $u_*(\aug{\F})(\emptyset)=\F$
	while $u_*(\aug{\F})(x)\cong \F^n$ for all $x\in X-\{\emptyset\}$.
	(The conceptual reason for this is that   $\topl{u}_+:\topl{Y}_+\to \topl{X}_+$
	is generally not a degree-$n$ covering of topological spaces.)

\section{Sheaf Cohomology is a Right Derived Functor}
\label{sec:derived}

Throughout, $X$ is a possibly-infinite simplicial complex. Recall that $\Sh(X)$ denotes the category
of sheaves on $X$ and let $\mathrm{Ab}$ denote the category 
of abelian groups. 
Then $\HH^0(X,-)$ defines a left exact functor
from $\Sh(X)$ to $\mathrm{Ab}$. The purpose of this appendix is to prove
that the higher  cohomology groups $\HH^i(X,-)$
defined in  \S\ref{subsec:sheaf-coh} are the   \emph{right derived functors}
of $\HH^0(X,-)$. 
In particular, the category $\Sh(X)$ has \emph{enough injectives}  so that 
the right derived functors of $\HH^0(X,-)$ are defined.
The necessary material about derived
functors can be found in \cite{Iversen_1986_cohomology_of_sheaves}
and \cite[Tag \href{https://stacks.math.columbia.edu/tag/010P}{010P}]{DeJong_2020_stacks_project}, for instance.

\medskip

We begin by  introducing the following construction.

\begin{construction}\label{CN:skyscraper}
Let $x\in X-\{\emptyset\}$
and let $A$ be an abelian group. Define a sheaf $A_x=A_{X,x}$ on $X$ by
\[
A_x(y) =\left\{\begin{array}{ll}
A & y\subseteq x \\
0 & y\nsubseteq x
\end{array}\right.,
\qquad
\res^{A_x}_{z\from y} =
\left\{\begin{array}{ll}
\id_A & z\subseteq x \\
0 & z\nsubseteq x
\end{array}\right.
\]
for all $y,z\in X-\{\emptyset\}$ with $y\subsetneq z$.
\end{construction}

\begin{remark}
	Under the equivalence of Theorem~\ref{TH:sheaf-category-equiv},
	the sheaf $A_x$   corresponds to a \emph{skyscraper} sheaf.
\end{remark}

\begin{lem}\label{LM:skyscrapter-coh}
	With notation as in Construction~\ref{CN:skyscraper}, we have $\HH^i(X,A_x)=0$
	for all $i\geq 1$.
\end{lem}

\begin{proof}
	Let $d=\dim x$.
	We many forget about all faces in $X$ not contained in $x$,
	and thus assume that $X$ has   a single $d$-face $x$,
	and $A_x$ is the constant sheaf $A$. 
	In this case, the topological realization $|X|$ of $X$
	is contractible, so by Corollary~\ref{CR:cohomology-comparison},
	$\HH^i(X,A)=\HH^i(|X|,A)=0$ for $i\geq 1$.
\end{proof}

\begin{lem}\label{LM:maps-to-skyscraper}
	Let $X,x,A$ be as in Construction \ref{CN:skyscraper}
	and let $\calF$ be any sheaf on $X$.
	There is a natural (in $\calF$ and $A$) bijection between
	$\Hom_{\Sh(X)}(\calF,A_{ x})$ and $\Hom_{\mathrm{Ab}}(\calF(x),A)$
	given by $\vphi\mapsto \vphi_x$.
\end{lem}

\begin{proof}
	Let us first show that $\vphi\mapsto \vphi_x: \Hom_{\Sh(X)}(\calF,A_{ x})
	\to \Hom_{\mathrm{Ab}}(\calF(x),A)$ is injective.
	Let $\psi:\calF\to A_x$ be another morphism with $\psi_x=\vphi_x$,
	and let $y\in X-\{\emptyset\}$.
	If $y\nsubseteq x$, then we must have $\psi_y=0=\vphi_x$, because $A_x(y)=0$.
	On the other hand, if $y\subseteq x$, then
	$\psi_y=\res^{A_x}_{x\from y}\circ \psi_y=\psi_x\circ \res^{\calF}_{x\from y}
	=\vphi_x\circ \res^{\calF}_{x\from y}=\res^{A_x}_{x\from y}\circ \vphi_y=
	\vphi_y$  (with $\res^{\calF}_{x\from x}$ being $\id_{\calF(x)}$). We conclude that $\psi=\vphi$.
	
	Conversely, given an abelian group homomorphism $\vphi_0:\calF(x)\to A$,
	we can define a morphism $\vphi:\calF\to A_x$ by
	setting $\vphi_y=0$ if $y\nsubseteq x$ and $\vphi_y = \vphi_0\circ \res^{\calF}_{x\from y}$
	if $y\subseteq x$ (with $\res^{\calF}_{x\from x}$ being $\id_{\calF(x)}$).
	It is routine to check that $\vphi$ is indeed a sheaf morphism 
	and $\vphi_x=\vphi_0$, so the map in the lemma is   onto.
	
	That $\vphi\mapsto \vphi_x$ is natural in $\calF$ and $A$ is   straightforward.
\end{proof}

We can now prove the following key lemma.
 
\begin{lem}\label{LM:coh-is-effaceable}
	Let $\calF$ be a sheaf on $X$. Then there exists 
	a sheaf $\calG$ on $X$ and a monomorphism $j:\calF\to \calG$
	such that $\HH^i(X,\calG)=0$ for all $i\geq 1$. Moreover, $\calG$
	can be taken to be injective in $\Sh(X)$.
\end{lem}

\begin{proof}
	Let $x\in X-\{\emptyset\}$. By Lemma~\ref{LM:maps-to-skyscraper}
	(applied with $A=\calF(x)$),
	the identity map $\id:\calF(x)\to \calF(x)$
	gives rise to  a sheaf morphism $j^{(x)}:\calF\to \calF(x)_x$.
	Let $\calG=\prod_{x\in X-\{\emptyset\}}\calF(x)_x$ (note that the product may be infinite).
	Then the morphisms $\{j^{(x)}\}_{x\in X-\{\emptyset\}}$
	determine a morphism $j:\calF\to \calG$
	given by $j_y(f)=(j^{(x)}_y(f))_{x\in X-\{\emptyset\}}$.
	Since the the $y$-component of $j_y(f)$ is just $f$, we have $\ker j=0$.
	Moreover, by Lemma~\ref{LM:skyscrapter-coh},
	$\HH^i(X,\calF(x)_x)=0$ for all $i\geq 1$,
	which means that the cochain complex
	$C^\bullet(X,\calF(x)_x)$ is exact in degrees $ \geq 1$.
	Since $C^\bullet(X,\calG)$ is the product of the cochain
	complexes $\{C^\bullet(X,\calF(x)_x)\}_{x\in X-\{\emptyset\}}$,
	it is also exact in degrees $ \geq 1$, and we conclude that
	$\HH^i(X,\calG)=0$ for $i\geq 1$.
	
	In order to choose $\calG$ that is also injective, for every $x\in X-\{\emptyset\}$,
	choose an embedding $i_x:\calF(x)\to E^{(x)}$ of $\calF(x)$
	into an injective $\Z$-module $E^{(x)}$, and use the $i_x$ to construct
	the $j^{(x)}:\calF\to (E^{(x)})_x$ and $j:\calF\to \calG:=\prod_x (E^{(x)})_x$.
	Lemma~\ref{LM:maps-to-skyscraper} implies readily that each of the
	sheaves $(E^{(x)})_x$ is injective, and therefore, so is $\calG$.
\end{proof}

\begin{thm}\label{TH:cohomology-is-derived-func}
Let $X$ be a (possibly infinite) simplicial complex. Then:
\begin{enumerate}[label=(\roman*)]
	\item The   abelian category $\Sh(X)$ has enough injectives.
	\item The functor  $ \HH^i(X,-) :\Sh(X)\to\mathrm{Ab}$
	of \S\ref{subsec:sheaf-coh} is the $i$-th right derived
	functor of $\HH^0(X,-)$.
\end{enumerate}
\end{thm}

\begin{proof} 
	(i) This is Lemma~\ref{LM:coh-is-effaceable}. Alternatively, 
	by Theorem~\ref{TH:sheaf-category-equiv}, 
	$\Sh(X)$ is equivalent to $\Sh(\topl{X})$ and the latter
	is well-known to have enough injectives \cite[II, Theorem~3.1]{Iversen_1986_cohomology_of_sheaves}. 

	(ii) 
	The derived functors of $\HH^0(X,-)$ form a universal cohomological $\delta$-functor,
	and we observed in \S\ref{subsec:sheaf-coh}
	that  the    functors $\{H^i(X,-)\}_{i\geq 0}$ 
	form a   cohomological $\delta$-functor. Since universal $\delta$-functors
	are unique up to natural isomorphism, it is enough 
	to show that the $\{H^i(X,-)\}_{i\geq 0}$  are  universal. 
	By \cite[Tag \href{https://stacks.math.columbia.edu/tag/010P}{010T}]{DeJong_2020_stacks_project}, 
	this will follow if we show that every
	sheaf $\calF$ on $X$ embeds in a sheaf $\calG$ with $\HH^i(X,\calG)=0$
	for all $i\geq 1$, and that is exactly what Lemma~\ref{LM:coh-is-effaceable}
	tells us.
\end{proof}

\begin{remark}
	Let $R$ be a ring, and let $\Sh_R(X)$ denote
	the category of sheaves of left $R$-modules.
	The cohomology groups of a sheaf in $\Sh_R(X)$
	defined in \S\ref{subsec:sheaf-coh}
	are naturally left $R$-modules, so we may regard $\HH^i(X,-)$
	as a functor from $\Sh_R(X)$ to the category of left $R$-modules,
	denoted $\lMod{R}$.
	The  same argument as in the proof of  Theorem~\ref{TH:cohomology-is-derived-func}
	can be used to show  that $\HH^i(X,-):\Sh_R(X)\to \lMod{R}$
	is the $i$-th right derived functor of $\HH^0(X,-):\Sh_R(X)\to \lMod{R}$.
\end{remark}

\bibliographystyle{alpha}
\bibliography{MyBib_18_05}



\end{document}

%% file: amsart_header_18_05.tex
\usepackage{amsfonts}
\usepackage{amssymb}
\usepackage{amsmath}
\usepackage{hyperref}
\usepackage{mathrsfs}
\usepackage{centernot}
\usepackage{mathdots}
\usepackage{stmaryrd}
\usepackage[all]{xy}

\usepackage{mathtools}

\usepackage{color}

\newtheorem{thm}{Theorem}[section]
\newtheorem{lem}[thm]{Lemma}
\newtheorem{prp}[thm]{Proposition}
\newtheorem{cor}[thm]{Corollary}
\newtheorem{dfn}[thm]{Definition}

\newtheorem{cnj}[thm]{Conjecture}

\newtheorem{que}[thm]{Question}

\theoremstyle{definition}

\newtheorem{notation}[thm]{Notation}
\newtheorem{construction}[thm]{Construction}
\newtheorem{example}[thm]{Example}
\newtheorem{remark}[thm]{Remark}

\theoremstyle{plain}

\newcommand{\rem}[1]{}

\newcommand{\gap}[1]{{\color{blue} #1}}

\newcommand{\C}{\mathbb{C}}

\newcommand{\F}{\mathbb{F}}
\newcommand{\K}{\mathbb{K}}
\newcommand{\N}{\mathbb{N}}
\newcommand{\Q}{\mathbb{Q}}
\newcommand{\R}{\mathbb{R}}
\newcommand{\Z}{\mathbb{Z}}

\newcommand{\HH}{{\mathrm{H}}}

\newcommand{\frakm}{{\mathfrak{m}}}

\newcommand{\calA}{{\mathcal{A}}}

\newcommand{\calC}{{\mathcal{C}}}

\newcommand{\calF}{{\mathcal{F}}}
\newcommand{\calG}{{\mathcal{G}}}
\newcommand{\calH}{{\mathcal{H}}}

\newcommand{\calO}{{\mathcal{O}}}
\newcommand{\calP}{{\mathcal{P}}}

\newcommand{\calV}{{\mathcal{V}}}

\newcommand{\bbA}{{\mathbb{A}}}

\newcommand{\bbK}{{\mathbb{K}}}

\newcommand{\bfG}{{\mathbf{G}}}
\newcommand{\bfH}{{\mathbf{H}}}

\newcommand{\what}[1]{\widehat{#1}}

\newcommand{\veps}{\varepsilon}
\newcommand{\vphi}{\varphi}

\newcommand{\idealof}{\unlhd} 

\newcommand{\suchthat}{\,:\,}
\newcommand{\where}{\,|\,}

\newcommand{\quo}[1]{\overline{#1}}

\newcommand{\leftmod}{{\setminus}}

\newcommand{\Trings}[1]{\left< #1 \right>}

\newcommand{\trings}[1]{\langle {#1} \rangle}

\newcommand{\floor}[1]{\lfloor {#1} \rfloor}

\newcommand{\ceil}[1]{\lceil {#1} \rceil}

\DeclareMathOperator{\Aut}{Aut} %
\DeclareMathOperator{\Char}{char} %
\DeclareMathOperator{\coker}{coker} %
\DeclareMathOperator{\End}{End} %
\DeclareMathOperator{\Hom}{Hom} %
\DeclareMathOperator{\id}{id} %
\DeclareMathOperator{\im}{im} %
\DeclareMathOperator{\ord}{ord} %
\DeclareMathOperator{\rank}{rank}
\DeclareMathOperator{\Spec}{Spec} %
\DeclareMathOperator{\supp}{supp} %
\newcommand{\nGL}[2]{\mathrm{GL}_{#2}({#1})}

\newcommand{\nSL}[2]{\mathrm{SL}_{#2}({#1})}

\newcommand{\nSU}[2]{\mathrm{SU}_{#2}({#1})}

\newcommand{\nMat}[2]{\mathrm{M}_{#2}(#1)}

\newcommand{\uSL}{{\mathbf{SL}}}

\newcommand{\uSp}{{\mathbf{Sp}}}

\newcommand{\uSU}{{\mathbf{SU}}}

\newcommand{\nGm}[1]{{\mathbf{G}}_{\mathbf{m},{#1}}}

\newcommand{\dirlim}{\underrightarrow{\lim}\,}
\newcommand{\invlim}{\underleftarrow{\lim}\,}

\newcommand{\units}[1]{{#1^\times}}

